\documentclass[11pt]{article}

\usepackage{csquotes} 
\usepackage{amsmath}%
\usepackage{amsfonts}%
\usepackage{amssymb}%
\usepackage{graphicx}
\usepackage{mathtools}
\usepackage{afterpage} 

\usepackage{mathrsfs}

\usepackage{subcaption}

\usepackage{xcolor}

\usepackage{amsthm} 
\usepackage[colorlinks]{hyperref} 
\AtBeginDocument{\hypersetup{
    citecolor=magenta,
    colorlinks=true,
    linkcolor=blue,
    filecolor=magenta,      
    urlcolor=cyan,    
    }}
\usepackage{enumerate}
\usepackage[all,cmtip]{xy}
\usepackage{doi}

\DeclareMathOperator{\id}{id}

\DeclareMathOperator{\sll}{sl}
\DeclareMathOperator{\Ad}{Ad}
\DeclareMathOperator{\ad}{ad}

\DeclareMathOperator{\tr}{tr}
\DeclareMathOperator{\diag}{diag}
\DeclareMathOperator{\GL}{GL}
\DeclareMathOperator{\SL}{SL}
\DeclareMathOperator{\gl}{gl}
\DeclareMathOperator{\rank}{rank}

\newcommand{\dif}{\left.\frac{d}{dt}\right|_{t=0}}

\newcommand{\gc}{\mathcal{GC}}

\newcommand{\bg}{\mathbf{\Gamma}}
\newcommand{\std}{\text{std}}

\newcommand{\op}{\mathrm{op}}
\newcommand{\togamma}{{\mathring{\phantom{\gamma}}\kern-6.3pt\tilde{\gamma}}}
\newcommand{\togammas}{\togamma^*}
\newcommand{\const}{\mathrm{const}}
\newcommand{\ogamma}{\mathring{\gamma}}
\newcommand{\ogammas}{\mathring{\gamma}^*}
\newcommand{\sgamma}{\check{\gamma}}
\newcommand{\sgammas}{\check{\gamma}^*}
\newcommand{\onu}{\mathring{\nu}}
\newcommand{\onus}{\mathring{\nu}^*}

\newcommand{\ggo}{\left[ \frac{\gamma^*}{1-\gamma^*}\otimes 1\right]}
\newcommand{\ogg}{\left[1\otimes \frac{\gamma^*}{1-\gamma^*}\right]}
\newcommand{\goo}{\left[ \frac{1}{1-\gamma^*}\otimes 1\right]}
\newcommand{\oog}{\left[1\otimes \frac{1}{1-\gamma^*}\right]}
\newcommand{\adir}{\Ad_{\rho(U)^{-1}}}

\newcommand{\sinv}[1]{\,\hspace*{-3pt}\left. #1 \hspace*{-4pt}\right.^{-1}  }
\newcommand{\rum}{\sinv{U_-}}

\newcommand{\dqdq}{[d_U^R\mathcal{Q}\otimes d_U^R\mathcal{Q}]}
\newtheorem*{theorem*}{Theorem}
\newtheorem*{proposition*}{Proposition}
\theoremstyle{definition}\newtheorem{definition}{Definition}[section]
\theoremstyle{plain}
\newtheorem{theorem}[definition]{Theorem}
\newtheorem{lemma}[definition]{Lemma}
\newtheorem{proposition}[definition]{Proposition}
\theoremstyle{definition}\newtheorem{remark}[definition]{Remark}
\newtheorem{corollary}[definition]{Corollary}
\numberwithin{equation}{section}

\begin{document}
\setcounter{tocdepth}{2} 

\SetBlockThreshold{1}
\title{Generalized cluster structures related to Poisson duals of $\mathrm{SL}_n$ }
\date{}
\author{Michael Gekhtman and Dmitriy Voloshyn}
\maketitle

\makeatletter
\def\blfootnote{\xdef\@thefnmark{}\@footnotetext}
\makeatother

\begin{abstract}
We study Poisson varieties $(\SL_n,\pi_{\bar{\bg}}^{\dagger})$ parameterized by Belavin--Drinfeld quadruples $\bar{\bg}:=(\bg,r_0)$ of type $A_{n-1}$ along with generalized cluster structures $\gc^{\dagger}(\bg)$ in $\mathbb{C}[\SL_n]$ compatible with $\pi_{\bar{\bg}}^{\dagger}$. The Poisson structure $\pi_{\bar{\bg}}^{\dagger}$ is a pushforward of the Poisson structure $\pi_{\bar{\bg}}^*$ of the Poisson dual $\SL_n^*$ of $(\SL_n,\pi_{\bar{\bg}})$. We prove that the generalized upper cluster algebra of $\gc^{\dagger}(\bg)$ is naturally isomorphic to $\mathbb{C}[\SL_n]$. Moreover, for any connected reductive complex group $G$ and a BD quadruple~$(\bg,r_0)$, we produce a Poisson birational map $\mathcal{Q}:(G,\pi_{(\bg_{\std},r_0)}^{\dagger})\dashrightarrow(G,\pi_{(\bg,r_0)}^{\dagger})$, and when $G \in \{\SL_n,\GL_n\}$, we show that $\mathcal{Q}$ is a birational quasi-isomorphism between $\gc^\dagger(\bg_{\std})$ and $\gc^\dagger(\bg)$. Lastly, for any pair of BD triples $\tilde{\bg} \prec \bg$ of type $A_{n-1}$ comparable in the natural order, we use the map $\mathcal{Q}$ to construct a birational quasi-isomorphism between $\gc^{\dagger}(\tilde{\bg})$ and $\gc^{\dagger}(\bg)$.
\end{abstract}


\tableofcontents

\newpage
\section{Introduction}
\subsection{Overview}\label{s:overview}
Cluster algebras were introduced by S. Fomin and A. Zelevinsky in~\cite{fathers} as an algebraic framework for studying dual canonical bases and total positivity. Fomin and Zelevinsky suggested that many interesting varieties in Lie theory carry the structure of an (upper) cluster algebra. Known examples include double Bruhat cells~\cite{upper_bounds}, Grassmannians~\cite{scott}, Bott--Samelson cells~\cite{bott_shen}, simply connected simple complex algebraic groups~\cite{ios-aequ-glocal,oya-note,qy-quant-bruhat}, and open Richardson varieties~\cite{open_rich}.

M. Gekhtman, M. Shapiro, and A. Vainshtein observed in~\cite{roots} that the problem of finding cluster structures on rational varieties can be approached via Poisson geometry. In earlier works~\cite{roots,conj,grass_poiss}, they showed that the cluster structures on double Bruhat cells, Grassmannians, and simply connected simple complex algebraic groups are compatible with certain Poisson brackets. For the groups, they conjectured in~\cite{conj} that there exist multiple cluster structures compatible with Poisson brackets from the Belavin--Drinfeld class. We refer to this conjecture as the \emph{GSV conjecture}.

Before we state the GSV conjecture in precise terms, let us briefly review relevant concepts from the theory of Poisson--Lie groups and cluster algebras (a more detailed account can be found in Section~\ref{s:background}). Given a set of simple roots $\Pi$ of $G$, a \emph{Belavin--Drinfeld triple} (or, for short, a \emph{BD triple}) is a triple $\bg:=(\Gamma_1,\Gamma_2,\gamma)$ where $\Gamma_1,\Gamma_2\subset \Pi$ and $\gamma:\Gamma_1 \rightarrow \Gamma_2$ is a nilpotent isometry. Let $\mathfrak{h}$ be the Cartan subalgebra of the Lie algebra $\mathfrak{g}$ of $G$ corresponding to $\Pi$. As shown by Belavin and Drinfeld in~\cite{bd,bd2}, with an additional continuous parameter $r_0 \in \mathfrak{h}\otimes \mathfrak{h}$, the \emph{BD quadruples} $\bar{\bg}:=(\Gamma_1,\Gamma_2,\gamma,r_0)$  parameterize the moduli space of factorizable quasi-triangular Poisson structures $\pi_{\bar{\bg}}$ on $G$. The Poisson structure $\pi_{\bar{\bg}}$ induces a Lie algebra structure on the dual space $\mathfrak{g}^*$, and upon integration, one obtains a \textit{dual Poisson--Lie group} $(G^*,\pi_{\bar{\bg}}^*)$. It is possible to choose a Poisson bivector $\pi_{\bar{\bg}}^{D}$ on $D(G):=G\times G$ and a dual group $(G^*,\pi_{\bar{\bg}}^*)$ in such a way that 1) $G$ is the diagonal subgroup of $D(G)$ and $G^*$ is an immersed Lie subgroup of $D(G)$; 2) $\pi_{\bar{\bg}}^D|_{G} = \pi_{\bar{\bg}}$ and $\pi^{D}_{\bar{\bg}}|_{G^*} = \pi^{*}_{\bar{\bg}}$. Consider the map
\begin{equation}\label{eq:themap}
D(G)\ni (g,h)\ \ \longmapsto\ \ g^{-1}h \in G,
\end{equation}
and define the Poisson bivector $\pi_{\bar{\bg}}^{\dagger}$ as the pushforward of $\pi^D_{\bar{\bg}}$ under~\eqref{eq:themap}. The image of $G^*$ under~\eqref{eq:themap} is open in $G$, so the Poisson manifold $(G,\pi^{\dagger}_{\bar{\bg}})$ locally models $(G^*,\pi_{\bar{\bg}}^*)$. The Poisson bivectors $\pi_{\bar{\bg}}$ and $\pi_{\bar{\bg}}^D$ can be generalized to Poisson-homogeneous bivectors
$\pi_{(\bar{\bg}^r,\bar{\bg}^c)}$ and $\pi_{(\bar{\bg}^r,\bar{\bg}^c)}^D$ that are parameterized by pairs $(\bar{\bg}^r,\bar{\bg}^c)$ of Belavin--Drinfeld quadruples (see~\cite{plethora, multdouble}). When $\bar{\bg}^r = \bar{\bg}^c = \bar{\bg}$, one recovers $\pi_{\bar{\bg}}$ and $\pi_{\bar{\bg}}^D$.

Now, let us turn to the theory of cluster algebras. In the context of the GSV conjecture, one starts with an \emph{initial extended cluster} $(x_1,x_2,\ldots,x_{N+M})$, which is a collection of rational functions satisfying $\mathbb{C}(G) = \mathbb{C}(x_1,x_2,\ldots,x_{N+M})$ where $N+M = \dim G$. The variables $x_1,\ldots,x_{N}$ are called \emph{cluster variables}, and $x_{N+1},\ldots,x_{N+M}$ are the \emph{frozen variables}. Given a skew-symmetrizable $N \times (N+M)$ matrix $B$ with integer entries, for each $i \in [1,N]$ one can obtain a new extended cluster from the initial one by replacing the cluster variable $x_i$ with a new cluster variable $x_i^{\prime}$ via the equation
\begin{equation}\label{eq:iordmut}
x_i x_i^{\prime} = \prod_{j:\,b_{ij}>0} x_j^{b_{ij}} + \prod_{j:\,b_{ij} < 0} x_{j}^{-b_{ij}}.
\end{equation}
The above transformation is called a \emph{mutation.} For each new extended cluster, the matrix $B$ also changes according to a specific rule. Iterating the above process, one obtains new extended clusters, and the collection $\mathcal{C}$ of those is called a \emph{cluster structure} on $G$. The \emph{cluster algebra} $\mathcal{A}_{\mathbb{C}}$ is defined as the polynomial algebra over $\mathbb{C}$ generated by all cluster and frozen variables. With each extended cluster $\mathbf{x}$, we associate a ring of Laurent polynomials $\mathcal{L}_{\mathbb{C}}(\mathbf{x}):=\mathbb{C}[x_1^{\pm 1},\ldots,x_{N}^{\pm 1},x_{N+1},\ldots,x_{N+M}]$, and define the upper cluster algebra $\bar{\mathcal{A}}_{\mathbb{C}}$ as the intersection of all $\mathcal{L}_{\mathbb{C}}(\mathbf{x})$ for $\mathbf{x}\in\mathcal{C}$ (possibly with an additional localization at some frozen variables). The celebrated Laurent phenomenon states that $\mathcal{A}_{\mathbb{C}} \subseteq \bar{\mathcal{A}}_{\mathbb{C}}$. We say that $\mathcal{C}$ is \emph{regular} if $\mathcal{A}_{\mathbb{C}} \subseteq \mathbb{C}[G]$, and $\mathcal{C}$ is \emph{complete} if $\bar{\mathcal{A}}_{\mathbb{C}} = \mathbb{C}[G]$ (in this case, we also say that $\bar{\mathcal{A}}_{\mathbb{C}}$ is \emph{naturally isomorphic} to $\mathbb{C}[G]$). An extended cluster $(x_1,\ldots,x_{N+M})$ is called \emph{log-canonical} if $\{x_i,x_j\} = \omega_{ij} x_i x_j$, where $\{\cdot,\cdot\}$ denotes the Poisson bracket corresponding to the chosen Poisson bivector, and $\omega_{ij} \in \mathbb{C}$. If every extended cluster is log-canonical, the cluster algebra (or structure) is said to be \emph{compatible} with the Poisson structure. A generalized cluster structure $\gc$ is defined similarly except~\eqref{eq:iordmut} may contain more than two monomials on the right-hand side along with some coefficients.

\paragraph{The GSV conjecture.} In its current form, the conjecture is stated as follows:
\blockquote{
\textit{Let $G$ be a simply connected simple complex algebraic group and $(\bar{\bg}^r,\bar{\bg}^c)$ be a pair of Belavin--Drinfeld quadruples defined relative to a root system of $G$. Then for each of the Poisson varieties $(G,\pi_{(\bar{\bg}^r,\bar{\bg}^c)})$, $(D(G),\pi_{(\bar{\bg}^r,\bar{\bg}^c)}^D)$ and\footnote{The reader might notice that the Poisson structure $\pi_{\bar{\bg}^c}^\dagger$ depends only on $\bar{\bg}^c$ and not on the pair $(\bar{\bg}^r,\bar{\bg}^c)$. The reason for that is, the pushforward of $\pi_{(\bar{\bg}^r,\bar{\bg}^c)}^{D}$ under the map $(X,Y) \mapsto X^{-1}Y$ "forgets" the BD quadruple $\bar{\bg}^r$.}  $(G,\pi_{\bar{\bg}^c}^{\dagger})$, there exists a complete and compatible generalized cluster structure.}
}
Although the Poisson bivectors in each case depend on the parameters $r_0^r,r_0^c \in \mathfrak{h}\otimes\mathfrak{h}$, the corresponding generalized cluster structures depend only on the pair of BD triples $(\bg^r,\bg^c)$.
Let us call a BD triple $\bg$ \emph{trivial} if $\bg = \bg_{\std} := (\Gamma_1,\Gamma_2,\gamma)$ with $\Gamma_1=\Gamma_2 = \emptyset$, and a BD quadruple $\bar{\bg}$ \emph{trivial} if $\bar{\bg} = (\bg_{\std},\frac{1}{2}\Omega_0)$, where $\Omega_0 \in \mathfrak{h}\otimes \mathfrak{h}$ is the $2$-tensor corresponding to the chosen form on $\mathfrak{h}$. A pair of BD triples $(\bg^r,\bg^c)$ (resp. BD quadruples $(\bar{\bg}^r,\bar{\bg}^c)$) is \emph{trivial} if both $\bg^r$ and $\bg^c$ (resp. $\bar{\bg}^r$ and $\bar{\bg}^c$) are trivial. A pair $(\bg^r,\bg^c)$ is called \emph{aperiodic} if the map $\gamma_r (-w_0) \gamma_c^{-1}(-w_0^{-1})$ is nilpotent, where $w_0$ is the longest Weyl group element. In type $A_{n-1}$, identifying the set of simple roots $\Pi$ with the interval of integers $[1,n-1]$, the pair $(\bg^r,\bg^c)$ is called \emph{oriented} if both $\gamma_r$ and $\gamma_c$ are increasing maps. The current state of the GSV conjecture is as follows.
\begin{enumerate}[1)]
\item The conjecture was verified for $(G,\pi_{\bar{\bg}_{\std}})$, for arbitrary $G$. Specifically, the cluster structure was constructed in~\cite{fz-dbl-bruhat} (the same as for the open double Bruhat cell $G^{w_0,w_0}$), and its compatibility with $\pi_{\bar{\bg}_{\std}}$ was verified in~\cite{conj}. The completeness was proved in~\cite{qy-quant-bruhat}, and in types other than $F_4$, it was strengthened to the equality $\mathbb{C}[G] = \mathcal{A}_{\mathbb{C}}(\mathcal{C})$ in~\cite{ios-aequ-glocal,oya-note}. 
\item For $(\SL_n(\mathbb{C}),\pi_{(\bar{\bg}^r,\bar{\bg}^c)})$, the conjecture was verified for aperiodic oriented BD pairs in~\cite{plethora}, and later improved to aperiodic BD pairs (not necessarily oriented) in~\cite{rho}.
\item For $(D(\SL_n(\mathbb{C})),\pi^{D}_{(\bar{\bg}^r,\bar{\bg}^c)})$, the conjecture was verified for the trivial BD pair in~\cite{double}, and later for aperiodic oriented BD pairs in~\cite{multdouble}.
\item For $(\SL_n(\mathbb{C}),\pi^{\dagger}_{\bar{\bg}})$, the conjecture was verified only for the trivial BD quadruple in~\cite{double}. The results of the current paper prove the conjecture for $(\SL_n(\mathbb{C}),\pi^{\dagger}_{\bar{\bg}})$ associated with arbitrary $\bar{\bg}$.
\end{enumerate}
All results for $\SL_n(\mathbb{C})$ can be upgraded to $\GL_n(\mathbb{C})$ in a natural way.

\paragraph{Connection to other problems.} We envision a connection between our work and some of the following known problems in mathematical physics: 
\begin{enumerate}[1)]
    \item There is a conjecture due to D. Gaiotto on the existence of cluster structures on  Coulomb branches in $3d$ $\mathcal{N}=4$ supersymmetric gauge theories. The conjecture was verified in the case of quiver gauge theories by A. Shapiro and G. Schrader in~\cite{sasha_gauge}. It would be interesting to see if $(\SL_n,\pi_{\bar{\bg}}^{\dagger})$ for a nontrivial $\bar{\bg}$ can be interpreted as a Coulomb branch (for the trivial BD quadruple, see \cite[Remark 3.9]{double}). 
    \item G. Schrader, A. Shapiro and I. Ip in~\cite{ivan_ip, sasha_group} proposed a quantum $\mathcal{X}$-cluster realization of the quantum group $U_g(\mathfrak{g})$ for simple complex Lie algebras $\mathfrak{g}$. As a byproduct, they derived a factorization into quantum dilogarithms of the quantum $R$-matrix associated with the trivial BD triple. An explicit classification of factorizable quasi-triangular quantum $R$-matrices was given in~\cite{etingofdyn}. We expect that there exists a connection between compatible generalized cluster structures on $(G,\pi_{(\bg,r_0)}^{\dagger})$ and quantum $\mathcal{A}$-cluster realizations of $U_{q}(\mathfrak{g})$.
\end{enumerate}
As we mentioned earlier, $(G,\pi_{\bar{\bg}}^{\dagger})$ serves as a local model for the dual Poisson--Lie group $(G^*,\pi_{\bar{\bg}}^{*})$. If $G$ is of adjoint type, there are other cluster realizations of $(G^*,\pi_{\bar{\bg}_{\std}}^*)$. In~\cite{dual_shen}, based on an earlier work~\cite{gonch-shen} with A. Goncharov, L. Shen proved that $\mathbb{C}[G^*]$ embeds into a quotient of a cluster Poisson algebra with a Weyl group action.

\paragraph{Current challenges in proving the conjecture.} For compatible generalized cluster structures on $(G,\pi_{(\bar{\bg}^r,\bar{\bg}^c)})$, there is a significant distinction between aperiodic and non-aperiodic BD pairs. In the case of $G \in \{ \SL_n,\GL_n\}$ and an aperiodic BD pair, the initial cluster variables are obtained as trailing minors of the so-called $\mathcal{L}$-matrices (see~\cite{plethora,rho}); however, some of these matrices become infinite when the pair is non-aperiodic. Moreover, at least for $(\SL_n,\pi_{(\bar{\bg}^r,\bar{\bg}^c)})$, one has to consider generalized mutations when $(\bar{\bg}^r,\bar{\bg}^c)$ is non-aperiodic (for an example, see~\cite{periodic}). However, in the case of $(\SL_n,\pi_{\bar{\bg}}^{\dagger})$, there is no distinction between aperiodic and non-aperiodic cases, and for any BD triple in our construction, there is always only one generalized mutation in each seed. Nevertheless, the case of $(G,\pi_{\bar{\bg}}^{\dagger})$ for other $G$ remains open due to the difficulty of finding a Lie-theoretic generalization of the so-called  $\varphi$-functions (see equation~\eqref{eq:s_def}), which form a part of the initial cluster. Difficulties in $(D(G),\pi_{(\bar{\bg}^r,\bar{\bg}^c)})$ combine the previous two cases.

\paragraph{A new approach.} Compared to earlier papers~\cite{plethora,exotic,multdouble} dealing with nontrivial BD triples, the proofs in the present work are significantly simpler. Our new approach, first outlined in~\cite{plethora,multdouble}, involves studying properties of certain Poisson birational quasi-isomorphisms, which appear to be the appropriate type of equivalence in the context of the GSV conjecture. These maps (almost) preserve all relevant structures on a given variety; hence, once the results are established for the trivial Belavin–Drinfeld data, the Poisson birational quasi-isomorphisms can be used to extend them to the nontrivial case.
\begin{enumerate}[1)]
\item Previously, the proofs of compatibility between the chosen Poisson bracket and the (generalized) cluster structure consisted in direct and involved computations~\cite{plethora,exotic,multdouble}. In this paper, given any connected reductive complex algebraic group $G$, any BD quadruple $(\bg,r_0)$ and a Poisson bivector $\pi_{(\bg,r_0)}^{\dagger}$ on $G$, we prove that there exists a Poisson birational map $\mathcal{Q}:(G,\pi_{(\bg_{\std},r_0)}^{\dagger}) \dashrightarrow (G,\pi_{(\bg,r_0)}^{\dagger})$. When $G \in \{\SL_n,\GL_n\}$, the map $\mathcal{Q}$ is a birational quasi-isomorphism between the associated generalized cluster structures. We expect that generalized cluster structures for nontrivial BD triples in other Lie types can be constructed by means of the map~$\mathcal{Q}$. The compatibility of $\pi_{(\bg,r_0)}^{\dagger}$ with the corresponding generalized cluster structure on $G$ is a consequence of the Poisson property of $\mathcal{Q}$.
\item To prove completeness, we mainly rely on the results of~\cite{zametka} and an inductive argument on the size of Belavin--Drinfeld triples, which was first suggested in~\cite{plethora}. The proof is reduced to verifying certain properties of the so-called marked variables in the case $|\Gamma_1| = 1$.
\item The existence of a global toric action of maximal rank follows from certain equivariance properties of $\mathcal{Q}$, and the proof is reduced to the verification of some properties of the marked variables in the case $|\Gamma_1| = 1$.
\end{enumerate}
In the case of $(G,\pi_{(\bar{\bg}^r,\bar{\bg}^c)})$, a similar approach was developed in~\cite{rho}. For every pair of BD quadruples $((\bg^r,r_0^r),(\bg^c,r_0^c))$, the authors  constructed a Poisson rational map 
\begin{equation*}
    \mathbf{h} : (G,\pi_{((\bg_{\std},r_0^r),(\bg_{\std},r_0^c))}) \dashrightarrow (G,\pi_{((\bg^r,r_0^r),(\bg^c,r_0^c))}).
\end{equation*} 
If the BD triple $(\bg^r,\bg^c)$ is aperiodic, the map $\mathbf{h}$ is birational, and can therefore be used to construct an initial extended cluster associated with $(\bg^r,\bg^c)$. 

\paragraph{Organization of the paper.} In Section~\ref{s:descr_h}, we describe the initial extended seed in the \emph{$h$-convention} (an alternative construction, referred to as the \emph{$g$-convention}, is given in Section~\ref{s:descr_g}). In Section~\ref{s:results}, we state the main results of the paper. In Section~\ref{s:background}, we present an expanded overview on Poisson geometry, generalized cluster theory, and the theory of birational quasi-isomorphisms. In Section~\ref{s:comp}, for any connected reductive complex algebraic group $G$ and any BD quadruple $(\bg,r_0)$, we construct the birational map $\mathcal{Q}:(G,\pi_{(\bg_{\std},r_0)}^{\dagger}) \dashrightarrow (G,\pi_{(\bg,r_0)}^{\dagger})$ and prove that the map $\mathcal{Q}$ is Poisson. In Section~\ref{s:birat}, we construct various birational quasi-isomorphisms for $G \in \{\SL_n,\GL_n\}$. In Section~\ref{s:complet}, we prove our main Theorems~\ref{thm:maingln}--\ref{thm:mainsln}. In Section~\ref{s:other}, we derive explicit formulas for the initial variables and show that the frozen variables are Poisson. Explicit examples of generalized cluster structures on $(\SL_n,\pi_{(\bg,r_0)}^{\dagger})$ are provided in Appendix~\ref{s:exs}. Supplementary material for this work is available at~\cite{github}.

\paragraph{Acknowledgements.} Michael Gekhtman was supported by the NSF grant DMS-2100785.
Dmitriy Voloshyn was supported by the grant IBS-R003-D1.

\subsection{\texorpdfstring{Construction of $\gc_h^{\dagger}(\bg)$}{Construction of the initial seed in the h-convention}}\label{s:descr_h}
Let $\Pi$ be a fixed set of simple roots of type $A_{n-1}$ identified with the interval $[1,n-1]$. In this subsection, given a BD triple $\bg :=(\Gamma_1,\Gamma_2,\gamma)$ and a BD quadruple $\bar{\bg}:=(\bg,r_0)$, we describe the initial extended cluster of $\gc_h^{\dagger}(\bg)$ for $(\GL_n,\pi_{\bar{\bg}}^{\dagger})$ and $(\SL_n,\pi_{\bar{\bg}}^{\dagger})$ in the \emph{$h$-convention}. For the trivial BD quadruple, the description was provided in~\cite{double}. In Section~\ref{s:descr_g}, we provide an alternative construction referred to as the \emph{$g$-convention}. The initial extended cluster comprises three types of functions: $c$-functions, $\varphi$-functions, and $h$-functions. Only the description of the $h$-functions depends on the choice of the BD triple. The initial quiver is described in Appendix~\ref{s:ainiquivh}, and selected examples are provided in Appendix~\ref{s:exs}.

\paragraph{Notation for $\gc$.} The notation $\gc^{\dagger}(\bg)$ stands for the generalized cluster structure on $(\SL_n,\pi_{\bar{\bg}}^{\dagger})$ or $(\GL_n,\pi_{\bar{\bg}}^{\dagger})$ in the $h$- or $g$-convention. The subscript $h$ or $g$ specifies the convention. When we need to specify the variety, we write $\gc^{\dagger}(\bg,\GL_n)$ or $\gc^{\dagger}(\bg,\SL_n)$. The same convention applies to the (upper) generalized cluster algebra, the ground ring, etc. The notation $(G,\gc^{\dagger}(\bg))$ (with or without the subscript $h$ or $g$) refers to a group $G \in \{\SL_n,\GL_n\}$ equipped with $\gc^{\dagger}(\bg,G)$. A choice of $r_0$ is irrelevant for the construction of the initial extended seed. The ground field is assumed throughout to be $\mathbb{C}$ (so that $\SL_n$ denotes $\SL_n(\mathbb{C})$).

\paragraph{Description of $\varphi$- and $c$-functions.} For an element $U \in \GL_n$, let us set
\begin{equation}\label{eq:big_phi_def_h}
\Phi_{kl}(U):=\begin{bmatrix}(U^0)^{[n-k+1,n]} & U^{[n-l+1,n]} & (U^2)^{\{n\}} & \cdots & (U^{n-k-l+1})^{\{n\}}\end{bmatrix}, \ \ k,l \geq 1, \ k+l \leq n;
\end{equation}
\begin{equation}\label{eq:s_def}
s_{kl}:=\begin{cases}
(-1)^{k(l+1)} \ &n \ \text{is even},\\
(-1)^{(n-1)/2 + k(k-1)/2 + l(l-1)/2} \ & n \ \text{is odd}.
\end{cases}
\end{equation}
Then, the $\varphi$-functions are given by
\begin{equation}\label{eq:phi_h_def}
\varphi_{kl}(U):=s_{kl} \det \Phi_{kl}(U).
\end{equation}
The $\varphi$-functions have the following invariance properties. Given a unipotent lower-triangular matrix $N_-$ and a nondegenerate diagonal matrix $T$,
\begin{equation}\label{eq:phi_invar}\begin{aligned}
\varphi_{kl}(N_- U \sinv{N_-}) &= \varphi_{kl}(U), \\ \varphi_{kl}(TUT^{-1}) &= \chi_{\varphi_{kl}}(T)\varphi_{kl}(U)
\end{aligned}
\end{equation}
where $\chi_{\varphi_{kl}}$ denotes a character. The $c$-functions are uniquely defined via
\begin{equation}\label{eq:c_def}
\det(I+\lambda U) = \sum_{i=0}^{n} \lambda^{i} s_i c_i(U)
\end{equation}
where $s_i := (-1)^{i(n-1)}$, and $I$ is the identity matrix. Note that $c_0(U) = 1$ and $c_n(U) = \det U$. 

\begin{remark}
For $(X,Y) \in \GL_n\times \GL_n$, set $U:= X^{-1}Y$. Then, the definitions of the $\varphi$- and $c$-functions are identical to those given in~\cite{double,multdouble} (up to a factor of $\det X$).
\end{remark}

\paragraph{Description of the $h$-functions.} For any generic $n \times n$ matrix $A$, let us denote by $A_{\oplus}$ and $A_-$ the upper-triangular and  unipotent lower-triangular matrices, respectively, such that $A = A_\oplus A_-$. For a given BD triple $\bg = (\Gamma_1,\Gamma_2,\gamma)$ of type $A_{n-1}$, the map $\gamma^*$ can be extended to a Lie algebra homomorphism $\gamma^*:\mathfrak{n}_-(\Gamma_2)\rightarrow \mathfrak{n}_-(\Gamma_1)$, where $\mathfrak{n}_-(\Gamma_i)$ are the Lie subalgebras of $\sll_n(\mathbb{C})$ of nilpotent lower-triangular matrices generated by $\Gamma_i$, $i \in \{1,2\}$. Then, $\gamma^*$ can be extended by zero to the map $\gamma^* : \mathfrak{n}_- \rightarrow \mathfrak{n}_-$, which remains a Lie algebra homomorphism; therefore, it can be integrated to a group homomorphism $\tilde{\gamma}^* : \mathcal{N}_- \rightarrow \mathcal{N}_-$, where $\mathcal{N}_-$ is the subgroup of $\SL_n$ consisting of unipotent lower-triangular matrices (see Sections~\ref{s:ba_poisson} and \ref{s:bd_map} for details). Define the following sequence of rational maps:
\begin{equation*}
\mathcal{F}_0(U):=U, \ \ \mathcal{F}_k(U):= \tilde{\gamma}^*([\mathcal{F}_{k-1}(U)]_-)\cdot U, \ \ U \in \GL_n, \ k \geq 1.
\end{equation*}
We study properties of $\{\mathcal{F}_{k}\}_{k \geq 0}$ in Section~\ref{s:map_f}. The sequence stabilizes due to the nilpotency of $\gamma$:
\begin{equation*}
\mathcal{F}_{\deg \gamma}(U) = \mathcal{F}_{\deg \gamma + 1}(U) = \mathcal{F}_{\deg \gamma+2} = \ldots
\end{equation*}
(for $\deg \gamma$, see Definition~\ref{d:degnilp}). Set
\begin{equation*}
\mathcal{F}(U):=\mathcal{F}_{\deg \gamma}(U),\ \ U \in \GL_n.
\end{equation*}

\begin{definition}\label{d:g_str}
For $\alpha_0 \in \Pi \setminus \Gamma_2$, set $\alpha_t:=\gamma(\alpha_{t-1})$, $t \geq 1$. The sequence $S^{\gamma}(\alpha_0):=\{\alpha_{t}\}_{t \geq 0}$ is called the \emph{$\gamma$-string} associated to $\alpha_0$.
\end{definition}
Note that $\gamma$-strings partition the set of simple roots $\Pi$. For each $\gamma$-string $S^{\gamma}(\alpha_0) = \{\alpha_0,\alpha_1,\ldots,\alpha_{m}\}$, for each $i \in [0,m]$ and $j \in [\alpha_i+1,n]$, set
\begin{equation}\label{eq:h_fun}
h_{\alpha_i+1,j}(U) := (-1)^{\varepsilon_{\alpha_i+1,j}}\det [\mathcal{F}(U)]^{[j,n]}_{[\alpha_i+1,n-j+\alpha_i+1]} \prod_{t \geq i+1}^m \det [\mathcal{F}(U)]^{[\alpha_t+1,n]}_{[\alpha_t+1,n]}
\end{equation}
where $\varepsilon_{ij}$ is defined as
\begin{equation}\label{eq:h_sign}
\varepsilon_{ij} := (j-i)(n-i), \ \ 1 \leq i \leq j \leq n.
\end{equation}
We refer to the functions $h_{ij}$, $2 \leq i \leq j \leq n$, together with $h_{11}(U):=\det U$, as the \emph{$h$-functions}. Note that the flag minors of $\mathcal{F}(U)$ can be expressed as
\begin{equation}\label{eq:f_flags}
(-1)^{\varepsilon_{ij}} \det [\mathcal{F}(U)]_{[i,n-j+i]}^{[j,n]} = \begin{cases}
\frac{h_{ij}(U)}{h_{\mu\mu}(U)} \ &\text{if} \ i-1\in \Gamma_1,\\
h_{ij}(U) \ &\text{otherwise}
\end{cases}
\end{equation}
where $\mu := \gamma(i-1)+1$ for $i-1 \in \Gamma_1$. When $\bg = \bg_{\std}$, we obtain that
\begin{equation*}
h_{ij}(U) = (-1)^{\varepsilon_{ij}} \det U^{[j,n]}_{[i,n-j+i]}, \ \ 2 \leq i \leq j \leq n,
\end{equation*}
which coincides with the earlier description given in~\cite{double}. We will show in Section~\ref{s:p_explicit} that the $h$-functions are regular as elements of $\mathbb{C}[\GL_n]$, and we will also derive an explicit formula in terms of minors of $U\in \GL_n$. Moreover, in Section~\ref{s:complet_final} we will show that the $h$-functions are irreducible and pairwise coprime with their mutations.

\paragraph{Frozen variables.} In the case of $\gc_h^{\dagger}(\bg,\GL_n)$, the frozen variables are given by the set
\begin{equation*}
\{c_1,c_2,\ldots,c_{n-1}\} \cup \{h_{i+1,i+1} \ | \ i \in \Pi \setminus \Gamma_2\} \cup \{h_{11}\}.
\end{equation*}
In the case of $\gc_h^{\dagger}(\bg,\SL_n)$, $h_{11}(U) = 1$, so this variable is absent. We will show in Section~\ref{s:p_frozen} that the zero loci of the frozen variables foliate into unions of symplectic leaves of the ambient Poisson varieties $(\GL_n,\pi_{\bar{\bg}}^{\dagger})$ or $(\SL_n,\pi_{\bar{\bg}}^{\dagger})$. In the same section, we will also show that the frozen $h$-variables do not vanish on the image of $\SL_n^*$ under the map~\eqref{eq:themap}.

\paragraph{Initial extended cluster.} The initial extended cluster $\Psi_0$ of $\gc_h^{\dagger}(\bg,\GL_n)$ is given by the set
\begin{equation*}
    \{h_{ij} \ | \ 2 \leq i \leq j \leq n\} \cup \{\varphi_{kl} \ | \ k,l \geq 1, \ k+l\leq n\} \cup \{c_1,\ldots,c_{n-1}\}\cup\{h_{11}\}.
\end{equation*}
The initial extended cluster of $\gc_h^{\dagger}(\bg,\SL_n)$ is obtained from $\Psi_0$ by removing $h_{11}$.

\paragraph{Ground rings.} The ground rings are defined as follows.
\begin{align*}
\hat{\mathbb{A}}_{\mathbb{C}}(\bg,\GL_n):=&\mathbb{C}[h_{11}^{\pm 1}, c_1,\ldots,c_{n-1}, h_{i+1,i+1} \ | \ i \in \Pi\setminus \Gamma_2];\\
\hat{\mathbb{A}}_{\mathbb{C}}(\bg,\SL_n):=&\mathbb{C}[c_1,\ldots,c_{n-1}, h_{i+1,i+1} \ | \ i \in \Pi\setminus \Gamma_2].
\end{align*}

\paragraph{A generalized cluster mutation.} In the initial extended cluster, only the variable $\varphi_{11}$ is equipped with a nontrivial \emph{string}, which is given by $(1,c_1,\ldots,c_{n-1},1)$ (see Definition~\ref{def:ext_seed}). The generalized mutation relation for $\varphi_{11}$ is given by
\begin{equation}\label{eq:p11mut}
\varphi_{11} \varphi_{11}^{\prime} = \sum_{r=0}^n c_r \varphi_{21}^{r} \varphi_{12}^{n-r}.
\end{equation}
Other mutations of the initial extended cluster follow the usual pattern from the theory of cluster algebras of geometric type.

\paragraph{Toric action.} Let $\mathfrak{h}$ be the Cartan subalgebra of $\sll_n(\mathbb{C})$. Set
\begin{equation}\label{eq:toralg}
\mathfrak{h}_{\bg}:=\{h \in \mathfrak{h} \ | \ \alpha(h) = \beta(h) \ \text{if} \ \gamma^j(\alpha) = \beta \ \text{for some} \ j\}.
\end{equation}
The dimension of $\mathfrak{h}_{\bg}$ is given by
\begin{equation}\label{eq:toralgdim}
k_{\bg}:= \dim {\mathfrak{h}_{\bg}} = |\Pi \setminus \Gamma_2|.
\end{equation}
Set $\mathcal{H}_{\bg}$ to be the connected subgroup of $\SL_n$ that corresponds to $\mathfrak{h}_{\bg}$. We let $\mathcal{H}_{\bg}$ to act upon $\GL_n$ and $\SL_n$ via
\begin{equation}\label{eq:toract}
(T,U) \mapsto TU T^{-1}, \ \ T \in \mathcal{H}_{\bg}, \ U \in \GL_n.
\end{equation}
In the case of $\GL_n$, we consider the action of $\mathcal{H}_{\bg} \times \mathbb{C}^*$, where $\mathbb{C}^*$ acts by scalar multiplication:
\begin{equation*}
(a,U) \mapsto a\cdot U, \ \ a \in \mathbb{C}^*, \ U \in \GL_n.
\end{equation*}
In Section~\ref{s:toric}, we will show that the action of $\mathcal{H}_{\bg}$ on $\SL_n$ and the action of $\mathcal{H}_{\bg}\times \mathbb{C}^*$ on $\GL_n$ induce global toric actions upon the corresponding generalized cluster structures.

\paragraph{The Poisson bracket.} A general formula for the Poisson bracket is given in~\eqref{eq:bdag}. In the cases of $\GL_n$ and $\SL_n$, the Poisson bracket corresponding to $\pi_{\bar{\bg}}^{\dagger}$ is given as follows. First, for a smooth function $f \in C^{\infty}(\GL_n)$, set 
\begin{equation}\label{eq:glngrad}
\nabla_U f := \left( \frac{\partial f}{\partial u_{ij}} \right)_{j,i=1}^n, \ \ U \in \GL_n;
\end{equation}
second, equip $\gl_n(\mathbb{C})$ and $\sll_n(\mathbb{C})$ with the trace form
\[
\langle A,B \rangle := \tr(AB), \ \ A, B \in \gl_n(\mathbb{C}).
\]
Let $R:=R_+:\gl_n(\mathbb{C}) \rightarrow \gl_n(\mathbb{C})$ be the $R$-matrix that corresponds to $\bar{\bg}$ (see formula~\eqref{eq:rplusm} below). The Poisson bracket on $(\GL_n,\pi_{\bar{\bg}}^{\dagger})$ is given by
\begin{equation}\label{eq:brackgln}
    \{f,g\}_{\bar{\bg}}^{\dagger} = \langle R([\nabla_U f,U]),[\nabla_U g, U]\rangle -\langle [\nabla_Uf,U],\nabla_U g \cdot U\rangle.
\end{equation}
In the case of $(\SL_n,\pi_{\bar{\bg}}^{\dagger})$, the formula is exactly the same, except one chooses $R$ as a linear endomorphism of $\sll_n(\mathbb{C})$. 

\paragraph{Casimirs.} For any\footnote{In our work on the Drinfeld double~\cite{multdouble}, the functions $c_1,\ldots,c_{n-1}$ defined on $(D(\GL_n),\pi_{\bg}^D)$ are Casimirs if and only if $R_0(I) 
 = (1/2)I$. However, since $[\nabla_U c_i, U] = 0$ for any $i$, the functions $c_1,\ldots,c_{n-1}$ are Casimirs for any choice of $R_0$ for $(\GL_n,\pi_{\bar{\bg}}^{\dagger})$.} choice of $r_0$ for $(\GL_n,\pi_{\bar{\bg}}^{\dagger})$ and $(\SL_n,\pi_{\bar{\bg}}^{\dagger})$, the functions $c_1,c_2,\ldots,c_{n-1}$, $h_{11}$ are Casimirs of $\pi_{\bar{\bg}}^{\dagger}$.

 \begin{remark}
 Let $(\bg^r,\bg^c)$ be an aperiodic oriented BD pair and let $\gc^D(\bg^r,\bg^c)$ be the generalized cluster structure on $D(\SL_n)$ (or on $D(\GL_n)$) compatible with $\pi_{(\bar{\bg}^r,\bar{\bg}^c)}^D$, which was constructed in our previous work~\cite{multdouble}. The initial seed of $\gc^{\dagger}_h(\bg_{\std})$ was identified as a subseed of some extended seed in $\gc^D(\bg_{\std},\bg_{\std})$. It is thus natural to ask whether the initial extended seed of $\gc_h^{\dagger}(\bg)$ can be obtained as a subseed of an extended seed from $\gc^D(\bg,\bg)$. We verified for $n=3$ and $n=4$, for all BD triples $\bg$, that this is indeed the case. The corresponding mutation sequences are listed in the supplementary note~\cite{github}.
 \end{remark}

\subsection{Main results}\label{s:results}
Let $G$ be a connected reductive complex algebraic group, $\mathfrak{g}$ be its Lie algebra with a fixed decomposition $\mathfrak{g}=\mathfrak{n}_-\oplus\mathfrak{h}\oplus\mathfrak{n}_+$ and a set of simple roots $\Pi$. Let $\mathcal{B}_+$ and $\mathcal{N}_-$ be the connected subgroups of $G$ corresponding to $\mathfrak{n}_+\oplus\mathfrak{h}$ and $\mathfrak{n}_-$. Decompose a generic element $U \in G$ as $U = U_{\oplus}U_-$, $U_\oplus \in \mathcal{B}_+$, $U_-\in \mathcal{N}_-$. Let $\bg:=(\Gamma_1,\Gamma_2,\gamma)$ be a BD triple and $\bar{\bg} :=(\bg,r_0)$ be a BD quadruple.  Define a rational map $\mathcal{Q}:G\dashrightarrow G$ as
\begin{equation}\label{eq:qdef0}
    \mathcal{Q}(U) := \rho(U)^{-1}U \rho(U), \ \ \rho(U) := \prod_{i=1}^{\rightarrow} (\tilde{\gamma}^*)^i(U_-), \ U \in G.
\end{equation}

\begin{proposition}\label{p:qpoiss}
The map $\mathcal{Q}:(G,\pi_{(\bg_{\std},r_0)}^{\dagger})\dashrightarrow (G,\pi_{(\bg,r_0)}^{\dagger})$ is Poisson.
\end{proposition}
Section~\ref{s:comp} contains the proof of Proposition~\ref{p:qpoiss}. In Section~\ref{s:qfc}, we restrict to $G\in\{\SL_n,\GL_n\}$ and show that $\mathcal{Q}:(G,\gc_h^{\dagger}(\bg_{\std}))\dashrightarrow(G,\gc_h^{\dagger}(\bg))$ is a birational quasi-isomorphism. In Section~\ref{s:g_birat}, we construct the map $\mathcal{Q}^{\op}:(G,\gc_g^{\dagger}(\bg_{\std})) \dashrightarrow (G,\gc_g^{\dagger}(\bg))$, which is a birational quasi-isomorphism in the $g$-convention. In Section~\ref{s:q_transp}, we prove a version of Proposition~\ref{p:qpoiss} for $\mathcal{Q}^{\op}$.

\begin{theorem}\label{thm:maingln}
    The generalized cluster structure $\gc^{\dagger}(\bg,\GL_n(\mathbb{C}))$ has the following properties:
    \begin{enumerate}[1)]
        \item The number of frozen variables is $k_{\bg}+n$, and the exchange matrix has full rank;\label{i:gln_rk}
        \item It is complete, and the initial extended cluster consists of irreducible elements of $\mathbb{C}[\GL_n]$ that are pairwise coprime with their mutations;\label{i:gln_ni}
        \item The global toric action of $(\mathbb{C}^*)^{k_{\bg}+1}$ upon $\gc^{\dagger}(\bg,\GL_n(\mathbb{C}))$ is induced by the action of $\mathcal{H}_{\bg} \times \mathbb{C}^*$ on $\GL_n(\mathbb{C})$;\label{i:gln_tor}
        \item It is compatible with the Poisson bivector $\pi^{\dagger}_{\bar{\bg}}$.\label{i:gln_comp}
    \end{enumerate}
\end{theorem}

A parallel statement is valid in the case of $\SL_n$:

\begin{theorem}\label{thm:mainsln}
    The generalized cluster structure $\gc^{\dagger}(\bg,\SL_n(\mathbb{C}))$ has the following properties:
    \begin{enumerate}[1)]
        \item The number of frozen variables is $k_{\bg}+(n-1)$, and the exchange matrix has full rank;
        \item It is complete, and the initial extended cluster consists of irreducible elements of $\mathbb{C}[\SL_n]$ that are pairwise coprime with their mutations;\label{i:sln_ni}
        \item The global toric action of $(\mathbb{C}^*)^{k_{\bg}}$ upon $\gc^{\dagger}(\bg,\SL_n(\mathbb{C}))$ is induced by the action of $\mathcal{H}_{\bg}$ on $\SL_n(\mathbb{C})$;\label{i:sln_tor}
        \item It is compatible with the Poisson bivector $\pi^{\dagger}_{\bar{\bg}}$.
    \end{enumerate}
\end{theorem}
Both theorems are proved in Section~\ref{s:complet} for $\gc_h^{\dagger}(\bg)$ and in Section~\ref{s:thm_pf_g} for $\gc_g^{\dagger}(\bg)$.

Given two BD triples $\bg = (\Gamma_1,\Gamma_2,\gamma)$ and $\tilde{\bg} = (\tilde{\Gamma}_1,\tilde{\Gamma}_2,\theta)$ of the same Lie type, let us write $\tilde{\bg} \prec \bg$ if $\tilde{\Gamma}_1 \subsetneq \Gamma_1$, $\tilde{\Gamma}_2 \subsetneq \Gamma_2$ and $\gamma|_{\tilde{\Gamma}_1} = \theta$.  If $\tilde{\bg}\prec \bg$ and $(\bg,r_0)$ is a BD quadruple for some $r_0$, then $(\tilde{\bg},r_0)$ is also a BD quadruple.

\begin{proposition}\label{p:birat}
    Let $\tilde{\bg}\prec \bg$ be two BD triples of type $A_{n-1}$, and let $G \in \{\SL_n(\mathbb{C}),\GL_n(\mathbb{C})\}$. Then there exists a birational quasi-isomorphism $\mathcal{G}:(G,\gc^{\dagger}(\tilde{\bg}))\dashrightarrow (G,\gc^{\dagger}(\bg))$. Moreover, $\mathcal{G}: (G,\pi_{(\tilde{\bg},r_0)}^{\dagger}) \dashrightarrow (G,\pi_{(\bg,r_0)}^{\dagger})$ is Poisson.
\end{proposition}
If $\tilde{\bg} = \bg_{\std}$, then $\mathcal{G}=\mathcal{Q}$ in the $h$-convention and $\mathcal{G} = \mathcal{Q}^{\op}$ in the $g$-convention. In Sections~\ref{s:genbirat}-\ref{s:bsimpl}, we provide a more explicit formula for $\mathcal{G}$ and prove the proposition for the $h$-convention. In the $g$-convention, the proof and the formulas are provided in Section~\ref{s:g_birat}.

Let $\mathcal{P}_+(\Gamma_2)$ and $\mathcal{P}_-(\Gamma_1)$ be the parabolic subgroups of $\GL_n$ generated by the BD data $(\Gamma_1,\Gamma_2,\gamma)$. Define the subgroup $\mathcal{D} \subset \mathcal{P}_+(\Gamma_1) \times \mathcal{P}_-(\Gamma_2)$ via
\begin{equation}\label{eq:dgroup_semi}
\mathcal{D} := \{(g^\prime,g) \in \mathcal{P}_+(\Gamma_1) \times \mathcal{P}_-(\Gamma_2) \ | \ \togamma(\mathring{\Pi}_{\Gamma_1}(g^\prime)) = \mathring{\Pi}_{\Gamma_2}(g)\}
\end{equation}
where $\mathring{\Pi}_{\Gamma_1} : \mathcal{P}_+(\Gamma_1) \rightarrow \GL_n(\Gamma_1)$ and $\mathring{\Pi}_{\Gamma_2} : \mathcal{P}_-(\Gamma_2) \rightarrow \GL_n(\Gamma_2)$ are the group projections upon the Levi factors of $\mathcal{P}_+(\Gamma_1)$ and $\mathcal{P}_-(\Gamma_2)$, and the map $\togamma$ is a certain group lift of $\gamma$ (see Section~\ref{s:bd_map} for details). Define an action of $\mathcal{D}$ upon $\GL_n$ via
\begin{equation}\label{eq:daction}
((g^\prime, g),U) \mapsto g^\prime U g^{-1}, \ \ (g^\prime,g) \in \mathcal{D} ,\ U \in \GL_n.
\end{equation}
\begin{proposition}\label{p:frozvar}
Let $\psi$ be a frozen variable in $\gc^{\dagger}(\bg,G)$,  $G \in \{\SL_n(\mathbb{C}),\GL_n(\mathbb{C})\}$. Then the following statements hold:
\begin{enumerate}[1)]
\item The variable $\psi$ is semi-invariant with respect to the action~\eqref{eq:daction};
\item The variable $\psi$ is log-canonical with any $u_{ij}$;
\item The zero locus of $\psi$ foliates into a union of symplectic leaves of $(G,\pi_{\bar{\bg}}^{\dagger})$.
\end{enumerate}
\end{proposition}
In the $h$-convention, we prove Proposition~\ref{p:frozvar} in Section~\ref{s:p_frozen}; in the $g$-convention, we prove it in Section~\ref{s:descr_gg}.
\section{Background}\label{s:background}
In this section, we provide a necessary background on Poisson geometry, generalized cluster structures and birational quasi-isomorphisms. We will also use some well-known determinantal identities, which are listed in Appendix~\ref{s:aidentminors}.

\subsection{Poisson geometry
}\label{s:ba_poisson}
In this section, we briefly recall relevant concepts from Poisson geometry. A more detailed account can be found in \cite{chari}, \cite{etingof}, and \cite{rs}.

\paragraph{Poisson--Lie groups.} A \emph{Poisson bracket $\{\cdot, \cdot\}$} on a commutative algebra is a Lie bracket that satisfies the Leibniz rule in each slot. Given a manifold $M$, a \emph{Poisson bivector} on $M$ is a section $\pi \in \Gamma(M,\bigwedge^2 TM)$ such that $\{f,g\}:=\pi(df\wedge dg)$ is a Poisson bracket on the space of smooth functions on $M$. The pair $(M,\pi)$ is called a \emph{Poisson manifold}.
\begin{definition}
   Given a Lie group $G$ and a Poisson bivector $\pi \in \Gamma(G,\Lambda^2TG)$, the pair $(G,\pi)$  is called a \emph{Poisson--Lie group} if 
\begin{equation*}
\pi_{gh} = [dL_g \otimes dL_g](\pi_h) + [dR_h \otimes dR_h](\pi_g), \ \ g, h \in G,
\end{equation*}
where $L_g$ and $R_h$ denote the left and right translations by $g$ and $h$, respectively. Equivalently, the multiplication map from $G\times G$ equipped with the direct product Poisson structure to $G$ is Poisson.
\end{definition}
Let $\mathfrak{g}$ be the Lie algebra of $G$ and $r \in \mathfrak g \otimes \mathfrak g$. If $G$ is a connected Lie group, then the bivector\footnote{If $G$ is semisimple and complex, then any bivector $\pi$ that yields the structure of a Poisson--Lie group on $G$ is of this form for some $r \in \mathfrak{g}\otimes\mathfrak{g}$.} \begin{equation}\label{eq:pigroup}
\pi_g := [dL_g\otimes dL_g] (r) - [dR_g \otimes dR_g](r), \ \ g \in G,
\end{equation}
defines the structure of a Poisson--Lie group on $G$ if and only if the following conditions are satisfied:
\begin{enumerate}[1)]
\item The symmetric part of $r$ is $\ad$-invariant;\label{i:rc1}
\item The $3$-tensor $[r,r]:=[r_{12},r_{13}] + [r_{12},r_{23}] + [r_{13},r_{23}]$ is $\ad$-invariant, where $(a\otimes b)_{12} = a \otimes b \otimes 1$, $(a \otimes b)_{13} = a \otimes 1 \otimes b$ and $(a\otimes b)_{23} = 1 \otimes a \otimes b$, $a,b \in \mathfrak{g}$. \label{i:rc2}
\end{enumerate}
The \emph{Classical Yang--Baxter equation (CYBE)} is the equation $[r,r] = 0$. For simple complex Lie algebras~$\mathfrak{g}$, Belavin and Drinfeld in~\cite{bd,bd2} classified solutions of the CYBE that have a nondegenerate $\ad$-invariant symmetric part. The classification was partially extended by Hodges in~\cite{hodges} to the case of reductive complex Lie algebras (however, Hodges required the symmetric part of $r$ to be a multiple of the Casimir element). A full classification of solutions of the CYBE with an arbitrary nondegenerate $\ad$-invariant symmetric part in the case of reductive complex Lie algebras was obtained by Delorme in~\cite{bdr}.

\paragraph{Belavin--Drinfeld data.} Let $\mathfrak{g}$ be a reductive complex Lie algebra with a fixed decomposition $\mathfrak{g} = \mathfrak{n}_- \oplus \mathfrak{h} \oplus \mathfrak{n}_+$, let $\Pi$ be a set of simple roots and $\langle\,,\,\rangle$ be a nondegenerate invariant symmetric  bilinear form on~$\mathfrak{g}$. 
\begin{definition}
A \emph{Belavin--Drinfeld triple} (for conciseness, a \emph{BD triple}) is a triple $(\Gamma_1, \Gamma_2, \gamma)$ with $\Gamma_1, \Gamma_2 \subset \Pi$ and $\gamma: \Gamma_1 \rightarrow \Gamma_2$ a nilpotent\footnote{The nilpotency condition means that for any $\alpha \in \Gamma_1$ there exists a number $j$ such that $\gamma^j(\alpha) \notin \Gamma_1$.} isometry. A \emph{Belavin--Drinfeld quadruple} (a \emph{BD quadruple}) is a quadruple $\bar{\bg}:=(\bg,r_0)$ where $r_0 \in \mathfrak{h}\otimes \mathfrak{h}$ is an element satisfying the system of equations
\begin{align}
    &r_0 + r_0^t = \Omega_0;\label{eq:r0} \\
    &[\gamma(\alpha)\otimes 1](r_0) + [1\otimes \alpha](r_0) = 0, \ \ \alpha \in \Gamma_1 \label{eq:ralg}
\end{align}
where $(a\otimes b)^{t} = b \otimes a$, $a,b\in \mathfrak{g}$.
\end{definition}
For every positive root $\alpha$, choose $e_{\alpha} \in \mathfrak{g}_\alpha$ and $e_{-\alpha} \in \mathfrak{g}_{-\alpha}$ such that $\langle e_\alpha,e_{-\alpha}\rangle = 1$, and set $h_{\alpha}:=[e_{\alpha},e_{-\alpha}]$. Let $\mathfrak{g}_{\Gamma_1}$ and $\mathfrak{g}_{\Gamma_2}$ be the semisimple Lie subalgebras of $\mathfrak{g}$ generated by $\Gamma_1$ and $\Gamma_2$. Extend $\gamma$ to an isomorphism $\mathbb{Z}\Gamma_1 \xrightarrow{\sim} \mathbb{Z}\Gamma_2$, and then define $\gamma : \mathfrak{g}_{\Gamma_1} \rightarrow \mathfrak{g}_{\Gamma_2}$ by $\gamma(e_{\alpha}) = e_{\gamma(\alpha)}$ and $\gamma(h_{\alpha}) = h_{\gamma(\alpha)}$. Let $\gamma^* : \mathfrak{g}_{\Gamma_2} \rightarrow \mathfrak{g}_{\Gamma_1}$ be the conjugate of $\gamma$ with respect to the form on $\mathfrak{g}$. Extend both $\gamma$ and $\gamma^*$ by zero to $\mathfrak{g}$ via setting $\gamma|_{\mathfrak{g}_{\Gamma_1}^{\perp}} = 0$ and $\gamma^*|_{\mathfrak{g}_{\Gamma_2}^{\perp}} = 0$. Note that the extensions $\gamma,\gamma^*:\mathfrak{g}\rightarrow\mathfrak{g}$ are still adjoint to one another.

\paragraph{The Belavin--Drinfeld classification.} Continuing in the setup of the previous paragraph, let $\Omega$ be the $2$-tensor corresponding to the form $\langle\,,\,\rangle$. Let $\pi_{>}:\mathfrak{g} \rightarrow \mathfrak{n}_+$, $\pi_{<}:\mathfrak{g}\rightarrow \mathfrak{n}_-$, and $\pi_0 : \mathfrak{g}\rightarrow \mathfrak{h}$ be the orthogonal projections, and let $\Omega_{>}:=[1\otimes\pi_{>}]\Omega$, $\Omega_{<}:=[1\otimes \pi_{<}]\Omega$, and $\Omega_0 :=[1\otimes\pi_0]\Omega$ be the $2$-tensors corresponding to the projections. 

\begin{theorem}[Belavin \& Drinfeld]\label{thm:bd} Under the above setup, for a given BD quadruple $(\bg,r_0)$, if
\begin{equation}\label{eq:rplus}
r = r_0 + \left[\frac{1}{1-\gamma^*} \otimes 1\right](\Omega_{>}) - \left[1 \otimes \frac{\gamma^*}{1-\gamma^*}\right](\Omega_{<})
\end{equation}
then $r$ satisfies CYBE. Moreover,
\begin{equation}\label{eq:plusmin}
r + r^t = \Omega.
\end{equation}
Conversely, if $r$ is a solution of the CYBE such that $r$ satisfies~\eqref{eq:plusmin}, then $r$ assumes the form~\eqref{eq:rplus} for a suitable decomposition of $\mathfrak{g}$, for some BD quadruple and some choice of the root vectors $e_{\alpha}$.
\end{theorem}
The element $r$ from the theorem is called a \emph{classical $r$-matrix}. If $\mathfrak{g}$ is simple, then Theorem~\ref{thm:bd} provides a full classification of solutions $r \in \mathfrak{g}\otimes \mathfrak{g}$ of the CYBE with nondegenerate $\ad$-invariant symmetric parts. Sometimes we will use the notation $r_{\bar{\bg}}$ along with $\pi_{\bar{\bg}}$ for the Poisson bivector on $G$ defined by $r_{\bar{\bg}}$ via the formula~\eqref{eq:pigroup}.

Let us define $R_+ : \mathfrak{g}\rightarrow \mathfrak{g}$ and $R_-:\mathfrak{g}\rightarrow \mathfrak{g}$ as the linear maps satisfying the equation 
\begin{equation}\label{eq:rp_rm}
    r = [1\otimes R_+]\Omega = -[R_-\otimes 1]\Omega.
\end{equation}
Then $R_+ - R_- = \id$ and
\begin{equation}\label{eq:rplusm}
    R_+ = \frac{1}{1-\gamma} \pi_{>} - \frac{\gamma^*}{1-\gamma^*}\pi_{<} + R_0 \pi_0.
\end{equation}
Defining the map $R_0 :\mathfrak{h}\rightarrow \mathfrak{h}$  by $r_0 = [1\otimes R_0]\Omega_0$, the system~\eqref{eq:r0}-\eqref{eq:ralg} can be rewritten 
\begin{align}
&R_0 + R_0^* = \id;\label{eq:r0m}\\
&R_0(h_{\alpha}-\gamma(h_{\alpha})) = h_{\alpha}, \ \ \alpha \in \Gamma_1\label{eq:ralgm}
\end{align}
where $R_0^*$ denotes the adjoint of $R_0$ in $\mathfrak{h}$.

\paragraph{The Drinfeld Double.} Continuing in the setup of the previous paragraphs, let $G$ be equipped with a  Poisson bivector $\pi_{\bar{\bg}}$ defined by a classical $r$-matrix $r_{\bar{\bg}}$ relative to a BD quadruple $\bar{\bg}$. Let $R_+$ and $R_-$ be given by $r_{\bar{\bg}}$ from the equation~\eqref{eq:rp_rm}, and set $\mathfrak{d} := \mathfrak g \oplus \mathfrak g$ to be the direct sum of Lie algebras. Define a nondegenerate symmetric invariant bilinear form on $\mathfrak{d}$ by
\begin{equation*}
\langle (x_1,y_1), (x_2,y_2) \rangle = \langle x_1, x_2 \rangle - \langle y_1, y_2 \rangle, \ \ x_1, \,x_2,\, y_1,\, y_2 \in \mathfrak{g}.
\end{equation*}
As a vector space, $\mathfrak{d}$ splits into the direct sum of the following isotropic Lie subalgebras:
\begin{equation*}
\mathfrak{g}^{\delta} := \{(x, x) \ | \ x \in \mathfrak{g}\}, \ \ \mathfrak{g}^* := \{(R_+(x), R_-(x)) \ | \ x \in \mathfrak{g}\}.
\end{equation*}
Set $R^{\mathfrak{d}}_+ := P_{\mathfrak{g}^*}$ where $P_{\mathfrak{g}^*}$ is the projection of $\mathfrak{d}$ onto $\mathfrak{g}^*$, and let $r^{\mathfrak{d}} \in \mathfrak{d}\otimes \mathfrak{d}$ be the $2$-tensor corresponding to $R^{\mathfrak{d}}_+$. Then $R_+^{\mathfrak{d}}$ yields the structure of a Poisson--Lie group on the Lie group $D(G):=G \times G$ via the Poisson bivector 
\begin{equation*}
(\pi^{D}_{\bar{\bg}})_{(g,h)} := [dL_{(g,h)}\otimes dL_{(g,h)}] (r^{\mathfrak{d}}) - [dR_{(g,h)}\otimes dR_{(g,h)}](r^{\mathfrak{d}}), \ \ {(g,h)} \in D(G).
\end{equation*}
\begin{definition}
    The Poisson--Lie group $(D(G),\pi_{\bar{\bg}}^{D})$ is called the \emph{Drinfeld double} of $(G,\pi_{\bar{\bg}})$.
\end{definition} 
The Poisson bracket on $D(G)$ can be written in the form
\begin{equation*}
\{f_1,f_2\}_{\bar{\bg}}^D = \langle R_+(E_L f_1), E_Lf_2 \rangle - \langle R_+(E_R f_1), E_Rf_2 \rangle + \langle \nabla_X^R f_1, \nabla^R_Y f_2\rangle - \langle \nabla_X^L f_1, \nabla_Y^L f_2\rangle
\end{equation*}
where $\nabla^L f_i = (\nabla_X^L f_i, -\nabla_Y^L f_i)$ and $\nabla^R f_i = (\nabla_X^R f_i, -\nabla_Y^R f_i)$ are the left and right gradients, respectively, while $E_Lf_i = \nabla^L_X f_i+ \nabla^L_Yf_i$ and $E_Rf_i = \nabla^R_Xf_i + \nabla^R_Yf_i$. We define the gradients on $G$ as\footnote{This convention is opposite to the one in~\cite{plethora} and~\cite{rs}, but this convention ensures that the left gradient is the gradient in the left trivialization, and the right gradient is the gradient in the right trivialization of the group.}
\begin{equation*}
\langle \nabla^L f|_g, x\rangle := \left. \frac{d}{dt}\right|_{t=0} f(g\exp (tx)), \ \ \ \langle \nabla^R f|_g, x\rangle := \left. \frac{d}{dt}\right|_{t=0} f(\exp(tx)g), \ \ g \in G, \ x \in \mathfrak{g}.
\end{equation*}
 The group $G$ can be identified with the diagonal subgroup $G^{\delta}$ of $D(G)$, which corresponds to the Lie subalgebra $\mathfrak{g}^\delta$. The Poisson bracket $\{\cdot, \cdot \}_{\bar{\bg}}$ on $G$ can be expressed as
\begin{equation*}
\{f_1,f_2\}_{\bar{\bg}} = \langle R_+(\nabla^L f_1), \nabla^L f_2 \rangle - \langle R_+(\nabla^R f_1), \nabla^R f_2 \rangle.
\end{equation*}
Additionally, the connected (immersed) Poisson--Lie subgroup $(G^*,\pi_{\bar{\bg}}^*)$ of $D(G)$ corresponding to $\mathfrak{g}^*$ is called the \emph{dual Poisson--Lie group of $G$}.

\paragraph{The Poisson dual $G^{\dagger}$.} In the setup of the previous paragraph, consider the Drinfeld double $(D(G),\pi^{D}_{\bar{\bg}})$ along with $(G^*,\pi^*_{\bar{\bg}})$ immersed into $D(G)$. Consider the map
\begin{equation}\label{eq:dagpush}
D(G) \ni (g,h) \ \longmapsto \ g^{-1}h \in G.
\end{equation}
The pushforward of $\pi^D_{\bar{\bg}}$ under the above map is denoted as $\pi^{\dagger}_{\bar{\bg}}$, and the image of $G^*$ is an open subset of $G$ denoted as $G^{\dagger}$. The Poisson bivector $\pi^{\dagger}_{\bar{\bg}}$ is given by
\begin{equation}\label{eq:pidag}\begin{split}
    (\pi^{\dagger}_{\bar{\bg}})_g = \left[d L_g \otimes d L_g-d R_g \right.&\left.\otimes d L_g-d L_g \otimes d R_g+d R_g \otimes d R_g\right](r) -\\- &\left[d L_g \otimes d L_g - d R_g \otimes d L_g\right](\Omega),  \ \ g \in G,
    \end{split}
\end{equation}
and the corresponding Poisson bracket is given by
\begin{equation}\label{eq:bdag}
\{f_1,f_2\}^{\dagger}_{\bar{\bg}} = \langle R_+(\nabla^Lf_1-\nabla^R f_1), \nabla^L f_2 - \nabla^R f_2\rangle - \langle \nabla^Lf_1-\nabla^R f_1, \nabla^L f_2\rangle.
\end{equation}

\paragraph{Symplectic foliation and Poisson submanifolds.} Let $(M,\pi)$ be a Poisson manifold. An immersed submanifold $S \subseteq M$ is called a \emph{Poisson submanifold} if $\pi|_S \in \Gamma(S,\bigwedge^2TS)$. Define a morphism of vector bundles $\pi^\# : TM^* \rightarrow TM$ via $\langle\pi^\#(\xi),\eta\rangle := \langle \pi,\xi\wedge \eta\rangle$, $\xi,\eta \in T_p^*M$, $p \in M$. 
The Poisson bivector $\pi$ is called \emph{nondegenerate} if $\pi^\#$ is an isomorphism of vector bundles. A \emph{symplectic leaf} is a maximal (by inclusion) connected Poisson submanifold $S$ of $M$ for which $\pi|_S$ is nondegenerate. It is a theorem that any Poisson manifold $M$ is a union of its symplectic leaves.

In this paper, we will be interested in symplectic leaves of the zero loci of frozen variables (see Proposition~\ref{p:frozvar}). Assume that $M$ is a complex affine variety. Recall that an ideal $J \subseteq \mathbb{C}[M]$ is a \emph{Poisson ideal} if $\{\mathbb{C}[M],J\} \subseteq J$. The following proposition is a restatement of \cite[Remark 2.4]{yakimov_det} in the setup of the current paper:

\begin{proposition}\label{p:yakim}
    Let $p \in \mathbb{C}[M]$ be a prime element. Then the following conditions are equivalent:
    \begin{enumerate}[(i)]
    \item The principal ideal $(p)$ is a Poisson ideal;\label{pc:ppois}
    \item\label{pc:zerl} The zero locus of $p$ is a union of symplectic leaves of $(M,\pi)$.
    \end{enumerate}
\end{proposition}
Proposition~\ref{p:yakim} can also be generalized to the case of not necessarily smooth affine varieties; see \cite[Proposition 2.3]{yakimov_locus}. A prime element $p \in \mathbb{C}[M]$ satisfying condition~\ref{pc:ppois} is also called a \emph{Poisson prime element}.

\subsection{Generalized cluster algebras}\label{s:ba_cluster}
In this subsection, we briefly review relevant definitions and propositions from the theory of generalized cluster structures. An expanded account with proofs can be found in~\cite{double}. The concept of a generalized cluster algebra is based on an earlier definition given in~\cite{earlier} (see~\cite[Remark 2.1]{double}). Throughout the subsection, we let $\mathcal{F}$ to be a field of rational functions in $N+M$ independent variables with coefficients in $\mathbb{Q}$. We refer to $\mathcal{F}$ as the \emph{ambient field}.

\begin{definition}[Seeds]\label{def:ext_seed}
    An \emph{extended seed} is a triple $\Sigma:=(\mathbf{x},B,\mathcal{P})$ that consists of the following data:
    \begin{enumerate}[1)]
        \item A collection of elements $\mathbf{x}:=(x_1,x_2,\ldots,x_{N+M})$ of $\mathcal{F}$ such that $\mathcal{F} = \mathbb{Q}(x_1,x_2,\ldots,x_{N+M})$ (an \emph{extended cluster}). The elements $x_1,x_2,\ldots,x_{N}$ are called \emph{cluster variables} and the elements $x_{N+1},\ldots,x_{N+M}$ are called \emph{frozen variables};
        \item An integer matrix $B:=(b_{ij} \ | \ i \in [1,N], \ j \in [1,N+M])$ such that the principal part $B^{[1,N]}$ is skew-symmetrizable\footnote{Recall that a square matrix $A$ of size $N\times N$ is \emph{skew-symmetrizable} if there exists a diagonal matrix $D$ with positive entries such that $DA$ is skew-symmetric.} (an \emph{exchange matrix});
        \item For each $i \in [1,N]$, set $d_i$ to be a factor of $\gcd(b_{ij} \ | \ j \in [1,N])$ (the \emph{multiplicity} of $x_i$). For each $i \in [1,N]$, define the $i$\emph{th string} $p_i:=(p_{ir})_{r \in [0,d_i]}$ to be a collection of monomials $p_{ir}$ in the frozen variables such that $p_{i0} = p_{id_i} = 1$. The collection of \emph{strings} $\mathcal{P}$ is given by $\mathcal{P}:=\{p_i\}_{i \in [1,N]}$.
    \end{enumerate}
\end{definition}
\begin{definition}[$\tau$-monomials.] Let $\Sigma:=(\mathbf{x},B,\mathcal{P})$ be an extended seed. For $k \in [1,N]$, the \emph{cluster $\tau$-monomials} $u_{k;>}$ and $u_{k;<}$ are monomials in the cluster variables of $\Sigma$ given by
\begin{equation*}
    u_{k;>} := \prod_{\substack{1 \leq i \leq N, \\ b_{ki} > 0}} x_{i}^{b_{ki}/d_k}, \ \ u_{k;<} := \prod_{\substack{1 \leq i \leq N, \\ b_{ki} < 0}} x_{i}^{-b_{ki}/d_k}.
\end{equation*}
    For $k \in [1,N]$ and $r \in [0,d_k]$, the \emph{stable $\tau$-monomials} $v_{k;>}^{[r]}$ and $v_{k;<}^{[r]}$ are monomials in the frozen variables given by
    \begin{equation*}
        v_{k;>}^{[r]} := \prod_{\substack{N+1\leq i \leq N+M,\\ b_{ki} > 0}} x_i^{\lfloor rb_{ki}/d_k \rfloor}, \ \ v_{k;<}^{[r]} := \prod_{\substack{N+1 \leq i \leq N+M,\\ b_{ki} < 0}} x_i^{\lfloor-rb_{ki}/d_k\rfloor},
    \end{equation*}
    where the product over an empty set by definition equals $1$ and $\lfloor m \rfloor$ denotes the floor of $m \in \mathbb{Z}$.
\end{definition}
\begin{definition}[Generalized cluster mutations]\label{d:gcm} Let $\Sigma:=(\mathbf{x},B,\mathcal{P})$ be an extended seed. For $k \in [1,N]$, a \emph{generalized mutation of $\Sigma$ in direction $k$} is an extended seed $\Sigma^\prime:=(\mathbf{x}^{\prime},B^\prime,\mathcal{P}^\prime)$ defined as follows.
\begin{enumerate}[1)]
    \item Let $x_k^\prime$ be an element given by the \emph{generalized exchange relation}:
    \begin{equation*}\label{eq:exchr}
    x_k x_k^\prime := \sum_{r=0}^{d_k} p_{kr} u_{k;>}^{r} v_{k;>}^{[r]} u_{k;<}^{d_k-r} v_{k;<}^{[d_k-r]}
    \end{equation*}
    where $u_{k;>}$ and $u_{k;<}$ are the cluster $\tau$-monomials and $v_{k;<}$ and $v_{k;<}$ are the stable $\tau$-monomials constructed relative $\Sigma$.
    The extended cluster $\mathbf{x}^\prime$ is given by $\mathbf{x}^\prime:=(\mathbf{x}\setminus \{x_k\})\cup \{x_k^\prime\}$.
    \item The exchange matrix $B^\prime:=(b_{ij}^\prime \ | \ i \in [1,N], \ j \in [1,N+M])$ is given by
    \begin{equation*}
        b_{ij}^\prime := \begin{cases}
    -b_{ij} \ \ &\text{if } i=k \text{ or }j=k;\\
    b_{ij} + \dfrac{|b_{ik}|b_{kj} + b_{ik} |b_{kj}|}{2} \ \ &\text{otherwise.}
    \end{cases}
    \end{equation*}
    \item Define $p_k^\prime:=(p_{kr})_{r \in [0,d_k]}$ as $p_{kr}:=p_{k,d_k-r}$, $r \in [0,d_k]$. The collection of strings $\mathcal{P}^\prime$ is given by $\mathcal{P}^\prime:=(\mathcal{P}\setminus \{p_k\}) \cup \{p_k^\prime\}$.\label{i:defstr}
\end{enumerate}
 The extended seeds $\Sigma$ and $\Sigma^\prime$ are also called \emph{adjacent}. 
 \end{definition}
 
Let us make a few remarks on the above definition:
\begin{enumerate}[1)]
\item A frozen variable $x_i$ is called \emph{isolated} if $b_{ji} = 0$ for all $j \in [1,N]$. The definition of $b_{ij}^\prime$ implies that if $x_i$ is isolated in $\Sigma$, it is isolated in $\Sigma^\prime$ as well;
\item Since $\gcd\{b_{ij} \ | \ 1 \leq j \leq N\} = \gcd\{b_{ij}^\prime \ | \ 1 \leq j \leq N\}$, the numbers $d_1,\ldots, d_N$ are the same in $\Sigma$ and $\Sigma^\prime$;
\item If the string $p_k$ is trivial, then the generalized exchange relation \eqref{eq:exchr} becomes the exchange relation from the theory of cluster algebras of geometric type:
\begin{equation}\label{eq:ordexchrel}
x_k x_k^\prime = \prod_{i:\, b_{ki>0}} x_i^{b_{ki}} + \prod_{i:\,b_{ki<0}} x_i^{-b_{ki}}.
\end{equation}
In fact, each extended seed of the generalized cluster structures studied in this paper has only one generalized mutation. The other mutations follow the rule~\eqref{eq:ordexchrel}.
\item For each $i\in[1,N]$, set $v_{i;>} := v_{i;>}^{[d_i]}$, $v_{i;<} := v_{i;<}^{[d_i]}$; define
\[
q_{ir} := \frac{v_{i;>}^r v_{i; < }^{d_i-r}}{(v_{i;>}^{[r]} v_{i;<}^{[d_i-r]})^{d_i}}, \ \ \hat{p}_{ir} := \frac{p_{ir}^{d_i}}{q_{ir}}, \ \ i \in [1,N],\ r\in [0,d_i].
\]
Then the relation~\eqref{eq:exchr} can be rewritten as
\begin{equation}
x_k x_k^\prime = \sum_{r=0}^{d_k} (\hat{p}_{kr} v_{k;>}^r v_{k;<}^{d_k-r})^{1/d_k} u_{k;>}^r u_{k;<}^{d_k-r}.
\end{equation}
The expression $(\hat{p}_{kr} v_{k;>}^r v_{k;<}^{d_k-r})^{1/d_k}$ is a monomial in the frozen variables. Note also that the mutation rule for $\hat{p}_{ir}$ is the same as for $p_{ir}$ (see Part~\ref{i:defstr} of Definition~\ref{d:gcm}).
\end{enumerate}

\begin{definition}[Generalized cluster structures] Two extended seeds $\Sigma$ and $\Sigma^\prime$ are called \emph{mutation equivalent} if there's a sequence of extended seeds $\Sigma_1, \ldots, \Sigma_m$ where $\Sigma_1 := \Sigma$ and $\Sigma_m := \Sigma^\prime$, such that $\Sigma_{i+1}$ and $\Sigma_i$ are adjacent for each~$i$. For a fixed extended seed $\Sigma$, the set of all extended seeds that are mutation equivalent to $\Sigma$ is called a \emph{generalized cluster structure} and is denoted as $\gc(\Sigma)$ or simply as $\gc$. The \emph{rank} of $\gc$ is the number $N$.
\end{definition}

In practice, we fix one extended seed $\Sigma_0:=(\mathbf{x}_0,B_0,\mathcal{P}_0)$ and provide its full description. In this case, we call $\Sigma_0$ the \emph{initial extended seed} (correspondingly, $\mathbf{x}_0$ is the \emph{initial extended cluster} and $B_0$ is the \emph{initial exchange matrix}). Also, when several generalized cluster structures are considered at once, we will denote by $\mathcal{F}(\gc)$ the ambient field of $\gc$.

\paragraph{Algebras associated to $\gc$.} Let $\gc$ be a generalized cluster structure and $x_{N+1},\ldots,x_{N+M}$ be its set of frozen variables. Define the rings $\mathbb{A} := \mathbb{Z}[x_{N+1},\ldots,x_{N+M}]$ and $\bar{\mathbb{A}} := \mathbb{Z}[x_{N+1}^{\pm 1},\ldots, x_{N+M}^{\pm 1}]$. 
\begin{definition}
    A \emph{ground ring} $\hat{\mathbb{A}}$ is a subring of $\bar{\mathbb{A}}$ that contains $\mathbb{A}$.
\end{definition}
\begin{definition}
The $\hat{\mathbb{A}}$-algebra $\mathcal{A}:=\mathcal{A}(\gc)$ generated by all the cluster variables in $\gc$ is called the \emph{generalized cluster algebra}. 
\end{definition}
For a fixed ground ring $\hat{\mathbb{A}}$ and an extended cluster $\mathbf{x}:=(x_1,x_2,\ldots,x_{N+M})$ in $\gc$, define the $\hat{\mathbb{A}}$-algebras
\begin{equation}\label{eq:laurring}
\mathcal{A}(\mathbf{x}) := \hat{\mathbb{A}}[x_1,\ldots,x_N], \ \ \  \mathcal{L}(\mathbf{x}) := \hat{\mathbb{A}}[x_1^{\pm 1},\ldots,x_N^{\pm 1}].
\end{equation}
For an extended seed $\Sigma$ that contains the extended cluster $\mathbf{x}$, we will also write $\mathcal{L}(\Sigma):=\mathcal{L}(\mathbf{x})$ and $\mathcal{A}(\Sigma):=\mathcal{A}(\mathbf{x})$.

\begin{definition}
The $\hat{\mathbb{A}}$-algebra given by
\begin{equation}
\bar{\mathcal{A}} :=\bar{\mathcal{A}}(\gc):=
\bigcap_{\Sigma \in \gc}\mathcal{L}(\Sigma)
\end{equation}
is called the \emph{generalized upper cluster algebra}.
\end{definition}
One of the central results of the theory is the \emph{generalized Laurent phenomenon}, which states that there is an inclusion $\mathcal{A} \subseteq \bar{\mathcal{A}}$. For cluster algebras, the result was first proved in \cite[Theorem~3.1]{fathers}; later, it was extended to generalized cluster algebras in \cite[Theorem~2.5]{earlier}.

\paragraph{Upper bounds.} We continue in the setup of the previous paragraph. Let $\mathbb{T}_N$ be a labeled $N$-regular tree. Associate with each vertex an extended seed so that adjacent extended seeds are adjacent in the tree, and if an extended seed $\Sigma^\prime$ is adjacent to $\Sigma$ in direction $k$, label the corresponding edge in the tree with number $k$. A \emph{nerve} $\mathcal{N}$ in $\mathbb{T}_N$ is a subtree on $N+1$ vertices such that all its edges have different labels (for instance, a star is a nerve). 
\begin{definition}An \emph{upper bound} $\bar{\mathcal{A}}(\mathcal{N})$ is the $\hat{\mathbb{A}}$-algebra defined as
\begin{equation*}
\bar{\mathcal{A}}(\mathcal{N}) := \bigcap_{\Sigma \in V(\mathcal{N})}\mathcal{L}(\Sigma)
\end{equation*}
where $V(\mathcal{N})$ stands for the vertex set of $\mathcal{N}$. 
\end{definition}
Upper bounds were first defined and studied in~\cite{upper_bounds}. Let $L$ be the number of isolated frozen variables in $\gc$. For the $i$th nontrivial string in $\mathcal{P}$, let ${B}(i)$ be a $(d_{i}-1) \times L$ matrix such that the $r$th row consists of the exponents of the isolated variables in $p_{ir}$. The following result was proved in \cite[Theorem 3.11]{double}:

\begin{proposition}\label{p:upperb}
Assume that the extended exchange matrix has full rank and let $\rank {B}(i) = d_i - 1$ for any nontrivial string in $\mathcal{P}$. Then for any upper bound $\mathcal{N}$ in $\gc$, $\bar{\mathcal{A}}(\mathcal{N}) = \bar{\mathcal{A}}$.
\end{proposition}

\paragraph{Poisson structures in $\gc$.} Let $\{\cdot, \cdot\}$ be a Poisson bracket on $\mathcal{F}$ and let ${\mathbf{x}}$ be any extended cluster in $\gc$. We say that ${\mathbf{x}}$ is \emph{log-canonical} if for each $i,j \in [1,N+M]$, $\{x_i,x_j\} = \omega_{ij} x_i x_j$, where $\omega_{ij} \in \mathbb{Q}$. The matrix $\Omega:=(\omega_{ij})_{i,j=1}^{N+M}$ is called the \emph{coefficient matrix} of the Poisson bracket with respect to the extended cluster $\tilde{\mathbf{x}}$.
\begin{definition}
    The generalized cluster structure $\gc$ is called \emph{compatible} with a Poisson bracket $\{\cdot,\cdot\}$ if each extended cluster in $\gc$ is log-canonical.
\end{definition}

The following result is the key connection between cluster algebras and Poisson geometry. Its earliest version was proved in~\cite[Theorem 1.4]{roots}.

\begin{proposition}\label{p:compb}
Let $(\tilde{\mathbf{x}}, B, \mathcal{P}) \in \gc$ be an extended seed that satisfies the following properties:
\begin{enumerate}[1)]
\item The extended cluster ${\mathbf{x}}$ is log-canonical with respect to the Poisson bracket;\label{pr:c1}
\item For a diagonal matrix with positive entries $D$ such that $DB^{[1,N]}$ is skew-symmetric, there exists a diagonal $N \times N$ matrix $\Delta$ such that $B\Omega = \begin{bmatrix}\Delta & 0\end{bmatrix}$ and such that $D\Delta$ is a multiple of the identity matrix;\label{pr:c2}
\item The Laurent polynomials $\hat{p}_{ir}$ are Casimirs of the Poisson bracket.\label{pr:c3}
\end{enumerate}
Then any other extended seed in $\gc$ satisfies properties \ref{pr:c1}, \ref{pr:c2} (with the same $\Delta$) and \ref{pr:c3}. In particular, $\gc$ is compatible with $\{\cdot, \cdot \}$.
\end{proposition}
 Condition~\ref{pr:c2} has the following interpretation, which is used in practice. For each $i \in [1,N]$, define the \emph{$y$-variable} of the cluster variable $x_i$ as $y_i:=y(x_i) := \prod_{j=1}^{N+M} x_j^{b_{ij}}$. Then Condition~\ref{pr:c2} is equivalent to $\{\log y_i, \log x_j\} = \delta_{ij}\Delta_{ii}$, where $\delta_{ij}$ is the Kronecker symbol. Note also that Condition~\ref{pr:c2} implies that $B$ has full rank.

 \paragraph{$\gc$ on varieties.} For an extended cluster $\mathbf{x}$ in $\gc$, set
 \begin{align*}
    &\mathcal{A}_{\mathbb{C}}(\mathbf{x}) := \mathcal{A}(\mathbf{x})\otimes\mathbb{C}, \ &\   &\mathcal{L}_{\mathbb{C}}(\mathbf{x}) := \mathcal{L}(\mathbf{x})\otimes\mathbb{C};\\
     &\mathcal{A}_{\mathbb{C}}(\gc):=\mathcal{A}(\gc)\otimes \mathbb{C}, \ &\ &\bar{\mathcal{A}}_{\mathbb{C}}(\gc):=\bar{\mathcal{A}}(\gc)\otimes \mathbb{C};
 \end{align*}
 For the ambient field $\mathcal{F}$ of $\gc$, we also set $\mathcal{F}_{\mathbb{C}}:=\mathcal{F}\otimes \mathbb{C}$. 

 Let $V$ be a complex rational quasi-affine variety, $\mathbb{C}[V]$ be its coordinate ring and $\mathbb{C}(V)$ be its field of rational functions. Consider an embedding $\imath : \mathcal{A}_{\mathbb{C}}(\mathbf{x}_0)\hookrightarrow \mathbb{C}(V)$.
\begin{definition}\label{d:clustring}
    The triple $(V,\gc,\imath)$ is a \emph{generalized cluster structure on the variety $V$}. A variable $x$ in $\gc$ is called \emph{regular} if $x \in \mathbb{C}[V]$; an extended cluster $\mathbf{x}$ is called \emph{regular} if for each $x \in \mathbf{x}$, $x$ is regular; the generalized cluster structure $(V,\gc)$ is called \emph{regular} if $\mathcal{A}_{\mathbb{C}}(\gc)\subseteq \mathbb{C}[V]$. Finally, $(V,\gc)$ is called \emph{complete} if $\bar{\mathcal{A}}_{\mathbb{C}}(\gc) = \mathbb{C}[V]$.
\end{definition}
Note that once an embedding $\imath$ induces an embedding of $\mathcal{F}_{\mathbb{C}}$ into $\mathbb{C}(V)$. Hence, every cluster variable is realized as a rational function on $V$. Abusing notation, we will omit the embedding $\imath$ from the triple and speak only of the pair $(V,\gc)$.

 \paragraph{Toric actions.} Let $\mathbf{x}:=(x_1,x_2,\ldots,x_{N+M})$ be an extended cluster in some generalized cluster structure $\gc$, as above. Given an integer matrix $W:=(w_{ij} \ | \ i \in [1,N+M], \ j \in [1,s])$, consider an action of $(\mathbb{C}^{*})^s$ upon $\mathcal{F}_{\mathbb{C}}$ by field automorphisms: for every $\mathbf{t}:=(t_1,t_2,\ldots,t_s) \in (\mathbb{C}^{*})^{s}$, the corresponding field automorphism $\sigma_{\mathbf{t}}$ is given by
\begin{equation}\label{eq:tor_loc}
    \sigma_{\mathbf{t}}(x_i) = \left(\prod_{j=1}^{s}t_j^{w_{ij}}\right) \cdot x_i, \ \ i \in [1,N+M].
\end{equation}
\begin{definition}
    A \emph{local toric action} of rank $s$ upon an extended cluster $\mathbf{x}$ is the action by field automorphisms given by~\eqref{eq:tor_loc} such that $W$ is of rank $s$. It is called a \emph{global toric action} if the action is equivariant\footnote{In other words, given any other extended cluster $\mathbf{x}^\prime$ mutation equivalent to $\mathbf{x}$, there exists a matrix $W^\prime$ such that the action is still of the form~\eqref{eq:tor_loc}.} with respect to mutations. 
\end{definition}

The first version of following proposition was proved in~\cite[Lemma 2.3]{roots} for cluster algebras of geometric type, and it was later extended to generalized cluster algebras in~\cite[Proposition 2.6]{double}.

\begin{proposition}\label{p:toric}
A local toric action with a weight matrix $W$ is a global toric action if $BW = 0$ and the Laurent polynomials $\hat{p}_{ir}$ are invariant with respect to the action.
\end{proposition}

When all the strings are trivial, the statement of Proposition~\ref{p:toric} is \emph{if and only if}.

\paragraph{Quivers.} A \emph{quiver} is a directed multigraph with no $1$- and $2$-cycles. Pick an extended seed $\Sigma:=({\mathbf{x}},{B},\mathcal{P})$ and let $D:=\diag(d_1^{-1},\ldots, d_N^{-1})$ be a diagonal matrix with $d_i$'s defined as above. Assume that $DB^{[1,N]}$ is skew-symmetric. Then the matrix 
\begin{equation*}
\hat{B}:= \begin{bmatrix}
DB^{[1,N]} & {B}^{[N+1,N+M]} \\
-({B}^{[N+1,N+M]})^T & 0
\end{bmatrix}
\end{equation*}
is the adjacency matrix of a quiver $Q$, in which each vertex $i$ corresponds to a variable $x_i \in {\mathbf{x}}$. The vertices that correspond to cluster variables are called \emph{mutable}, the vertices that correspond to frozen variables are called \emph{frozen}, and the vertices that correspond to isolated variables are called \emph{isolated}. For each $i$, the number $d_i$ is called the \emph{multiplicity} of the $i$th vertex. If one mutates the extended seed $\Sigma$, then the quiver of the resulting extended seed can be obtained from $Q$ via the following steps:
\begin{enumerate}[1)]
\item For each path $i \rightarrow k \rightarrow j$, add an arrow $i \rightarrow j$;
\item If there is a pair of arrows $i \rightarrow j$ and $j \rightarrow i$, remove both;
\item Flip the orientation of all arrows going in and out of the vertex $k$.
\end{enumerate} 
The above process is called a \emph{quiver mutation in direction $k$} (or \emph{at vertex $k$}). Instead of describing the exchange matrix ${B}$, we describe the corresponding quiver and multiplicities $d_i$'s.

\subsection{Birational quasi-isomorphisms}\label{s:ba_birat}
In this subsection, we state some of the results on birational quasi-isomorphisms from~\cite{zametka} that are used in the proofs of Theorems~\ref{thm:maingln} and~\ref{thm:mainsln}. The statements are adjusted for generalized cluster structures in the coordinate rings of complex rational quasi-affine varieties.
 We assume that for all generalized cluster structures, the ground rings are chosen so that $\hat{\mathbb{A}} = \mathbb{A}$ (in other words, we do not localize at the frozen variables).

\paragraph{Markers and related triples.} Let $\gc$ and $\widetilde{\gc}$ be generalized cluster structures of ranks $N$ and $\tilde{N}$ with $M$ and $\tilde{M}$ frozen variables, respectively. Fix initial extended seeds $\Sigma_0 := (\mathbf{x}_0,B_0,\mathcal{P}_0)$ and $\tilde{\Sigma}_0 := (\tilde{\mathbf{x}}_0,\tilde{B}_0,\tilde{\mathcal{P}}_0)$. For a permutation $\kappa$ of the indices $[1,N+M]$, set
\begin{align*}
    &\tilde{\mathcal{I}}(\kappa) :=[\tilde{N}+1,\tilde{N}+\tilde{M}] \cap \kappa([1,N]);\\
    &\mathcal{I}(\kappa) := \kappa^{-1}(\tilde{\mathcal{I}}(\kappa)) = [1,N] \cap \kappa^{-1}([\tilde{N}+1,\tilde{N}+\tilde{M}]).
\end{align*}
Denote by $d_i$ and $\tilde{d}_j$ the multiplicities of the variables with indices $i$ and $j$ in $\gc$ and $\widetilde{\gc}$, respectively.

\begin{definition}\label{d:marker}
A \emph{marker} $\kappa$ for the pair $(\gc,\widetilde{\gc})$ is a bijection $k:[1,N+M]\rightarrow[1,N+M]$ such that $\kappa([N+1,N+M]) \subset [\tilde{N}+1,\tilde{N}+\tilde{M}]$, and for any $i \in [1,N]\setminus\mathcal{I}(\kappa)$, $d_{i} = \tilde{d}_{\kappa(i)}$. For a given marker $\kappa$, the triple $(\gc,\widetilde{\gc},\kappa)$ is called a \emph{related} triple. The sets $\mathcal{I}(\kappa)$ and $\tilde{\mathcal{I}}(\kappa)$ are called the sets of \emph{marked indices} in $\gc$ and $\widetilde{\gc}$, respectively.
\end{definition}

 Assume that $(\gc,\widetilde{\gc},\kappa)$ is a related triple. Define
\begin{align*}
&\mathbf{x}_0(\kappa):=\{x_i \in \mathbf{x}_0 \ | \ i \in \mathcal{I}(\kappa)\},\\
&\tilde{\mathbf{x}}_0(\kappa):=\{\tilde{x}_i \in \tilde{\mathbf{x}}_0 \ | \ i \in \tilde{\mathcal{I}}(\kappa)\}.\label{eq:tildx0kap}
\end{align*}
\begin{definition}
The elements of $\mathbf{x}_0(\kappa)$ and $\tilde{\mathbf{x}}_0(\kappa)$ are called the \emph{marked variables} in $\gc$ and $\widetilde{\gc}$.
\end{definition}

\begin{definition}
An extended seed $\Sigma:=(\mathbf{x},B,\mathcal{P})$ in $\gc$ is called \emph{marked} if $\Sigma$ is obtained from $\Sigma_0$ via a sequence of mutations
\begin{equation}\label{eq:emarkcl}
\Sigma_{i_0}\rightarrow \Sigma_{i_1} \rightarrow \cdots \rightarrow \Sigma_{i_m}
\end{equation}
where $\Sigma_{i_0}:=\Sigma_0$, $\Sigma_{i_m}:=\Sigma$, and $\Sigma_{i_j}$ is a mutation of $\Sigma_{i_{j-1}}$ in direction $i_j \in [1,N] \setminus \mathcal{I}(\kappa)$, $j \in [1,m]$. In this case, the extended cluster $\mathbf{x}$ is called \emph{marked}.
\end{definition}
In particular, the initial extended seed $\Sigma_0$ is marked. A marker $\kappa$ induces a natural bijection between extended marked seeds in $\gc$ and all extended seeds in $\widetilde{\gc}$, as follows. If an extended seed $\Sigma$ is obtained from $\Sigma_0$ via the sequence~\eqref{eq:emarkcl}, consider the sequence of mutations in $\widetilde{\gc}$
\begin{equation*}
\tilde{\Sigma}_{\kappa(i_0)} \rightarrow \tilde{\Sigma}_{\kappa(i_1)} \rightarrow \cdots \rightarrow \tilde{\Sigma}_{\kappa(i_m)},
\end{equation*}
where $\tilde{\Sigma}_{\kappa(i_0)} = \tilde{\Sigma}_0$. Then we set $\tilde{\Sigma}:=\kappa(\Sigma) := \tilde{\Sigma}_{\kappa(i_m)}$.

\begin{definition}\label{d:relclust}
Extended seeds $\Sigma:=(\mathbf{x},B,\mathcal{P}) \in \mathcal{C}$ and $\tilde{\Sigma}:=(\tilde{\mathbf{x}},\tilde{B},\tilde{\mathcal{P}}) \in \tilde{\mathcal{C}}$ are called \emph{related} if $\tilde{\Sigma} = \kappa(\Sigma)$. In this case, the extended clusters $\mathbf{x}$ and $\tilde{\mathbf{x}}$ are called \emph{related}, and for any $i \in [1,N+M]$, the variables $x_i \in \mathbf{x}$ and $\tilde{x}_{\kappa(i)}\in \tilde{\mathbf{x}}$ are called \emph{related}.
\end{definition}

By default, if $x$ denotes a variable from an extended cluster in $\gc$, $\tilde{x}$ denotes the corresponding related variable in $\widetilde{\gc}$.

\paragraph{Quasi-isomorphisms} Here, we introduce the notion of a quasi-isomorphism suitable for our purposes\footnote{As we explained in~\cite{zametka}, the closest notion in literature to our version is the one defined by C. Fraser in~\cite{fraser}. The difference is, C. Fraser assumes $N = \tilde{N}$ and arbitrary $M$ and $\tilde{M}$, whereas we insist upon $N+M = \tilde{N}+\tilde{M}$ and $\tilde{M}-M > 0$.} along with some basic results. We let $(\gc,\widetilde{\gc},\kappa)$ to be a related triple of generalized cluster structures with fixed initial extended clusters.

\begin{definition}\label{def:quasiiso}
A \emph{quasi-isomorphism} $\mathcal{Q}:\gc \rightarrow \widetilde{\gc}$ is 
a field isomorphism $\mathcal{Q}:\mathcal{F}(\gc)\rightarrow \mathcal{F}(\widetilde{\gc})$ such that the following conditions hold:
\begin{enumerate}[1)]
    \item For any extended marked cluster $\mathbf{x}$, there exists an integer matrix
\begin{equation*}
\Lambda(\mathbf{x}):=\left(\lambda_{ij} \ | \ i\in[1,N+M], \ j \in {\mathcal{I}}(\kappa)\right)
\end{equation*} 
such that
\begin{equation}\label{eq:quasi_formula}
\mathcal{Q}(x_i) = \begin{cases}
    \displaystyle\tilde{x}_{\kappa(i)} \prod_{j \in {\mathcal{I}}(\kappa)} \tilde{x}_{\kappa(j)}^{\lambda_{ij}} & i \in [1,N+M]\setminus \mathcal{I}(\kappa); \\[10pt] \displaystyle\prod_{j \in {\mathcal{I}}(\kappa)} \tilde{x}_{\kappa(j)}^{\lambda_{ij}} & i \in \mathcal{I}(\kappa).
\end{cases}
\end{equation}
\item For every $i \in [1,N]\setminus \mathcal{I}(\kappa)$, every pair $(x_i,\tilde{x}_{\kappa(i)})$ of related variables and the corresponding hatted strings $\{\hat{p}_{ir}\}_{r=0}^{d_i}$ and $\{\tilde{\hat{p}}_{\kappa(i)r}\}_{r=0}^{d_{\kappa(i)}}$, $\mathcal{Q}(\hat{p}_{ir}) = \tilde{\hat{p}}_{\kappa(i)r}$, $r \in [0,d_i]$.\label{i:condquas}
\end{enumerate}
\end{definition}

For an exchange matrix $\tilde{B}$ in $\widetilde{\gc}$, let $P$ be a permutation matrix such that $P^{T} \tilde{B}^{[1,\tilde{N}]} P$ is block diagonal, with blocks given by $D_1, D_2,\ldots, D_m$. We say that an integer matrix $\tilde{B}^\prime$ of size $\tilde{N}\times (\tilde{N}+\tilde{M})$ is obtained from $\tilde{B}$ by a \emph{flip of orientations} if $P^{T} [\tilde{B}^{\prime}]^{[1,\tilde{N}]} P$ is block diagonal, with blocks given by $\epsilon_1 D_1, \epsilon_2 D_2,\ldots, \epsilon_n D_n$ for some choices $\epsilon_i \in \{-1,1\}$. 

\begin{proposition}\label{p:q_y_var}
    Let $\mathcal{Q}: \mathcal{F}(\gc) \rightarrow \mathcal{F}(\widetilde{\gc})$ be a field isomorphism given by~\eqref{eq:quasi_formula} in the initial extended seed $(\mathbf{x}_0,B,\mathcal{P})$, and satisfying Condition~\ref{i:condquas} of Definition~\ref{def:quasiiso}. If $\mathcal{Q}$ satisfies the condition
    \begin{enumerate}[1)]
    \setcounter{enumi}{2}
        \item For each $y$-variable $y_i$ in $\mathbf{x}_0$, $i \notin \mathcal{I}(\kappa)$, $\mathcal{Q}(y_i) = \tilde{y}_{\kappa(i)}$,\label{i:ycond}
    \end{enumerate}
    then $\mathcal{Q}$ is a quasi-isomorphism. Conversely, if $\mathcal{Q}:\mathcal{F}(\gc) \rightarrow \mathcal{F}(\widetilde{\gc})$ is a quasi-isomorphism, then there exists a flip of orientations such that $\mathcal{Q}$ satisfies Condition~\ref{i:ycond}.
\end{proposition}

Let us introduce the following notation: if $A$ is an $n_1\times n_2$ matrix, $n_1,n_2\leq N+M$, then $\kappa^{-1}(A)$ is an $n_1\times n_2$ matrix such that $(\kappa^{-1}(A))_{ij} = A_{\kappa(i),\kappa(j)}$. If $B$ is another matrix of size $n_2 \times n_3$, $n_3\leq N+M$, then $\kappa^{-1}(AB) = \kappa^{-1}(A)\kappa^{-1}(B)$. In addition, let us extend $\kappa$ to a field isomorphism $\kappa:\mathcal{F}(\gc) \rightarrow \mathcal{F}(\widetilde{\gc})$ via $\kappa(x_i) = \tilde{x}_{\kappa(i)}$ in the initial extended clusters. The next proposition establishes a connection between toric actions and quasi-isomorphisms.

\begin{proposition}\label{p:quasi_toric}
    Suppose that $\mathcal{I}(\kappa) = \{m\}$ for some $m \in [1,N]$ and that $\gc$ and $\widetilde{\gc}$ admit global toric actions of rank $1$. Let $\omega$ and $\tilde{\omega}$ be the weight vectors of the toric actions in the initial extended clusters. Assume the following conditions:
    \begin{enumerate}[1)]
        \item For every $i \in [1,\tilde{N}]$ and $r \in [0,d_i]$, the Laurent monomials $\hat{p}_{ir}$ and $\tilde{\hat{p}}_{ir}$ are invariant with respect to the toric actions and $\kappa(\hat{p}_{ir}) = \tilde{\hat{p}}_{\kappa(i)r}$;
        \item For each $i \in [1,N+M]$, $(\omega_i-\tilde{\omega}_{\kappa(i)})\tilde{\omega}_{\kappa(m)}^{-1} \in \mathbb{Z}$ and $w_{m} = \tilde{w}_{\kappa(m)}$;
        \item The following equality holds:
        \begin{equation*}
        [\kappa^{-1}(\tilde{B})]_{[1,N]\setminus\{m\}}^{[1,N+M]\setminus\{m\}} = B^{[1,N+M]\setminus\{m\}}_{[1,N]\setminus\{m\}}.
        \end{equation*}
    \end{enumerate}
    Set
    \begin{equation*}
            \Lambda_i(\mathbf{x}_0):=\Lambda_i:=\begin{cases}
            \dfrac{w_i-\tilde{w}_{\kappa(i)}}{\tilde{w}_{\kappa(m)}} \ &\text{if} \ i \neq m\\[3pt]
            1 \ &\text{if}\ i = m.
        \end{cases}
    \end{equation*}
    Then $\Lambda$ defines a quasi-isomorphism $\mathcal{Q}:\gc\rightarrow\widetilde{\gc}$. Moreover, if $(\mathbf{x},\tilde{\mathbf{x}})$ is a pair of related extended clusters and $w^\prime$ and $\tilde{w}^\prime$ are the corresponding weight vectors of the toric actions, then the corresponding matrix $\Lambda(\mathbf{x})$ is given by
    \begin{equation*}
            \Lambda_i(\mathbf{x}):=\begin{cases}
            \dfrac{w_i^\prime-\tilde{w}_{\kappa(i)}^\prime}{\tilde{w}_{\kappa(m)}} \ &\text{if} \ i \neq m\\[3pt]
            1 \ &\text{if}\ i = m.
        \end{cases}
    \end{equation*}
\end{proposition}
An example of a quasi-isomorphism that arises from toric actions is given in Section~\ref{s:bsimpl}.  For other examples, see \cite[Proposition 5.3]{multdouble}.

\paragraph{Birational quasi-isomorphisms.}  In addition to the setup of the previous paragraph, we now assume that both $\gc$ and $\widetilde{\gc}$ are generalized cluster structures on some complex rational quasi-affine varieties $V$ and $\widetilde{V}$, respectively.

\begin{definition}\label{d:birat}
A birational map $\mathcal{Q}:\widetilde{V}\dashrightarrow V$ is called a \emph{birational quasi-isomorphism} if the following conditions hold:
\begin{enumerate}[1)]
    \item The marked variables in $\gc$ and $\widetilde{\gc}$ are regular;
    \item The map $\mathcal{Q}$ induces an isomorphism of the $\mathbb{C}$-algebras:\footnote{By $\mathbf{x}_0(\kappa)^{\pm 1}$ we mean the localization at each $x_i \in \mathbf{x}_0(\kappa)$.}
    \begin{equation*}
        \mathcal{Q}^*:\mathbb{C}[V][\mathbf{x}_0(\kappa)^{\pm 1}]\xrightarrow{\sim}\mathbb{C}[\widetilde{V}][\tilde{\mathbf{x}}_0(\kappa)^{\pm 1}];
    \end{equation*}
    \item The field isomorphism $\mathcal{Q}^*:\mathbb{C}(V)\xrightarrow{\sim}\mathbb{C}(\widetilde{V})$ restricts to a quasi-isomorphism 
    \[
    \mathcal{Q}^*:\mathcal{F}_{\mathbb{C}}(\gc) \rightarrow \mathcal{F}_{\mathbb{C}}(\widetilde{\gc}).
    \]
\end{enumerate}
\end{definition} 
We typically write $\mathcal{Q}:(\widetilde{V},\widetilde{\gc}) \dashrightarrow (V,\gc)$ to express the fact that $\mathcal{Q}$ is a birational quasi-isomorphism. 

\begin{definition}\label{d:iadmissible}
    Let $A$ be a subset of initial variables. We say that $A$ is \emph{i-admissible} if it satisfies the following conditions:
    \begin{enumerate}[1)]
        \item All variables in $A$ are regular and irreducible;
        \item For each cluster variable $x_i \in A$, $x_i^\prime$ is regular and coprime with $x_i$.
    \end{enumerate}
\end{definition}
Likewise, we define an i-admissible set of initial variables in $\widetilde{\gc}$.

\begin{proposition}\label{p:birat_fsingle}
    Le $\mathbb{C}[V]$ and $\mathbb{C}[\widetilde{V}]$ be factorial\footnote{We mean that, as a commutative ring, it is a UFD.} $\mathbb{C}$-algebras such that $\mathbb{C}[V]^* = \mathbb{C}[\widetilde{V}]^* = \mathbb{C}^*$, and let $\gc$ and $\widetilde{\gc}$ be generalized cluster structures on $V$ and $\widetilde{V}$, respectively, with fixed initial extended seeds $(\mathbf{x}_0,B_0,\mathcal{P}_0)$ and $(\tilde{\mathbf{x}}_0,\tilde{B}_0,\tilde{\mathcal{P}}_0)$. Assume the following:
     \begin{enumerate}[1)]
         \item There is a birational quasi-isomorphism $\mathcal{Q}:(\widetilde{V},\widetilde{\gc})\dashrightarrow (V,\gc)$;\label{i:fbirsing1}
         \item The initial extended cluster $\tilde{\mathbf{x}}_0$ is i-admissible;\label{i:fbirsing2}
         \item The set of marked variables $\mathbf{x}_0(\kappa)$ is i-admissible;\label{i:fbirsing3}
         \item The initial extended cluster $\mathbf{x}_0$ is regular;\label{i:fbirsing4}
        \item Each nonmarked cluster variable in $\mathbf{x}_0$ is coprime with each marked variable.\label{i:fbirsing5}
    \end{enumerate}
    Then $\mathbf{x}_0$ is i-admissible, as well as $\bar{\mathcal{A}}_{\mathbb{C}}(\widetilde{\gc}) \subseteq \mathbb{C}[\widetilde{V}]$ and $\bar{\mathcal{A}}_{\mathbb{C}}(\gc) \subseteq \mathbb{C}[V]$.
\end{proposition}

\begin{proposition}\label{p:upper_cont1}
    Le $\mathbb{C}[V]$ and $\mathbb{C}[\widetilde{V}]$ be factorial $\mathbb{C}$-algebras such that $\mathbb{C}[V]^* = \mathbb{C}[\widetilde{V}]^* = \mathbb{C}^*$, and let $\gc$ and $\widetilde{\gc}$ be generalized cluster structures on $V$ and $\widetilde{V}$, respectively, with fixed initial extended seeds $(\mathbf{x}_0,B_0,\mathcal{P}_0)$ and $(\tilde{\mathbf{x}}_0,\tilde{B}_0,\tilde{\mathcal{P}}_0)$. Assume the following:
    \begin{enumerate}[1)]
        \item There is a birational quasi-isomorphism $\mathcal{Q}:(\widetilde{V},\widetilde{\gc})\dashrightarrow (V,\gc)$;
        \item The matrices $B_0$ and $\tilde{B}_0$ are of full rank;
        \item The reverse inclusion $\bar{\mathcal{A}}_{\mathbb{C}} (\widetilde{\gc})\supseteq \mathbb{C}[\widetilde{V}]$ holds;
        \item For each $i \in \mathcal{I}(\kappa)$, $\mathcal{L}(\mathbf{x}_i^\prime)\supseteq \mathbb{C}[V]$.
    \end{enumerate}
    Then $\bar{\mathcal{A}}_{\mathbb{C}}(\gc)\supseteq \mathbb{C}[V]$.
\end{proposition}

Let $(\widehat{V},\widehat{\gc})$ be another pair of a complex rational quasi-affine variety~$\widehat{V}$ equipped with a generalized cluster structure $\widehat{\gc}$. 

\begin{definition}\label{d:birat_compl}
    Assume that $(\gc,\widetilde{\gc},\kappa)$ and $(\gc,\widehat{\gc},\kappa^\prime)$ are related triples and that there exist birational quasi-isomorphisms $\mathcal{Q}:(\widetilde{V},\widetilde{\gc}) \dashrightarrow (V,{\gc})$ and $\mathcal{Q}^{\prime}:(\widehat{V},\widehat{\gc}) \dashrightarrow (V,{\gc})$ such that $\mathcal{I}(\kappa) \cap \mathcal{I}(\kappa^\prime) = \emptyset$. Then $\mathcal{Q}$ and $\mathcal{Q}^{\prime}$ are called \emph{complementary}.
\end{definition}

\begin{proposition}\label{p:birat_double_ufd}
    Let $\mathbb{C}[V]$, $\mathbb{C}[\widetilde{V}]$, and $\mathbb{C}[\widehat{V}]$ be factorial $\mathbb{C}$-algebras such that $\mathbb{C}[V]^* = \mathbb{C}[\widetilde{V}]^* = \mathbb{C}[\widehat{V}]^* = \mathbb{C}^*$, and let $\gc$, $\widetilde{\gc}$, and $\widehat{\gc}$ be generalized cluster structures on $V$, $\widetilde{V}$, and $\widehat{V}$, respectively, with fixed initial extended seeds $(\mathbf{x}_0,B_0,\mathcal{P}_0)$, $(\tilde{\mathbf{x}}_0,\tilde{B}_0,\tilde{\mathcal{P}}_0)$, and $(\hat{\mathbf{x}}_0,\hat{B}_0,\hat{\mathcal{P}}_0)$. Assume the following:
    \begin{enumerate}[1)]
        \item There is a pair of complementary birational quasi-isomorphisms $\mathcal{Q}:(\widetilde{V},\widetilde{\gc}) \dashrightarrow (V,{\gc})$ and $\mathcal{Q}^{\prime}:(\widehat{V},\widehat{\gc}) \dashrightarrow (V,{\gc})$;
        \item The initial extended clusters $\tilde{\mathbf{x}}_0$ and $\hat{\mathbf{x}}_0$ are i-admissible;\label{ii:fbirat_ufd2}
        \item The initial extended cluster $\mathbf{x}_0$ is regular.
    \end{enumerate}
    Then $\mathbf{x}_0$ is i-admissible, as well as $\bar{\mathcal{A}}_{\mathbb{C}}(\widehat{\gc}) \subseteq \mathbb{C}[\widehat{V}]$, $\bar{\mathcal{A}}_{\mathbb{C}}(\widetilde{\gc}) \subseteq \mathbb{C}[\widetilde{V}]$ and $\bar{\mathcal{A}}_{\mathbb{C}}(\gc) \subseteq \mathbb{C}[V]$. 
\end{proposition}

\begin{proposition}\label{p:upper_cont2}
    Let $\mathbb{C}[V]$, $\mathbb{C}[\widetilde{V}]$, and $\mathbb{C}[\widehat{V}]$ be factorial $\mathbb{C}$-algebras such that $\mathbb{C}[V]^* = \mathbb{C}[\widetilde{V}]^* = \mathbb{C}[\widehat{V}]^* = \mathbb{C}^*$, and let $\gc$, $\widetilde{\gc}$, and $\widehat{\gc}$ be generalized cluster structures on $V$, $\widetilde{V}$, and $\widehat{V}$, respectively, with fixed initial extended seeds $(\mathbf{x}_0,B_0,\mathcal{P}_0)$, $(\tilde{\mathbf{x}}_0,\tilde{B}_0,\tilde{\mathcal{P}}_0)$, and $(\hat{\mathbf{x}}_0,\hat{B}_0,\hat{\mathcal{P}}_0)$. Assume the following:
        \begin{enumerate}[1)]
            \item There is a pair of complementary birational quasi-isomorphisms $\mathcal{Q}:(\widetilde{V},\widetilde{\gc}) \dashrightarrow (V,{\gc})$ and $\mathcal{Q}^{\prime}:(\widehat{V},\widehat{\gc}) \dashrightarrow (V,{\gc})$;
            \item The matrices $\tilde{B}_0$, $\hat{B}_0$ and $B_0$ are of full rank;
            \item The inclusions $\bar{\mathcal{A}}_{\mathbb{C}}(\widetilde{\gc}) \supseteq \mathbb{C}[\widetilde{V}]$ and $\bar{\mathcal{A}}_{\mathbb{C}}(\hat{\gc}) \supseteq \mathbb{C}[\widehat{V}]$ hold.
        \end{enumerate}
        Then $\bar{\mathcal{A}}_{\mathbb{C}}(\gc) \supseteq \mathbb{C}[V]$.
\end{proposition}
\section{\texorpdfstring{Poisson properties of the map  $\mathcal{Q}$}{Poisson properties of the map Q}}\label{s:comp}
Let $G$ be a connected complex reductive algebraic group and $\mathfrak{g}$ be its Lie algebra with a fixed set of simple roots $\Pi$ and a nondegenerate invariant symmetric bilinear form $\langle\cdot,\cdot\rangle$. We decompose $\mathfrak{g}$ as $\mathfrak{g}:=\mathfrak{n}_- \oplus \mathfrak{h} \oplus \mathfrak{n}_+$, where $\mathfrak{n}_{\pm}$ are nilpotent subalgebras and $\mathfrak{h}$ is a Cartan subalgebra. We set $\mathfrak{b}_{\pm}:=\mathfrak{n}_{\pm} \oplus \mathfrak{h}$, and we define $\mathcal{B}_{\pm}$, $\mathcal{H}$, and $\mathcal{N}_{\pm}$ to be the connected subgroups of $G$ that correspond to $\mathfrak{b}_{\pm}$, $\mathfrak{h}$, and $\mathfrak{n}_{\pm}$, respectively.  For a generic element $U \in G$, we set $U_{\oplus} \in \mathcal{B}_+$ and $U_- \in \mathcal{N}_-$ to be the elements such that $U = U_{\oplus}U_-$. Let $\bg:=(\Gamma_1,\Gamma_2,\gamma)$ be a BD triple and $\bar{\bg}:=(\bg,r_0)$ be a BD quadruple. Recall that $\gamma : \Gamma_1 \rightarrow \Gamma_2$ defines a Lie algebra homomorphism $\gamma^*:\mathfrak{n}_-\rightarrow\mathfrak{n}_-$, which in turn integrates to the group homomorphism $\tilde{\gamma}^* : \mathcal{N}_- \rightarrow \mathcal{N}_-$. We set $\pi_{(\bg,r_0)}^{\dagger}$ and $\pi_{(\bg_{\std},r_0)}^{\dagger}$ to be the Poisson bivectors~\eqref{eq:pidag} associated with the BD quadruples $(\bg,r_0)$ and $(\bg_{\std},r_0)$, respectively. We define the rational map $\mathcal{Q}:(G,\pi_{(\bg_{\std},r_0)}^{\dagger}) \dashrightarrow (G,\pi_{(\bg,r_0)}^{\dagger})$ via
\begin{equation}\label{eq:qdef}
    \mathcal{Q}(U) := \rho(U)^{-1}U \rho(U), \ \ \rho(U) := \prod_{i=1}^{\rightarrow} \tilde{\gamma}^*(U_-), \ U \in G.
\end{equation}
Note that since $\gamma^*$ is nilpotent, the above product is finite. The objective of this section is to prove that the map $\mathcal{Q}$ is a Poisson map; that is, $\mathcal{Q}_*(\pi_{(\bg_{\std},r_0)}^{\dagger}) = \pi_{(\bg,r_0)}^{\dagger}$. In Section~\ref{s:birat}, we will show that $\mathcal{Q}$ is a birational quasi-isomorphism when $G \in\{\SL_n,\GL_n\}$.\footnote{We expect that for a simply connected simple complex algebraic group $G$, there exists $\gc^{\dagger}(\bg,G)$ compatible with $\pi_{(\bg,r_0)}^{\dagger}$ and such that $\mathcal{Q}$ is a birational quasi-isomorphism as well.} In Section~\ref{s:complet}, we will apply the obtained result in order to conclude that $\gc_h^{\dagger}(\bg)$ is compatible with~$\pi_{(\bg,r_0)}^{\dagger}$.

\paragraph{The choice of trivialization.} In this section, all computations are performed in the right trivialization of $G$. In particular, given a smooth map $f : G \rightarrow H$ between Lie groups $G$ and $H$, its differential $d^R f$ in the right trivialization is given by
\begin{equation*}
    d_U^R f(x) := \dif f(\exp(tx)U)f(U)^{-1}, \ \ U \in G, \ x \in \mathfrak{g},
\end{equation*}
where $I$ is the identity element.
The following properties of $d^R$ will be used below: if $G,H,K$ are Lie groups and  $g:G \rightarrow H$ and $h:H \rightarrow K$ are smooth maps, then the product rule for $d^R$ reads $d_U^R[f\cdot g](x) = d_U^Rf(x) + \Ad_{f(U)}d_U^R g(x)$, and the chain rule $d_U^R[h \circ f](x) = d_{f(U)}^Rh\circ d_U^Rf(x)$, $x \in \mathfrak{g}$. We also set
\begin{equation*}
    \Ad_U^{\circ} := \Ad_U - \id, \ \ U \in G.
\end{equation*}
In the right trivialization, the Poisson bivector $\pi_{(\bg,r_0)}^{\dagger}$ is given by
\begin{equation}\label{eq:pigr}
    (\pi_{(\bg,r_0)}^{\dagger,R})_U = [\Ad_{U}^{\circ} \otimes \Ad_U^{\circ}](r_{(\bg,r_0)}) - [\Ad_U^\circ \otimes \Ad_U](\Omega).
\end{equation}

\subsection{Preliminary identities}
In this subsection, we list various auxiliary identities that are referenced in the proofs below. We let $\pi_{>}$, $\pi_0$, $\pi_{<}$, $\pi_{\geq}$, $\pi_{\leq}$ to be the orthogonal projections upon $\mathfrak{n}_+$, $\mathfrak{h}$, $\mathfrak{n}_-$, $\mathfrak{b}_+$ and $\mathfrak{b}_-$, respectively, and we let $\Omega_{>}$, $\Omega_0$, $\Omega_{<}$, $\Omega_{\geq}$ and $\Omega_{\leq}$ to be the $2$-tensors corresponding to the projections with respect to the chosen form on $\mathfrak{g}$ (for instance, $\Omega_{>} = [1\otimes \pi_{>}]\Omega$).

\begin{lemma}
For any $N \in \mathcal{N}_-$, the following identities hold:
\begin{align}\label{eq:adpd}
&[(\pi_0 \circ \Ad_{N}) \otimes 1](\Omega_{<}) = [1 \otimes (\pi_{<} \circ \Ad_{N^{-1}})](\Omega_0);
\\
\label{eq:admp}
 &[(\pi_{\geq} \circ \Ad_{N}) \otimes \Ad_{N}](\Omega_{<}) + [1\otimes (\pi_{<} \circ \Ad_{N})](\Omega_0) = \Omega_{<};
 \\ \label{eq:admn}
 &[\Ad_N \otimes \Ad_N](\Omega_{>}) + [(\pi_{<} \circ \Ad_N) \otimes \Ad_N](\Omega_{\leq}) = \Omega_{>}.
\end{align}
\end{lemma}
\begin{proof}
Formula \eqref{eq:adpd} is straightforward. To verify~\eqref{eq:admp}, we first use~\eqref{eq:adpd} to derive
\[\begin{split}
[(\pi_0 \circ \Ad_N) \otimes \Ad_N](\Omega_{<}) = [1\otimes \Ad_{N}] \circ [1 \otimes (\pi_{<} \circ \Ad_{N^{-1}})](\Omega_0) =\Omega_0 - [1\otimes \Ad_{N}](\Omega_0);
\end{split}
\]
now, using the latter relation together with $[(\pi_{>} \circ \Ad_{N}) \otimes \Ad_{N}](\Omega_{<}) = \Omega_{<}$, we see that
\[\begin{split}
[(\pi_{\geq} \circ \Ad_{N}) \otimes \Ad_{N}](\Omega_{<}) + [1\otimes (\pi_{<}& \circ \Ad_{N})](\Omega_0) = \\ = &\Omega_{<}  + \Omega_0 - [1\otimes \Ad_{N}](\Omega_0) + [1\otimes (\pi_{<} \circ \Ad_{N})](\Omega_0).
\end{split}
\]
The last three terms cancel each other out, and hence~\eqref{eq:admp} holds. Formula~\eqref{eq:admn} follows after rewriting $\Omega_{\leq}$ as $\Omega_{\leq} = \Omega - \Omega_{>}$ and noticing that $\Omega_{>} \in \mathfrak{n}_- \otimes \mathfrak{n}_+$ and $[(\pi_{<} \circ \Ad_N) \otimes \Ad_N](\Omega) = \Omega_{>}$.
\end{proof}

Let us define rational maps
\begin{align*}
\mathcal{B}: G \dashrightarrow G,& \ \ \ \  \mathcal{B}(U) = U_\oplus;\\
\mathcal{N}: G \dashrightarrow G,& \ \ \ \ \mathcal{N}(U) = U_-.
\end{align*}
\begin{lemma}\label{l:dum}
The following identities hold:
\begin{align}
    &d_U^R\mathcal{N}\Ad_U = \Ad_{U_-}\pi_{<} + \pi_{<}\Ad_{U_-}\pi_{>} + \Ad^{\circ}_{U_-}\pi_0;\label{eq:dnl}\\
    &d_U^R \mathcal{N}\pi_{>} = 0.\label{eq:dnr}
\end{align}
\end{lemma}
\begin{proof}
Differentiating $\mathcal{N}(U) = [\mathcal{B}(U)]^{-1}\cdot U$, we see that $d_U^R \mathcal{N}(\cdot) = \pi_{<}[\Ad_{\sinv{U_{\oplus}}}(\cdot)]$. The identities follow from the latter formula.
\end{proof}

\subsection{First computations}
We continue in the setup of the section. Here, we compute some of the first formulas for the proof that $\mathcal{Q}$ is a Poisson map. 

\paragraph{The Poisson bivectors.} We will work with the Poisson bivectors in the right trivialization of $G$ in the following form:

\begin{equation}\label{eq:pidag1}\begin{split}
(\pi^{\dagger,R}_{(\bg,r_0)})_U = &[\Ad_{U}^{\circ} \otimes \Ad_U](r_0) - [\Ad_{U}^{\circ} \otimes \Ad_U](\Omega_0) + \\ + &[\Ad_{U}^{\circ} \otimes \Ad_U]\left[ \frac{\gamma^*}{1-\gamma^*}\otimes 1\right](\Omega_{>}) - [\Ad_{U}^{\circ} \otimes 1]\left[\frac{1}{1-\gamma^*}\otimes 1\right](\Omega_{>}) - \\ - &[\Ad_{U}^{\circ} \otimes \Ad_U]\left[1\otimes \frac{1}{1-\gamma^*}\right](\Omega_{<}) + [\Ad_{U}^{\circ} \otimes 1]\left[1\otimes \frac{\gamma^*}{1-\gamma^*}\right](\Omega_{<});
\end{split}
\end{equation}
\begin{equation}\label{eq:pidagstd1}
\begin{split}
(\pi^{\dagger,R}_{(\bg_{\std},r_0)})_U =  [\Ad_{U}^{\circ} \otimes \Ad_{U}^{\circ}](r_0) - [\Ad_{U}^{\circ}& \otimes \Ad_U](\Omega_0) - \\ & - [\Ad_U^{\circ} \otimes \Ad_U](\Omega_{<}) - [\Ad_U^{\circ} \otimes 1](\Omega_{>}).
\end{split}
\end{equation}
The objective is to prove that $\mathcal{Q}_*\left((\pi_{(\bg_{\std},r_0)}^{\dagger,R})_U\right) = (\pi^{\dagger,R}_{(\bg,r_0)})_{\mathcal{Q}(U)}$.

\begin{lemma}
    The following identities hold:
    \begin{equation}\label{eq:qswap}
    \begin{aligned}
        \adir \Ad_U &= \Ad_{\mathcal{Q}(U)}\adir;\\
        \adir \Ad_U^\circ &= \Ad_{\mathcal{Q}(U)}^{\circ} \adir.
    \end{aligned}
    \end{equation}
\end{lemma}
\begin{proof}
    The identities follow from the definition of $\mathcal{Q}$ (see formula~\eqref{eq:qdef}).
\end{proof}

\begin{lemma}\label{l:rhodif}
The differential of the map $\rho(U) = \prod_{i=1}^{\rightarrow} (\tilde{\gamma}^*)^{i}(U_-)$ is given by
\begin{equation}\label{eq:rhodif}
\adir d_U^R \rho(\cdot)  = \frac{1}{1-\gamma^*} \Ad_{\rho(U)^{-1}} \gamma^*(d_U^R\mathcal{N}(\cdot)).
\end{equation}
\end{lemma}
\begin{proof}
Recall that $\rho(U)$ satisfies the relation
\begin{equation}\label{eq:rhorel}
\rho(U) = \tilde{\gamma}^*(U_-)\tilde{\gamma}^*(\rho(U)).
\end{equation}
Differentiating~\eqref{eq:rhorel} and applying $\adir$ to both sides, we see that
\begin{equation}\label{eq:rhofir}
\adir d_U^R \rho = \adir \gamma^* d_U^R \mathcal{N} + \Ad_{\rho(U)^{-1}\tilde{\gamma}^*(U_-)}\gamma^* d_U^R \rho.
\end{equation}
The last term can be rewritten as
\begin{equation}\label{eq:adrho}
    \Ad_{\rho(U)^{-1}\tilde{\gamma}^*(U_-)}\gamma^* d_U^R \rho = \Ad_{\tilde{\gamma}^*(\rho(U))}\gamma^* d_U^R\rho = \gamma^*\adir d_U^R \rho.
\end{equation}
Plugging~\eqref{eq:adrho} back into~\eqref{eq:rhofir} and solving for $d_U^R\rho$, the result follows.
\end{proof}

\begin{lemma}
The differential of $\mathcal{Q}$ is given by
\begin{equation}\label{eq:dq}
d_U^R \mathcal{Q} = \Ad_{\rho(U)^{-1}} + \Ad_{\mathcal{Q}(U)}^{\circ} \frac{1}{1-\gamma^*}\left[ \Ad_{\rho(U)^{-1}} \gamma^* d_U^R \mathcal{N}(\cdot)  \right].
\end{equation}
\end{lemma}
\begin{proof}
Differentiating the product $\mathcal{Q}(U) = \rho(U)^{-1}U\rho(U)$, we see that
\[
d_U^R \mathcal{Q} = -\adir d_U^R \rho + \Ad_{\rho(U)^{-1}} + \adir \Ad_U d_U^R \rho.
\]
Rewriting the third term as $\adir \Ad_U d_U^R \rho = \Ad_{\mathcal{Q}(U)}\adir d_U^R \rho$ and applying Lemma~\ref{l:rhodif} to the first and the third terms, we obtain~\eqref{eq:dq}.
\end{proof}
\begin{lemma}\label{l:adgon}
    The following identities hold:
    \begin{align}\label{eq:adgon}
        \frac{1}{1-\gamma^*} \Ad_{\rho(U)^{-1}} \gamma^* d_U^R \mathcal{N}(\Ad_U \pi_{<}(\cdot))  &= \frac{\gamma^*}{1-\gamma^*}\Ad_{\rho(U)^{-1}}\pi_{<}(\cdot);\\
        \frac{1}{1-\gamma^*} \Ad_{\rho(U)^{-1}} \gamma^* d_U^R \mathcal{N}(\Ad_U \pi_{0}(\cdot))  &=\frac{1}{1-\gamma^*} \pi_{<} [\gamma^*,\adir]\pi_0(\cdot)\label{eq:adgonz}
    \end{align}
    where $[\gamma^*,\adir] = \gamma^*\adir  - \adir\gamma^*$.
\end{lemma}
\begin{proof}
The first identity follows from formula~\eqref{eq:dnl} and formula~\eqref{eq:rhorel}. As for the second formula, an application of~\eqref{eq:dnl} yields
    \[\begin{split}
    \adir \gamma^* d_U^R \mathcal{N}\Ad_U \pi_0 = &\adir \gamma^* \Ad_{U_-}^{\circ} \pi_0 = \adir \left( \Ad_{\tilde{\gamma}^*(U_-)}\gamma^* - \gamma^*\right) \pi_0 = \\ = &\left( \Ad_{\tilde{\gamma}^*(\rho(U)^{-1})}\gamma^* - \adir \gamma^*\right)\pi_0 = [\gamma^*,\adir]\pi_0.\qedhere
    \end{split}
    \]
\end{proof}

\begin{lemma}
The following formula holds:
\begin{equation}\label{eq:dqrr}\begin{split}
-\left[d_U^R \mathcal{Q}\otimes d_U^R \mathcal{Q}\right]&[\Ad_U^{\circ}\otimes \Ad_U](\Omega_{<}) = -D_0 - D_1-\\ = -&\left[\Ad_{\mathcal{Q}(U)}^{\circ}\otimes \Ad_{\mathcal{Q}(U)}\right]\left[\adir\otimes \adir\right](\Omega_{<}) - \\ - &\left[\Ad_{\mathcal{Q}(U)}^{\circ} \otimes \Ad_{\mathcal{Q}(U)}^{\circ}\right]\left[\adir \otimes \frac{\gamma^*}{1-\gamma^*}\adir\right](\Omega_{<})
\end{split}
\end{equation}
where
\begin{align}\label{eq:d0}
&D_0 := \left[\Ad_{\mathcal{Q}(U)}^{\circ}\otimes \Ad_{\mathcal{Q}(U)}\right]\left[\frac{1}{1-\gamma^*}\adir \gamma^*\pi_{<}\Ad_{U_-}\otimes \adir\right](\Omega_{<}) ;\\
&D_1 := \left[\Ad_{\mathcal{Q}(U)}^{\circ}\otimes \Ad_{\mathcal{Q}(U)}^\circ\right]\left[\frac{1}{1-\gamma^*}\otimes\frac{\gamma^*}{1-\gamma^*}\right][\adir\gamma^*\pi_{<}\Ad_{U_-}\otimes \adir ](\Omega_{<}).\label{eq:d1}
\end{align}
\end{lemma}
\begin{proof}
Recall that $\Omega_{<} \in \mathfrak{n}_+ \otimes \mathfrak{n}_-$. One applies formula~\eqref{eq:dq} (the differential of $\mathcal{Q}$) to $\Ad_U^{\circ}\otimes \Ad_U$, and then one invokes formulas~\eqref{eq:dnl} and \eqref{eq:dnr} (the differential of $\mathcal{N}$) with \eqref{eq:adgon} and \eqref{eq:qswap}.
\end{proof}
\begin{lemma}
    The following formula holds:
    \begin{equation}\label{eq:qqadu01}\begin{split}
         -[d_U^R\mathcal{Q}\otimes d_U^R \mathcal{Q}]&[\Ad_U^\circ \otimes 1](\Omega_{>})  = D_0-[\Ad_{\mathcal{Q}(U)}^{\circ}\otimes 1]\left[\frac{1}{1-\gamma^*}\adir \otimes \adir\right](\Omega_{>}) +\\&+[\Ad_{\mathcal{Q}(U)}^{\circ}\otimes \Ad_{\mathcal{Q}(U)}]\left[\frac{\gamma^*}{1-\gamma^*}\adir \otimes \adir\right](\Omega_{>}) + \\ &+[\Ad_{\mathcal{Q}(U)}^{\circ}\otimes\Ad_{\mathcal{Q}(U)}]\left[\frac{1}{1-\gamma^*}\pi_{<}[\gamma^*,\adir]\otimes \adir \right](\Omega_0),
    \end{split}
    \end{equation}
    where $D_0$ is defined in~\eqref{eq:d0}.
\end{lemma}
\begin{proof}
    Applying formula~\eqref{eq:dq} to $-[\Ad_U^\circ \otimes 1]$ and using formula~\eqref{eq:dnr}, we see that
    \begin{equation}\label{eq:qqadu010}
    \begin{split}
        -[d_U^R\mathcal{Q}\otimes d_U^R \mathcal{Q}][\Ad_U^\circ \otimes 1](\Omega_{>}) =&-[\Ad_{\mathcal{Q}(U)}^{\circ}\otimes 1]\left[\frac{1}{1-\gamma^*}\adir \otimes \adir\right](\Omega_{>}) + \\ &+[\Ad_{\mathcal{Q}(U)}^{\circ}\otimes 1]\left[\frac{1}{1-\gamma^*}\adir \gamma^*d_U^R\mathcal{N}\otimes \adir\right](\Omega).
    \end{split}
    \end{equation}
    Since $\Omega = [\Ad_U \otimes \Ad_U](\Omega)$, the last term can be expanded as
    \begin{equation}\label{eq:qqadu012}
    \begin{split}
        [\Ad_{\mathcal{Q}(U)}^{\circ}\otimes 1]&\left[\frac{1}{1-\gamma^*}\adir \gamma^*d_U^R\mathcal{N}\otimes \adir\right](\Omega) = D_0 + \\ + &[\Ad_{\mathcal{Q}(U)}^{\circ}\otimes \Ad_{\mathcal{Q}(U)}]\left[\frac{1}{1-\gamma^*}\adir \gamma^*\Ad_{U_-}\otimes \adir\right](\Omega_{>})  +\\+&[\Ad_{\mathcal{Q}(U)}^{\circ}\otimes \Ad_{\mathcal{Q}(U)}]\left[\frac{1}{1-\gamma^*}\adir \gamma^*\Ad^\circ_{U_-}\otimes \adir\right](\Omega_{0}).
    \end{split}
    \end{equation}
    Now one can apply~\eqref{eq:adgon} and~\eqref{eq:adgonz} to~\eqref{eq:qqadu012} and combine the result with formula~\eqref{eq:qqadu010}.
\end{proof}

\subsection{\texorpdfstring{Pushforward of the diagonal part of $\pi^{\dagger}$}{Pushforward of the diagonal part of pi dagger}}
The objective of this subsection is to compute the term
\begin{equation*}
\dqdq\left([\Ad_U^\circ\otimes \Ad_U^\circ](r_0)-[\Ad_U^\circ \otimes \Ad_U](\Omega_0)  \right).
\end{equation*}
Let us define the orthogonal projections $\pi_{\Gamma_1}:\mathfrak{g}\rightarrow\mathfrak{g}_{\Gamma_1}$ and $\pi_{\hat{\Gamma}_1}:\mathfrak{g}\rightarrow\mathfrak{g}_{\Gamma_1}^{\perp}$.
\begin{lemma}
The following identities hold:
\begin{align}\label{eq:ogr0}
    &[1 \otimes (1-\gamma^*)](r_0) = [1\otimes \pi_{\hat{\Gamma}_1}](r_0) - [1\otimes \gamma^*](\Omega_0);
\\ \label{eq:gor0}
    &[(1-\gamma^*)\otimes 1](r_0) = [\pi_{\Gamma_1}\otimes \pi_{\Gamma_1}](\Omega_0) + [\pi_{\hat{\Gamma}_1}\otimes 1](r_0);
\\ \label{eq:ggr0}
    &[(1-\gamma^*)\otimes(1-\gamma^*)](r_0) = [\pi_{\hat{\Gamma}_1}\otimes \pi_{\hat{\Gamma}_1}](r_0) + [\pi_{\Gamma_1}\otimes\pi_{\Gamma_1}](\Omega_0) - [1\otimes \gamma^*](\Omega_0).
\end{align}
\end{lemma}
\begin{proof}
    The formulas follow from~\eqref{eq:ralg}.
\end{proof}
\begin{lemma}
    The following identities hold:
    \begin{align}\label{eq:dqad0}
    &d_U^R\mathcal{Q} \Ad_U^ \circ \pi_0 = \Ad_{\mathcal{Q}(U)}^{\circ} \pi_0 + \Ad_{\mathcal{Q}(U)}^{\circ} \frac{1}{1-\gamma^*}\pi_{<} \adir (1-\gamma^*)\pi_0;\\ \label{eq:dql0}
    &d_U^R\mathcal{Q}\Ad_U \pi_0 = \Ad_{\rho(U)^{-1}}\pi_0 + d_U\mathcal{Q} \Ad_U^\circ \pi_0.
    \end{align}
\end{lemma}
\begin{proof}
    It follows from Lemma~\ref{l:dum} that $d^R_U \mathcal{N}  \Ad_U^\circ \pi_0 = d^R_U \mathcal{N} \Ad_U \pi_0$; combining this with formulas~\eqref{eq:adgonz}, 
 \eqref{eq:dq} and \eqref{eq:qswap}, we see that
    \begin{equation*}\begin{split}
d_U^R\mathcal{Q} \Ad_U^ \circ \pi_0 = \Ad_{\mathcal{Q}(U)}^{\circ}\adir & \pi_0 + \Ad_{\mathcal{Q}(U)}^{\circ} \frac{\gamma^*}{1-\gamma^*} \pi_{<} \adir \pi_0 -\\- &\Ad_{\mathcal{Q}(U)}^{\circ} \frac{1}{1-\gamma^*}\pi_{<} \adir \gamma^* \pi_0.     
\end{split}
    \end{equation*}
    Now, formula~\eqref{eq:dqad0} follows after noticing that
    \begin{equation*}
    \Ad_{\mathcal{Q}(U)}^{\circ}\adir \pi_0 = \Ad_{\mathcal{Q}(U)}^{\circ}\pi_{<}\adir \pi_0 + \Ad_{\mathcal{Q}(U)}^{\circ}\pi_0
    \end{equation*}
    and combining the terms. As for formula~\eqref{eq:dql0}, it follows from the fact that $d_U^R\mathcal{N}\pi_0 = 0$ and formula~\eqref{eq:dq}.
\end{proof}

\begin{lemma}
    The following identity holds:
    \begin{equation}\label{eq:dqdqr0}
    \begin{split}
    [d_U^R&\mathcal{Q} \otimes d^R_U \mathcal{Q}][\Ad_U^\circ \otimes \Ad_U^\circ](r_0) = [\Ad_{\mathcal{Q}(U)}^{\circ}\otimes\Ad_{\mathcal{Q}(U)}^{\circ}](r_0) -\\- &[\Ad_{\mathcal{Q}(U)}^{\circ}\otimes\Ad_{\mathcal{Q}(U)}^{\circ}]\oog[1\otimes \pi_{<}\adir][1\otimes \gamma^*](\Omega_0) + \\ + &[\Ad_{\mathcal{Q}(U)}^{\circ}\otimes\Ad_{\mathcal{Q}(U)}^{\circ}]\goo[\pi_{<}\adir \otimes 1][\pi_{\Gamma_1}\otimes \pi_{\Gamma_1}](\Omega_0) - \\ - &[\Ad_{\mathcal{Q}(U)}^{\circ}\otimes\Ad_{\mathcal{Q}(U)}^{\circ}] \left[\frac{1}{1-\gamma^*}\otimes\frac{1}{1-\gamma^*} \right][\pi_{<}\adir \otimes \pi_{<}\adir ][(1-\gamma^*)\otimes \gamma^*](\Omega_0).
    \end{split}
    \end{equation}
\end{lemma}
\begin{proof}
It follows from formula~\eqref{eq:dqad0} that 
\begin{equation}\label{eq:dqdqr0_1}
    \begin{split}
        [d_U^R\mathcal{Q}&\otimes d_U^R\mathcal{Q}] [\Ad_U^\circ \otimes \Ad_U^\circ](r_0) = [\Ad_{\mathcal{Q}(U)}^{\circ}\otimes\Ad_{\mathcal{Q}(U)}^{\circ}](r_0) + \\ &+[\Ad_{\mathcal{Q}(U)}^{\circ}\otimes\Ad_{\mathcal{Q}(U)}^{\circ}] \oog [1\otimes \pi_{<} \adir][1 \otimes (1-\gamma^*)](r_0) + \\ &+ [\Ad_{\mathcal{Q}(U)}^{\circ}\otimes\Ad_{\mathcal{Q}(U)}^{\circ}] \goo [\pi_{<}\adir \otimes 1][(1-\gamma^*)\otimes 1](r_0) + \\ &+ [\Ad_{\mathcal{Q}(U)}^{\circ}\otimes\Ad_{\mathcal{Q}(U)}^{\circ}] \left[\frac{1}{1-\gamma^*}\otimes \frac{1}{1-\gamma^*}\right][\pi_{<}\adir \otimes \pi_{<} \adir]\circ\\ &\hspace{3in}\circ [(1-\gamma^*)\otimes (1-\gamma^*)](r_0).
    \end{split}
\end{equation}
Now formula~\eqref{eq:dqdqr0} follows from applying the identities~\eqref{eq:ogr0}, \eqref{eq:gor0} and \eqref{eq:ggr0} to formula~\eqref{eq:dqdqr0_1} and combining the terms (in the process, note also that $\pi_{<}\adir \pi_{\hat{\Gamma}_1}\pi_0 = 0$).
\end{proof}

\begin{lemma}
    The following identity holds:
    \begin{equation}\label{eq:dqdqdiag}\begin{split}
        \dqdq&\left([\Ad_U^\circ\otimes \Ad_U^\circ](r_0)-[\Ad_U^\circ \otimes \Ad_U](\Omega_0)\right) = \\ = &[\Ad_{\mathcal{Q}(U)}^{\circ}\otimes\Ad_{\mathcal{Q}(U)}^{\circ}] (r_0) - [\Ad_{\mathcal{Q}(U)}^{\circ}\otimes\Ad_{\mathcal{Q}(U)}](\Omega_0) - D_2- \\-&[\Ad_{\mathcal{Q}(U)}^{\circ}\otimes\Ad_{\mathcal{Q}(U)}]\oog [1\otimes \pi_{<}\adir](\Omega_0) + \\ + &[\Ad_{\mathcal{Q}(U)}^{\circ}\otimes1] \ogg [1\otimes \pi_{<}\adir](\Omega_0) + \\ +&[\Ad_{\mathcal{Q}(U)}^{\circ}\otimes\Ad_{\mathcal{Q}(U)}^{\circ}] \goo [\pi_{<}\adir \otimes 1][\gamma^*\otimes 1](\Omega_0) - \\ - &[\Ad_{\mathcal{Q}(U)}^{\circ}\otimes1]\goo [\pi_{<}\adir \otimes \adir][(1-\gamma^*)\otimes 1](\Omega_0)
        \end{split}
    \end{equation}
    where
    \begin{equation}\label{eq:d2}
    D_2 := [\Ad_{\mathcal{Q}(U)}^{\circ}\otimes\Ad_{\mathcal{Q}(U)}^{\circ}] \left[\frac{1}{1-\gamma^*}\otimes\frac{1}{1-\gamma^*} \right][\pi_{<}\adir \otimes \pi_{<}\adir ][(1-\gamma^*)\otimes 1](\Omega_0).
    \end{equation}
\end{lemma}
\begin{proof}
    Notice that
    \begin{equation}\label{eq:o0not}\begin{split}
        [\Ad_{\mathcal{Q}(U)}^{\circ}\otimes\Ad_{\mathcal{Q}(U)}^{\circ}](\Omega_0&) + [\Ad_{\mathcal{Q}(U)}^{\circ}\otimes1] [1\otimes \adir](\Omega_0)=\\ = &[\Ad_{\mathcal{Q}(U)}^{\circ}\otimes\Ad_{\mathcal{Q}(U)}](\Omega_0) +[\Ad_{\mathcal{Q}(U)}^{\circ}\otimes1][1\otimes \pi_{<}\adir](\Omega_0).
        \end{split}
    \end{equation}
    Let us expand $\dqdq[\Ad_U^\circ\otimes \Ad_U](\Omega_0)$ via~\eqref{eq:dqad0} and~\eqref{eq:dql0}, and apply~\eqref{eq:o0not} to the resulting expression. We obtain 
    
    \begin{equation}\label{eq:dqdqo0}
    \begin{split}
        -\dqdq &[\Ad_U^\circ \otimes \Ad_U](\Omega_0) =-[\Ad_{\mathcal{Q}(U)}^{\circ}\otimes\Ad_{\mathcal{Q}(U)}](\Omega_0) -\\ - &[\Ad_{\mathcal{Q}(U)}^{\circ}\otimes1][1\otimes \pi_{<}\adir](\Omega_0) - \\ - &[\Ad_{\mathcal{Q}(U)}^{\circ}\otimes\Ad_{\mathcal{Q}(U)}^{\circ}] \oog [1\otimes \pi_{<}\adir][1\otimes (1-\gamma^*)](\Omega_0) -\\-&[\Ad_{\mathcal{Q}(U)}^{\circ}\otimes\Ad_{\mathcal{Q}(U)}^{\circ}] \goo [\pi_{<}\adir \otimes 1][(1-\gamma^*)\otimes 1](\Omega_0) - \\ -&[\Ad_{\mathcal{Q}(U)}^{\circ}\otimes1] \goo [\pi_{<}\adir \otimes \adir][(1-\gamma^*)\otimes 1](\Omega_0)- \\ - &[\Ad_{\mathcal{Q}(U)}^{\circ}\otimes\Ad_{\mathcal{Q}(U)}^{\circ}] \left[\frac{1}{1-\gamma^*}\otimes \frac{1}{1-\gamma^*}\right][\pi_{<}\adir \otimes \pi_{<} \adir]\circ\\ & \hspace{2.4in}\circ [(1-\gamma^*)\otimes(1-\gamma^*)](\Omega_0).
        \end{split}
    \end{equation}
    Now, let us indicate how the terms in \eqref{eq:dqdqr0} and in~\eqref{eq:dqdqo0} combine into the terms of~\eqref{eq:dqdqdiag}. The term $-D_2$ in~\eqref{eq:dqdqdiag} is obtained via combining the last term in~\eqref{eq:dqdqr0} with the last term in~\eqref{eq:dqdqo0}. The fourth and the fifth terms in~\eqref{eq:dqdqdiag} are obtained via combining the second term in~\eqref{eq:dqdqr0} with the second and the third terms in~\eqref{eq:dqdqo0}. To obtain the sixth term in~\eqref{eq:dqdqdiag}, one combines the third term in~\eqref{eq:dqdqr0} with the fourth term in~\eqref{eq:dqdqo0}, together with the fact that $\pi_{<}\adir \pi_{\hat{\Gamma}_1}\pi_0 = 0$. The last term in~\eqref{eq:dqdqdiag} is the penultimate term in~\eqref{eq:dqdqo0}.
\end{proof}

\begin{lemma}
    The following identity holds:
    \begin{equation}\label{eq:specialterm}
    \begin{split}
        -D_1-D_2=&[\Ad_{\mathcal{Q}(U)}^{\circ}\otimes\Ad_{\mathcal{Q}(U)}^{\circ}][\pi_{<}\adir\otimes\adir](\Omega_{<}) +\\+&[\Ad_{\mathcal{Q}(U)}^{\circ}\otimes\Ad_{\mathcal{Q}(U)}^{\circ}] \ogg [\pi_{<}\adir \otimes \adir](\Omega_{<}) + \\ + &[\Ad_{\mathcal{Q}(U)}^{\circ}\otimes\Ad_{\mathcal{Q}(U)}^{\circ}] \ggo [\pi_{<}\adir \otimes \adir](\Omega_{<}) + \\ +&[\Ad_{\mathcal{Q}(U)}^{\circ}\otimes\Ad_{\mathcal{Q}(U)}^{\circ}]\goo [\pi_{<}\adir \otimes \pi_{<}\adir][\gamma^*\otimes 1](\Omega_0).
    \end{split}
    \end{equation}
    where $D_1$ is given by~\eqref{eq:d1} and $D_2$ is given by~\eqref{eq:d2}.
\end{lemma}
\begin{proof}
    For convenience, let us drop the operator $[\Ad_{\mathcal{Q}(U)}^{\circ}\otimes\Ad_{\mathcal{Q}(U)}^{\circ}]$ from the computations. A part of $D_1$ can be rewritten in the following way:
    \begin{equation}\label{eq:d1d2_1}
    \begin{split}
    \left[\frac{1}{1-\gamma^*}\right.\left.\otimes \frac{\gamma^*}{1-\gamma^*}\right]&[\adir \gamma^* \pi_{<} \Ad_{U_-} \otimes \adir](\Omega_{<}) = \\ = &\left[\frac{1}{1-\gamma^*}\otimes \frac{\gamma^*}{1-\gamma^*}\right][\pi_{<}\adir \gamma^*  \Ad_{U_-} \otimes \adir](\Omega_{<}) - \\ -& \left[\frac{1}{1-\gamma^*}\otimes \frac{\gamma^*}{1-\gamma^*}\right][\pi_{<}\adir \gamma^* \pi_{0} \Ad_{U_-} \otimes \adir](\Omega_{<}) - \\ -& \left[\frac{1}{1-\gamma^*}\otimes \frac{\gamma^*}{1-\gamma^*}\right][\pi_{<}\adir \gamma^* \pi_{>} \Ad_{U_-} \otimes \adir](\Omega_{<}).
    \end{split}
    \end{equation}    
    Now let us work with the three terms in the right-hand side of~\eqref{eq:d1d2_1}. To rewrite the first term, we apply formula~\eqref{eq:dnl} and~\eqref{eq:adgon}:
    \begin{equation}\label{eq:d1d2_11}\begin{split}
        \left[\frac{1}{1-\gamma^*}\otimes \frac{\gamma^*}{1-\gamma^*}\right]&[\pi_{<}\adir \gamma^*  \Ad_{U_-} \otimes \adir](\Omega_{<}) = \\= &\left[\frac{\gamma^*}{1-\gamma^*}\otimes \frac{\gamma^*}{1-\gamma^*}\right][\pi_{<}\adir \otimes \adir](\Omega_{<}).
        \end{split}
    \end{equation}
    Next, let us notice the following formula that follows from~\eqref{eq:rhorel}:
    \begin{equation}\label{eq:ggrhos}
        \frac{\gamma^*}{1-\gamma^*}\Ad_{\rho(U)^{-1}\rum }\pi_{<} = \frac{1}{1-\gamma^*}\adir \gamma^* \pi_{<}.
    \end{equation}
    For the second term in the right-hand side of~\eqref{eq:d1d2_1}, we apply formula~\eqref{eq:adpd}:
    \begin{equation}\label{eq:d1d2_12}\begin{split}
        - \left[\frac{1}{1-\gamma^*}\otimes \frac{\gamma^*}{1-\gamma^*}\right]&[\pi_{<}\adir \gamma^* \pi_{0} \Ad_{U_-} \otimes \adir](\Omega_{<}) = \\ =&-\left[\frac{1}{1-\gamma^*}\otimes \frac{\gamma^*}{1-\gamma^*}\right][\pi_{<}\adir\gamma^* \otimes \adir \pi_{<} \Ad_{\rum}](\Omega_0);
    \end{split} 
    \end{equation}
    then, we write 
    \[
    -\adir \pi_{<} \Ad_{\rum} \pi_0 = \pi_{<}\adir \pi_0  - \pi_{<} \Ad_{\rho(U)^{-1}\rum} \pi_0
    \]
    and apply formula~\eqref{eq:ggrhos}:
    \begin{equation}\label{eq:d1d2_122}\begin{split}
        - \left[\frac{1}{1-\gamma^*}\otimes \frac{\gamma^*}{1-\gamma^*}\right]&[\pi_{<}\adir \gamma^* \pi_{0} \Ad_{U_-} \otimes \adir](\Omega_{<}) = \\ =& \left[\frac{1}{1-\gamma^*} \otimes \frac{\gamma^*}{1-\gamma^*} \right][\pi_{<}\adir \otimes \pi_{<}\adir ][\gamma^*\otimes 1](\Omega_0) - \\ &-\left[\frac{1}{1-\gamma^*}\otimes \frac{1}{1-\gamma^*}\right][\pi_{<}\adir \otimes \pi_{<}\adir][\gamma^*\otimes \gamma^*](\Omega_0).
    \end{split} 
    \end{equation}
    To rewrite the third term in the right-hand side of~\eqref{eq:d1d2_1}, we apply the identity \[[\pi_{>}\Ad_{U_-} \otimes 1](\Omega_{<}) = [1\otimes \Ad_{\rum}](\Omega_{<})\] together with formula~\eqref{eq:ggrhos}:
    \begin{equation}\label{eq:d1d2_3}\begin{split}
        - \left[\frac{1}{1-\gamma^*}\otimes \frac{\gamma^*}{1-\gamma^*}\right]&[\pi_{<}\adir \gamma^* \pi_{>} \Ad_{U_-} \otimes \adir](\Omega_{<}) = \\ = &-\left[\frac{1}{1-\gamma^*}\otimes \frac{\gamma^*}{1-\gamma^*}\right][\pi_{<}\adir\gamma^*\otimes \Ad_{\rho(U)^{-1}\rum}](\Omega_{<}) = \\ =&-\left[\frac{1}{1-\gamma^*}\otimes \frac{1}{1-\gamma^*}\right][\pi_{<}\adir\otimes \Ad_{\rho(U)^{-1}}][\gamma^*\otimes\gamma^*](\Omega_{<}).
    \end{split}
    \end{equation}
    Combining formulas~\eqref{eq:d1d2_11}, \eqref{eq:d1d2_12} and \eqref{eq:d1d2_3}, we obtain
    \begin{equation}\begin{split}\label{eq:expd1}
         \left[\frac{1}{1-\gamma^*}\right.\left.\otimes \frac{\gamma^*}{1-\gamma^*}\right]&[\adir \gamma^* \pi_{<} \Ad_{U_-} \otimes \adir](\Omega_{<}) = \\ = &\left[\frac{\gamma^*}{1-\gamma^*}\otimes \frac{\gamma^*}{1-\gamma^*} \right][\pi_{<}\adir \otimes \adir](\Omega_<) + \\ + &\left[\frac{1}{1-\gamma^*}\otimes \frac{\gamma^*}{1-\gamma^*}\right][\pi_{<}\adir \otimes \pi_{<} \adir][\gamma^*\otimes 1](\Omega_0) - \\ -&\left[ \frac{1}{1-\gamma^*}\otimes \frac{1}{1-\gamma^*}\right][\pi_{<}\adir \otimes \pi_{<}\adir][\gamma^*\otimes \gamma^*](\Omega_{\leq}).
         \end{split}
    \end{equation}
    Now, let us turn to $D_2$:
    \begin{equation}\label{eq:d2first}\begin{split}
        -\left[\frac{1}{1-\gamma^*}\otimes \frac{1}{1-\gamma^*}\right]&[\pi_{<}\adir \otimes \pi_{<}\adir][(1-\gamma^*)\otimes 1](\Omega_0) = \\ = &-\left[\frac{1}{1-\gamma^*}\otimes \frac{1}{1-\gamma^*}\right][\pi_{<}\adir \otimes \pi_{<}\adir](\Omega_0) + \\ &+\left[\frac{1}{1-\gamma^*}\otimes \frac{\gamma^*}{1-\gamma^*}\right][\pi_{<}\adir \otimes \pi_{<}\adir][\gamma^*\otimes 1](\Omega_0) + \\ &+ \goo [\pi_{<}\adir \otimes \pi_{<}\adir][\gamma^*\otimes 1](\Omega_0).
        \end{split}
    \end{equation}
    The negative of~\eqref{eq:expd1} together with~\eqref{eq:d2first} combine into
    \begin{equation}\label{eq:d1d2comb}\begin{aligned}
        &-\left[\frac{1}{1-\gamma^*}\right.\left.\otimes \frac{\gamma^*}{1-\gamma^*}\right][\adir \gamma^* \pi_{<} \Ad_{U_-} \otimes \adir](\Omega_{<})  -\\ &- \left[\frac{1}{1-\gamma^*}\otimes \frac{1}{1-\gamma^*}\right][\pi_{<}\adir \otimes \pi_{<}\adir][(1-\gamma^*)\otimes 1](\Omega_0) = \\&\hspace{1.1in}=-\left[\frac{\gamma^*}{1-\gamma^*}\otimes \frac{\gamma^*}{1-\gamma^*}\right][\pi_{<}\adir \otimes \adir](\Omega_{<}) + \\  &\hspace{1.14in}\hphantom{=}+\left[\frac{1}{1-\gamma^*}\otimes \frac{1}{1-\gamma^*}\right][\pi_{<}\adir \otimes \pi_{<}\adir][\gamma^*\otimes \gamma^*](\Omega_{<}) + \\  &\hspace{1.14in}\hphantom{=}+ \goo [\pi_{<}\adir \otimes \pi_{<}\adir][\gamma^*\otimes 1](\Omega_0).
        \end{aligned}
    \end{equation}
    Let us write
    \begin{equation*}\begin{split}
        \left[\frac{1}{1-\gamma^*}\otimes\frac{1}{1-\gamma^*}\right]&[\pi_{\Gamma_1}\otimes\pi_{\Gamma_1}] - \left[\frac{\gamma^*}{1-\gamma^*}\otimes\frac{\gamma^*}{1-\gamma^*}\right] = \\ = &\left[\frac{1}{1-\gamma^*}\otimes\frac{1}{1-\gamma^*}\right][\pi_{\hat{\Gamma}_1}\otimes\pi_{\hat{\Gamma}_1}] + 1\otimes 1 + \ggo + \ogg,
        \end{split}
    \end{equation*}
    and let us observe that $\pi_{<}\adir \pi_{\hat{\Gamma}_1} \pi_{>} = 0$, so~\eqref{eq:d1d2comb} can be updated as
    \begin{equation}\label{eq:d1d2final}
    \begin{split}
        &-\left[\frac{1}{1-\gamma^*}\right.\left.\otimes \frac{\gamma^*}{1-\gamma^*}\right][\adir \gamma^* \pi_{<} \Ad_{U_-} \otimes \adir](\Omega_{<})  -\\ &- \left[\frac{1}{1-\gamma^*}\otimes \frac{1}{1-\gamma^*}\right][\pi_{<}\adir \otimes \pi_{<}\adir][(1-\gamma^*)\otimes 1](\Omega_0) = \\&\hspace{1.1in}= [\pi_{<}\adir\otimes\adir](\Omega_{<}) +\\&\hspace{1.14in}\hphantom{=}+ \ogg [\pi_{<}\adir \otimes \adir](\Omega_{<}) + \\  &\hspace{1.14in}\hphantom{=}+\ggo [\pi_{<}\adir \otimes \adir](\Omega_{<}) + \\ &\hspace{1.14in}\hphantom{=}+\goo [\pi_{<}\adir \otimes \pi_{<}\adir][\gamma^*\otimes 1](\Omega_0).
    \end{split}
    \end{equation}
    Applying $[\Ad_{\mathcal{Q}(U)}^{\circ}\otimes\Ad_{\mathcal{Q}(U)}^{\circ}]$, we finally obtain~\eqref{eq:specialterm}.
\end{proof}

\subsection{Proof of Proposition~\ref{p:qpoiss}}
Let us restate Proposition~\ref{p:qpoiss} for convenience:

\begin{proposition*}
For any BD quadruple $(\bg,r_0)$, the map $\mathcal{Q}:(G,\pi_{(\bg_{\std},r_0)}^{\dagger})\dashrightarrow (G,\pi_{(\bg,r_0)}^{\dagger})$ is Poisson.
\end{proposition*}
\begin{proof}
    We apply the pushforward $\mathcal{Q}_*$ to each term of~\eqref{eq:pidagstd1} and  combine the output in the form~\eqref{eq:pidag1}. The strategy is to collect the terms \eqref{eq:dqrr}, \eqref{eq:qqadu01}, \eqref{eq:dqdqdiag} and~\eqref{eq:specialterm}, and then apply either~\eqref{eq:admp} or~\eqref{eq:admn} to cancel out $\adir$ from the resulting expressions. Let us notice first that all the terms that end with $[\gamma^*\otimes 1](\Omega_0)$ cancel each other out (this includes a part of the last term in~\eqref{eq:qqadu01}). As for the rest, let us indicate the correspondence between the terms in the formulas~\eqref{eq:dqrr}, \eqref{eq:qqadu01}, \eqref{eq:dqdqdiag} and~\eqref{eq:specialterm}, and the terms in~\eqref{eq:pidag1}.
    \begin{enumerate}[1)]
        \item To obtain the term $[\Ad_{\mathcal{Q}(U)}^\circ\otimes\Ad_{\mathcal{Q}(U)}]\ggo(\Omega_{>})$, one combines the last two terms of \eqref{eq:qqadu01} with the $[\Ad_{\mathcal{Q}(U)}^{\circ}\otimes \Ad_{\mathcal{Q}(U)}]$ part of the penultimate term in~\eqref{eq:specialterm}, and then one applies~\eqref{eq:admn};
        \item To obtain the term $-[\Ad_{\mathcal{Q}(U)}^\circ\otimes 1]\goo(\Omega_{>})$, one combines the first term from~\eqref{eq:qqadu01},  the last term from~\eqref{eq:dqdqdiag} and the $[\Ad^{\circ}_{\mathcal{Q}(U)}\otimes 1]$ parts from the first and the penultimate terms of~\eqref{eq:specialterm}; then, one applies~\eqref{eq:admn};
        \item To obtain the term $-[\Ad_{\mathcal{Q}(U)}^\circ\otimes\Ad_{\mathcal{Q}(U)}]\oog(\Omega_{<})$, one combines the $[\Ad_{\mathcal{Q}(U)}^\circ\otimes\Ad_{\mathcal{Q}(U)}]$ parts of the first two terms in~\eqref{eq:dqrr}, the first two terms in~\eqref{eq:specialterm} and the fourth term in~\eqref{eq:dqdqdiag}; then, one applies~\eqref{eq:admp};
        \item To obtain the term $[\Ad_{\mathcal{Q}(U)}^\circ\otimes 1]\ogg(\Omega_{<})$, one combines the fifth term from \eqref{eq:dqdqdiag} together with the $\Ad_{\mathcal{Q}(U)}^\circ\otimes 1$ parts of the second term of~\eqref{eq:dqrr} and the second term of~\eqref{eq:specialterm}; then, one applies~\eqref{eq:admp}.
    \end{enumerate}
    Thus the proposition is proved.
\end{proof}

\subsection{\texorpdfstring{An alternative to $\mathcal{Q}$}{An alternative to Q}}\label{s:q_transp}
In this subsection, we define an alternative version of $\mathcal{Q}$ for the \emph{$g$-convention} (see Section~\ref{s:descr_g}). Continuing in the setup of the current section, let us decompose a generic element $U \in G$ into the product $U = U_+ U_{\ominus}$ where $U_+ \in \mathcal{N}_+$ and $U_{\ominus} \in \mathcal{B}_-$. Let $\bg^{\op}:=(\Gamma_2,\Gamma_1,\gamma^*)$ be the opposite BD triple. Note that $\bar{\bg}^{\op}:=(\bg^{\op},r_0^t)$ is a BD quadruple. Define the rational map $\mathcal{Q}^{\op}:(G,\pi_{(\bg_\std,r_0)}^{\dagger})\dashrightarrow (G,\pi_{(\bg,r_0)}^{\dagger})$ via
\begin{equation}\label{eq:qopdef}
    \mathcal{Q}^{\op}(U) := \rho^{\op}(U)U \rho^{\op}(U)^{-1}, \ \ \rho^{\op}(U) := \prod_{i=1}^{\leftarrow} \tilde{\gamma}(U_+), \ U \in G.
\end{equation}
Define a map $T_*: \mathfrak{g}\rightarrow\mathfrak{g}$ by setting $T_*|_{\mathfrak{h}} = \id|_{\mathfrak{h}}$, $T_*(e_{-\alpha}) = e_{\alpha}$, $T_*(e_{\alpha}) = e_{-\alpha}$, where $\langle e_{\alpha},e_{-\alpha}\rangle = 1$, as in Section~\ref{s:ba_birat}. Then $T_*$ integrates to the \emph{transposition map} $T : G \rightarrow G$.
\begin{lemma}\label{l:transpoiss}
The transposition map $T:(G,\pi_{\bar{\bg}}^{\dagger}) \rightarrow (G,-\pi_{\bar{\bg}^{\op}}^{\dagger})$ is a Poisson isomorphism.
\end{lemma}
\begin{proof}
It follows from formula~\eqref{eq:rplus} that the classical $r$-matrix $r_{\bar{\bg}}$ can be expressed as
\begin{equation*}
r_{\bar{\bg}} = r_0 + \left[1\otimes\frac{1}{1-\gamma}\right](\Omega_{>}) - \left[\frac{\gamma}{1-\gamma}\otimes 1\right](\Omega_{<}).
\end{equation*}
We see that
\begin{equation*}
[T_*\otimes T_*](r_{\bar{\bg}}) = r_0 + \left[1\otimes\frac{1}{1-\gamma}\right](\Omega_{<}) - \left[\frac{\gamma}{1-\gamma}\otimes 1\right](\Omega_{>}) = (r_{\bar{\bg}^{\op}})^{t}.
\end{equation*}
Also, $d_U^R T(\Ad_U x) = T_*(x)$ and $d_U^R T(\Ad_U^\circ x) = -\Ad_{T(U)}^\circ T_*(x)$, $x \in \mathfrak{g}$. Now, applying $[d_U^R T \otimes d_U^R T]$ to formula~\eqref{eq:pigr}, we conclude that $[d_U^R T\otimes d_U^RT]((\pi_{\bar{\bg}}^{\dagger,R})_U) = -(\pi_{\bar{\bg}^{\op}}^{\dagger,R})_{T(U)}$.
\end{proof}
\begin{proposition}\label{p:qop_poiss}
For any BD quadruple $(\bg,r_0)$, the map $\mathcal{Q}^{\op}:(G,\pi_{(\bg_{\std},r_0)}^{\dagger}) \dashrightarrow (G,\pi_{(\bg,r_0)}^{\dagger})$ is Poisson.
\end{proposition}
We emphasize that the target Poisson bivector is $\pi_{\bar{\bg}}^{\dagger}$ and not $\pi_{\bar{\bg}^{\op}}^{\dagger}$.
\begin{proof}
Indeed, let $\mathcal{Q}:(G,\pi_{(\bg_\std,r_0)}^{\dagger}) \dashrightarrow (G,\pi_{(\bg^{\op},r_0^t)}^{\dagger})$ be defined relative to $\bg^{\op}$. Note that the  diagram
\begin{equation}\label{eq:qqopt}
\xymatrix{
(G,\pi_{(\bg_{\std},r_0)}^{\dagger})\ar@{-->}[r]^{\mathcal{Q}^{\op}}\ar[d]^{T} & (G,\pi_{(\bg,r_0)}^{\dagger})\ar[d]^{T} \\
(G,-\pi_{(\bg_{\std},r_0)}^{\dagger})\ar@{-->}[r]^{\mathcal{Q}} & (G,-\pi_{(\bg^{\op},r_0^t)}^{\dagger})
}
\end{equation}
is commutative. Thus $\mathcal{Q}^{\op}$ is a Poisson map.
\end{proof}
\section{Birational quasi-isomorphisms}\label{s:birat}
In this section, we describe various birational quasi-isomorphisms between the generalized cluster structures in the $h$-convention (see the description in Section~\ref{s:descr_h}). More specifically, given a BD triple $\bg:=(\Gamma_1,\Gamma_2,\gamma)$, a BD quadruple $\bar{\bg}:=(\bg,r_0)$, and $G \in \{\SL_n,\GL_n\}$, we describe the following maps:

\begin{itemize}
\item The birational map $\mathcal{F}:(G,\pi_{(\bg,r_0)}^{\dagger}) \dashrightarrow (G,\pi_{(\bg_{\std},r_0)}^{\dagger})$, which is a basic ingredient for defining the $h$-variables and birational quasi-isomorphisms;
\item The birational quasi-isomorphism $\mathcal{Q}:(G,\gc_h^{\dagger}(\bg_{\std})) \dashrightarrow (G,\gc_h^{\dagger}(\bg))$ and its inverse $\mathcal{F}^c:=\mathcal{Q}^{-1}$. By  results of Section~\ref{s:comp}, $\mathcal{Q}:(G,\pi_{(\bg_{\std},r_0)}^{\dagger}) \dashrightarrow (G,\pi_{(\bg,r_0)}^{\dagger})$ is a Poisson map;
\item For any $\tilde{\bg}\prec \bg$, we construct a birational quasi-isomorphism $\mathcal{G}:(G,\gc_h^{\dagger}(\tilde{\bg}))\dashrightarrow (G,\gc_h^{\dagger}(\bg))$. It is Poisson as a map $\mathcal{G}:(G,\pi_{(\tilde{\bg},r_0)}^{\dagger})\dashrightarrow (G,\pi_{(\bg,r_0)}^{\dagger})$. If $\tilde{\bg}$ is obtained from $\bg$ by removing a pair of roots, we provide an even more explicit description of $\mathcal{G}$ in Section~\ref{s:bsimpl}.
\end{itemize}

If $\tilde{\bg}\prec \bg$ are BD triples of type $A_{n-1}$, we will decorate the $h$-variables in $\gc^{\dagger}_h(\tilde{\bg})$ with a tilde, to distinguish them from the $h$-variables in $\gc^{\dagger}_h(\bg)$. The marker (see Definition~\ref{d:marker}) will always be chosen in such a way that each $\varphi_{ij}$ and $c_i$ are related to themselves, and each $h_{ij}$ is related to $\tilde{h}_{ij}$. The marked variables in $\gc_h^{\dagger}(\bg)$ relative to the pair $(\gc_h^{\dagger}(\bg),\gc_h^{\dagger}(\tilde{\bg}))$ are given by 
\begin{equation}
\{h_{i+1,i+1}\ | \ i\in \Gamma_2 \setminus \tilde{\Gamma}_2\}.
\end{equation}

\subsection{\texorpdfstring{Belavin--Drinfeld data for $\gl_n(\mathbb{C})$}{Belavin--Drinfeld data for gln(C)}}\label{s:bd_map}
In this section, we introduce various notions and formulas pertaining to Belavin--Drinfeld triples. 

\begin{definition}\label{d:degnilp}
Let $(\Gamma_1,\Gamma_2,\gamma)$ be a Belavin--Drinfeld triple. The \emph{degree of nilpotency} of $\gamma$ is defined as the number
\begin{equation*}
\deg \gamma := \max_{\alpha\in\Gamma_1}\{m \in \mathbb{Z}_+ \ | \ \gamma^{m}(\alpha) \notin \Gamma_1\}.
\end{equation*}
\end{definition}

\begin{definition}\label{d:bdremov}
    Let $\bg := (\Gamma_1,\Gamma_2,\gamma)$ and $\tilde{\bg}:=(\tilde{\Gamma}_1,\tilde{\Gamma}_2,\theta)$ be two Belavin--Drinfeld triples. We write $\tilde{\bg} \prec \bg$ if $\tilde{\Gamma}_1 \subset \Gamma_1$, $\tilde{\Gamma}_2 \subset \Gamma_2$ and $\gamma|_{\tilde{\Gamma}_1} = \theta$, and we say that \textit{$\tilde{\bg}$ is obtained from $\bg$ by removing the roots} $\Gamma_1\setminus \tilde{\Gamma}_1$ from $\Gamma_1$ and $\Gamma_2 \setminus \tilde{\Gamma}_2$ from $\Gamma_2$.
\end{definition}

\paragraph{Runs and orientability of $\gamma$.} Let us fix a set of simple roots $\Pi$ of type $A_{n-1}$ and identify it with the interval $[1,n-1]$. Let $\bg := (\Gamma_1,\Gamma_2,\gamma)$ be a BD triple of type $A_{n-1}$. For an arbitrary $i \in \Pi$, set
\begin{equation*}
i_- := \max\{ j \in [0,n] \setminus \Gamma_1 \ | \ j < i\}, \ \ i_+ := \min \{j \in [1,n] \setminus \Gamma_1 \ | \ j \geq i \}.
\end{equation*}
An \emph{$X$-run} of $i$ is the interval $\Delta(i) := [i_-+1, i_+]$. Replacing $\Gamma_1$ with $\Gamma_2$ in the above formulas, we obtain the definition of a \emph{$Y$-run} $\bar{\Delta}(i)$ of $i$. The $X$-runs partition the set $[1,n]$, and likewise the $Y$-runs. A run is called \emph{trivial} if it consists of a single element. Evidently, the map $\gamma$ can be viewed as a bijection between the set of nontrivial $X$-runs and the set of nontrivial $Y$-runs. We say that a $Y$-run $\bar{\Delta}$ \emph{corresponds} to an $X$-run $\Delta$ if $\gamma(i) = j$ for some $i \in \Delta \cap \Gamma_1$ and some $j \in \bar{\Delta}\cap \Gamma_2$. The notion of runs was introduced in \cite{plethora}.

\begin{definition}
    Let $\bg = (\Gamma_1,\Gamma_2,\gamma)$ be a BD triple of type $A_{n-1}$, and let $\Delta$ be a nontrivial $X$-run. We say that $\gamma$ is \emph{positively oriented on $\Delta$} if $\gamma(i+1) = \gamma(i)+1$ for each $i,i+1 \in \Delta \cap \Gamma_1$. 
    Likewise, if $\gamma(i+1) = \gamma(i)-1$ for each $i,i+1 \in \Delta \cap \Gamma_1$, then $\gamma$ is called \emph{negatively oriented on $\Delta$}. The former circumstance will be denoted as $\gamma|_{\Delta \cap \Gamma_1} > 0$, while the latter as\footnote{If an $X$-run $\Delta$ consists of $2$ elements, we will view $\gamma$ as being positively oriented on $\Delta$. However, the distinction is not necessary in this case.} $\gamma|_{\Delta \cap \Gamma_1} < 0$.
\end{definition}

\begin{remark}
    In~\cite{rho,plethora}, a BD triple was called \emph{oriented} if $\gamma|_{\Delta\cap \Gamma_1} > 0$ for every nontrivial $X$-run $\Delta$; it was called \emph{nonoriented} if there exists a nontrivial $X$-run~$\Delta$ such that $\gamma_{\Delta\cap\Gamma_1} < 0$.
\end{remark}

\paragraph{The map of indices $\bar{\gamma}$.} Let us set $\bar{\Gamma}_1$ to be the union of all the nontrivial $X$-runs\footnote{Alternatively, $\bar{\Gamma}_1$ is the union of $\Gamma_1$ and all the right endpoints of all the nontrivial $X$-runs. Similarly for $\bar{\Gamma}_2$. } and $\bar{\Gamma}_2$ to be the union of all the nontrivial $Y$-runs. We define a map $\bar{\gamma} : \bar{\Gamma}_1 \rightarrow \bar{\Gamma}_2$ in the following way: for each nontrivial $X$-run $\Delta$ and the corresponding $Y$-run $\bar{\Delta}$,
\begin{enumerate}[a)]
    \item If $\gamma|_{\Delta \cap \Gamma_1} > 0$, then the restriction $\bar{\gamma}|_{\Delta} : \Delta \rightarrow \bar{\Delta}$ is set to be the unique increasing bijection;
    \item If $\gamma|_{\Delta \cap \Gamma_1} < 0$, then the restriction $\bar{\gamma}|_{\Delta} : \Delta \rightarrow \bar{\Delta}$ is set to be the unique decreasing bijection\footnote{Note that if $\gamma|_{\Delta\cap \Gamma_1} > 0$, then $\bar{\gamma}|_{\Delta \cap \Gamma_1} = \gamma|_{\Delta \cap \Gamma_1}$; however, if $\gamma|_{\Delta\cap \Gamma_1} < 0$, then $\gamma(i)+1 = \bar{\gamma}(i)$, $i \in \Delta \cap \Gamma_1$.}.
\end{enumerate}
A similar recipe applies to $\bar{\gamma}^*$.

\paragraph{The map of subsets $\check{\gamma}$.} Let $\Delta$ be any nontrivial $X$-run and $\bar{\Delta}$ be the corresponding $Y$-run. We define the map $\check{\gamma}:2^{\Delta} \rightarrow 2^{\Delta}$ in the following way. For any $A = \{a_{i}\}_{i=1}^k \subseteq \Delta$, 
\begin{enumerate}[a)]
\item If $\gamma|_{\Delta \cap \Gamma_1} > 0$, then $\check{\gamma}(A) := \{\bar{\gamma}(a_{i})\}_{i=1}^k$;
\item If $\gamma|_{\Delta \cap \Gamma_1} < 0$, then $\check{\gamma}(A) := \bar{\Delta} \setminus \{\bar{\gamma}(a_{i})\}_{i=1}^k$.
\end{enumerate}

\paragraph{The extension $\mathring{\gamma}:\gl_n(\mathbb{C}) \rightarrow \gl_n(\mathbb{C})$.} As was shown in~\cite{plethora}, in the case of $\GL_n$, it is expedient to extend the BD map $\gamma : \Gamma_1 \rightarrow \Gamma_2$ in a different way\footnote{In our previous work~\cite{multdouble}, we denoted the resulting map by $\gamma$ as well, for the main focus of the paper was the case of $\GL_n$. In this paper, we keep the original notation introduced in~\cite{plethora}.}. Let $\Delta_1,\ldots,\Delta_k$ be the list of all nontrivial $X$-runs, and set $\bar{\Delta}_1,\ldots,\bar{\Delta}_k$ to be the list of the corresponding $Y$-runs (in other words, $\gamma(\Delta_i\cap \Gamma_1) = \bar{\Delta}_i\cap \Gamma_2$, $i \in [1,k]$). Set $\gl(\Delta_i)$ to be a subalgebra of $\gl_n(\mathbb{C})$ of the matrices that are zero outside of the block $\Delta_i \times \Delta_i$ (and similarly for $\gl(\bar{\Delta}_i)$). For each $i \in [1,k]$ define a map $\mathring{\gamma}_i:\gl(\Delta_i) \rightarrow \gl(\bar{\Delta}_i)$ in the following way: first, canonically identify $\gl(\bar{\Delta}_i)$ and $\gl(\Delta_i)$ with $\gl_{|\Delta_i|}(\mathbb{C})$; second,  
\begin{enumerate}[a)]
    \item If $\gamma|_{\Delta_i\cap \Gamma_1} > 0$, then $\mathring{\gamma}_i:\gl(\Delta_i) \rightarrow \gl(\bar{\Delta}_i)$ acts as the identity map under the identifications $\gl(\Delta_i) \cong \gl_{|\Delta_i|}(\mathbb{C}) $ and $\gl(\bar{\Delta}_i)\cong \gl_{|\Delta_i|}(\mathbb{C})$; in other words, $\mathring{\gamma}_i(e_{ks}) = e_{\bar{\gamma}(k)\bar{\gamma}(s)}$, $k,s \in \Delta_i$;
    \item If $\gamma|_{\Delta_i\cap \Gamma_1} < 0$, then let us set $w_0 := \sum_{j=1}^{|\Delta_i|}(-1)^{j-1} e_{|\Delta_i|-j+1,j}$; under the identifications, if $A \in \gl_{|\Delta_i|}(\mathbb{C})$, then $\mathring{\gamma}_i:\gl(\Delta_i) \rightarrow \gl(\bar{\Delta}_i)$ acts as
    \begin{equation*}
    \mathring{\gamma}_i(A) = w_0 (-A)^T w_0^{-1}
    \end{equation*}
    where $A^T$ is the transpose of $A$. In other words, $\mathring{\gamma}_i(e_{ks}) = (-1)^{s+k-1} e_{\bar{\gamma}(s)\bar{\gamma}(k)}$, $k, s \in \Delta_i$.
\end{enumerate}
The map $\mathring{\gamma}: \gl_n(\mathbb{C}) \rightarrow \gl_n(\mathbb{C})$ is defined as the direct sum $\mathring{\gamma}:=\bigoplus_{i=1}^k\mathring{\gamma}_i$ extended by zero to $\gl_n(\mathbb{C})$. Similarly, one defines the pieces $\mathring{\gamma}_i^* : \gl_n(\bar{\Delta}_i)\rightarrow \gl_n(\Delta_i)$, and then the map $\mathring{\gamma}^* : \gl_n(\mathbb{C}) \rightarrow \gl_n(\mathbb{C})$ is obtained as the direct sum $\mathring{\gamma}^*:=\bigoplus_{i=1}^k\mathring{\gamma}_i^*$ extended by zero to $\gl_n(\mathbb{C})$. 

\paragraph{The group homomorphism $\togamma$.} In the case of $\GL_n$, the restrictions of $\ogamma$ to the Borel subalgebras $\mathfrak{b}_+$ and $\mathfrak{b}_-$ of $\gl_n(\mathbb{C})$ are Lie algebra homomorphisms. Therefore, one can integrate $\ogamma$ in each case to a group homomorphism $\togamma: \mathcal{B}_\pm \rightarrow \mathcal{B}_\pm$, where $\mathcal{B}_+$ and $\mathcal{B}_-$ are the subgroups of upper and lower triangular matrices. Likewise, one integrates $\ogammas : \mathfrak{b}_\pm \rightarrow \mathfrak{b}_\pm$ to a group homomorphism $\togammas:\mathcal{B}_\pm \rightarrow \mathcal{B}_\pm$. Note: since $\gamma$ and $\ogamma$ coincide on $\mathfrak{n}_+$ and $\mathfrak{n}_-$, $\tilde{\gamma}(N_\pm) = \togamma(N_{\pm})$, $N_{\pm} \in \mathcal{N}_{\pm}$. We  prefer the notation $\tilde{\gamma}$ over $\togamma$ if the argument is a unipotent lower or upper triangular matrix. Next, set
 \begin{equation*}
\gl_n({\Gamma_1}) := \bigoplus_{|\Delta|>1} \gl_n(\Delta), \ \  \gl_n({\Gamma_2}):=\bigoplus_{|\bar{\Delta}|>1} \gl_n(\bar{\Delta}),
\end{equation*} 
where $\gl_n(\Delta)$ and $\gl_n(\bar{\Delta})$ are defined above. Let $\GL_n(\Gamma_i)$ be the connected subgroup of $\GL_n$ that corresponds to $\gl_n(\Gamma_i)$, $i \in \{1,2\}$. The restriction $\ogamma|_{\gl_n(\Gamma_1)}:\gl_n(\Gamma_1) \rightarrow \gl_n(\Gamma_2)$ can be integrated to the group homomorphism $\togamma : \GL_n(\Gamma_1) \rightarrow \GL_n(\Gamma_2)$; likewise, $\ogammas:\gl_n(\Gamma_2) \rightarrow \gl_n(\Gamma_1)$ integrates to $\togammas : \GL_n(\Gamma_2) \rightarrow \GL_n(\Gamma_1)$.

\begin{remark}
    Note that $\mathring{\gamma}|_{\sll_n(\mathbb{C})}$ coincides with the map $\gamma:\sll_{n}(\mathbb{C}) \rightarrow \sll_n(\mathbb{C})$ introduced in Section~\ref{s:ba_poisson} only on the nilpotent subalgebras $\mathfrak{n}_-$ and $\mathfrak{n}_+$: the maps are different on the Cartan subalgebra of $\sll_n(\mathbb{C})$. Moreover, the version of $\gamma$ from Section~\ref{s:ba_poisson} does not integrate to a homomorphism of Borel subgroups.
\end{remark}

\paragraph{Projections.} We equip $\gl_n(\mathbb{C})$ with the trace form and define the orthogonal projections $\mathring{\pi}_{\Gamma_i} : \gl_n(\mathbb{C}) \rightarrow \gl_n(\Gamma_i)$ and $\mathring{\pi}_{\hat{\Gamma}_i}:=\gl_n(\mathbb{C}) \rightarrow \gl_n(\Gamma_i)^{\perp}$, $i \in \{1,2\}$. Given the parabolic subalgebras $\mathfrak{p}_{\pm}(\Gamma_i)$ of $\gl_n(\mathbb{C})$ and the corresponding parabolic subgroups $\mathcal{P}_{\pm}(\Gamma_i)$, we integrate the restrictions $(\mathring{\pi}_{\Gamma_i})|_{\mathfrak{p}_{\pm}(\Gamma_i)}$ to the group projections $\mathring{\Pi}_{\Gamma_i} :\mathcal{P}_{\pm}(\Gamma_i) \rightarrow \GL_n(\Gamma_i)$, $i \in \{1,2\}$. Let $\mathcal{H}$ be the Cartan subgroup of $\GL_n$, $\mathfrak{h}(\Gamma_i)$ the Cartan subalgebra of $\gl_n(\Gamma_i)$, $\mathcal{H}(\Gamma_i)$ and $\mathcal{H}(\hat{\Gamma}_i)$ be the connected subgroups of $\GL_n$ that correspond to $\mathfrak{h}(\Gamma_i)$ and $\mathfrak{h}(\Gamma_i)^{\perp}$, $i \in \{1,2\}$. We set $\mathring{\Pi}_{\hat{\Gamma}_i}:\mathcal{H}\rightarrow\mathcal{H}(\hat{\Gamma}_i)$ to be the group homomorphism that corresponds to $\mathring{\pi}_{\hat{\Gamma}_i}|_{\mathfrak{h}}$.

\paragraph{The Gauss decomposition and $\tilde{\gamma}^*$.} Let $\bg = (\Gamma_1,\Gamma_2,\gamma)$ be a BD triple of type $A_{n-1}$, and let $U = U_\oplus U_-$ be the decomposition of a generic element $U \in \GL_n$ into an upper triangular matrix $U_\oplus$ and a unipotent lower triangular matrix $U_-$. Let $\Delta$ be a nontrivial $X$-run and $i,j \in \Delta$, $i > j$. We will need the following formulas:
\begin{enumerate}[a)]
    \item If $\gamma|_{\Delta\cap \Gamma_1}> 0$, then
    \begin{equation*}
    [\tilde{\gamma}^*(U_-)]_{ij} = \frac{\det U^{\{\bar{\gamma}(j)\}\cup [\bar{\gamma}(i)+1,n]}_{[\bar{\gamma}(i),n]}}{\det U^{[\bar{\gamma}(i),n]}_{[\bar{\gamma}(i),n]}};
    \end{equation*}
    \item If $\gamma|_{\Delta \cap \Gamma_1}<0$, then 
    \begin{equation*}
    [\tilde{\gamma}^*(U_-)]_{ij} = \frac{\det U^{[\bar{\gamma}(i),n]\setminus \{\bar{\gamma}(j)\}}_{[\bar{\gamma}(i)+1,n]} }{\det U^{[\bar{\gamma}(i)+1,n]}_{[\bar{\gamma}(i)+1,n]}}.
    \end{equation*}
\end{enumerate}
The above formulas follow from Lemma~\ref{l:min_ump} and Lemma~\ref{l:min_umn}. 
Furthermore, if $N$ is a unipotent lower-triangular matrix and $i,j \in \Delta$ such that $\gamma|_{\Delta \cap \Gamma_1} < 0$, then
\begin{equation}\label{eq:gammann}
[\tilde{\gamma}^*(N)]_{ij} = (-1)^{i+j} [N^{-1}]_{\bar{\gamma}(j)\bar{\gamma}(i)} = \det N^{[1,n]\setminus \{\bar{\gamma}(j)\}}_{[1,n]\setminus \{\bar{\gamma}(i)\}}.
\end{equation}
Lastly, observe that the denominators of $\tilde{\gamma}^*(U_-)$ form the set
\begin{equation}\label{eq:Qdenoms}
    \{\det U_{[i+1,n]}^{[i+1,n]} \ | \ i \in \Gamma_2\}.
\end{equation}

 \paragraph{Ringed version of $R_0$ for $\GL_n$.} As in our previous work~\cite{multdouble}, it is expedient to choose $R_0$ with special properties in the case of $\GL_n$. Any such special solution will be denoted as $\mathring{R}_0$ and called a \emph{ringed} $R_0$ element. The ringed element $\mathring{R}_0$ is a solution of~\eqref{eq:r0m}-\eqref{eq:ralgm} that in addition satisfies any of the following equivalent identities:
 \begin{equation}\label{eq:ringedr0}
 \begin{aligned}
     \mathring{R}_0(1-\ogamma) &= \mathring{\pi}_{\Gamma_1} + \mathring{R}_0 \mathring{\pi}_{\hat{\Gamma}_1} \ \ & \ \ \mathring{R}_0(1-\ogammas) &= -\ogammas + \mathring{R}_0\mathring{\pi}_{\hat{\Gamma}_2} \\ \mathring{R}_0^*(1-\ogamma) &= -\ogamma + \mathring{R}_0^* \mathring{\pi}_{\hat{\Gamma}_1}\ \  &\ \ \mathring{R}_0^*(1-\ogammas) &= \mathring{\pi}_{\Gamma_2} + \mathring{R}_0^* \mathring{\pi}_{\hat{\Gamma}_2}
 \end{aligned}
 \end{equation}
 where all the maps are assumed to be restricted to the diagonal matrices.

\subsection{\texorpdfstring{The map $\mathcal{F}$}{The map F}}\label{s:map_f}
Fix a set of simple roots $\Pi$ of type $A_{n-1}$ and a BD triple $\bg:=(\Gamma_1,\Gamma_2,\gamma)$ of type $A_{n-1}$. Let $\mathcal{B}_{\pm}$ and $\mathcal{N}_\pm$ be the Borel subgroups and their unipotent radicals, and set $\mathcal{H}$ to be the Cartan subgroup. We decompose a generic element $U \in \GL_n$ as $U = U_\oplus U_-$, where $U_- \in \mathcal{N}_-$ and $U_{\oplus} \in \mathcal{B}_+$. Define a sequence $\mathcal{F}_k:\GL_n\dashrightarrow \GL_n$ via
\begin{equation}\label{eq:fseq}
\mathcal{F}_0(U):=U, \ \ \mathcal{F}_k(U) = \tilde{\gamma}^*([\mathcal{F}_{k-1}(U)]_-)U, \ \ U \in \GL_n, \ k \geq 1.
\end{equation}
The rational map $\mathcal{F}:G \dashrightarrow G$ is defined as the limit
\begin{equation}\label{eq:fdef}
\mathcal{F}(U):=\lim_{k\rightarrow\infty}\mathcal{F}_k(U), \ \ U \in \GL_n.
\end{equation}

\begin{lemma}\label{l:fzerocond}
Let $N \in \mathcal{N}_-$ and $k \in [0,\deg\gamma+1]$. Let $\{\hat{\mathcal{F}}_i : \GL_n \dashrightarrow \GL_n\}_{i \geq 0}$ be a sequence of rational maps defined via
\begin{equation*}
\hat{\mathcal{F}}_0(U) := (\tilde{\gamma}^*)^k(N)U, \ \ \hat{\mathcal{F}}_{i}(U):=\tilde{\gamma}^*([\hat{\mathcal F}_{i-1}(U)]_-)U, \ \ U \in G, \ i \geq 1.
\end{equation*}
Then $\hat{\mathcal{F}}_{\deg \gamma + 1 - k}(U) = \mathcal{F}_{\deg \gamma + 1 -k}(U)$.
\end{lemma}

\begin{proof} Define $\mathcal{S}:\mathcal{N}_-\rightarrow\mathcal{N}_-$ by setting $\mathcal{S}(N):=\tilde{\gamma}^*([NU]_-)$, and let 
\begin{equation*}
    S_\ell := \{\beta \in \Phi_+ \ | \ \beta\succ (\gamma^*)^\ell(\alpha) \ \text{for some} \ \alpha\in A\cap \Gamma_1\}
\end{equation*}
where $\Phi_+$ is the set of positive roots and where we imply the usual order on the roots. The statement of the lemma is an observation that the entries of $N$ in the matrix $\mathcal{S}^{\ell}(N)$ are visible only at matrix entries that correspond to the root subspaces $(\gl_n)_{-\beta}$ for $\beta \in S_{\ell}$. 
\end{proof}

\begin{proposition}\label{p:fq0i}
    The sequence $\mathcal{F}_k$ stabilizes at $k = \deg \gamma$ and $\mathcal{F}(U) = \mathcal{F}_{\deg \gamma}(U)$. Moreover, the rational map $\mathcal{F}:\GL_n \dashrightarrow\GL_n$ satisfies the relation
    \begin{equation}\label{eq:frel}
    \mathcal{F}(U) = \tilde{\gamma}^*([\mathcal{F}(U)]_-) U,
    \end{equation}
    and in particular, $\mathcal{F}$ is birational, with the inverse given by
    \begin{equation}\label{eq:finv}
    \mathcal{F}^{-1}(U) = \tilde{\gamma}^*(U_-)^{-1}U.
    \end{equation}
\end{proposition}
\begin{proof}
The proposition follows from Lemma~\ref{l:fzerocond} via setting $N := [\mathcal{F}_i(U)]_-$ for $i\geq 1$.
\end{proof}

\begin{proposition}\label{p:invfk}
Let $N\in\mathcal{N}_-$ and $T\in \mathcal{H}$. Set $N_0:=N$ and $N_k := (\tilde{\gamma}^*)^k(N)$, $k \in [1,\deg\gamma]$. For any $k \in [0,\deg\gamma]$, the following formulas hold:
\begin{align*}
&\mathcal{F}_{\deg\gamma-k}(N_{k+1} U N_{k}^{-1}) = \mathcal{F}_{\deg\gamma-k}(U)N_{k}^{-1};\\ 
&\mathcal{F}_{\deg \gamma - k}(\tilde{\gamma}^*(T)UT^{-1}) = \tilde{\gamma}^*(T)\mathcal{F}_{\deg \gamma-k}(U)T^{-1};\\ 
&\mathcal{F}_{\deg \gamma-k}(T^{-1}U\tilde{\gamma}(T)) = T^{-1}\mathcal{F}_{\deg \gamma-k}(U) \tilde{\gamma}(T).
\end{align*}
\end{proposition}
\begin{proof}
Let us set
\[
\hat{\mathcal{F}}_i(U):= \mathcal{F}_i(N_{k+1}UN_{k}^{-1})N_k, \ \ i \geq 0.
\]
Observe that $\hat{\mathcal{F}}_0(U) = N_{k+1}U$ and that $\tilde{\mathcal{F}}_i(U)$ satisfies the recurrence relation $\hat{\mathcal{F}}_i(U) = \tilde{\gamma}^*([\hat{\mathcal{F}}_{i-1}(U)]_-)U$, $i \geq 0$. Indeed, from the recurrence relation for $\mathcal{F}_i(U)$, we see that
\[
\hat{\mathcal{F}}_{i+1}(U) = \mathcal{F}_i(N_{k+1}UN_{k}^{-1})N_k = \tilde{\gamma}^*([\mathcal{F}_{i-1}(N_{k+1}UN_{k}^{-1})]_-)N_{k+1}UN_k^{-1}N_k = \tilde{\gamma}^*([\hat{\mathcal{F}}_{i}(U)]_-)U.
\]
By Lemma~\ref{l:fzerocond}, we conclude that $\hat{\mathcal{F}}_{\deg\gamma-k}(U) = \mathcal{F}_{\deg\gamma-k}(U)$, which yields the first formula of the proposition.

The second formula is obvious for $k = \deg \gamma$; the other cases follow easily from induction:
\[\begin{split}
\mathcal{F}_{\deg \gamma-k+1}(\togammas(T)UT^{-1}) = &\togammas( [\mathcal{F}_{\deg \gamma-k}(\togammas(T)UT^{-1})]_-)\togammas(T)UT^{-1} = \\ = &\togammas( [\togammas(T)\mathcal{F}_{\deg \gamma-k}(U)T^{-1}]_-)\togammas(T)UT^{-1} = \\ = &\togammas(T[\mathcal{F}_{\deg \gamma-k}(U)]_- T^{-1})\togammas(T)UT^{-1} = \\ = & \togammas(T) \mathcal{F}_{\deg \gamma -k+1}(U)T^{-1}.
\end{split}
\]

The last formula is also evident for $k = \deg \gamma$. Assuming it holds for some $k$, we see that
\[\begin{split}
\mathcal{F}_{\deg\gamma-k+1}(T^{-1}U\togamma(T)) = &\togammas\left( [\mathcal{F}_{\deg\gamma-k}(T^{-1}U\togamma(T))]_- \right) T^{-1}U \togamma(T) = \\ = &\togammas\left(\togamma(T^{-1}) [\mathcal{F}_{\deg\gamma-k}(U)]_-\togamma(T) \right) T^{-1}U \togamma(T) = \\ = &\mathring{\Pi}_{\Gamma_1}(T^{-1}) \togammas\left([\mathcal{F}_{\deg\gamma-k}(U)]_-  \right)\mathring{\Pi}_{\hat{\Gamma}_1}(T^{-1})U \togamma(T) = \\ = &\mathring{\Pi}_{\Gamma_1}(T^{-1})\mathring{\Pi}_{\hat{\Gamma}_1}(T^{-1}) \togammas\left([\mathcal{F}_{\deg\gamma-k}(U)]_-  \right)U \togamma(T) = \\ = &T^{-1}\mathcal{F}_{\deg\gamma-k+1}(U)\togamma(T). 
\end{split}
\]
Thus the identities hold.
\end{proof}
\begin{corollary}\label{c:f_invar}
The following identities hold :
\begin{align}\label{eq:invar1}
&\mathcal{F}(\tilde{\gamma}^*(N)UN^{-1}) = \mathcal{F}(U)N^{-1};
\\ 
&\mathcal{F}({\togamma}^*(T)UT^{-1})) = {\togamma}^*(T) \mathcal{F}(U)T^{-1};\notag \\
&\mathcal{F}(T^{-1}U{\togamma}(T)) = T^{-1}\mathcal{F}(U) {\togamma}(T).\notag
\end{align}
\end{corollary}
\begin{proof}
Set $k = 0$ in Proposition~\ref{p:invfk} and invoke Proposition~\ref{p:fq0i}.
\end{proof}
\begin{corollary}\label{c:hinvarn}
    Let $h_{ij}$ be any $h$-variable from the initial extended cluster of $\gc_h^{\dagger}(\bg,\GL_n)$. Then there exist characters $\chi$ and $\chi^\prime$ such that the following identities hold:
    \begin{align}
    &h_{ij}(\tilde{\gamma}^*(N)U N^{-1}) = h_{ij}(U);\label{eq:hninv}\\
   &h_{ij}(\togammas(T)UT^{-1}) = \chi(T) h_{ij}(U);\notag \\ 
  &h_{ij}(T^{-1}U\togamma(T)) = \chi^\prime(T) h_{ij}(U).\notag
    \end{align}
\end{corollary}
\begin{corollary}\label{c:h_inf}
    Let $h_{ij}$ be any $h$-variable from the initial extended cluster of $\gc^{\dagger}_h(\bg,\GL_n)$. Then the following identities hold:
    \begin{align*}
        &\nabla_U h_{ij} \cdot U - \ogamma\left(U \cdot \nabla_U h_{ij}\right) \in \mathfrak{b}_-;\\
        &\pi_0\left(\nabla_U \log h_{ij} \cdot U - \ogamma\left(U \cdot \nabla_U \log h_{ij}\right)\right) = \const;\\ 
        &\pi_0\left(U \cdot \nabla_U \log h_{ij} - \ogammas\left(\nabla_U \log h_{ij} \cdot U\right)\right) = \const.
    \end{align*}
\end{corollary}
\begin{proof}
    The formulas follow from differentiating the identities in Corollary~\ref{c:hinvarn}.
\end{proof}
\begin{remark}
    The birational map $\mathcal{F}:(\GL_n,\pi_{(\bg,r_0)}^{\dagger})\dashrightarrow (\GL_n,\pi_{(\bg_\std,r_0)}^{\dagger})$ is not a Poisson map. Nevertheless, we believe that the pushforward Poisson bivector $\mathcal{F}_*(\pi_{(\bg,r_0)}^{\dagger})$, as well as the induced compatible generalized cluster structure on $\GL_n$ might be of interest. We also believe that there might exist generalized cluster structures on some interesting subvarieties of $\GL_n$ whose variables are constructed as the flag minors of $\mathcal{F}_m$ (akin to the formula~\eqref{eq:f_flags}), and these generalized cluster structures might be compatible with some other Poisson structures. For a discussion of these problems, as well as for supporting evidence the reader might refer to the supplementary file~\cite{github}.
\end{remark}

\subsection{\texorpdfstring{The maps $\mathcal{Q}$ and $\mathcal{F}^c$}{The maps Q and Fc}}\label{s:qfc}
Recall from Section~\ref{s:comp} that the rational map $\mathcal{Q}: (\GL_n,\pi_{(\bg_\std,r_0)}^{\dagger}) \dashrightarrow (\GL_n,\pi_{(\bg,r_0)}^{\dagger})$ is given by
\begin{equation*}
\mathcal{Q}(U):=\rho(U)^{-1}U\rho(U), \ \ \rho(U):=\prod_{i\geq 1}^{\rightarrow}(\tilde{\gamma}^*)^i(U_-).
\end{equation*}
Let us define the rational map $\mathcal{F}^c : (\GL_n,\pi_{(\bg,r_0)}^{\dagger}) \dashrightarrow (\GL_n,\pi_{(\bg_\std,r_0)}^{\dagger})$ via
\begin{equation*}
\mathcal{F}^c(U) := \mathcal{F}(U)U[\mathcal{F}(U)]^{-1}.
\end{equation*}
\begin{proposition}\label{p:fc_q_inv}
The rational maps $\mathcal{F}^c$ and $\mathcal{Q}$ are inverse to each other.
\end{proposition}
\begin{proof}
Recall that $\rho(U)$ satisfies the relation
\[
\tilde{\gamma}^*(U_-)\tilde{\gamma}^*(\rho(U)) = \rho(U).
\]
Applying formulas~\eqref{eq:finv} and~\eqref{eq:invar1}, we see that
\[
\mathcal{F}(\mathcal{Q}(U)) = \mathcal{F}\left(\tilde{\gamma}^*(\rho(U))^{-1}\left[\tilde{\gamma}^*(U_-)^{-1}U\right]\rho(U)\right) = U \cdot \rho(U);
\]
therefore,
\[
\mathcal{F}^c(\mathcal{Q}(U)) = \left[U\cdot \rho(U)\right] \left[\rho(U)^{-1} U\rho(U)\right]\left[\rho(U)^{-1}U^{-1}\right] = U.\qedhere
\]
\end{proof}

\begin{remark}
The map $\mathcal{F}^c$ can also be written as
\begin{equation}\label{eq:fc_alt}
\mathcal{F}^c(U) = \mathcal{F}(U)\tilde{\gamma}^*(\mathcal{F}(U)_-)^{-1} = \tilde{\gamma}^*([\mathcal{F}(U)]_-) U \tilde{\gamma}^*([\mathcal{F}(U)]_-)^{-1}.
\end{equation}
The identities follow from~\eqref{eq:frel}.
\end{remark}

\begin{proposition}\label{p:q_quasih}
Let $\psi$ be any variable from the initial extended cluster of $\gc^\dagger_h(\bg)$. Then the following identities hold:
\begin{enumerate}[1)]
\item If $\psi$ is a $\varphi$- or $c$-function, then $\mathcal{Q}^*(\psi(U)) = \psi(U)$;\label{i:bqlvc}
\item\label{i:bqplh} If $\psi = h_{ij}$ for some $h$-function $h_{ij}(U)$, let $\{\alpha_t\}_{t=0}^m$ be the $\gamma$-string such that $i-1 = \alpha_k$ for some $k$; then,
\begin{equation}\label{eq:q_upon_h}
\mathcal{Q}^*(h_{ij}(U)) = \tilde{h}_{ij}(U)\prod_{t \geq k+1}\tilde{h}_{\alpha_t+1,\alpha_t+1}(U).
\end{equation}
\end{enumerate}
Moreover, $\mathcal{Q}:(G,\gc_h^{\dagger}(\bg_{\std}))\dashrightarrow (G,\gc_h^{\dagger}(\bg))$ is a birational quasi-isomorphism for $G \in \{\SL_n,\GL_n\}$.
\end{proposition}
Note that the regularity of the $h$-functions is established independently of the above statement in Proposition~\ref{p:h_expansion} below.
\begin{proof}
The statement in Part~\ref{i:bqlvc} follows from the invariance properties of $\varphi$- and $c$-functions (see~\eqref{eq:phi_invar} and~\eqref{eq:c_def}). For Part~\ref{i:bqplh}, notice that for any $1 \leq j \leq n$ and any $I \subseteq [1,n]$ of size $|I| = n-j+1$,
\[
\mathcal{Q}^*\left(\det\mathcal{F}(U)^{[j,n]}_{I}\right) = \det U^{[j,n]}_{I},
\]
and then invoke the definition of the $h$-variables~\eqref{eq:h_fun}. The last statement can be verified via Proposition~\ref{p:q_y_var}. The hatted strings $\hat{p}_{ir}$ are given by the $c$-variables, which in turn are invariant with respect to $\mathcal{Q}$ by Part~\ref{i:bqlvc}. To verify that $\mathcal{Q}^*(y_i) = \tilde{y}_{\kappa(i)}$ for each nonmarked index $i$, we consider the neighborhoods of the variables as given in Appendix~\ref{s:ainiquivh}. For instance, from Figure~\ref{f:nbd_h} we read that $y(h_{ij})$ for $2 \leq i\leq j \leq n-1$ is given by 
\begin{equation*}
    y(h_{ij}) = \frac{h_{i+1,j}h_{i-1,j-1}h_{i,j+1}}{h_{i-1,j}h_{i,j-1},h_{i+1,j+1}}.
\end{equation*}
It follows from~\eqref{eq:q_upon_h} that for any $2 \leq j \leq n-1$ and any $2 \leq s,k\leq j$,
\begin{equation*}
    \mathcal{Q}^*\left(\frac{h_{kj}}{h_{sj}}\right) = \frac{\tilde{h}_{kj}}{\tilde{h}_{sj}}.
\end{equation*}
Therefore, $\mathcal{Q}^*(y(h_{ij})) = y(\tilde{h}_{ij})$. A similar reasoning applies to all the other neighborhoods. By Proposition~\ref{p:q_y_var}, we conclude that $\mathcal{Q}$ is a  quasi-isomorphism. To verify that $\mathcal{Q}$ is a birational quasi-isomorphism (see   Definition~\ref{d:birat}), we need to verify that $\mathcal{Q}^*$ induces an isomorphism 
\begin{equation}\label{eq:q_birat_iso}
\mathcal{Q}^*:\mathbb{C}[\mathrm{Mat}_n][h_{i+1,i+1}^{\pm 1}(U) \ | \ i \in \Gamma_2] \xrightarrow{\sim}\mathbb{C}[\mathrm{Mat}_n][\tilde{h}_{i+1,i+1}^{\pm 1}(U) \ | \ i \in \Gamma_2]
\end{equation} 
where $\mathrm{Mat}_n\supseteq \GL_n$ denotes the space of $n\times n$ matrices with complex entries. Observe that the denominators of $\mathcal{Q}(U)$ are the denominators of $\tilde{\gamma}^*(U_-)$. The latter are given by equation~\eqref{eq:Qdenoms}, and therefore,
\begin{equation*}
    \mathcal{Q}^*(\mathbb{C}[\mathrm{Mat}_n]) \subseteq \mathbb{C}[\mathrm{Mat}_n][\tilde{h}_{i+1,i+1}^{\pm 1}(U) \ | \ i \in \Gamma_2].
\end{equation*}
Likewise, the denominators of $\mathcal{Q}^{-1}(U)$ are given by the denominators of $\tilde{\gamma}^*([\mathcal{F}(U)]_-)$, which in turn are given by $h_{i+1,i+1}^{\pm 1}$ for $i \in \Gamma_2$. By equation~\eqref{eq:q_upon_h}, $\mathcal{Q}$ is an isomorphism between the rings of Laurent monomials in $h_{i+1,i+1}(U)$ and $\tilde{h}_{i+1,i+1}(U)$, $i \in \Gamma_2$. Thus equation~\eqref{eq:q_birat_iso} holds.
\end{proof}

\paragraph{Equivariance property of $\mathcal{Q}$.} Let $\mathfrak{h}^s$ be the Cartan subalgebra of $\sll_n(\mathbb{C})$.
Recall that the Lie subalgebra $\mathfrak{h}_{\bg}\subseteq \mathfrak{h}^s$ is given by
\begin{equation*}
\mathfrak{h}_{\bg}=\{h \in \mathfrak{h}^s \ | \ \alpha(h) = \beta(h) \ \text{if} \ \gamma^j(\alpha) = \beta \ \text{for some} \ j\}.
\end{equation*}
Let $\mathcal{H}_{\bg}$ be the connected subgroup of $\SL_n$ that corresponds to $\mathfrak{h}_{\bg}$.
\begin{proposition}\label{p:q_equiv}
For any $T \in \mathcal{H}_{\bg}$ and any $U \in \GL_n$, $\mathcal{Q}(TUT^{-1}) = T\mathcal{Q}(U)T^{-1}$.
\end{proposition}
\begin{proof}
We see that
\begin{equation*}
T^{-1}\rho(TUT^{-1}) = \prod_{i\geq 1}(\togammas)^{i-1}(T^{-1}\togammas(T))(\togammas)^i(U_-).
\end{equation*}
It suffices to show that $T^{-1}\togammas(T)$ commutes with $\togammas(U_-)$. Indeed, we can decompose $T^{-1}\togammas(T)$ as
\begin{equation*}
T^{-1}\togammas(T) = \mathring{\Pi}_{\Gamma_1}(T^{-1}\togammas(T))\mathring{\Pi}_{\hat{\Gamma}_1}(T^{-1}\togammas(T)).
\end{equation*}
Evidently, the second factor commutes with $\togammas(U_-)$; as for the first factor, we see that
\begin{equation*}
\mathring{\Pi}_{\Gamma_1}(T^{-1}\togammas(T)) = \exp\left( \sum_{\Delta}I_{\Delta}t_{\Delta}\right)
\end{equation*}
where the sum is over nontrivial $X$-runs $\Delta$, $I_{\Delta}:=\sum_{i \in \Delta}e_{ii}$ and $t_{\Delta} \in \mathbb{C}$ are some scalars. Clearly, the last expression also commutes with ${\togamma}^*(U_-)$. Thus the statement holds.
\end{proof}

\subsection{\texorpdfstring{Birational quasi-isomorphisms $\mathcal{G}$ for $\tilde{\bg}\prec\bg$}{Birational quasi-isomorphisms G for tilde Gamma < Gamma}}\label{s:genbirat}
Let $\tilde{\bg}:=(\tilde{\Gamma}_1,\tilde{\Gamma}_2,\theta)$ and $\bg:=(\Gamma_1,\Gamma_2,\gamma)$ be a pair of BD triples (of type $A_{n-1}$), and let $\tilde{r}_0$ (resp. $r_0$) be an element such that $\bar{\bg}:=(\bg,r_0)$ (resp. $\tilde{\bar{\bg}}:=(\tilde{\bg},\tilde{r}_0)$) is a BD quadruple. In this subsection, we construct a birational map $\mathcal{G} : (\GL_n,\pi_{\tilde{\bar{\bg}}}^{\dagger}) \dashrightarrow (\GL_n,\pi_{\bar{\bg}}^{\dagger})$ from the birational maps $\mathcal{Q}:(\GL_n,\pi_{(\bg_{\std},r_0)}^{\dagger}) \dashrightarrow (\GL_n,\pi_{\bar{\bg}}^{\dagger})$ and $\tilde{\mathcal{Q}}:(\GL_n,\pi_{(\bg_\std,\tilde{r}_0)}^{\dagger}) \dashrightarrow (\GL_n,\pi_{\tilde{\bar{\bg}}}^{\dagger})$. Proposition~\ref{p:b_gen} below is a more detailed version of Proposition~\ref{p:birat} for $\gc_h^{\dagger}(\bg)$. Let us set
\begin{equation}\label{eq:geng0}
    G_0(U) := \tilde{\theta}^*(\tilde{\mathcal{F}}(U)_-)^{-1}\tilde{\gamma}^*(\tilde{\mathcal{F}}(U)_-), \ \ U \in \GL_n,
\end{equation}
and let $\tilde{\mathcal{F}}$ and $\mathcal{F}$ be the maps given by~\eqref{eq:fdef} and constructed relative to $\tilde{\bg}$ and $\bg$, respectively.

\begin{proposition}\label{p:b_gen}
    In the setup of the subsection, set $\mathcal{G}:=\mathcal{Q}\circ \tilde{\mathcal{Q}}^{-1}$ and $G \in \{\SL_n,\GL_n\}$. Then $\mathcal{G}$ is a birational map given by the following formula:
    \begin{equation}\label{eq:genbirat}
        \mathcal{G}(U) = G(U)^{-1} \cdot U \cdot G(U), \ \ G(U) := G_0(U)\prod_{i \geq 1}^{\rightarrow} (\tilde{\gamma}^*)^i(G_0(U)).
    \end{equation}
    Moreover, if $\tilde{\bg} \prec \bg$, then $\mathcal{G}:(G,\gc_h^{{\dagger}}(\tilde{\bg}))\dashrightarrow (G,\gc_h^{\dagger}(\bg))$ is a birational quasi-isomorphism, and if in addition $\tilde{r}_0 = r_0$, then $\mathcal{G}:(G,\pi_{\tilde{\bar{\bg}}}^{\dagger}) \dashrightarrow (G,\pi_{\bar{\bg}}^{\dagger})$ is a Poisson map.
\end{proposition}
In Section~\ref{s:bsimpl}, we will further refine the construction in the case $\tilde{\bg}$ is obtained from $\bg$ by removing a pair of roots. There, we will give a more explicit description of the matrix $G_0(U)$.
\begin{proof}
    Let us start with the definition $\mathcal{G}= \mathcal{Q}\circ \tilde{\mathcal{Q}}^{-1}$ and derive formula~\eqref{eq:genbirat}. Recall from~\eqref{eq:fc_alt} that 
		\[
		\tilde{\mathcal{Q}}^{-1}(U) = \tilde{\mathcal{F}}(U)\tilde{\theta}^*(\tilde{\mathcal{F}}(U)_-)^{-1} =\tilde{\theta}^*(\tilde{\mathcal{F}}(U)_-)U \tilde{\theta}^*(\tilde{\mathcal{F}}(U)_-)^{-1}.
		\]
		We see that
    \begin{equation*}
        \mathcal{G}(U) = \rho\left(\tilde{\mathcal{Q}}^{-1}(U)\right)^{-1} \tilde{\theta}^*(\tilde{\mathcal{F}}(U)_-) \cdot U \cdot \tilde{\theta}^*(\tilde{\mathcal{F}}(U)_-)^{-1} \rho\left(\tilde{\mathcal{Q}}^{-1}(U)\right).
    \end{equation*}
    Set $G(U) := \tilde{\theta}^*(\tilde{\mathcal{F}}(U)_-)^{-1} \rho\left(\tilde{\mathcal{Q}}^{-1}(U)\right)$. Since $\tilde{\gamma}^*(\rho(U)) = \tilde{\gamma}^*(U_-)^{-1}\rho(U)$, we see that
    \begin{equation*}\begin{split}
        \tilde{\gamma}^*(G(U)) = &\tilde{\gamma}^*\left(\tilde{\theta}^*(\tilde{\mathcal{F}}(U)_-)^{-1}\right) \tilde{\gamma}^*\left(\left[\tilde{\mathcal{F}}(U)\tilde{\theta}^*(\tilde{\mathcal{F}}(U)_-)^{-1}\right]_-\right)^{-1}\rho\left(\tilde{\mathcal{Q}}^{-1}(U)\right) = \\ =& \tilde{\gamma}^*([\tilde{\mathcal{F}}(U)_-])^{-1}\rho(\tilde{\mathcal{Q}}^{-1}(U)) = G_0(U)^{-1}G(U).
        \end{split}
    \end{equation*}
    Now, formula~\eqref{eq:genbirat} for $G(U)$ follows from the identity $G_0(U)\tilde{\gamma}^*(G(U)) = G(U)$.

    To show that $\mathcal{G}$ is a birational quasi-isomorphism, let us consider an increasing chain $\bg_1 \prec \bg_2 \prec \cdots \prec \bg_m$ of BD triples $\bg_i:=(\Gamma_{1,i},\Gamma_{2,i},\gamma_i)$ such that $\bg_1 := \tilde{\bg}$, $\bg_m := \bg$ and $|\Gamma_{1,i+1}\setminus \Gamma_{1,i}| = 1$, $i \in [1,m-1]$. Let $\mathcal{G}_i : (\GL_n,\pi_{(\bg_i,r_0)}^{\dagger})\dashrightarrow (\GL_n,\pi_{(\bg_{i+1},r_0)}^{\dagger})$ be defined relative $\bg_i$ and $\bg_{i+1}$, as above. In Proposition~\ref{p:quasi_iso_simple} below we show that each $\mathcal{G}_i:(\GL_n,\gc_h^{\dagger}(\bg_i)) \dashrightarrow (\GL_n,\gc_h^{\dagger}(\bg_{i+1}))$ is a birational quasi-isomorphism. In particular, by Proposition~\ref{p:q_y_var}, each $\mathcal{G}_i^*$ sends a nonmarked $y$-variable in the initial extended cluster of $\gc_h^{\dagger}(\bg_i)$ to a $y$-variable in the initial extended cluster of $\gc_h^{\dagger}(\bg_{i+1})$. Therefore, the same applies for $\mathcal{G}^*$, and by Proposition~\ref{p:q_y_var}, it is a birational quasi-isomorphism.
\end{proof}

\begin{remark}\label{r:birinv} 
Note that since we did not assume $\tilde{\bg}\prec \bg$ in the derivation of formula~\eqref{eq:genbirat}, the formula covers both $\mathcal{G}$ and its inverse $\mathcal{G}^{-1}$. Explicitly, since $\mathcal{G}^{-1} = \tilde{\mathcal{Q}}\circ \mathcal{Q}^{-1}$, we set
\begin{equation*}
    \tilde{G}_0(U):=\tilde{\theta}^*(\mathcal{F}(U)_-)^{-1}\tilde{\gamma}^*(\mathcal{F}(U)_-), \ \ U \in \GL_n;
\end{equation*}
then, $\mathcal{G}^{-1}$ is given by
\begin{equation*}
    \mathcal{G}^{-1}(U) = \tilde{G}(U) U \tilde{G}(U)^{-1}, \ \ \tilde{G}(U) =\left[\prod_{i \geq 1}^{\leftarrow}(\tilde{\theta}^*)^{i}(\tilde{G}_0(U)) \right] \cdot \tilde{G}_0(U).
\end{equation*}
If $\tilde{\bg}$ is the trivial BD triple, one recovers $\mathcal{G}^{-1} = \mathcal{F}^c$.
\end{remark}

\subsection{\texorpdfstring{Birational quasi-isomorphisms $\mathcal{G}$ for $|\Gamma_1\setminus\tilde{\Gamma}_1| = 1$}{Birational quasi-isomorphisms G for |Gamma 1 minus tilde Gamma 1| = 1}}\label{s:bsimpl}
Let $\bg = (\Gamma_1,\Gamma_2,\gamma)$ be a BD triple of type $A_{n-1}$. Fix any pair of roots $p \in \Gamma_1$ and $q:=\gamma(p) \in \Gamma_2$, and let $\tilde{\bg}$ be obtained from $\bg$ by removing the roots $p$ and $q$ (see Definition~\ref{d:bdremov}). Let $\mathcal{F}$ and $\tilde{\mathcal{F}}$ be the birational maps defined by~\eqref{eq:fdef} relative to the BD triples $\bg$ and $\tilde{\bg}$, respectively. Let $\mathcal{G}:=\mathcal{Q}\circ \tilde{\mathcal{Q}}^{-1}$ be the birational map constructed in Section~\ref{s:genbirat}. The objective of this subsection is to further refine formulas~\eqref{eq:geng0} and~\eqref{eq:genbirat} (the result is formula~\eqref{eq:g0hyp} below).

Let $\Delta$ and $\bar{\Delta}$ be the $X$- and $Y$-runs that contain $p$ and $q$, respectively. Set
\begin{equation*}
\tilde{\Delta}_1 := [p_-(\Gamma_1)+1,p], \ \ \ \tilde{\Delta}_2 := [p+1,p_+(\Gamma_1)].
\end{equation*}
 For $1 \leq i \leq |\tilde{\Delta}_2|$ and $1 \leq j \leq |\tilde{\Delta}_1|$, define the following sets of indices:
\begin{equation*}
J_{ij} := \begin{cases}
\{\bar{\gamma}(p_-(\Gamma_1)+j)\}\cup\left([q+1,n]\setminus\{\bar{\gamma}(p+i)\}\right) \ &\text{if} \ \gamma|_{\Delta} >0;\\
\{\bar{\gamma}(p+i)\}\cup\left([q+1,n]\setminus\{\bar{\gamma}(p_-(\Gamma_1)+j)\}\right) \ &\text{if} \ \gamma|_{\Delta}<0.\\
\end{cases}
\end{equation*}
Formula~\eqref{eq:geng0} for the matrix $G_0(U)$ can be refined as follows.

\begin{proposition}\label{p:g0expl}
Under the setup of the current subsection, let us set
\begin{equation*}
\alpha_{ij}(U) := (-1)^{i+1} \frac{\det [\tilde{\mathcal{F}}(U)]^{J_{ij}}_{[q+1,n]}}{\det [\tilde{\mathcal{F}}(U)]_{[q+1,n]}^{[q+1,n]}}.
\end{equation*}
Then
\begin{equation}\label{eq:g0hyp}
G_0(U) = I + \sum_{i=1}^{|\tilde{\Delta}_2|}\sum_{j=1}^{|\tilde{\Delta}_1|}\alpha_{ij}(U)e_{p+i,p_-(\Gamma_1)+j}.
\end{equation}
\end{proposition}

\begin{proof}
Due to formula~\eqref{eq:geng0}, the only nontrivial part of $G_0(U)$ is given by the submatrix $[G_0(U)]_{\tilde{\Delta}_2}^{\tilde{\Delta}_1}$. If $\gamma|_{\Delta \cap \Gamma_1} > 0$, we see that
\[\begin{split}
\alpha_{ij}(U) = &(-1)^{i+1}\det \left[\tilde{\mathcal{F}}(U)_-\right]_{[q+1,n]}^{\{q_-(\Gamma_2)+j\}\cup ([q+1,n]\setminus\{q+i\})} =\\= &\sum_{k \in [1,i]} (-1)^{i+k} \left[\tilde{\mathcal{F}}(U)_-\right]_{q+k,q_-(\Gamma_2)+j} \cdot \det \left[\tilde{\mathcal{F}}(U)_-\right]_{[q+1,n]\setminus \{q+k\}}^{[q+1,n]\setminus \{q+i\}} = \\ =& \sum_{k \in [1,i]}\left[\tilde{\gamma}^*\left(\tilde{\mathcal{F}}(U)_-\right)\right]_{p+k,p_-(\Gamma_1)+j} \left[\tilde{\theta}^*\left(\tilde{\mathcal{F}}(U)_-\right)^{-1}\right]_{p+i,p+k} = \\ =&G_0(U)_{p+i,p_-(\Gamma_1)+j}.
\end{split}
\] 
In the case of $\gamma|_{\Delta \cap \Gamma_1} < 0$, let us recall that $\bar{\gamma}$ is defined in such a way that $\bar{\gamma}(p) = q+1$ (see Section~\ref{s:bd_map}). Set $\tau:=q_+(\Gamma_2)-q+j$. We see that
\[\begin{split}
\alpha_{ij}(U) = &(-1)^{i+1} \det \left[\tilde{\mathcal{F}}(U)_-\right]_{[q+1,n]}^{\{q-i+1\}\cup \left([q+1,n]\setminus \{q_+(\Gamma_2)-j+1\}\right)} =\\ = &(-1)^{\tau+1} \det \left[\sinv{\tilde{\mathcal{F}}(U)_-}\right]_{[q-i+2,q]\cup\{q_+(\Gamma_2)-j+1\}}^{[q-i+1,q]} = \\ = &(-1)^{\tau+1} \sum_{k=1}^{j} (-1)^{k-1} \det \left[\sinv{\tilde{\mathcal{F}}(U)_-}\right]_{[q-i+1,q]\setminus\{q-i+1\}}^{[q-i+1,q]\setminus\{q-k+1\}} \left[\sinv{\tilde{\mathcal{F}}(U)_-}\right]_{q_+(\Gamma_2)-j+1,q-k+1} = \\ = &\sum_{k=1}^i \left[\tilde{\theta}^*\left(\sinv{\tilde{\mathcal{F}}(U)_-}\right)\right]_{p+i,p+k} \left[\tilde{\gamma}^*\left( \tilde{\mathcal{F}}(U)_-\right)  \right]_{p+k,p_-(\Gamma_1)+j} = G_0(U)_{p+i,p_-(\Gamma_1)+j}.
\end{split}
\]
where we also used formula~\eqref{eq:gammann}. The result follows.
\end{proof}

We will use the following lemma in order to prove that $\mathcal{G}$ is a quasi-isomorphism induced by toric actions (in the sense of Proposition~\ref{p:quasi_toric}), as well as to prove Part~\ref{i:gln_tor} of Theorem~\ref{thm:maingln} in Section~\ref{s:toric}.
\begin{lemma}\label{l:toric_glob_scalar}
    The action of $\mathbb{C}^*$ upon $\GL_n$ given by $(a,U) \mapsto a\cdot U$, $a \in \mathbb{C}^*$, $U \in \GL_n$, induces a global toric action of rank $1$ upon $\gc_h^{\dagger}(\bg,\GL_n)$.
\end{lemma}
\begin{proof}
    To prove that the local toric action is global, we appeal to Proposition~\ref{p:toric}. Notice that all the variables in the initial extended cluster $\Psi_0$ of $\gc_h^{\dagger}(\bg)$ are homogeneous polynomials. The weight vector $\omega$ of the induced local toric action upon $\Psi_0$ is given by the polynomial degrees of the variables; however, we assign the value $0$ to the components of $\omega$ that correspond to the $c$-variables (this ensures that the coefficients $\hat{p}_{ir}$ from the proposition are invariant with respect to the toric action). Showing that $B\omega = 0$ for the initial extended exchange matrix $B$ amounts to showing that the $y$-variables in $\Psi_0$ have polynomial degree $0$. From~\eqref{eq:big_phi_def_h}, we see that
    \begin{equation*}
        \deg \varphi_{kl} = l-1+\frac{1}{2}(n-k-l+1)(n-k-l+2).
    \end{equation*}
    In Section~\ref{s:p_explicit}, we will also show that for any $\gamma$-string $\{\alpha_i\}_{i=0}^m$ and each $j \in [\alpha_i+1,n]$, the polynomial degree of $h_{\alpha_i+1,j}$ is given by
    \begin{equation}\label{eq:degh2}
        \deg h_{\alpha_i+1,j} = n-j+1 + \sum_{\ell=1}^{m-i} (n-\alpha_{i+\ell}).
    \end{equation}
    Now, to show that the $y$-variables in $\Psi_0$ have polynomial degree $0$, we analyze each of the neighborhoods from Section~\ref{s:quiv_h_triv} and Section~\ref{s:h_quiv_expl}. For instance, consider Figure~\ref{f:inbd_hjj_3}, which is the neighborhood of $h_{jj}$ for $2 \leq j \leq n-1$, $j-1,j \in \Gamma_2$, $i:=\gamma^*(j-1)+1$ and $\gamma|_{\Delta(i)\cap \Gamma_1} < 0$. We see that
    \begin{equation*}\begin{split}
        \deg y(h_{jj}) = &(\deg h_{j,j+1}-\deg h_{in}) - (\deg h_{j-1,j} - \deg h_{j-1,j-1}) - \\ &- (\deg h_{j+1,j+1} - \deg h_{i-1,n}) = 2 - 1 - 1 = 0.
        \end{split}
    \end{equation*}
    Thus the action $(a,U) \mapsto a \cdot U$, $a \in \mathbb{C}^*$, $U \in \GL_n$ induces a global toric action of rank $1$ upon $\gc_h^{\dagger}(\bg,\GL_n)$.
\end{proof}

\begin{proposition}\label{p:quasi_iso_simple}
Let $\psi$ be any variable from the initial extended cluster of $\gc^\dagger_h(\bg)$. Then the following identities hold:
\begin{enumerate}[1)]
\item If $\psi$ is a $\varphi$- or $c$-function, then $\mathcal{G}^*(\psi(U)) = \psi(U)$;\label{i:bsimplvc}
\item\label{i:bsimplh} If $\psi = h_{ij}$ for some $h$-function $h_{ij}(U)$, then 
\begin{equation*}
\mathcal{G}^*(h_{ij}(U)) = \begin{cases}
\tilde{h}_{ij}(U)\tilde{h}_{q+1,q+1}(U) \ &\text{if} \ p = \gamma^{\nu}(i-1) \ \text{for some}\ \nu > 0;\\
\tilde{h}_{ij}(U) \ &\text{otherwise.}
\end{cases}
\end{equation*}
\end{enumerate}
Moreover, $\mathcal{G}$ is a birational quasi-isomorphism.
\end{proposition}
\begin{proof}
Due to the invariance properties of $\varphi$- and $c$-functions (see~\eqref{eq:phi_invar} and~\eqref{eq:c_def}) and the fact that $G(U)$ is a unipotent lower-triangular matrix, we see that Part~\ref{i:bsimplvc} holds. For the second part, we see that for any $j \in [1,n]$ and $I \subseteq [1,n]$, $|I| = n-j+1$, 
\[
\det[\mathcal{F}(\mathcal{G}(U))]_{I}^{[j,n]} = \det[\tilde{\mathcal{F}}(U)]_{I}^{[j,n]}.
\]
We also notice that if $S^{\gamma}(\alpha_0)$ is a $\gamma$-string such that $q \notin S^{\gamma}(\alpha_0)$, then $S^{\gamma}(\alpha_0) = S^{\theta}(\alpha_0)$; otherwise, if $q \in S^{\gamma}(\alpha_0)$, then $S^{\gamma}(\alpha_0) = S^{\theta}(\alpha_0)\cup S^{\theta}(q)$. Now Part~\ref{i:bsimplh} follows from an application of $\mathcal{G}^*$ to the definition of the $h$-functions given in~\eqref{eq:h_fun}. 

The last statement can be proved either via Proposition~\ref{p:q_y_var} or Proposition~\ref{p:quasi_toric}. Let us use the latter proposition, for it also shows that, as a quasi-isomorphism, $\mathcal{G}^*$  is generated by toric actions when $|\Gamma_1\setminus\tilde{\Gamma}_1| = 1$. First of all, by Lemma~\ref{l:toric}, there are global toric actions of rank $1$ acting upon $\gc_h^{\dagger}(\bg,\GL_n)$ and $\gc_h^{\dagger}(\tilde{\bg},\GL_n)$, whose weights are given by the polynomial degrees of the variables. From formula~\eqref{eq:degh2}, we also see that
\begin{equation*}
    \deg h_{ij} - \deg \tilde{h}_{ij} = \begin{cases}
        \deg \tilde{h}_{q+1,q+1} \ &\text{if} \ p = \gamma^{\nu}(i-1) \ \text{for some} \ \nu > 0;\\
        0 \ &\text{otherwise.}
    \end{cases}
\end{equation*}
It follows from Proposition~\ref{p:quasi_toric} that $\mathcal{G}^*$ is a quasi-isomorphism. Since $\mathcal{G}^*$ induces an isomorphism between $\mathbb{C}[\GL_n][h_{q+1,q+1}^{\pm 1}]$ and $\mathbb{C}[\GL_n][\tilde{h}_{q+1,q+1}^{\pm 1}]$, we conclude that it is a birational quasi-isomorphism.
\end{proof}
\section{\texorpdfstring{Proofs of Theorems~\ref{thm:maingln} and~\ref{thm:mainsln} for $\gc_h^{\dagger}(\bg)$}{Proofs of Theorems~\ref{thm:maingln} and~\ref{thm:mainsln} for for the h-convention}}\label{s:complet}
In this section, we prove Theorem~\ref{thm:maingln} and Theorem~\ref{thm:mainsln} for $\gc_h^{\dagger}(\bg)$. Each subsection corresponds to one part of a theorem (note that Part~\ref{i:gln_rk} follows from Part~\ref{i:gln_comp} due to Part~\ref{pr:c2} of Proposition~\ref{p:compb}).  The main idea is to use an inductive argument on the size $|\Gamma_1|$ along with the birational quasi-isomorphisms constructed in Section~\ref{s:birat}. The statements for $|\Gamma_1| = 0$ follow from the earlier work~\cite{double}, the statements for $|\Gamma_1|=1$ require a manual verification only for the marked variables, and the statements for $|\Gamma_1|>1$ follow from the existence of complementary birational quasi-isomorphisms.

\paragraph{Conventions.} Throughout the section, when the indices of the $h$-variables are out of range, the following conventions are used:
\begin{align}\label{eq:hconv}
&h_{1j}(U) := \varphi_{j-1,n-j+1}(U) = (-1)^{\varepsilon_{1j}}\det U^{[j,n]}_{[1,n-j+1]}, \ &j \in [2,n];\\
&h_{i,n+1}(U) := \begin{cases}
    h_{jj}(U) \ &\text{if} \ i-1 \in \Gamma_2\\
    1 \ &\text{otherwise}
\end{cases} \ &i \in [1,n]
\notag
\end{align}
where $j := \gamma^*(i-1)+1$ if $i-1 \in \Gamma_2$.

\subsection{Completeness}\label{s:completeness}
In this subsection, we prove that $\gc_h^{\dagger}(\bg,\GL_n)$ and $\gc_h^{\dagger}(\bg,\SL_n)$ are complete. Along the way, we prove that all the variables in the initial extended cluster are irreducible (as elements of $\mathbb{C}[\GL_n]$) and (for the initial cluster variables) coprime with their mutations. The final result of the subsection is Proposition~\ref{p:compl}, which corresponds to Part~\ref{i:gln_ni} of Theorem~\ref{thm:maingln} and Part~\ref{i:sln_ni} of Theorem~\ref{thm:mainsln}.

\subsubsection{The case of $|\Gamma_1|=|\Gamma_2| = 1$ (shifted cluster).}
Let $\bg = (\Gamma_1,\Gamma_2,\gamma)$ be a BD triple of type $A_{n-1}$ such that $\Gamma_1:=\{p\}$ and $\Gamma_2:=\{q\}$ for some simple roots $p$ and $q$.  We decorate the $h$-variables in $\gc_h^{\dagger}(\bg_{\std})$ with a tilde, to distinguish them from the $h$-variables in $\gc_h^{\dagger}(\bg)$. Recall that we choose a marker (see Definition~\ref{d:marker}) between $\gc_{h}^{\dagger}(\bg_{\std})$ and $\gc_h^{\dagger}(\bg)$ in such a way that each $\varphi$- and $c$- variable are related to themselves, and each $h_{ij}$ is related to $\tilde{h}_{ij}$. The only marked variable is given by $\psi_{\square}:=h_{q+1,q+1}$. The map $\mathcal{F}^c$ assumes the form
\begin{equation*}
\mathcal{F}^c(U) = \tilde{\gamma}^*(U_-) U \tilde{\gamma}^*(U_-)^{-1}, \ \ \tilde{\gamma}^*(U_-) = I + \alpha(U)e_{p+1,p}, \ \ \alpha(U) = \frac{1}{\psi_{\square}(U)}\det U^{\{q\}\cup[q+2,n]}_{[q+1,n]}.
\end{equation*}

\paragraph{Shifting the initial cluster.} We will prove $\bar{\mathcal{A}}_{\mathbb{C}}(\gc_h^{\dagger}(\bg,\GL_n)) = \mathbb{C}[\GL_n]$ starting not in the initial extended cluster, but in a more convenient one. Let us set $\Psi$ to be the extended cluster in $\gc_h^{\dagger}(\bg)$ obtained by mutating the initial extended cluster $\Psi_0$ along the following mutation sequence:\footnote{If $p = n-1$, the mutation sequence is empty, so  we set $\Psi$ to be the initial extended cluster.}
\begin{equation*}
{h}_{p+1,n} \rightarrow {h}_{p+1,n-1} \rightarrow \cdots \rightarrow {h}_{p+1,p+2}.
\end{equation*}
We set $\tilde{\Psi}$ to be the extended cluster in $\gc_h^{\dagger}(\bg_{\std})$ related to $\Psi$ (see Definition~\ref{d:relclust}) and $\Psi^{\prime}$ to be the extended cluster in $\gc_h^{\dagger}(\bg)$ obtained by mutating $\Psi$ in the direction of $\psi_{\square}$. 

\begin{lemma}\label{l:relpsi}
    The extended cluster $\Psi$ contains the variables
    \begin{equation}\label{eq:hpps}
    {h}^{\prime}_{p+1,n-s}(U) = -(-1)^{s(n-(p+1))} \det U_{\{p\}\cup[p+2,s+p+2]}^{[n-s-1,n]}, \ \ s \in [0,n-p-2].
    \end{equation}
    Moreover, for any $\psi \in \Psi\setminus\{h_{p+1,p+1}\}$ and the related variable $\tilde{\psi} \in \tilde{\Psi}\setminus\{\tilde{h}_{p+1,p+1}\}$, \mbox{$\tilde{\psi}(U) = \psi(U)$}.
\end{lemma}
\begin{proof}
It is well-known (see, for instance, \cite[Lemma 5.3]{multdouble}, with adjustments of signs) that the related extended cluster $\tilde{\Psi}$ contains the variables
    \begin{equation}\label{eq:thpps}
    {h}^{\prime}_{p+1,n-s}(U) = -(-1)^{s(n-(p+1))} \det U_{\{p\}\cup[p+2,s+p+2]}^{[n-s-1,n]}, \ \ s \in [0,n-p-2].
    \end{equation}
Applying $(\mathcal{F}^c)^*$ to~\eqref{eq:thpps}, we see that
\begin{equation*}
(\mathcal{F}^c)^*[\tilde{h}^{\prime}_{p+1,n-s}(U)] = -(-1)^{s(n-(p+1))} \det U_{\{p\}\cup[p+2,s+p+2]}^{[n-s-1,n]}.
\end{equation*}
Since $\mathcal{Q}$ is a quasi-isomorphism, we conclude that $\Psi$ contains the variables~\eqref{eq:hpps} as well. In particular, $\Psi$ and $\tilde{\Psi}$ differ only in the variable $h_{p+1,p+1}$.
\end{proof}

For the next lemma, let us note that the frozen variable $h_{p+1,p+1}$ in $\gc_h^{\dagger}(\bg)$ is given by 
\begin{equation*}
h_{p+1,p+1}(U) = \det U^{[p+1,n]}_{[p+1,n]} \det U^{[q+1,n]}_{[q+1,n]} + \det U^{[p+1,n]}_{\{p\}\cup [p+2,n]} \det U^{\{q\}\cup[q+2,n]}_{[q+1,n]}.
\end{equation*} 

\begin{lemma}
The mutation of $\Psi$ at the marked variable $\psi_{\square}$ produces the variable
\begin{equation}\label{eq:psisqm}
\psi_{\square}^{\prime}(U) = (-1)^{\varepsilon_{\square}} \left[\det U^{[p+1,n]}_{[p+1,n]} \det U_{I_{\square}}^{[q+1,n]} + \det U^{[p+1,n]}_{\{p\}\cup [p+2,n]} \det U^{\{q\}\cup[q+2,n]}_{I_{\square}}\right]
\end{equation}
where 
\begin{equation}\label{eq:intsq}
\varepsilon_{\square} := \begin{cases} n-q \ &\text{if}\ q-p\neq 1,\\ n-q+1 \ &\text{if} \ q-p = 1; \end{cases} \ \ \ \ I_{\square} := \begin{cases}
[q,n-1] \ &\text{if} \ q-p \neq 1,\\
\{q-1\}\cup[q+1,n-1] \ &\text{if} \ q-p=1.
\end{cases}
\end{equation}
\end{lemma}
\begin{proof}
\emph{Case $q-p = 1$, $q \leq n-2$.} The mutation relation for $\psi_{\square}$ reads
\begin{equation*}
\psi_{\square}\psi_{\square}^{\prime} = (h_{p+1,p+2}^\prime)^2 h_{q+1,q+2} + h_{p+1,p+3}^\prime h_{p+1,p+1}.
\end{equation*}
Expanding the right-hand side, we see that
\begin{equation}\label{eq:complcso}
\begin{split}
(-1)^{n-p} &\psi_{\square}(U)\psi_{\square}^{\prime}(U) = \left[\det U^{[p+1,n]}_{\{p\}\cup[p+2,n]}\right]^{2} \det U^{[p+3,n]}_{[p+2,n-1]} +\\ +  &\det U^{[p+2,n]}_{\{p\}\cup [p+2,n-1]} \left[ \det U^{[p+1,n]}_{[p+1,n]} \det U^{[q+1,n]}_{[q+1,n]} + \det U^{\{q\}\cup [q+2,n]}_{[q+1,n]} \det U^{[p+1,n]}_{\{p\}\cup [p+2,n]} \right].
\end{split}
\end{equation}
Applying the Desnanot--Jacobi identity (see Proposition~\ref{p:dj22}) to the matrix $A:= U^{[p+1,n]}_{\{p\}\cup[p+2,n]}$, we see that
\begin{equation}\label{eq:compldj}\begin{split}
\det U^{[p+1,n]}_{\{p\}\cup[p+2,n]} \det U^{[p+3,n]}_{[p+2,n-1]} + \det& U^{[p+2,n]}_{\{p\}\cup[p+2,n-1]} \det U^{\{p+1\}\cup[p+3,n]}_{[p+2,n]}=\\ = &\det U^{[p+2,n]}_{[p+2,n]} \det U^{\{p+1\}\cup[p+3,n]}_{\{p\}\cup[p+2,n-1]}.
\end{split}
\end{equation}
Now, combining the first and the third terms in~\eqref{eq:complcso} via~\eqref{eq:compldj}, and then dividing both sides by $(-1)^{n-p} \det U^{[q+1,n]}_{[q+1,n]}$, formula \eqref{eq:psisqm} follows.

\emph{Case $q-p = 1$, $q = n -1$.} In this case, the mutation at $\psi_{\square}$ reads
\begin{equation*}
\psi_{\square}\psi_{\square}^{\prime} = h_{n-2,n}h_{n-1,n-1} + (h_{n-1,n}^{\prime})^{2}.
\end{equation*}
Expanding the latter relation, we see that
\begin{equation*}\begin{split}
\psi_{\square}(U) \psi_{\square}^{\prime}(U) = &u_{n-2,n} \left[\det U^{[n-1,n]}_{[n-1,n]} u_{nn} + \det U^{[n-1,n]}_{\{n-2\}\cup\{n\}}u_{n-1,n}\right] + \left[ \det U^{[n-1,n]}_{\{n-2\}\cup\{n\}}\right]^{2} = \\ = &u_{nn} \left[u_{n-2,n-1} \det U^{[n-1,n]}_{\{n-2\}\cup\{n\}} + u_{n-2,n} \det U^{[n-1,n]}_{[n-1,n]} \right].
\end{split}
\end{equation*}
Dividing both sides by $\psi_{\square}(U) = u_{nn}$, we obtain \eqref{eq:psisqm}.

\emph{Case $q-p \neq 1$.} In this case, the mutation at $\psi_{\square}$ reads
\begin{equation*}
\psi_{\square}\psi_{\square}^{\prime} = h_{p+1,p+2}^{\prime} h_{qq} h_{q+1,q+2} + h_{q,q+1} h_{p+1,p+1}.
\end{equation*}
Expanding the right-hand side, we see that
\begin{equation}\label{eq:kakoet}
\begin{split}
(-1)^{n-q}\psi_{\square}(U)&\psi_{\square}^{\prime}(U) = \det U^{[p+1,n]}_{\{p\}\cup[p+2,n]} \det U^{[q,n]}_{[q,n]} \det U^{[q+2,n]}_{[q+1,n-1]} + \\ + & \det U^{[q+1,n]}_{[q,n-1]} \left[ \det U^{[p+1,n]}_{[p+1,n]} \det U^{[q+1,n]}_{[q+1,n]} + \det U^{[p+1,n]}_{\{p\}\cup[p+2,n]} \det U^{\{q\}\cup[q+2,n]}_{[q+1,n]}\right].
\end{split}
\end{equation}
Applying the Desnanot--Jacobi identity (see Proposition~\ref{p:dj22}) to the matrix $A:= U^{[q,n]}_{[q,n]}$, we see that
\begin{equation}\label{eq:kakoetj}
\det U^{[q,n]}_{[q,n]} \det U^{[q+2,n]}_{[q+1,n-1]} + \det U^{[q+1,n]}_{[q,n-1]} \det U^{\{q\}\cup [q+2,n]}_{[q+1,n]} = \det U^{[q+1,n]}_{[q+1,n]} \det U^{\{q\}\cup[q+2,n]}_{[q,n-1]}.
\end{equation}
Combining the first and the third terms in~\eqref{eq:kakoet} with~\eqref{eq:kakoetj} and dividing both sides by $(-1)^{n-q} \psi_{\square}$, we obtain formula~\eqref{eq:psisqm}.
\end{proof}

\begin{lemma}\label{l:basecoprim}
    The regular functions $\psi_{\square}$ and $\psi^\prime_{\square}$ are coprime.
\end{lemma}
\begin{proof}
    Evidently, $\psi_{\square}$ is irreducible, so it suffices to show that $\psi^{\prime}_{\square}(U)$ is not divisible by $\psi_{\square}(U)$. If $p+1 < q+1$, then $\psi_{\square}^{\prime}$ is generically nonzero on the variety given by $u_{n,q+1} = u_{n,q+2} = \cdots = u_{nn} = 0$, whereas $\psi_{\square}$ vanishes. If $p+1 > q+1$, then let $V$ be the variety given by the equations 
    \begin{equation*}
    u_{i,q+1} - u_{i,q+2} = 0, \ \ i \in [q+1,n].
    \end{equation*}
    We see that $\psi_{\square}(U)|_{V} = 0$; on the other hand,
    \begin{equation*}
    \psi^{\prime}_{\square}(U)|_{V} = (-1)^{\varepsilon_{\square}} \det U^{[p+1,n]}_{\{p\}\cup[p+2,n]} \det U^{\{q\}\cup [q+2,n]}_{[q,n-1]},
    \end{equation*}
    which is generically nonzero on $V$. Thus $\psi_{\square}$ and $\psi_{\square}^{\prime}$ are coprime.
\end{proof}
\begin{lemma}\label{l:transc}
Let $\mathcal{F}$ be a field of characteristic zero, and let $\alpha$ and $\beta$ be distinct transcendental elements over $\mathcal{F}$ such that $\mathcal{F}(\alpha) = \mathcal{F}(\beta)$. If there is a relation
\begin{equation*}
    \sum_{k=1}^m (\alpha^k-\beta^k)p_k = 0
\end{equation*}
for $p_k \in \mathcal{F}$, then all $p_k = 0$.
\end{lemma}
\begin{proof}
    See \cite[Lemma 5.15]{multdouble}.
\end{proof}

\begin{lemma}\label{l:fixptincl}
    For any $f \in \mathbb{C}[\GL_n]$, if $\mathcal{Q}^*(f) = f$, then $f \in \mathcal{L}_{\mathbb{C}}(\Psi^{\prime})$.
\end{lemma}
\begin{proof}
Since $\mathbb{C}[\GL_n] \subseteq \mathcal{L}_{\mathbb{C}}(\tilde{\Psi})$, one can find a sequence of elements $f_k \in \mathcal{L}_{\mathbb{C}}(\tilde{\Psi})$ that do not involve the variable $h_{p+1,p+1}$, such that
\begin{equation}\label{eq:ffkhp}
f(U) = \sum_{k \geq 0} f_{k}(U) [\tilde{h}_{p+1,p+1}(U)]^k.
\end{equation}
Note that $\tilde{h}_{p+1,p+1}$ is frozen in $\gc_h^{\dagger}(\bg_{\std})$, so the expression is polynomial in $\tilde{h}_{p+1,p+1}$. Applying $(\mathcal{F}^c)^*$ to~\eqref{eq:ffkhp} and using Lemma~\ref{l:relpsi} together with the assumption $(\mathcal{F}^c)^*(f) = f$ yields
\begin{equation}\label{eq:fffk}
f(U) = \sum_{k \geq 0} f_{k}(U) \left[\frac{h_{p+1,p+1}(U)}{\psi_{\square}(U)}\right]^k.
\end{equation}
Subtracting~\eqref{eq:ffkhp} from~\eqref{eq:fffk} and invoking Lemma~\ref{l:transc}, we conclude that $f(U) = f_0(U)$. Since $\tilde{\psi}_{\square}(U)$ is frozen in $\gc_h^{\dagger}(\tilde{\bg})$, it enters $f_0(U)$ only with nonnegative exponents, so one can rewrite it in terms of a fraction with $\psi_{\square}^{\prime}$ in the denominator. Thus $f \in \mathcal{L}_{\mathbb{C}}(\Psi^{\prime})$.
\end{proof}

\begin{proposition}\label{p:baseshif}
Under the setup of the current subsection, the following statements hold:
\begin{enumerate}[1)]
    \item All the variables in $\Psi$ are regular and irreducible;
    \item For any cluster variable $\psi \in \Psi$, its mutation $\psi^\prime$ is regular and coprime with $\psi$;
    \item The generalized cluster structures $\gc_h^{\dagger}(\bg,\GL_n)$ and $\gc_h^{\dagger}(\bg,\SL_n)$ are complete.
\end{enumerate}
\end{proposition}
\begin{proof}
The marked variable $\psi_{\square}$ is regular and irreducible, and by Lemma~\ref{l:basecoprim}, its mutation $\psi_{\square}^{\prime}$ is regular and coprime with $\psi_{\square}$. By Proposition~\ref{p:birat_fsingle}, $\bar{\mathcal{A}}_{\mathbb{C}}(\gc^{\dagger}_h(\Gamma,\GL_n))\subseteq \mathbb{C}[\GL_n]$. 

By Proposition~\ref{p:upper_cont1}, it remains to show that $\mathbb{C}[\GL_n] \subseteq \mathcal{L}_{\mathbb{C}}(\Psi^{\prime})$. Since $\mathbb{C}[\GL_n] \subseteq \mathcal{L}_{\mathbb{C}}(\tilde{\Psi})$, it suffices to show that $\mathcal{L}_{\mathbb{C}}(\tilde{\Psi})\subseteq \mathcal{L}_{\mathbb{C}}(\Psi^{\prime})$; in turn, due to Lemma~\ref{l:relpsi}, proving the latter inclusion reduces to showing that $\det U^{[p+1,n]}_{[p+1,n]} \in \mathcal{L}_{\mathbb{C}}(\Psi^\prime)$. First, let us notice that
\begin{equation*}
(\mathcal{F}^c)^*\left( \det U^{\{q\}\cup [q+2,n]}_{I_{\square}}\right) = \det U^{\{q\}\cup [q+2,n]}_{I_{\square}},
\end{equation*} 
where $I_{\square}$ is given by~\eqref{eq:intsq}. By Lemma~\ref{l:fixptincl},  $\det U^{\{q\}\cup[q+2,n]}_{I_{\square}} \in \mathcal{L}_{\mathbb{C}}(\Psi^{\prime})$. Second, notice that $\det U^{[q+1,n]}_{I_{\square}}$ is a cluster variable in $\Psi^{\prime}$ (hence an invertible element of $\mathcal{L}_{\mathbb{C}}(\Psi^{\prime})$), for
\begin{equation*}
\det U^{[q+1,n]}_{I_{\square}} = \begin{cases}
    (-1)^{n-q} h_{q,q+1}(U) \ &\text{if} \ q-p \neq 1;\\
    (-1)^{n-p} h^{\prime}_{p+1,p+3}(U) \ &\text{if}\ q-p = 1, \ p \leq n-3;\\
    (-1)^{n-p} h_{pn}(U) \ &\text{if} \ q-p = 1, \ p = n-2.
\end{cases}
\end{equation*}
It follows that one can solve the equation~\eqref{eq:psisqm} for $\det U^{[p+1,n]}_{[p+1,n]}$ in the ring $\mathcal{L}_{\mathbb{C}}(\Psi^{\prime})$. Thus $\gc^{\dagger}_h(\bg,\GL_n)$ is complete.
\end{proof}

\subsubsection{The case of $|\Gamma_1|=|\Gamma_2| = 1$ (initial cluster).}
Although in the previous subsection we proved  $\bar{\mathcal{A}}(\gc_h^{\dagger}(\bg,\GL_n)) = \mathbb{C}[\GL_n]$ for $\bg = (\Gamma_1,\Gamma_2,\gamma)$ with $|\Gamma_1| = 1$, we proved regularity, irreducibility and coprimality with mutations not in the initial extended cluster $\Psi_0$, but in a certain shifted extended cluster $\Psi$. According to Proposition~\ref{p:birat_fsingle}, in order to claim the same properties for $\Psi_0$, we need to verify them for the marked variable in $\Psi_0$. As above, let us assume that $\Gamma_1 = \{p\}$ and $\Gamma_2 = \{q\}$, and set $\psi_{\square} := h_{q+1,q+1}$.

\begin{lemma}\label{l:inimarked}
The mutation of $\Psi_0$ at the marked variable $\psi_{\square}$ produces the variable
\[
\psi_{\square}^\prime(U) = \begin{cases} 
(-1)^{p+q+1} \left[ u_{p+1,n}\det U^{[q+1,n]}_{[q,n-1]}  + u_{pn}\det U^{\{q\}\cup[q+2,n]}_{[q,n-1]}  \right] \  &q\neq p+1;\\[10pt]
\det U^{[q+1,n]}_{[q+1,n]}\left[ u_{p+1,n}\det U^{[q+1,n]}_{[q,n-1]}  + u_{pn}\det U^{\{q\}\cup[q+2,n]}_{[q,n-1]}  \right] + \\[7pt] + \det U^{\{q\}\cup[q+2,n]}_{[q+1,n]}\left[u_{p+1,n}\det U^{[q+1,n]}_{\{q-1\}\cup[q+1,n-1]} + u_{pn} \det U^{\{q\}\cup[q+2,n]}_{\{q-1\}\cup[q+1,n-1]}  \right] \ &q = p+1.
\end{cases}
\]
\end{lemma}
\begin{proof}
The exchange relation is given by
\begin{equation*}
\psi_{\square}\psi_{\square}^\prime = h_{q,q+1} h_{p+1,n} + h_{qq} h_{pn} h_{q+1,q+2}.
\end{equation*}
Let us study two cases:

\emph{Case $q \neq p+1$.} Writing out all the functions, the relation becomes
\begin{equation}\label{eq:abca}
\begin{split}
(-1)^{p+q+1}\psi_{\square}(U) \psi_{\square}^\prime(U) = \det & U^{[q+1,n]}_{[q,n-1]} u_{p+1,n} \det U^{[q+1,n]}_{[q+1,n]} + \\ + &u_{pn} \left(\det U^{[q+1,n]}_{[q,n-1]} \det U^{\{q\}\cup [q+2,n]}_{[q+1,n]} + \det U^{[q,n]}_{[q,n]} \det U^{[q+2,n]}_{[q+1,n-1]} \right).
\end{split}
\end{equation}
Applying the Desnanot--Jacobi relation (see Proposition~\ref{p:dj22}) to the last term, we see that 
\[
\det U^{[q+1,n]}_{[q,n-1]} \det U^{\{q\}\cup [q+2,n]}_{[q+1,n]} + \det U^{[q,n]}_{[q,n]} \det U^{[q+2,n]}_{[q+1,n-1]} = \det U^{[q+1,n]}_{[q+1,n]} \det U^{\{q\}\cup [q+2,n]}_{[q,n-1]};
\]
hence, after dividing both sides by $\det U^{[q+1,n]}_{[q+1,n]}$ in~\eqref{eq:abca}, the result follows.

\emph{Case $q = p+1$.} Writing out all functions and rearranging the terms, the relation becomes
\begin{equation}\label{eq:mutq1e}
\begin{split}
\psi_{\square}^\prime(U) = u_{p+1,n}&\left[ \det U^{[q+1,n]}_{[q+1,n]} \det U^{[q+1,n]}_{[q,n-1]} + \det U^{[q+1,n]}_{\{q-1\}\cup[q+1,n-1]} \det U^{\{q\}\cup[q+2,n]}_{[q+1,n]} \right] + \\ + u_{pn} &\left[\det U^{[q+1,n]}_{[q,n-1]} \det U^{\{q\}\cup[q+2,n]}_{[q+1,n]} + \det U^{[q,n]}_{[q,n]} \det U^{[q+2,n]}_{[q+1,n-1]} \right] +\\+&  \frac{u_{pn} \det U^{\{q\}\cup[q+2,n]}_{[q+1,n]}}{\det U^{[q+1,n]}_{[q+1,n]}} \cdot A
\end{split}
\end{equation}
where $A$ is given by
\begin{equation}\label{eq:mutae}
A =  \det U^{[q+1,n]}_{\{q-1\}\cup[q+1,n-1]} \det U^{\{q\}\cup[q+2,n]}_{[q+1,n]} + \det U^{[q,n]}_{\{q-1\}\cup [q+1,n]} \det U^{[q+2,n]}_{[q+1,n-1]}.
\end{equation}
The second term in~\eqref{eq:mutq1e} can be simplified via the Desnanot--Jacobi identity as 
\begin{equation*}
\det U^{[q+1,n]}_{[q,n-1]} \det U^{\{q\}\cup[q+2,n]}_{[q+1,n]} + \det U^{[q,n]}_{[q,n]} \det U^{[q+2,n]}_{[q+1,n-1]} = \det U^{[q+1,n]}_{[q+1,n]} \det U^{\{q\}\cup[q+2,n]}_{[q,n-1]};
\end{equation*}
likewise, formula~\eqref{eq:mutae} can be simplified as
\begin{equation*}
\begin{split}\det U^{[q+1,n]}_{\{q-1\}\cup[q+1,n-1]}\det U^{\{q\}\cup[q+2,n]}_{[q+1,n]} +  \det& U^{[q,n]}_{\{q-1\}\cup [q+1,n]} \det U^{[q+2,n]}_{[q+1,n-1]} =\\= &\det U^{[q+1,n]}_{[q+1,n]} \det U^{\{q\}\cup[q+2,n]}_{\{q-1\}\cup[q+1,n-1]}.
\end{split}
\end{equation*}
Combining the above terms, the formula for $\psi_{\square}^\prime$ follows.
\end{proof}

\begin{proposition}\label{p:baseini}
Under the setup of the current subsection, the following statements hold:
    \begin{enumerate}[1)]
    \item All the variables in $\Psi_0$ are regular and irreducible;
    \item For any cluster variable $\psi \in \Psi_0$, its mutation $\psi^\prime$ is regular and coprime with $\psi$.
    \end{enumerate}
\end{proposition}
\begin{proof}
    According to Proposition~\ref{p:birat_fsingle}, it is enough to verify the above statements only for $\psi_{\square}$. Evidently, $\psi_{\square}$ is irreducible and $\psi_{\square}^{\prime}$ is regular (see Lemma~\ref{l:inimarked}). If $q \neq p+1$, we see that $\psi_{\square}$ vanishes on the variety given by $u_{n,q+1} = u_{n,q+2} = \cdots = u_{nn} = 0$, whereas $\psi_{\square}^\prime(U)$ is generically nonzero. If $q = p+1$, then it suffices to check that
    \begin{equation*}
        u_{p+1,n}\det U^{[q+1,n]}_{\{q-1\}\cup[q+1,n-1]} + u_{pn} \det U^{\{q\}\cup[q+2,n]}_{\{q-1\}\cup[q+1,n-1]}  
    \end{equation*}
    is not divisible by $\psi_{\square}(U)$, which immediately follows from the same argument.
\end{proof}

\subsubsection{Completeness for any BD triple}\label{s:complet_final}
In this subsection, we finalize the proofs of Part~\ref{i:gln_ni} of Theorem~\ref{thm:maingln} and Part~\ref{i:sln_ni} of Theorem~\ref{thm:mainsln}.

\begin{proposition}\label{p:compl}
    Let $\bg := (\Gamma_1,\Gamma_2,\gamma)$ be a BD triple of type $A_{n-1}$ and $G \in \{\SL_n,\GL_n\}$. Then the following statements hold:
    \begin{enumerate}[1)]
    \item The generalized cluster structure $\gc_h^{\dagger}(\bg,G)$ is regular;
    \item All the variables in the initial extended cluster of $\gc_h^{\dagger}(\bg,G)$ are irreducible (as elements of $\mathbb{C}[G]$) and (for cluster variables) coprime with their mutations;
    \item The generalized cluster structure $\gc_h^{\dagger}(\bg,G)$ is complete. 
    \end{enumerate}
\end{proposition}
\begin{proof}
The regularity of the initial extended cluster follows from Proposition~\ref{p:h_expansion}. As in the previous papers~\cite{plethora,multdouble}, we run an induction on the size $|\Gamma_1|$ for a fixed $n$. For $|\Gamma_1| = 0$, the statements were proved in~\cite[Theorem 3.8, Theorem 3.10]{double}; for $|\Gamma_1| = 1$, the statements are contained in Proposition~\ref{p:baseshif} and Proposition~\ref{p:baseini}. For $|\Gamma_1| > 1$, the statements follow from the existence of complementary birational quasi-isomorphisms $\mathcal{G}$ and Propositions~\ref{p:birat_double_ufd}-\ref{p:upper_cont2}.
\end{proof}
\subsection{Toric action}\label{s:toric}
In this subsection, we prove Part~\ref{i:gln_tor} of Theorems~\ref{thm:maingln} and \ref{thm:mainsln}. Given a BD triple $\bg:=(\Gamma_1,\Gamma_2,\gamma)$ of type $A_{n-1}$, let us recall that
\begin{equation*}
\mathfrak{h}_{\bg}=\{h \in \mathfrak{h}^s \ | \ \alpha(h) = \beta(h) \ \text{if} \ \gamma^j(\alpha) = \beta \ \text{for some} \ j\}.
\end{equation*}
where $\mathfrak{h}^s$ is the Cartan subalgebra of $\sll_n(\mathbb{C})$. The dimension of $\mathfrak{h}_{\bg}$ is equal to $k_{\bg}:=|\Pi\setminus\Gamma_1|$. Let $\mathcal{H}_{\bg}$ be the connected subgroup of $\SL_n$ with Lie algebra $\mathfrak{h}_{\bg}$. We let the group $\mathcal{H}_{\bg}$ act upon $\SL_n$ by conjugation. 

\begin{lemma}\label{l:toric}
Let $\Gamma_1 = \{p\}$ and $\Gamma_2 = \{q\}$ for some simple roots $p$ and $q$. Let $\psi_{\square}:=h_{q+1,q+1}$ be the marked variable. Then the $y$-variable $y(\psi_{\square})$ is invariant with respect to the action by $\mathcal{H}_{\bg}$.
\end{lemma}
\begin{proof}
Let us recall from Proposition~\ref{p:q_equiv} that $\mathcal{Q}$ is equivariant with respect to the action by $\mathcal{H}_{\bg}$. In the setup of the lemma, $\mathcal{H}_{\bg}$ is given by
\begin{equation*}
\mathcal{H}_{\bg} = \{T = \diag(t_1,t_2,\ldots,t_{n}) \ | \ \det T = 1, \ t_p t_{p+1}^{-1}t_{q+1}t_q^{-1} = 1\}.
\end{equation*}
Applying $\mathcal{Q}^*$ to $y(\psi_{\square})$, we see that
\begin{equation*}
\mathcal{Q}^*(y(\psi_{\square})) = \frac{\tilde{h}_{qq}}{\tilde{h}_{q,q+1}} \frac{\tilde{h}_{pn}}{\tilde{h}_{p+1,n}}\frac{\tilde{h}_{q+1,q+2}}{\tilde{h}_{q+1,q+1}}.
\end{equation*}
Now, it is a simple observation that $\tilde{h}_{q+1,q+1}(U)$ and $\tilde{h}_{qq}(U)$ are invariant with respect to $\mathcal{H}_{\bg}$, as well as
\begin{equation*}\begin{aligned}
&\tilde{h}_{pn}(TUT^{-1}) =t_{p}t_{n}^{-1}\tilde{h}_{pn}(U), \ & \ &\tilde{h}_{q+1,q+2}(TUT^{-1}) = t_{q+1}t_{n}^{-1}\tilde{h}_{q+1,q+2}(U), \\
&\tilde{h}_{q,q+1}(TUT^{-1}) = t_qt_n^{-1}\tilde{h}(U), \ & \ &\tilde{h}_{p+1,n}(TUT^{-1}) =t_{p+1}t_{n}^{-1}\tilde{h}_{p+1,n}(U).
\end{aligned}
\end{equation*}
Therefore,
\begin{equation*}
\mathcal{Q}^*\left(y(\psi_{\square})(TUT^{-1})\right) = \mathcal{Q}^*\left(y(\psi_{\square})(U)\right) \frac{t_pt_n^{-1}t_{q+1}t_{n}^{-1}}{t_qt_n^{-1}t_{p+1}t_n^{-1}} = \mathcal{Q}^*\left(y(\psi_{\square})(U)\right).
\end{equation*}
It follows that $y(\psi_{\square})$ is invariant with respect to the action by $\mathcal{H}_{\bg}$.
\end{proof}

\begin{proposition}\label{p:toric_h}
Let $\bg$ be a BD triple of type $A_{n-1}$. The following statements hold:
\begin{enumerate}[1)]
    \item The action of $\mathcal{H}_{\bg}$ upon $\SL_n$ induces a global toric action upon $\gc_h^{\dagger}(\bg,\SL_n)$ of rank $k_{\bg}$;
    \item The action of $\mathcal{H}_{\bg}\times \mathbb{C}^*$ upon $\GL_n$ induces a global toric action upon $\gc_h^{\dagger}(\bg,\GL_n)$ of rank $k_{\bg}+1$.
\end{enumerate}
\end{proposition}
\begin{proof}
The statement for $\mathbb{C}^*$ was already proved in Lemma~\ref{l:toric_glob_scalar}, so let us focus on $\mathcal{H}_{\bg}$; the proof is based on Proposition~\ref{p:toric}. Let us notice that due to Proposition~\ref{p:q_equiv}, all birational quasi-isomorphisms $\mathcal{G}$ that were introduced in Sections~\ref{s:genbirat} and \ref{s:bsimpl} are equivariant with respect to $\mathcal{H}_{\bg}$. This implies that $\mathcal{H}_{\bg}$ induces a local toric action upon the initial extended cluster of $\gc_h^{\dagger}(\bg)$. The fact that the action is of maximal rank follows from a standard argument\footnote{Let us briefly recall the argument. Assume that the toric action is of rank less than maximal. This means that there exists a $1$-dimensional subgroup of $\mathcal{H}_{\bg}$ that leaves all cluster and frozen variables from the initial extended cluster $\Psi_0$ invariant. However, since $\mathbb{C}[\GL_n]\subseteq \mathcal{L}_{\mathbb{C}}(\Psi_0)$, we see that the action leaves \emph{all} regular functions invariant. In other words, it is trivial.} that was utilized in \cite{plethora,multdouble}. The invariance of the Casimirs $\hat{p}_{ir}$ (which are the $c$-variables in our case) was proved before in~\cite[Theorem 6.1]{double}; therefore, it remains to show that all the $y$-variables in the initial extended cluster $\Psi_0$ of $\gc_h^{\dagger}(\bg)$ are invariant with respect to the action by $\mathcal{H}_{\bg}$.

We proceed by induction on the size $|\Gamma_1|$. If $|\Gamma_1| = 0$, then the statement was already proved in~\cite[Theorem 6.1]{double}. Let $\bg$ be a BD triple of type $A_{n-1}$ and $\tilde{\bg}$ be obtained from $\bg$ via removal of a pair of roots. Let $\psi_{\square}$ be the corresponding marked variable. Notice that $\mathfrak{h}_{\tilde{\bg}}\supset \mathfrak{h}_{\bg}$. If $|\Gamma_1| = 1$ and $y(\psi)$ is a $y$-variable of $\psi\neq \psi_{\square}$, we see that
\begin{equation*}
\mathcal{G}^*\left(y(\psi)(TUT^{-1})\right) = y(\psi)(T\mathcal{G}(U)T^{-1}) = y(\psi)(\mathcal{G}(TUT^{-1})) = y(\tilde{\psi})(U).
\end{equation*}
If $\psi = \psi_{\square}$, then the invariance follows from Lemma~\ref{l:toric}. The case of $|\Gamma_1| > 1$ follows from the existence of complementary birational quasi-isomorphisms $\mathcal{G}$ (see Section~\ref{s:genbirat}).
\end{proof}
\subsection{Compatibility}\label{s:compatibility}
In this subsection, we use the Poisson properties of the birational quasi-isomorphism $\mathcal{Q}$ in order to prove Part~\ref{i:gln_comp} of Theorem~\ref{thm:maingln} and Theorem~\ref{thm:mainsln}; that is, the generalized cluster structure $\gc_h^{\dagger}(\bg)$ is compatible with $\pi^{\dagger}_{\bar{\bg}}$ for any BD quadruple $\bar{\bg}$ of type~$A_{n-1}$.

\begin{lemma}\label{l:compmarked}
    Let $\bg:=(\Gamma_1,\Gamma_2,\gamma)$ be a BD triple such that $\Gamma_1 = \{p\}$, $\Gamma_2 = \{q\}$ for some simple roots $p$ and $q$, and let $\bar{\bg}:=(\bg,r_0)$ be a BD quadruple. Then for any cluster or frozen variable $\psi$ from the initial extended cluster of $\gc^{\dagger}_h(\bg)$,
    \begin{equation*}
    \{\log y(\psi_{\square}), \log \psi\}_{\bar{\bg}}^{\dagger} = \begin{cases}
        1 \ &\text{if} \ \psi = h_{q+1,q+1},\\
        0 \ &\text{otherwise}
    \end{cases}
    \end{equation*}
    where $\psi_{\square}:=h_{q+1,q+1}$ is the marked variable.
\end{lemma}
\begin{proof}
Let us apply the birational quasi-isomorphism $\mathcal{Q}$ and prove an equivalent statement for the trivial BD triple. We see that
\begin{equation*}
    \mathcal{Q}^* (y(\psi_{\square})) = \frac{\tilde{h}_{qq}}{\tilde{h}_{q,q+1}} \frac{\tilde{h}_{pn}}{\tilde{h}_{p+1,n}}\frac{\tilde{h}_{q+1,q+2}}{\tilde{h}_{q+1,q+1}}.
\end{equation*}
Let us first consider the case when $\psi$ is an $h$-variable. If $h_1$ and $h_2$ are two functions on $\GL_n$ such that $h_i(UN_-) = h_i(U)$, where $N_-$ is a unipotent lower-triangular matrix and $i \in \{1,2\}$, then the Poisson bracket $\{\cdot,\cdot\}_{(\bg_\std,r_0)}^{\dagger}$ between them is given by
\begin{equation*}
\{h_1,h_2\}_{(\bg_\std,r_0)}^{\dagger} = \langle \pi_{>} \nabla_U^R h_1,\nabla_U^R h_2\rangle + \langle R_0[\nabla_U h_1,U],[\nabla_U h_2,U]\rangle - \langle \pi_0[\nabla_U h_1,U],\nabla_U h_2 U\rangle.
\end{equation*}
It follows from Lemma~\ref{l:grrhijks} that 
\begin{equation*}
\langle \pi_{>}\left( \nabla_U^R \log \frac{\tilde{h}_{pn}}{\tilde{h}_{p+1,n}}\right),\nabla_U^R \log \tilde{h}_{ks} \rangle = -\chi_{n-s=p-k}
\end{equation*}
\begin{equation*}
\langle \pi_{>}\left(\nabla_U^R \log \frac{\tilde{h}_{qq}}{\tilde{h}_{q,q+1}}\frac{\tilde{h}_{q+1,q+2}}{\tilde{h}_{q+1,q+1}}  \right), \nabla_U^R \log \tilde{h}_{ks}\rangle = -\chi_{k \leq q}\cdot \chi_{s = q+1}
\end{equation*}
where $\chi_{(\cdot)} \in \{0,1\}$ is an indicator that returns true or false depending on whether the corresponding condition is satisfied. It follows from Lemma~\ref{l:hijlrd} that
\begin{equation}\label{eq:gradqyps}
\begin{aligned}
&\pi_0\left( \nabla_U^L \log \mathcal{Q}^* y(\psi_{\square}) \right) = e_{qq}-e_{q+1,q+1};\\
&\pi_0\left( \nabla_U^R \log \mathcal{Q}^* y(\psi_{\square}) \right) = e_{pp}-e_{p+1,p+1};
\end{aligned}
\end{equation}
hence,
\begin{equation}\label{eq:r0qypsi}\begin{split}
R_0[\nabla_U \log \mathcal{Q}^*& y(\psi_{\square}),U] = R_0(e_{qq}-e_{q+1,q+1} - (e_{pp}-e_{p+1,p+1})) = \\=&-R_0(1-\gamma)(e_{pp}-e_{p+1,p+1}) = -e_{pp}+e_{p+1,p+1}.
\end{split}
\end{equation}
Let us set $\Delta(i,j) := \sum_{k = i}^{j} e_{kk}$. We see that
\begin{equation}\label{eq:lastbraq}
\begin{split}
\{\log \mathcal{Q}^*y(\psi_{\square}), \log &\hat{h}_{ks}\}_{(\bg_\std,r_0)}^{\dagger} = -\chi_{n-s=p-k} - \chi_{k \leq q}\cdot \chi_{s = q+1} +\\+& \langle e_{pp}-e_{p+1,p+1},\Delta(k,n-s+k) \rangle - \langle e_{qq}-e_{q+1,q+1}, \Delta(s,n) \rangle = \\ = &-\chi_{k = p+1}+ (1-\chi_{k\leq q})\chi_{s=q+1} = -\chi_{k = p+1}+\chi_{k=q+1}\chi_{s=q+1}.
\end{split}
\end{equation}
Now let us recall that $\mathcal{Q}^*({h_{ks}}) = \tilde{h}_{ks}$ if $k \neq p+1$ and $\mathcal{Q}^*(h_{p+1,s}) = \tilde{h}_{p+1,s}\tilde{h}_{q+1,q+1}$. Together with formula~\eqref{eq:lastbraq}, the result follows in the case $\psi$ is an $h$-function. If $\psi = \varphi_{kl}$ for some $\varphi$-function $\varphi_{kl}$, then for any regular function $h$ such that $h(UN_-) = h(U)$,
\begin{equation*}
\{\log h,\log \varphi_{kl}\}_{(\bg_\std,r_0)}^{\dagger} = \langle R_0\pi_0[\nabla_U \log h,U],[\nabla_U\log \varphi_{kl},U]\rangle - \langle \pi_0 \nabla_U^R \log h,[\nabla_U\log \varphi_{kl},U]\rangle.
\end{equation*}
Setting $h:=\log \mathcal{Q}^*y(\psi_{\square})$ and applying $\pi_0[\nabla_U\log \varphi_{kl},U] = \const$ together with \eqref{eq:gradqyps} and \eqref{eq:r0qypsi}, we obtain that 
\begin{equation*}
\{\log \mathcal{Q}^*y(\psi_{\square}),\log \varphi_{kl}\}_{(\bg_\std,r_0)}^{\dagger} = 0. 
\end{equation*}
Thus the result follows.
\end{proof}

\begin{proposition}\label{p:compatibl}
For any BD quadruple $(\bg,r_0)$, the generalized cluster structure $\gc_h^{\dagger}(\bg)$ is compatible with $\pi_{(\bg,r_0)}^{\dagger}$.
\end{proposition}
\begin{proof}
Let us fix $n$ and run an induction on the size $|\Gamma_1|$. For $|\Gamma_1| = 0$, the result was proved in \cite[Theorem 3.8]{double}. Let $\tilde{\bg}$ be the BD triple obtained from $\bg$ by removing $p \in \Gamma_1$ and $q \in \Gamma_2$. Let $\mathcal{G}:(G,\pi_{(\tilde{\bg},r_0)}^{\dagger}) \dashrightarrow (G,\pi_{(\bg,r_0)}^{\dagger})$ be the birational quasi-isomorphism constructed in Section~\ref{s:bsimpl}, where $G \in \{\SL_n,\GL_n\}$. Let $\tilde{\Psi}_0$ and $\Psi_0$ be the initial extended clusters of $\gc_h^{\dagger}(\tilde{\bg},G)$ and $\gc_h^{\dagger}(\bg,G)$, and let $\psi_{\square} \in \Psi_0$ be the marked variable. Since $\tilde{\Psi}_0$ is assumed to be log-canonical and $\mathcal{G}$ is a Poisson isomorphism, $\Psi_0$ is log-canonical as well. Let $\psi_1 \in \Psi_0\setminus \{\psi_{\square}\}$ be a cluster variable and $\psi_2 \in \Psi_0$ be any cluster or frozen variable. Since
\begin{equation*}
\{\log y(\tilde{\psi}_1),\log \tilde{\psi}_2\}_{(\tilde{\bg},r_0)} = \chi_{\tilde{\psi}_1 = \tilde{\psi}_2},
\end{equation*}
an application of $(\mathcal{G}^*)^{-1}$ yields
\begin{equation}\label{eq:ypsir}
\{\log y(\psi_1),\log \psi_2\}_{(\bg,r_0)} = \chi_{\psi_1 = \psi_2}.
\end{equation}
If $|\Gamma_1| = 1$, it follows from Lemma~\ref{l:compmarked} (which treats the case of $\psi_1 = \psi_{\square}$) and Proposition~\ref{p:compb} that $\gc_h^{\dagger}(\bg,G)$ is compatible with $\pi_{(\bg,r_0)}^{\dagger}$. For $|\Gamma_1| > 1$, to verify~\eqref{eq:ypsir} for the marked variables, one can use another birational quasi-isomorphism $\mathcal{G}$ whose marked variables are disjoint from the former one. Thus the statement of the proposition holds.
\end{proof}

\section{\texorpdfstring{Further properties of $h$-variables}{Further properties of h-variables}}\label{s:other}
In this section, we provide several additional formulas for the $h$-variables, including their explicit expansion in terms of minors of an element $U \in \GL_n$.  Also, we study the Poisson properties of the frozen $h$-variables and prove Proposition~\ref{p:frozvar}.

\subsection{Preliminaries}\label{s:p_prelim}
In this subsection, we derive preliminary formulas that we will need in order to compute the $h$-variables explicitly in terms of minors of $U$.

\paragraph{Notation.} Let $A \subseteq [1,n]$, $i \in [1,n]$. Set $A^{\leq}(i) := A \cap [1,i]$, $A^{\geq}(i) := A \cap [i,n]$, and for a number $t \in [1,|A|]$ denote by ${A \choose t}$ the set of all subsets of $A$ of size $t$.

\begin{lemma}\label{l:h_frst_expn}
    Let $\bg := (\Gamma_1,\Gamma_2,\gamma)$ be a BD triple of type $A_{n-1}$. Let $i \in [1,n]$, $r \in [i,n]$, $k \geq 0$ and $J \subseteq [1,n]$, $|J| = r-i+1$. Let $\Delta$ be the $X$-run that contains $i$. For $i-1 \in \Gamma_1$, set
    \begin{equation}\label{eq:lir}
        L_{i,r} := \begin{cases}
            [\bar{\gamma}(i)_+(\Gamma_2)+1,n] \ &\text{if} \ r \notin \Delta \ \text{or} \ \gamma|_{\Delta\cap \Gamma_2}< 0;\\
            [\bar{\gamma}(r)+1,n] \ &\text{if} \ r \in \Delta \ \text{and} \ \gamma|_{\Delta \cap \Gamma_1} > 0;
        \end{cases}
    \end{equation}
    \begin{equation}\label{eq:dii}
    D_{i,I}(U) := \begin{cases}
        \dfrac{\det U^{\sgamma(I)\cup L_{i,r}}_{[\bar{\gamma}(i),n]}}{\det U^{[\bar{\gamma}(i),n]}_{[\bar{\gamma}(i),n]}} \ &\text{if} \ \gamma|_{\Delta \cap \Gamma_1} > 0;\\[20pt]
        \dfrac{\det U^{\sgamma(I)\cup L_{i,r}}_{[\bar{\gamma}(i)+1,n]}}{\det U^{[\bar{\gamma}(i)+1,n]}_{[\bar{\gamma}(i)+1,n]}} \ &\text{if} \ \gamma|_{\Delta \cap \Gamma_1} < 0.
    \end{cases}
    \end{equation}
    Then for any $k \geq 0$, \begin{equation}\label{eq:h_frst_expn}
        \det [\mathcal{F}_{k}(U)]_{[i,r]}^{J} = \begin{cases}\sum_{I \in {\Delta^{\leq}(r) \choose |\Delta \cap [i,r]|}} \mathcal{D}_{i,I}(\mathcal{F}_{k-1}(U) ) \cdot \det U^{J}_{I \cup [i_+(\Gamma_1)+ 1, r]} \ &\text{if} \ i-1 \in \Gamma_1,\\[12pt]
        \det U_{[i,r]}^{J} \ &\text{otherwise.}
        \end{cases}
    \end{equation}
    
    \end{lemma}
    \begin{proof}
        By formula~\eqref{eq:fseq} (the definition of $\mathcal{F}_k(U)$) and the Cauchy--Binet formula,
        \begin{equation*}
        \det [\mathcal{F}_k(U)]^{K}_{[i,r]} = \sum_{K, \ |K|=|i,r|} \det [\tilde{\gamma}^*(\mathcal{F}_{k-1}(U)_-)]_{[i,r]}^{K} \det U^{J}_{K}.
        \end{equation*}
       Let us notice that the shape of the matrix $[\tilde{\gamma}^*(U_-)]_{[i,n]}^{[1,n]}$ is given by
       \begin{equation*}
       [\tilde{\gamma}^*(U_-)]_{[i,n]}^{[1,n]} = \begin{bmatrix}
        0 & [\tilde{\gamma}^*(U_-)]_{[i,i_+(\Gamma_1)]}^{\Delta} & 0\\
        0 & 0 & [\tilde{\gamma}^*(U_-)]_{[i_+(\Gamma_1) + 1,n]}^{[i_+(\Gamma_1)+ 1,n]}
       \end{bmatrix};
       \end{equation*}
       therefore, $\det [\tilde{\gamma}^*(U_-)]_{[i,r]}^{K}$ is nontrivial for a generic $U\in \GL_n$ if and only if $K = I \cup [i_+(\Gamma_1)+1,r]$ where $I \in {\Delta^{\leq}(r)\choose |\Delta \cap [i,r]|}$. Hence $\det [\tilde{\gamma}^*(U_-)]^{K}_{[i,r]} = \det [\tilde{\gamma}^*(U_-)]^{I}_{\Delta \cap [i,r]}$. Now, formula~\eqref{eq:dii} follows from formulas~\eqref{eq:min_ump}--\eqref{eq:min_umn} and the definitions of $\bar{\gamma}$ and $\sgamma$ given in Section~\ref{s:bd_map}.
    \end{proof}

    \begin{lemma}\label{l:feqf} Let $\bg = (\Gamma_1,\Gamma_2,\gamma)$ be a BD triple of type $A_{n-1}$ and $S^{\gamma}(\alpha_0) := \{\alpha_i\}_{i \geq 0}^m$ be a $\gamma$-string. Then for any $i \in [0,m]$, $r \geq \alpha_{m-i}+1$, $J \subseteq [1,n]$, $|J| = |[\alpha_{m-i}+1,r]|$, and any $\ell \geq 0$,
    \begin{equation*}
        \det [\mathcal{F}(U)]_{[\alpha_{m-i}+1,r]}^J = \det [\mathcal{F}_{i + \ell}(U)]_{[\alpha_{m-i}+1,r]}^{J}.
    \end{equation*}
    \end{lemma}
    \begin{proof}
        Let us prove the statement by induction on $i$. For $i = 0$, the statement is  contained in the second line of formula~\eqref{eq:h_frst_expn}, so let us assume that the statement holds for $i \geq 0$ and prove it for $i+1$. For each $t \in [0,m]$, denote by $\Delta_t$ the $X$-run that contains $\alpha_t$. Notice that
        \begin{equation*}\begin{aligned}
        \bar{\gamma}(\alpha_t+1) &= \alpha_{t+1} + 1 \ &\text{if} \ \gamma|_{\Delta_i \cap \Gamma_1} > 0,\\
        \bar{\gamma}(\alpha_t+1) + 1 &= \alpha_{t+1} + 1 \ &\text{if} \ \gamma|_{\Delta_i \cap \Gamma_1} < 0.
        \end{aligned}
        \end{equation*}
        For $I \in {\Delta_{m-i-1}^{\leq}(r) \choose |\Delta_i \cap [\alpha_{m-i-1},r]}$, the latter formulas together with formula~\eqref{eq:dii} imply that
        \begin{equation}\label{eq:diia}
        D_{\alpha_{m-i-1}+1,I}(U) = \frac{\det U^{\sgamma(I)\cup L_{\alpha_{m-i-1}+1,r}}_{[\alpha_{m-i}+1,n]}}{\det U^{[\alpha_{m-i}+1,n]}_{[\alpha_{m-i}+1,n]}}
        \end{equation}
        regardless of the orientation of $\gamma$ on $\Delta_{m-i-1}$. Now, it follows from formula~\eqref{eq:h_frst_expn}, the assumption of the induction and formula~\eqref{eq:diia} that
        \[\begin{split}
        \det [\mathcal{F}_{i+1+\ell}(U)]^J_{[\alpha_{m-i-1}+1,r]} = &\sum_{I \in {\Delta_{m-i-1}^{\leq}(r) \choose |\Delta_{m-i-1} \cap [\alpha_{m-i-1}+1,r]}}\hspace{-10pt} D_{\alpha_{m-i-1}+1,I}(\mathcal{F}_{i+\ell}(U)) \det U^{J}_{I \cup [(\alpha_{m-i-1})_+(\Gamma_1)+1,r]} = \\ =&\sum_{I \in {\Delta_{m-i-1}^{\leq}(r) \choose |\Delta_{m-i-1} \cap [\alpha_{m-i-1}+1,r]}}\hspace{-10pt} D_{\alpha_{m-i-1}+1,I}(\mathcal{F}(U)) \det U^{J}_{I \cup [(\alpha_{m-i-1})_+(\Gamma_1)+1,r]} = \\ = &\det [\mathcal{F}(U)]_{[\alpha_{m-i-1}+1,r]}^{J}.
        \end{split}
        \]
        Thus the statement holds.
    \end{proof}

    As a consequence of Lemma~\ref{l:feqf}, the defining formula for the $h$-variables~\eqref{eq:h_fun} can be recast in the following way:
    
    \begin{proposition}\label{p:h_fun_enh}
        Let $\bg := (\Gamma_1,\Gamma_2,\gamma)$ be a BD triple of type $A_{n-1}$ and $S^{\gamma}(\alpha_0) := \{\alpha_i\}_{i=0}^m$ be a $\gamma$-string. For any $i \in [0,m]$ and $j \in [\alpha_i+1,n]$, the $h$-variable $h_{\alpha_i+1,j}(U)$ can be written as
        \begin{equation*}
            h_{\alpha_i+1,j}(U) = (-1)^{\varepsilon_{\alpha_i+1,j}}\det [\mathcal{F}_{m-i}(U)]^{[j,n]}_{[\alpha_i+1,n-j+\alpha_i+1]} \prod_{t \geq i+1}^m \det [\mathcal{F}_{m-t}(U)]^{[\alpha_t+1,n]}_{[\alpha_t+1,n]}
        \end{equation*}
        where $\varepsilon_{\alpha_i+1,j} \in \{0,1\}$ is given by formula~\eqref{eq:h_sign}.
    \end{proposition}
    \begin{proof}
        The statement is a straightforward consequence of Lemma~\ref{l:feqf}.
    \end{proof}

\subsection{\texorpdfstring{Explicit expansion of the $h$-variables}{Explicit expansion of the h-variables}}\label{s:p_explicit}
In this subsection, we derive an explicit formula for the $h$-variables in terms of minors of $U \in \GL_n$. Let us fix a BD triple $\bg := (\Gamma_1,\Gamma_2,\gamma)$ and a $\gamma$-string $S^{\gamma}(\alpha_0) := \{\alpha_i\}_{i=0}^m$. Define the following data:
\begin{equation}\label{eq:hdata}
\begin{aligned}
\Delta_{i}&:= \Delta(\alpha_i) = \text{the }X\text{-run that contains }\alpha_i, & &0 \leq i \leq m;\\
s_{ij} &:= n-j+\alpha_i + 1, & &i \in [0,m], \ j \in [\alpha_i+1,n];\\
K_i &:= [(\alpha_i)_+(\Gamma_1)+1,n],  & &i \in [0,m-1];\\
L_i &:= [(\alpha_i)_+(\Gamma_2)+1,n], & &i \in [2,m];\\
L^\prime_{i}(j) &:= \begin{cases}[\max\sgamma(\Delta_{i-1}^{\leq}(s_{i-1,j}))+1,n] \ &\text{if }\gamma|_{\Delta_i\cap \Gamma_1} > 0;\\
[(\alpha_i)_+(\Gamma_2)+1,n]  \ &\text{if }\gamma|_{\Delta_i\cap \Gamma_1} < 0.
\end{cases} & &i \in [1,m], \ j \in [\alpha_i+1,n];\\
\mathcal{N}_{i}(j) &:= {\Delta_{i}^{\leq}(s_{ij}) \choose |\Delta_{i} \cap [\alpha_{i}+1,s_{ij}]|}, \ \ \ \mathcal{N}_{i} := {\Delta_{i} \choose |\Delta_{i}^{\geq}(\alpha_i+1)|}, & & i \in [0,m-1], \ j \in [\alpha_i+1,n].
\end{aligned}
\end{equation}

\begin{proposition}\label{p:h_expansion}
For each $i \in [0,m]$ and each $j \in [\alpha_i+1,n]$, the $h$-variable $h_{\alpha_i+1,j}$ is given by:
\begin{enumerate}
\item If $i = m$, then
\begin{equation*}
h_{\alpha_m+1,j}(U) = (-1)^{\varepsilon_{\alpha_m+1,j}}\det U^{[j,n]}_{[\alpha_m+1,s_{mj}]};
\end{equation*}
\item If $i = m-1$, then
\begin{equation}\label{eq:halphao}
h_{\alpha_{m-1}+1,j}(U) = (-1)^{\varepsilon_{\alpha_{m-1}+1,j}}\hspace{-30pt}\sum_{I_{m-1}: \, I_{m-1} \in \mathcal{N}_{m-1}(j)} \det U^{[j,n]}_{I_{m-1} \cup K_{m-1}^{\leq}(s_{m-1,j})}  \det U^{\sgamma(I_{m-1})\cup L_m^\prime(j)}_{[\alpha_m+1,n]};
\end{equation}
\item If $i \in [0,m-2]$, then
\begin{equation*}\label{eq:h_alphai}
\begin{aligned}
h_{\alpha_i+1,j}&(U) = &\\ &=  (-1)^{\varepsilon_{\alpha_i+1,j}} \hspace{-20pt}\sum_{\substack{(I_{i},I_{i+1},\ldots,I_{m-1}) :\\ I_{i} \in \mathcal{N}_i(j), \ I_{t} \in \mathcal{N}_t,\\ \ t \in [i+1,m-1]}}
 &\det U^{[j,n]}_{I_i \cup K_{i}^{\leq}(s_{ij})} \det  U^{\sgamma(I_i) \cup L_{i+1}^\prime(j)}_{I_{i+1} \cup K_{i+1}} \det U^{\sgamma(I_{i+1}) \cup L_{i+2}}_{I_{i+2} \cup K_{i+2}}\cdots \\ & &\cdots\det U^{\sgamma(I_{m-2}) \cup L_{m-1}}_{I_{m-1} \cup K_{m-1}} \det U^{\sgamma(I_{m-1})\cup L_m}_{[\alpha_m+1,n]}.
\end{aligned}
\end{equation*}
\end{enumerate}
\end{proposition}

\begin{proof}
    The case $i=m$ is already contained in Lemma~\ref{l:h_frst_expn}. Comparing~\eqref{eq:lir} with $L_{i}^\prime(j)$ from~\eqref{eq:hdata}, we see that $L_{\alpha_i+1,s_{ij}} = L_{i+1}^\prime(j)$. Rewriting~\eqref{eq:h_frst_expn} in terms of~\eqref{eq:hdata}, we see that
    \begin{equation}\label{eq:frrstfmi}
    \begin{split}
        \det [\mathcal{F}_{m-i}(U)]_{[\alpha_i+1,s_{ij}]}^{[j,n]} \cdot &\det [\mathcal{F}_{m-i-1}(U)]_{[\alpha_{i+1}+1,n]}^{[\alpha_{i+1}+1,n]} = \\ = &\sum_{I_i \in \mathcal{N}_i(j)}\det U^{[j,n]}_{I_i \cup K_i^{\leq}(s_{ij})}\det [\mathcal{F}_{m-i-1}(U)]_{[\alpha_{i+1}+1,n]}^{\sgamma(I_i) \cup L^\prime_{i+1}(j)} .
        \end{split}
    \end{equation}
    Setting $i = m-1$ in formula~\eqref{eq:frrstfmi}, we obtain (up to a sign) formula~\eqref{eq:halphao}. If $m-i-1\geq 1$, let us expand $\det [\mathcal{F}_{m-i-1}(U)]_{[\alpha_{i+1}+1,n]}^{\sgamma(I_i)\cup L_{i+1}^{\prime}(j)}$ via formula~\eqref{eq:h_frst_expn}; in the notation of~\eqref{eq:hdata}, we obtain
\begin{equation*}
    \begin{split}
        \det [\mathcal{F}_{m-i}(U)&]_{[\alpha_i+1,s_{ij}]}^{[j,n]} \det [\mathcal{F}_{m-i-1}(U)]_{[\alpha_{i+1}+1,n]}^{[\alpha_{i+1}+1,n]} \det [\mathcal{F}_{m-i-2}(U)]_{[\alpha_{i+2}+1,n]}^{[\alpha_{i+2}+1,n]} = \\ = &\sum_{\substack{I_i \in \mathcal{N}_i(j),\\ I_{i+1}\in \mathcal{N}_{i+1}}}\det U^{[j,n]}_{I_i \cup K_i^{\leq}(s_{ij})}\det U_{[\alpha_{i+1}+1,n]}^{\sgamma(I_i)\cup L^\prime_{i+1}(j)}\det [\mathcal{F}_{m-i-2}(U)]_{[\alpha_{i+2}+1,n]}^{\sgamma(I_{i+1})\cup L_{i+2}} .
        \end{split}
    \end{equation*}
    Now, if $m-i-2 \geq 1$, we apply  formula~\eqref{eq:h_frst_expn} again to the term $\det [\mathcal{F}_{m-i-2}(U)]_{[\alpha_{i+2}+1,n]}^{\sgamma(I_{i+1})\cup L_{i+2}}$, and so with $\det [\mathcal{F}_{m-i-3}(U)]_{[\alpha_{i+3}+1,n]}^{\sgamma(I_{i+2})\cup L_{i+3}}$ in the resulting formula if $m-i-3 \geq 1$. Continuing in this fashion, we will finally reach the terms $\det [\mathcal{F}_0(U)]_{[\alpha_m+1,n]}^{\sgamma(I_{m-1})\cup L_{m}}$, which are equal to $\det U_{[\alpha_m+1,n]}^{\sgamma(I_{m-1})\cup L_{m}} $. By Proposition~\ref{p:h_fun_enh}, the right-hand side will be equal (up to a sign) to $h_{\alpha_i+1,j}(U)$, and the left-hand side will be the desired expansion in terms of minors of $U$.
\end{proof}
\begin{corollary}\label{c:degh}
    The polynomial degree of $h_{\alpha_i+1,j}$ is given by
    \begin{equation*}
        \deg h_{\alpha_i+1,j} = n-j+1 + \sum_{\ell = 1}^{m-i}(n-\alpha_{i+\ell}).
    \end{equation*}
\end{corollary}

\subsection{\texorpdfstring{Duality of the frozen $h$-variables}{Duality of the frozen h-variables}}\label{s:p_fr_dual}
Let $\bg := (\Gamma_1,\Gamma_2,\gamma)$ be a BD triple and $\bg^{\op}:=(\Gamma_2,\Gamma_1,\gamma^*)$ be the opposite BD triple. In this subsection, we  prove a certain duality property between the frozen $h$-variables in $\gc_h^{\dagger}(\bg)$ and $\gc_h^{\dagger}(\bg^{\op})$. We will subsequently use this property in two instances: 1) to prove Proposition~\ref{p:frozvar} (semi-invariance of the frozen $h$-variables); 2) to show that the frozen variables in $\gc_h^\dagger(\bg)$ coincide with the frozen variables in $\gc_g^\dagger(\bg)$ (Proposition~\ref{p:samefroz}).

\begin{proposition}\label{p:hfrozdual}
    Let $S^{\gamma}(\alpha_0) := \{\alpha_i\}_{i = 0}^{m}$ and $S^{\gamma^*}(\beta_0):= \{\beta_i\}_{i=0}^m$ be a $\gamma$- and a $\gamma^*$-string that start at $\alpha_0 \in \Pi\setminus \Gamma_2$ and $\beta_0\in \Pi\setminus\Gamma_1$, and such that $\beta_0 = \alpha_m$. Let $h_{\alpha_0+1,\alpha_0+1}$ be the frozen $h$-variable in $\gc_h^{\dagger}(\bg)$ associated with $\alpha_0$ and $\bar{h}_{\beta_0+1,\beta_0+1}$ be the frozen $h$-variable in $\gc_h^{\dagger}(\bg^{\op})$ associated with~$\beta_0$. Then
    \begin{equation}\label{eq:hfrozdual}
        \bar{h}_{\beta_0+1,\beta_0+1}(U^{T}) = h_{\alpha_0+1,\alpha_0+1}(U).
    \end{equation}
\end{proposition}
\begin{proof}
    We will compare the explicit expansions of the variables in terms of minors of $U$. Consider the data~\eqref{eq:hdata} relative to the BD triple $\bg^{\op}$ for the $\gamma^*$-string $S^{\gamma^*}(\beta_0)$; to distinguish it from the data for $S^{\gamma}(\alpha_0)$, let us decorate the former with a bar; for instance, we set $\bar{K}_i: = [(\beta_i)_+(\Gamma_2)+1,n]$. Now, if $m = 0$,  the statement is trivial; for $m > 0$, it follows from Proposition~\ref{p:h_expansion} that both variables can be expanded as
    \begin{equation}\label{eq:hbb}
        \bar{h}_{\beta_0+1,\beta_0+1}(U) = \sum_{\substack{(\bar{I}_0,\bar{I}_1,\ldots,\bar{I}_{m-1}):\\ \bar{I}_i \in \bar{\mathcal{N}}_i, \, 0 \leq i \leq m-1 }} \det U_{\bar{I}_0 \cup \bar{K}_0}^{[\beta_0,n]} \left[\prod_{i=1}^{m-1}\det U_{\bar{I}_i\cup\bar{K}_i}^{\sgammas(\bar{I}_{i-1})\cup \bar{L}_i} \right]\det U^{\sgammas(\bar{I}_{m-1})\cup\bar{L}_{m}}_{[\beta_m,n]};
    \end{equation}
    \begin{equation}
        h_{\alpha_0+1,\alpha_0+1}(U) = \sum_{\substack{(I_0,I_1,\ldots,I_{m-1}):\\ I_i \in \mathcal{N}_i,\, 0 \leq i \leq m-1 }} \det U^{[\alpha_0,n]}_{I_0\cup K_0}\left[\prod_{i=1}^{m-1}\det U^{\sgamma(I_{i-1})\cup L_i}_{I_i \cup K_i}  \right] \det U_{[\alpha_m,n]}^{\sgamma(I_{m-1})\cup L_{m}}.
    \end{equation}
    Let us establish a connection between the data associated with the two variables. We see that for each $i \in [0,m]$, $\beta_i = \alpha_{m-i}$, $\bar{K}_i = L_{m-i}$, $\bar{L}_i = K_{m-i}$; as for $\mathcal{N}_i$, we see that for each $i \in [0,m-1]$,
    $\sgamma|_{\mathcal{N}_i} : \mathcal{N}_i \rightarrow \bar{\mathcal{N}}_{m-i-1}$ 
    is a bijection; therefore, in formula~\eqref{eq:hbb} we can replace each $\bar{I}_i$ with $\sgamma(I_{m-i-1})$ and each $\sgammas(\bar{I}_{i-1})$ with $I_{m-i}$. Now, combining the latter with the substitution $U \mapsto U^T$, formula~\eqref{eq:hfrozdual} follows.
\end{proof}
\begin{corollary}\label{c:hposinvar}
   Let $h_{ii}$ be a frozen $h$-variable in $\gc_h^{\dagger}(\bg)$, and let $P$ be a unipotent upper-triangular matrix. Then the following identity holds:
   \begin{equation}\label{eq:hposinvar}
   h_{ii}(PU\tilde{\gamma}(P)^{-1}) = h_{ii}(U).
   \end{equation}
\end{corollary}
\begin{proof}
    By Proposition~\ref{p:hfrozdual}, there exists a frozen $h$-variable $\bar{h}_{jj}$ in $\gc_h^{\dagger}(\bg^{\op})$ such that $h_{ii}(U) = \bar{h}_{jj}(U^{T})$. Since $P^T$ is a unipotent lower-triangular matrix, applying formula~\eqref{eq:hninv}, we see that
    \begin{equation*}
    h_{ii}(U) = \bar{h}_{jj}(U^T) = \bar{h}_{jj}(\tilde{\gamma}(P^{-T})U^T P^T) = h_{ii}(PU \tilde{\gamma}(P)^{-1}).
    \end{equation*}
    Thus formula~\eqref{eq:hposinvar} holds. \end{proof}

    \begin{corollary}\label{l:hposdif}
        Let $h_{ii}$ be a frozen $h$-variable in $\gc_h^{\dagger}(\bg)$. Then for any $U \in \GL_n$,
        \begin{equation*}
            U\nabla_U h_{ii}- \ogammas\left(  \nabla_U h_{ii}\cdot U\right) \in \mathfrak{b}_+.
        \end{equation*}
    \end{corollary}
    \begin{proof}
        The statement follows from differentiating formula~\eqref{eq:hposinvar}.
    \end{proof}

\subsection{\texorpdfstring{Semi-invariance of the frozen $h$-variables}{Semi-invariance of the frozen h-variables}}\label{s:p_frozen}
In this subsection, we prove Proposition~\ref{p:frozvar} for $\gc_h^{\dagger}(\bg)$, and we also show that the frozen $h$-variables do not vanish on $\SL_n^{\dagger}$. Let $\bg:=(\Gamma_1,\Gamma_2,\gamma)$ be a BD triple and $\bar{\bg}:=(\bg,r_0)$ be a BD quadruple. We set $\mathcal{P}_+(\Gamma_2)$ and $\mathcal{P}_-(\Gamma_1)$ to be the parabolic subgroups of $\GL_n$ generated by $\bg$. The subgroup $\mathcal{D} \subset \mathcal{P}_+(\Gamma_1) \times \mathcal{P}_-(\Gamma_2)$ is defined as
\begin{equation*}
\mathcal{D} := \{(g^\prime,g) \in \mathcal{P}_+(\Gamma_1) \times \mathcal{P}_-(\Gamma_2) \ | \ \togamma(\mathring{\Pi}_{\Gamma_1}(g^\prime)) = \mathring{\Pi}_{\Gamma_2}(g)\}.
\end{equation*}
The group $\mathcal{D}$ acts upon $\GL_n$ via
\begin{equation}\label{eq:dactiont}
(g^\prime,g).U = g^\prime U g^{-1}, \ \ (g^\prime,g) \in \mathcal{D}, \ U \in \GL_n.
\end{equation}
\begin{proposition*}
Let $\psi$ be a frozen variable in $\gc^{\dagger}_h(\bg,G)$, $G \in \{\SL_n,\GL_n\}$. Then the following statements hold:
\begin{enumerate}[1)]
\item The variable $\psi$ is semi-invariant with respect to the action~\eqref{eq:dactiont};
\item The variable $\psi$ is log-canonical with any $u_{ij}$;\label{pc:logcan}
\item The zero locus of $\psi$ foliates into a union of symplectic leaves of $(G,\pi_{\bar{\bg}}^{\dagger})$.\label{pc:sympl}
\end{enumerate}
\end{proposition*}

\begin{proof}
The statements are evident for the $c$-variables, which are Casimirs of the Poisson bracket. Part~\ref{pc:sympl} for the $c$-variables follows from Proposition~\ref{p:yakim} and the fact that the $c$-variables are irreducible as elements of $\mathbb{C}[\GL_n]$ (see \cite[Theorem 3.10]{double}). From now on, we assume that $\psi$ is a frozen $h$-variable.
\begin{enumerate}[1)]
\item Pick a pair $(g^\prime,g) \in \mathcal{D}$ and decompose it as $g^\prime = g^\prime_+ g^\prime_0 g^\prime_-$ and $g =g_+ g_0 g_-$, where the subscript indicates if the matrix is unipotent upper, diagonal or unipotent lower, respectively. The defining condition of $\mathcal{D}$ implies that $\togammas(g_-) = g^\prime_-$ and $\togamma(g^\prime_+) = g_+$. It follows from Corollary~\ref{c:hinvarn} and Corollary~\ref{c:hposinvar} that
\begin{equation*}
\psi(g^\prime U g^{-1}) = \psi\left(g^\prime_+ g^\prime_0 \togammas(g_-) U \sinv{g_-}\sinv{g_0}\togamma(g^\prime_+)^{-1} \right) = \chi_{\psi}(g^\prime,g) \psi(U)
\end{equation*}
where $\chi_{\psi}$ is a character on $\mathcal{D}$. Thus $\psi$ is semi-invariant with respect to the action of $\mathcal{D}$.
\item For convenience\footnote{As we showed in Section 8.2 in \cite{multdouble}, log-canonicity does not depend on the choice of $R_0$.}, assume that $r_0$ is ringed, and denote the corresponding $R_0$ as $\mathring{R}_0$ (see formula~\eqref{eq:ringedr0}). Let us set 
\begin{align*}\onu(\psi) &:= \nabla_U \psi \cdot U - \ogamma(U\nabla_U\psi);\\
\onus(\psi) &:= U\nabla_U \psi-\ogammas(\nabla_U\psi\cdot U).
\end{align*}
It follows from Corollary~\ref{c:h_inf} and Corollary~\ref{l:hposdif} that $\onu(\psi) \in \mathfrak{b}_-$, $\onus(\psi) \in \mathfrak{b}_+$, $\pi_0(\onu(\log \psi)) = \const$ and $\pi_0(\onus(\log \psi)) = \const$. Moreover, one can write 
\begin{align*}
\onu(\psi) &= [\nabla_U \psi, U] + (1-\ogamma)(U\nabla_U \psi);\\
\onus(\psi) &= -[\nabla_U\psi, U] + (1-\ogammas)(\nabla_U \psi \cdot U).
\end{align*}
It follows that
\begin{align*}
\frac{1}{1-\ogamma}\pi_{>}[\nabla_U \psi, U] &= -\pi_{>}(U\nabla_U \psi);\\
\frac{\ogammas}{1-\ogammas}\pi_{<} [\nabla_U h, U] &= \pi_{<}(U\nabla_U \psi).
\end{align*}
Therefore,
\begin{equation}\label{eq:r0froz1}
R([\nabla_U \psi, U]) = -\pi_{>}(U\nabla_U \psi) - \pi_{<}(U\nabla_U \psi) +\mathring{R}_0 \pi_0 [\nabla_U \psi, U].
\end{equation}
Rewriting the last piece, we see that
\begin{equation*}\begin{split}
\mathring{R}_0 \pi_0 [\nabla_U \psi, U] = &\mathring{R}_0 \pi_0 \onu(\psi) - \mathring{R}_0(1-\ogamma)(U\nabla_U \psi) = \\ = &\mathring{R}_0 \pi_0\onu(\psi) - \pi_0\mathring{\pi}_{\Gamma_1}(U\nabla_U \psi) - \mathring{R}_0 \pi_0 \mathring{\pi}_{\hat{\Gamma}_1}(U\nabla_U \psi) = \\ =&\mathring{R}_0 \pi_0\onu(\psi) - \pi_0(U\nabla_U \psi) + \mathring{R}_0^* \pi_0 \mathring{\pi}_{\hat{\Gamma}_1}(U\nabla_U \psi);
\end{split}
\end{equation*}
hence, formula~\eqref{eq:r0froz1} can be updated as
\begin{equation*}
R([\nabla_U\psi,U]) = -U\nabla_U \psi + \mathring{R}_0 \pi_0 \onu (\psi) + \mathring{R}_0^*\pi_0 \mathring{\pi}_{\hat{\Gamma}_1}(U\nabla_U \psi).
\end{equation*}
Now, the Poisson bracket between $\psi$ and $u_{ij}$ is given by
\begin{equation*}\begin{split}
\{\psi,u_{ij}\}_{\bar{\bg}}^{\dagger} = &\langle R([\nabla_U \psi, U]), [\nabla_U u_{ij},U]\rangle - \langle [\nabla_U \psi, U],\nabla_U u_{ij} U\rangle = \\ = &-\langle U\nabla_U \psi, [\nabla_U u_{ij},U]\rangle - \langle [\nabla_U \psi, U],\nabla_U u_{ij} U\rangle + \\ + &\langle \mathring{R}_0 \pi_0 \onu(\psi)+  \mathring{R}_0^* \pi_0 \mathring{\pi}_{\hat{\Gamma}_1}(U\nabla_U \psi), \pi_0[\nabla_U u_{ij},U]\rangle. 
\end{split}
\end{equation*}
The first two terms in the last identity cancel each other out. As for the last term, notice that $\pi_0[\nabla_U u_{ij},U] = u_{ij}(e_{jj}-e_{ii})$, and since $\onu(\log\psi) = \const$ and $\onus(\log\psi) = \const$, we finally conclude that $\{\psi,u_{ij}\}_{\bar{\bg}}^{\dagger}$ is proportional to the product $\psi \cdot u_{ij}$.
\item It is a consequence of Proposition~\ref{p:compl} that $\psi$ is a prime element of $\mathbb{C}[G]$, and it follows from Part~\ref{pc:logcan} of the current proposition that the principal ideal $(\psi)$ is Poisson. Now the statement follows from Proposition~\ref{p:yakim}.\qedhere
\end{enumerate}
\end{proof}

\begin{remark} Let $\bg$ be an aperiodic oriented BD triple and let $\gc^D(\bg,\bg)$ be the generalized cluster structure on $D(\GL_n)$ constructed in~\cite{multdouble}. Then any frozen $g$- or $h$-variable in $\gc^D(\bg,\bg)$ is semi-invariant with respect to both the right and the left action of $\mathcal{D}$ upon $D(\GL_n)$.
Therefore, the frozen variables in $\gc^D(\bg,\bg)$ and in $\gc^{\dagger}_h(\bg,\GL_n)$ are semi-invariant with respect to the same group.
\end{remark}

\begin{proposition}\label{p:hnonzero}
The following statements hold:
\begin{enumerate}[1)]
\item For any choice of $r_0$, the frozen $h$-variables in $\gc^{\dagger}_h(\bg,\SL_n)$ do not vanish on $\SL_n^{\dagger}$;
\item If $r_0$ is ringed, then the frozen $h$-variables in $\gc^{\dagger}(\bg,\GL_n)$ do not vanish on $\GL_n^{\dagger}$.
\end{enumerate}
\end{proposition}
\begin{proof}
Both statements can be proved in the same way, so let us focus on the case of $\GL_n$. Let $r_0$ be chosen as a ringed element (see formula~\eqref{eq:ringedr0}), and let $\psi$ be a frozen $h$-variable in $\gc_h^{\dagger}(\bg,\GL_n)$. We set $\tilde{\psi} \in \mathbb{C}[D(\GL_n)]$ to be a lift of $\psi$ given by
\begin{equation*}
\tilde{\psi}(X,Y) := \psi(X^{-1}Y), \ \ (X,Y) \in D(\GL_n).
\end{equation*}
The choice of $r_0$ implies that $\ogamma R_+ = \mathring{\pi}_{\Gamma_2}R_-$, hence $\GL_n^* \subseteq \mathcal{D}$, where $\GL_n^*$ is viewed as an immersed Lie subgroup of $D(\GL_n)$. We see that
\begin{equation*}
\tilde{\psi}(g^\prime,g) = \psi((g^\prime)^{-1}g) = \chi_{\psi}(g^\prime,g)\psi(I)
\end{equation*}
where $\chi_{\psi}$ is a character on $\mathcal{D}$ from Proposition~\ref{p:frozvar}. Since $\mathcal{F}(I) = I$, it follows from the defining formula for the $h$-functions~\eqref{eq:h_fun} that $\psi(I) \neq 0$. Hence $\tilde{\psi}|_{\mathcal{D}}$ does not vanish at any point of $\mathcal{D}$, and thus $\psi$ is nonzero on $\GL_n^{\dagger}$ and $\SL_n^{\dagger}$.
\end{proof}

\section{\texorpdfstring{The $g$-convention}{The g-convention}}\label{s:descr_g}
In this section, for any Belavin--Drinfeld triple $\bg$ of type $A_{n-1}$, we describe the generalized cluster structure $\gc^{\dagger}_g(\bg)$ on $\GL_n$ and $\SL_n$ in the \emph{$g$-convention}. 

\subsection{\texorpdfstring{The map $\mathcal{F}^{\op}$}{The map F opposite}}
In this subsection, we recast the map $\mathcal{F}$ from Section~\ref{s:map_f} in the $g$-convention. For a fixed set of simple roots $\Pi$ of type $A_{n-1}$, let $\bg:=(\Gamma_1,\Gamma_2,\gamma)$ be a BD triple, $\mathcal{B}_{\pm}$, $\mathcal{N}_{\pm}$ be the Borel subgroups and their unipotent radicals, $\mathcal{H}$ be the corresponding Cartan subgroup of $\GL_n$. We decompose a generic element $U \in \GL_n$ as $U = U_{\oplus}U_- = U_+ U_{\ominus}$, where $U_{\ominus} \in \mathcal{B}_-$, $U_{\oplus} \in \mathcal{B}_+$, $U_- \in \mathcal{N}_-$, $U_+ \in \mathcal{N}_+$. Define a sequence of rational maps $\mathcal{F}^{\op}_{k}: \GL_n \dashrightarrow \GL_n$ via\footnote{To avoid confusion, let us emphasize that formula~\eqref{eq:mapfg} defines $\mathcal{F}_k^{\op}$ relative to the BD triple $\bg$ and not $\bg^{\op}$. If $\mathcal{F}_k^{\op}$ is defined relative to the opposite BD triple $\bg^{\op}$, then it satisfies the recurrence relation $\mathcal{F}_k^{\op}(U) = U\tilde{\gamma}^*(\mathcal{F}_{k-1}(U)_+)$. }
\begin{equation}\label{eq:mapfg}
    \mathcal{F}^{\op}_0(U) := U, \ \  \mathcal{F}^{\op}_{k}(U) = U\tilde{\gamma}(\mathcal{F}_{k-1}(U)_+), \ \  U \in \GL_n, \ k \geq 1.
\end{equation}
The rational map $\mathcal{F}^{\op} : \GL_n \dashrightarrow \GL_n$ is defined as the limit 
\begin{equation}\label{eq:fdefg}
\mathcal{F}^{\op}(U) := \lim_{k\rightarrow \infty}\mathcal{F}^{\op}_k(U), \ \ U \in \GL_n.
\end{equation}
The sequence $\mathcal{F}_k^{\op}$ stabilizes at $k = \deg \gamma$. As a corollary, $\mathcal{F}^{\op}(U)$ satisfies the relation
\begin{equation*}
\mathcal{F}^{\op}(U) = U \tilde{\gamma}(\mathcal{F}^{\op}(U)_+),
\end{equation*}
and $\mathcal{F}^{\op}$ is a birational map with the inverse given by 
\begin{equation*}
(\mathcal{F}^{\op})^{-1}(U) = U\tilde{\gamma}(U_+)^{-1}.
\end{equation*}
Let $T:\GL_n\rightarrow \GL_n$ be the transposition map; that is, $T(U) = U^T$, $U \in \GL_n$. 

\begin{proposition}\label{p:ffg_comm}
Let $\mathcal{F}:\GL_n \dashrightarrow \GL_n$ and $\mathcal{F}^{\op}:\GL_n \dashrightarrow \GL_n$ be constructed relative to the BD triples $\bg$ and $\bg^{\op}$, respectively. Then for any $k \geq 0$ the following diagram commutes:
\begin{equation*}
\xymatrix{
\GL_n\ar[d]_{T} \ar@{-->}[r]^{\mathcal{F}_k} &\GL_n \ar[d]^{T} \\ \GL_n \ar@{-->}[r]_{\mathcal{F}_k^{\op}} & \GL_n
}
\end{equation*}
\end{proposition}
\begin{proof}
    Indeed, the commutativity is evident for $k = 0$. Assuming by induction that the commutativity holds for $k$, we see that
    \begin{equation*}
    [\mathcal{F}_{k+1}(U^T)]^T = \left[\tilde{\gamma}^*(\mathcal{F}_{k}(U^T)_-)U^T\right]^T = U\tilde{\gamma}^*\left[(\mathcal{F}_k(U^T)^T)_+\right] = U \tilde{\gamma}^*(\mathcal{F}^{\op}_{k}(U)_+) = \mathcal{F}_{k+1}^{\op}(U).
    \end{equation*}
    Thus the diagram commutes.
\end{proof}
\begin{corollary}\label{c:fop_invar}
    For $P \in \mathcal{N}_+$ and $T \in \mathcal{H}$, the map $\mathcal{F}^{\op}$ has the following invariance properties:
    \begin{align}
        &\mathcal{F}^{\op}(PU\tilde{\gamma}(P)^{-1}) = P\cdot \mathcal{F}^{\op}(U);\label{eq:fopnpinv}\\
        &\mathcal{F}^{\op}(\togammas(T)UT^{-1}) = \togammas(T)\mathcal{F}^{\op}(U)T^{-1};\notag\\
        &\mathcal{F}^{\op}(T^{-1}U\togamma(T)) = T^{-1}\mathcal{F}^{\op}(U)\togamma(T).\notag
    \end{align}
\end{corollary}
\begin{proof}
    The identities follow from Corollary~\ref{c:f_invar} and Proposition~\ref{p:ffg_comm}.
\end{proof}

\subsection{\texorpdfstring{Description of $\gc_g^{\dagger}(\bg)$}{Initial seed in the g-convention}}\label{s:descr_gg}
In this subsection, we describe the initial extended cluster of $\gc_g^{\dagger}(\bg)$ for an arbitrary BD triple $\bg := (\Gamma_1,\Gamma_2,\gamma)$ of type $A_{n-1}$. It comprises three types of functions: $c$-functions, $\phi$-functions and $g$-functions. The initial quiver is described in Appendix~\ref{s:ainiquivg}.

\paragraph{Description of $\phi$- and $c$-functions.} The $c$-functions are given by~\eqref{eq:c_def} (as in the $h$-convention). The $\phi$-functions are similar to the $\varphi$-functions and are defined as follows. For an element $U \in \GL_n$, let us set
\begin{equation*}
\Phi_{kl}^\prime(U):=\begin{bmatrix}(U^0)^{[1,k]} & U^{[1,l]} & (U^2)^{\{1\}} & \cdots & (U^{n-k-l+1})^{\{1\}}\end{bmatrix}, \ \ k,l \geq 1, \ k+l \leq n;
\end{equation*}
Then the $\phi$-functions are given by
\begin{equation*}
\phi_{kl}(U):=s_{kl} \det \Phi_{kl}^\prime(U)
\end{equation*}
where $s_{kl} \in \{1,-1\}$ is a sign given by~\eqref{eq:s_def} (as in the $h$-convention). For $P \in \mathcal{N}_+$ and $T \in \mathcal{H}$, the $\phi$-functions have the following invariance properties:
\begin{align*}
        \phi_{kl}(PUP^{-1}) &= \phi_{kl}(U),\\
        \phi_{kl}(TUT^{-1}) &= \chi_{\phi_{kl}}(T)\phi_{kl}(U)
\end{align*}
where $\chi_{\phi_{kl}}$ is a character.
\begin{remark}\label{r:phivphi}
The $\varphi$- and $\phi$-functions are related as follows. Set $W_0 := \sum_{i=1}^{n-1}(-1)^{i+1}e_{n-i+1,i}$ and consider the map $w_0:\GL_n\rightarrow\GL_n$, $w_0(U) = W_0^{-1}UW_0$. If $\phi_{kl}$ is a $\phi$-function, then $\phi_{kl}(w_0(U)) = \varphi_{kl}(U)$ is a $\varphi$-function. It follows that the $\phi$-functions are irreducible as elements of $\mathbb{C}[\GL_n]$, and the mutations of $\phi_{kl}$ for $2\leq k+l\leq n-1$ are regular and coprime with $\phi_{kl}$. Moreover, if $\mathbb{C}[\GL_n]^{\mathcal{N}_\pm}$ are the rings of regular functions invariant with respect to the adjoint actions of $\mathcal{N}_+$ and $\mathcal{N}_-$, then $w_0^* : (\mathbb{C}[\GL_n]^{\mathcal{N}_+}, \{\cdot,\cdot\}_{\bar{\bg}}^{\dagger})\xrightarrow{\sim} (\mathbb{C}[\GL_n]^{\mathcal{N}_-}, \{\cdot,\cdot\}_{\bar{\bg}}^{\dagger})$ is a Poisson isomorphism for any $\bar{\bg}$. In particular, the $\phi$-functions are log-canonical.
\end{remark}
\begin{remark}
    The $\phi$-functions appeared before in~\cite{periodic}. They were used to define a generalized cluster structure on the Drinfeld double $D(\GL_n) = \GL_n \times \GL_n$, which is an alternative to the use of the $\varphi$-functions in~\cite{double}. It was shown in~\cite{doublerel} that the resulting constructions with $\phi$- and $\varphi$-functions (for the trivial BD triples) provide equivalent generalized cluster structures in $n=2$ and $n=3$; however, they are not equivalent in $n=4$ (in other words, the initial extended seeds are not mutation equivalent). In the supplementary note~\cite{github}, we show that $\gc_g^{\dagger}(\bg)$ and $\gc_h^{\dagger}(\bg^{\op})$ are equivalent for $n=3$ and any BD triple $\bg$. We conjecture that they are not equivalent for $n \geq 4$.
\end{remark}

\paragraph{Description of the $g$-functions.} Let $\mathcal{F}^{\op}:\GL_n \dashrightarrow \GL_n$ be the birational map defined in~\eqref{eq:fdefg}. For each $\alpha_0 \in \Pi\setminus\Gamma_1$ and the associated $\gamma^*$-string $S^{\gamma^*}(\alpha_0):= \{\alpha_i\}_{i=0}^m$ (see Definition~\ref{d:g_str}), for every $k \in [0,m]$ and $i \in [\alpha_k+1,n]$, set
\begin{equation}\label{eq:g_fun}
    g_{i,\alpha_k+1}(U) := \det [\mathcal{F}^{\op}(U)]_{[i,n]}^{[\alpha_k+1,n-i+\alpha_k+1]} \prod_{t \geq k+1}^{m} \det[\mathcal{F}^{\op}(U)]_{[\alpha_t+1,n]}^{[\alpha_t+1,n]}.
\end{equation}
It is evident from the above formula that the flag minors of $\mathcal{F}^{\op}(U)$ can be expressed as
\begin{equation*}
\det [\mathcal{F}^{\op}(U)]_{[i,n]}^{[j,n-i+j]} = \begin{cases}
    \frac{g_{ij}(U)}{g_{\mu\mu}(U)} \ &\text{if} \ j-1 \in \Gamma_2\\
    g_{ij}(U) \ &\text{otherwise}
\end{cases}
\end{equation*}
where $1 \leq i \leq n$, $1 \leq j \leq i$ and $\mu := \gamma^*(j-1)+1$ if $j-1 \in \Gamma_2$. 
\begin{proposition}\label{p:gh_t}
    Let $g_{ij}$ be a $g$-function from the initial extended cluster of $\gc_g^{\dagger}(\bg)$ and $h_{ji}$ be an $h$-function from the initial extended cluster of $\gc_h^{\dagger}(\bg^{\op})$. Then $g_{ij}(U) = (-1)^{\varepsilon_{ji}}h_{ji}(U^{T})$, where $\varepsilon_{ji} = (i-j)(n-j)$.
\end{proposition}
\begin{proof}
    The statement is a consequence of Proposition~\ref{p:ffg_comm} and the defining formulas~\eqref{eq:h_fun} and~\eqref{eq:g_fun} for the $h$- and $g$-functions.
\end{proof}
As a corollary, one can obtain an explicit expansion of the $g$-functions in terms of minors of $U$, as in Section~\ref{s:p_explicit}. One can also obtain invariance properties of the $g$-functions from Corollary~\ref{c:fop_invar}. 

\paragraph{Frozen variables.} The frozen variables in $\gc_g^{\dagger}(\bg,\GL_n)$ are given by the set
\begin{equation*}
    \{c_1,c_2,\ldots,c_{n-1}\} \cup \{g_{i+1,i+1} \ | \ i\in \Pi\setminus\Gamma_1\} \cup \{g_{11}\}.
\end{equation*}
In the case of $\gc_g^{\dagger}(\bg,\SL_n)$, $g_{11}(U) = 1$, so this variable is absent. As the next proposition shows, the frozen variables in $\gc_g^{\dagger}(\bg)$ and $\gc_h^{\dagger}(\bg)$ are identical.

\begin{proposition}\label{p:samefroz}
    Let $\{\alpha_t\}_{t = 0}^m$ be a $\gamma$-string for the BD triple $(\Gamma_1,\Gamma_2,\gamma)$. 
    Then \begin{equation*}
        h_{\alpha_0+1,\alpha_0+1}(U) = g_{\alpha_m+1,\alpha_m+1}(U)
    \end{equation*}
    where $h_{\alpha_0+1,\alpha_0+1}$ is a frozen variable in $\gc_h^{\dagger}(\bg)$ and $g_{\alpha_m+1,\alpha_m+1}$ is a frozen variable in $\gc_g^{\dagger}(\bg)$.
\end{proposition}
\begin{proof}
    Let us denote by $\bar{h}$ the $h$-variables in $\gc_h^{\dagger}(\bg^{\op})$. Since $\alpha_m \notin \Gamma_1$, the variable $\bar{h}_{\alpha_m+1,\alpha_m+1}$ is frozen in $\gc_h^{\dagger}(\bg^{\op})$. It follows from  Proposition~\ref{p:gh_t} that $g_{\alpha_m+1,\alpha_m+1}(U) = \bar{h}_{\alpha_m+1,\alpha_m+1}(U^T)$. By Proposition~\ref{p:hfrozdual}, $\bar{h}_{\alpha_m+1,\alpha_m+1}(U^T) = h_{\alpha_0+1,\alpha_0+1}(U)$. Thus the statement holds.
\end{proof}

The following proposition is a restatement of Proposition~\ref{p:frozvar} in the case of $\gc_g^{\dagger}(\bg)$.

\begin{proposition*}\label{p:frozvarg}
Let $\psi$ be a frozen variable in $\gc^{\dagger}_g(\bg,G)$, $G \in \{\SL_n,\GL_n\}$. Then the following statements hold:
\begin{enumerate}[1)]
\item The variable $\psi$ is semi-invariant with respect to the action~\eqref{eq:daction};
\item The variable $\psi$ is log-canonical with any $u_{ij}$;
\item The zero locus of $\psi$ foliates into a union of symplectic leaves of $(G,\pi_{\bar{\bg}}^{\dagger})$.
\end{enumerate}
\end{proposition*}
\begin{proof}
    It is a consequence of Proposition~\ref{p:frozvar} for $\gc_h^{\dagger}(\bg)$ (see Section~\ref{s:p_frozen}) and Proposition~\ref{p:samefroz}.
\end{proof}

\paragraph{Initial extended cluster.} The initial extended cluster $\Psi_0$ of $\gc^{\dagger}_g(\bg,\GL_n)$ is given by the set
\begin{equation*}
    \{g_{ij} \ | \ 2 \leq j \leq i \leq n\} \cup \{\phi_{kl} \ | \ k,l\geq 1, k+l\leq n\} \cup \{c_1,\ldots,c_{n-1}\}\cup \{g_{11}\}
\end{equation*}
where $g_{11}(U) := \det(U)$. In the case of $\SL_n$, the initial extended cluster is the same except the variable $g_{11}$ is removed.

\paragraph{Ground rings.} The ground rings are defined as follows.
\begin{align*}
\hat{\mathbb{A}}_{\mathbb{C}}(\bg,\GL_n):=&\mathbb{C}[g_{11}^{\pm 1}, c_1,\ldots,c_{n-1}, g_{i+1,i+1} \ | \ i \in \Pi\setminus \Gamma_1];\\
\hat{\mathbb{A}}_{\mathbb{C}}(\bg,\SL_n):=&\mathbb{C}[c_1,\ldots,c_{n-1}, g_{i+1,i+1} \ | \ i \in \Pi\setminus \Gamma_1].
\end{align*}

\paragraph{A generalized cluster mutation.} Only the variable $\phi_{11}$ is equipped with a nontrivial string, which is given by $(1,c_1,\ldots,c_{n-1},1)$. The generalized mutation relation for $\phi_{11}$ reads
\begin{equation*}
    \phi_{11}\phi_{11}^{\prime} = \sum_{r=0}^n c_r \phi_{21}^r \phi_{12}^{n-r},
\end{equation*}
which is similar to the relation~\eqref{eq:p11mut} in the $h$-convention. The mutations of all the other variables follow the usual pattern~\eqref{eq:ordexchrel} from the theory of cluster algebras of geometric type.

\subsection{Birational quasi-isomorphisms}\label{s:g_birat}
In this subsection, we recast the birational quasi-isomorphisms from Section~\ref{s:birat} in the $g$-convention. Let us recall from Section~\ref{s:q_transp} that the rational map $\mathcal{Q}^{\op}: (\GL_n,\pi_{(\bg_\std,r_0)}^{\dagger})\dashrightarrow(\GL_n,\pi_{(\bg,r_0)}^{\dagger})$ is given by
\begin{equation}\label{eq:qop2}
    \mathcal{Q}^{\op}(U) := \rho^{\op}(U)U \rho^{\op}(U)^{-1}, \ \ \rho^{\op}(U) := \prod_{i=1}^{\leftarrow} (\tilde{\gamma})^i(U_+), \ U \in \GL_n.
\end{equation}
We showed in Proposition~\ref{p:qop_poiss} that $\mathcal{Q}^{\op}$ is a Poisson map. Define the rational map $\mathcal{F}^{\op,c}:(\GL_n,\pi_{(\bg,r_0)}^{\dagger})\dashrightarrow(\GL_n,\pi_{(\bg_\std,r_0)}^{\dagger})$ as
\begin{equation*}
\mathcal{F}^{\op,c}(U) := [\mathcal{F}^{\op}(U)]^{-1} U \mathcal{F}^{\op}(U) = \tilde{\gamma}(\mathcal{F}^{\op}(U)_+)^{-1} \mathcal{F}^{\op}(U).
\end{equation*}
\begin{proposition}
    The rational maps $\mathcal{Q}^{\op}$ and $\mathcal{F}^{\op,c}$ are inverse to each other. 
\end{proposition}
\begin{proof}
    The statement follows from the commutative diagram~\eqref{eq:qqopt} and Proposition~\ref{p:fc_q_inv}.
\end{proof}
Similarly to the $h$-convention, if $\tilde{\bg} \prec \bg$, then the $g$-variables in $\gc_g^{\dagger}(\tilde{\bg})$ are decorated with a tilde, and the marker for the pair $(\gc_g^{\dagger}(\bg),\gc_g^{\dagger}(\tilde{\bg}))$ is chosen in such a way that each $\phi$- and $c$-function is related to itself, and each $g_{ij}$ is related to $\tilde{g}_{ij}$.

\begin{proposition}\label{p:qop_birat}
Let $\psi$ be any variable from the initial extended cluster of $\gc^\dagger_g(\bg)$. Then the following identities hold:
\begin{enumerate}[1)]
\item If $\psi$ is a $\phi$- or $c$-function, then $(\mathcal{Q}^{\op})^*(\psi(U)) = \psi(U)$;\label{i:bqlvcop}
\item\label{i:bqplgop} If $\psi = g_{ij}$ for some $g$-function $g_{ij}(U)$, let $\{\alpha_t\}_{t=0}^m$ be the $\gamma^*$-string such that $j-1 = \alpha_k$ for some $k$; then,
\begin{equation*}
(\mathcal{Q}^{\op})^*(g_{ij}(U)) = \tilde{g}_{ij}(U)\prod_{t \geq k+1}\tilde{g}_{\alpha_t+1,\alpha_t+1}(U).
\end{equation*}
\end{enumerate}
Moreover, $\mathcal{Q}^{\op}:(G,\gc_g^{\dagger}(\bg_{\std})) \dashrightarrow (G,\gc_g^{\dagger}(\bg))$ is a birational quasi-isomorphism for $G \in \{\SL_n,\GL_n\}$.
\end{proposition}
\begin{proof}
    Part~\ref{i:bqlvcop} follows from the invariance properties of $\phi$- and $c$-functions. Part~\ref{i:bqplgop} follows from the commutative diagram~\eqref{eq:qqopt}, Proposition~\ref{p:q_quasih} and Proposition~\ref{p:gh_t}. To show that $\mathcal{Q}^{\op}$ is a birational quasi-isomorphism, one applies Proposition~\ref{p:q_y_var} and proceeds along the same lines as in the proof of Proposition~\ref{p:q_quasih}.
\end{proof}
Let $\tilde{\mathcal{Q}}^{\op}$ denote the rational map $\tilde{\mathcal{Q}}^{\op}:(\GL_n,\pi_{(\bg_{\std},r_0)}^{\dagger}) \dashrightarrow (\GL_n,\pi_{(\tilde{\bg},\tilde{r}_0)}^{\dagger})$ given by formula~\eqref{eq:qop2} relative to $(\tilde{\bg},\tilde{r}_0)$. The following proposition is a restatement of Proposition~\ref{p:b_gen} for $\gc_g^{\dagger}(\bg)$, which in turn is a refinement of Proposition~\ref{p:birat}.
\begin{proposition}\label{p:b_gen_g}
    In the setup of the subsection, set $\mathcal{G}^{\op}:=\mathcal{Q}^{\op}\circ (\tilde{\mathcal{Q}}^{\op})^{-1}$. Then $\mathcal{G}^{\op}$ is a birational map given by the formula:
    \begin{align}\label{eq:genbirat_g}
        \mathcal{G}^{\op}(U) = G^{\op}(U) \cdot U \cdot G^{\op}(U)^{-1}, \ \ &G^{\op}(U) := \left[\prod_{i \geq 1}^{\leftarrow} (\tilde{\gamma})^i(G_0^{\op}(U))\right]\cdot G_0^{\op}(U),\\
        &G_0^{\op}(U):= \tilde{\gamma}(\tilde{\mathcal{F}}^{\op}(U)_+)\tilde{\theta}(\tilde{\mathcal{F}}^{\op}(U)_+)^{-1}.\notag
    \end{align}
    Moreover, for $\tilde{\bg} \prec \bg$, then  $\mathcal{G}^{\op}:(G,\gc_g^{\dagger}(\tilde{\bg}))\dashrightarrow(G,\gc_g^{\dagger}(\bg))$ is a birational quasi-isomorphism, and if in addition $\tilde{r}_0 = r_0$, then $\mathcal{G}^{\op}:(G,\pi_{\tilde{\bar{\bg}}}^{\dagger}) \dashrightarrow (G,\pi_{\bar{\bg}}^{\dagger})$ is a Poisson map.
\end{proposition}
\begin{proof}
    Formula~\eqref{eq:genbirat_g} follows from Proposition~\ref{p:b_gen} and the commutative diagram~\eqref{eq:qqopt}. The proof of the fact that $\mathcal{G}^{\op}$ is a birational quasi-isomorphism follows along the same lines as the proof of Proposition~\ref{p:b_gen}.
\end{proof}
Formula~\eqref{eq:genbirat_g} also determines the inverse of $\mathcal{G}^{\op}$, as in Remark~\ref{r:birinv}. In the case  $\tilde{\bg}\prec \bg$ and $|\Gamma_1\setminus\tilde{\Gamma}_1| = 1$, one can also derive a more explicit expression for $G_0^{\op}(U)$ via Proposition~\ref{p:g0expl}.

\subsection{\texorpdfstring{Proofs of Theorems~\ref{thm:maingln} and~\ref{thm:mainsln} for $\gc_g^{\dagger}(\bg)$}{Proofs of Theorems~\ref{thm:maingln} and~\ref{thm:mainsln} for in the g-convention}}\label{s:thm_pf_g}
The proofs of Parts~\ref{i:gln_ni}, \ref{i:gln_tor} and \ref{i:gln_comp} of the theorems are subdivided into the three propositions below (Part~\ref{i:gln_rk} is then a consequence of Part~\ref{i:gln_comp}). We will transfer the results from the case of $\gc_h^{\dagger}(\bg^{\op})$ to $\gc_g^{\dagger}(\bg)$. The key is the following correspondence between the variables:

\begin{enumerate}[1)]
    \item For $\phi$- and $\varphi$-functions, $\phi_{kl}(W_0^{-1}UW_0) = \varphi_{kl}(U)$ (see Remark~\ref{r:phivphi});
    \item For $g_{ij}$ and $h_{ji}$ from the initial extended clusters of $\gc_g^{\dagger}(\bg)$ and $\gc_h^{\dagger}(\bg^{\op})$, $g_{ij}(U) = (-1)^{\varepsilon_{ji}}h_{ji}(U^T)$ (see Proposition~\ref{p:gh_t}).
\end{enumerate}
Both correspondences preserve the Poisson brackets (see Remark~\ref{r:phivphi} and Lemma~\ref{l:transpoiss}).

\begin{proposition}\label{p:compl_g}
    Let $\bg := (\Gamma_1,\Gamma_2,\gamma)$ be a BD triple of type $A_{n-1}$ and $G \in \{\SL_n,\GL_n\}$. Then the following statements hold:
    \begin{enumerate}[1)]
    \item The generalized cluster structure $\gc_g^{\dagger}(\bg,G)$ is regular;
    \item All the variables in the initial extended cluster of $\gc_g^{\dagger}(\bg,G)$ are irreducible (as elements of $\mathbb{C}[G]$) and (for cluster variables) coprime with their mutations;
    \item The generalized cluster structure $\gc_g^{\dagger}(\bg,G)$ is complete.
    \end{enumerate}
\end{proposition}
\begin{proof}
    Let $\Psi_0$ be the initial extended cluster of $\gc_g^{\dagger}(\bg,G)$. In the case of $\bg = \bg_{\std}$, the proof can be executed in the same way as in~\cite{double} for $\gc_h^{\dagger}(\bg_{\std},G)$. There, the proof was based on various normal forms of a generic element $U \in \GL_n$. For instance, in the case of the initial extended cluster of $\gc_h^{\dagger}(\bg_{\std},G)$, one decomposes $U$ as $U=U_\oplus U_-$ and as $U = N_- B_+ C \sinv{N_-}$ for some $N_-,U_- \in \mathcal{N}_-$, $U_\oplus,B_+ \in \mathcal{B}_+$ and $C:=e_{1n} + \sum_{i=1}^{n-1}e_{i+1,i}$. Then $\varphi(U) = \varphi(B_+)$ and $h(U) = h(U_\oplus)$ for $\varphi$- and $h$-variables. In the case of $\gc_g^{\dagger}(\bg_{\std},G)$, one decomposes $U$ as $U = U_+ U_\ominus$ and $U = N_+B_-C^T\sinv{N_+}$, where $N_+,U_+ \in \mathcal{N}_+$ and $U_{\ominus},B_- \in \mathcal{B}_-$. Then, to express $u_{ij}$ in terms of $\mathcal{L}_{\mathbb{C}}(\Psi_0)$, it is enough to express the entries of $B_-$ and $U_{\ominus}$ in terms of $\mathcal{L}_{\mathbb{C}}(\Psi_0)$. This process is described in~\cite{double}. For $|\Gamma_1| = 1$, one checks the conditions of Propositions~\ref{p:birat_fsingle}-\ref{p:upper_cont1} for the marked variable, which in turn can be transferred from the case of $\gc_h^{\dagger}(\bg,G)$ via the transposition map. For $|\Gamma_1| > 1$, the proof proceeds along the same lines as in the case of $\gc_h^{\dagger}(\bg)$, and it is based on the inductive argument on the size of $|\Gamma_1|$ and Propositions~\ref{p:birat_double_ufd}-\ref{p:upper_cont2}, which in turn rely on the existence of birational quasi-isomorphisms $\mathcal{G}^{\op}$ with distinct sets of marked variables.
\end{proof}

Let us recall that the number $k_{\bg}$ is given by $k_{\bg}=|\Pi\setminus \Gamma_1|$.
\begin{proposition}\label{p:toric_g}
Let $\bg$ be a BD triple of type $A_{n-1}$. The following statements hold:
\begin{enumerate}[1)]
    \item The action of $\mathcal{H}_{\bg}$ upon $\SL_n$ induces a global toric action upon $\gc_g^{\dagger}(\bg,\SL_n)$ of rank $k_{\bg}$;
    \item The action of $\mathcal{H}_{\bg} \times \mathbb{C}^*$ upon $\GL_n$ induces a global toric action upon $\gc_g^{\dagger}(\bg,\GL_n)$ of rank $k_{\bg}+1$.
\end{enumerate}
\end{proposition}
\begin{proof}
The fact that the action by scalars $\mathbb{C}^*$ induces a global toric action upon $\gc_g^{\dagger}(\bg,\GL_n)$ can be proved in the same way as in the case of $\gc_h^{\dagger}(\bg,\GL_n)$ (see Lemma~\ref{l:toric_glob_scalar}). Due to Proposition~\ref{p:q_equiv} and the commutative diagram~\eqref{eq:qqopt}, $\mathcal{Q}^{\op}:(G,\pi_{(\bg_\std,r_0)}^{\dagger}) \dashrightarrow (G,\pi_{(\bg,r_0)}^{\dagger})$ is equivariant with respect to the action of~$\mathcal{H}_{\bg}$; as a consequence, the $g$-functions are semi-invariant with respect to the action of~$\mathcal{H}_{\bg}$. Furthermore, due to the semi-invariance of the $\varphi$-functions and Remark~\ref{r:phivphi}, the $\phi$-functions are also semi-invariant with respect to the action; hence, $\mathcal{H}_{\bg}$ induces a local toric action upon the initial extended cluster $\Psi_0$ of $\gc_g^{\dagger}(\bg)$. Due to Proposition~\ref{p:gh_t}, Remark~\ref{r:phivphi} and Proposition~\ref{p:toric_h}, all the $y$-variables in $\Psi_0$ are invariant with respect to $\mathcal{H}_{\bg}$ except $y(\phi_{kl})$, $k+l = n$, which require a separate check. In turn, this reduces to showing that $\pi_0([\nabla_U \log y(\phi_{kl}),U]) = 0$ for $\bg = \bg_{\std}$. Set $\Delta(i,j):=\sum_{k=i}^j e_{kk}$. It is not difficult to see that
\begin{align}
&\pi_0([\nabla_U \log \phi_{kl},U]) =  \Delta(1,l) - \Delta(k+1,n) + (n-k-l)\Delta(1,1);\label{eq:aduphikl}\\
&\pi_0([\nabla_U \log g_{ij},U]) = \Delta(j,n-i+j)-\Delta(i,n).\label{eq:adugij}
\end{align}
Now, the neighborhoods of $\phi_{kl}$ for $k+l=n$ are given by Figure~\ref{f:gnbd_phi1n1}, Figure~\ref{f:gnbd_phin11} and Figure~\ref{f:gnbd_phikl_b}, and it is a matter of a direct computation that $\pi_0([\nabla_U \log y(\phi_{kl}),U]) = 0$, $k+l=n$. The Casimirs $\hat{p}_{ir}$ from the statement of Proposition~\ref{p:toric} are the same as in $\gc_h^{\dagger}(\bg)$, hence they are invariant with respect to $\mathcal{H}_{\bg}$. It follows from Proposition~\ref{p:toric} that the local toric action of $\mathcal{H}_{\bg}$ upon $\gc_g^{\dagger}(\bg)$ is global. To verify that the action is of rank $k_{\bg}$, one applies the same argument as in Proposition~\ref{p:toric_h} based on Proposition~\ref{p:compl_g}.
\end{proof}

\begin{lemma}\label{l:log_g_phi}
    Let $g_{ij}$ be a $g$-function and $\phi_{kl}$ be a $\phi$-function from the initial extended cluster of $\gc_g^{\dagger}(\bg)$. Then
    \begin{equation}\label{eq:gijphikl}
        \{g_{ij},\phi_{kl}\}_{(\bg,r_0)}^{\dagger} = \langle R_0\pi_0[\nabla_Ug_{ij},U],[\nabla_U\phi_{kl},U]\rangle - \langle \pi_0(U\nabla_U g_{ij}),[\nabla_U\phi_{kl},U]\rangle.
    \end{equation}
    Moreover, the $g$-functions are log-canonical with the $\phi$-functions.
\end{lemma}
\begin{proof}
    Set $\mathring{\nu}^*(g_{ij}) :=U\nabla_U g_{ij} - \ogammas(\nabla_U g_{ij}\cdot U)$. Due to the invariance property~\eqref{eq:fopnpinv}, $\mathring{\nu}^*(g_{ij}) \in \mathfrak{b}_+$; therefore,
    \begin{equation*}
    -\frac{\ogammas}{1-\ogammas}\pi_{<}[\nabla_U g_{ij},U] = \pi_{<}(U\nabla_U g_{ij}).
    \end{equation*}
    Together with $[\nabla_U \phi_{kl},U] \in \mathfrak{b}_+$, expanding the bracket $\{g_{ij},\phi_{kl}\}_{(\bg,r_0)}^{\dagger}$ yields formula~\eqref{eq:gijphikl}. Since the map $\mathcal{Q}^{\op}:(\GL_n,\pi_{(\bg_\std,r_0)}^{\dagger}) \dashrightarrow (\GL_n,\pi_{(\bg,r_0)}^{\dagger})$ is a Poisson isomorphism (see Proposition~\ref{p:qop_poiss}) and a birational quasi-isomorphism (see Proposition~\ref{p:qop_birat}), it is enough to verify the log-canonicity for $\bg = \bg_{\std}$, and that readily follows from  combining the formulas~\eqref{eq:aduphikl}, \eqref{eq:adugij} and~\eqref{eq:gijphikl} along with $\pi_0(U\nabla_U \log g_{ij}) = \Delta(i,n)$.
\end{proof}

\begin{proposition}\label{p:compatibl_g}
For any BD quadruple $(\bg,r_0)$, the generalized cluster structure $\gc_g^{\dagger}(\bg)$ is compatible with $\pi_{(\bg,r_0)}^{\dagger}$.
\end{proposition}
\begin{proof}
    In the initial extended cluster $\Psi_0$ of $\gc_g^{\dagger}(\bg)$, the log-canonicity between the $\phi$-functions is a consequence of Remark~\ref{r:phivphi}, the log-canonicity between the $g$-functions is a consequence of Proposition~\ref{p:gh_t} and Lemma~\ref{l:transpoiss}, and the log-canonicity between the $g$- and the $\phi$-functions is a consequence of Lemma~\ref{l:log_g_phi}. Based on Proposition~\ref{p:compb}, to verify the compatibility of $\gc_g^{\dagger}(\bg)$ with $\pi_{(\bg,r_0)}^{\dagger}$, for any cluster variable $\psi_1 \in \Psi_0$ and any variable $\psi_2\in\Psi_0$, one verifies that
    \begin{equation}\label{eq:comp_gg}
        \{\log y(\psi_1),\psi_2\}_{(\bg,r_0)}^{\dagger} = \begin{cases}
            1 \ &\text{if} \ \psi_1 = \psi_2\\
            0 \ &\text{otherwise.}
        \end{cases}
    \end{equation}
    Due to Remark~\ref{r:phivphi}, if $\psi_1,\psi_2 \in \{\phi_{kl} \ | \ k+l < n\}$, then~\eqref{eq:comp_gg} holds, and due to Proposition~\ref{p:gh_t} and Lemma~\ref{l:transpoiss}, if $\psi_1,\psi_2\in \{g_{ij} \ | \ 2 \leq j \leq i \leq n\}$, \eqref{eq:comp_gg} holds as well (note that this covers the case of the marked variables of $\mathcal{Q}^{\op}$ as well). By Proposition~\ref{p:qop_poiss}, it remains to verify~\eqref{eq:comp_gg} in the case $\bg = \bg_{\std}$ and 1) when $\psi_1$ is a $\phi$-function and $\psi_2$ is a $g$-function; 2) when $\psi_1$ is a $g$-function and $\psi_2$ is a $\phi$-function; 3) when $\psi_1 = \phi_{kl}$ for $k+l=n$ and $\psi_2\in\Psi_0$ is any variable. This is a direct computation based on formulas~\eqref{eq:aduphikl}, \eqref{eq:adugij} and \eqref{eq:gijphikl} together with an analysis of the neighborhood of $\psi_1$ provided in Appendix~\ref{s:ainiquivg}.
\end{proof}
\appendix

\section{\texorpdfstring{Initial quiver for $\gc_h^{\dagger}(\bg)$}{Initial quiver for the h-convention}}\label{s:ainiquivh}
In this appendix, we describe the initial quiver $Q_h(\bg)$ for $\gc^\dagger_h(\bg)$. We start with describing the quiver $Q_h(\bg_{\std})$ in Section~\ref{s:quiv_h_triv}, which is exemplified in Figure~\ref{f:ex_n=5_std} for $n=5$. In Section~\ref{s:quiv_h_alg}, we explain how to obtain the quiver $Q_h(\bg)$ from $Q_h(\bg_{\std})$ for a nontrivial BD triple $\bg$. The procedure of obtaining $Q_h(\bg)$ from $Q_{h}(\bg_{\std})$ consists in adding extra arrows to $Q_{h}(\bg_{\std})$ based on $\bg$, as well as in unfreezing some of the vertices. Alternatively, in Section~\ref{s:h_quiv_expl} we provide explicit neighborhoods of the vertices of $Q_h(\bg)$ that are different from the corresponding neighborhoods in $Q_h(\bg_{\std})$.  For explicit examples of $Q_h(\bg)$ for nontrivial BD triples, see Appendix~\ref{s:exs} and the supplementary note~\cite{github}. Throughout the section, we assume that $n \geq 3$.
\subsection{The quiver for the trivial BD triple}\label{s:quiv_h_triv}
Below one can find pictures of the neighborhoods of all vertices of the initial quiver of $\gc_h^{\dagger}(\bg_{\std},\GL_n)$. A few remarks beforehand:
\begin{itemize}
\item The circled vertices are mutable (in the sense of ordinary exchange relations~\eqref{eq:ordexchrel}), the square vertices are frozen, the rounded square vertices may or may not be mutable depending on the indices, and the hexagon vertex is a mutable vertex with a generalized mutation relation (see formula~\eqref{eq:p11mut});
\item Since $c_1,\ldots,c_{n-1}$ are isolated variables, they are not shown on the resulting quiver;
\item For $k=2$ and $n > 3$, the vertices $\varphi_{1k}$ and $\varphi_{k-1,2}$ coincide; hence, the pictures provided below suggest that there are two edges pointing from $\varphi_{21}$ to $\varphi_{12}$ (however, there is only one arrow in $n = 3$);
\item When the indices of the $h$-variables are out of range, we use convention~\eqref{eq:hconv}.;
\item The initial quiver for $\gc_h^{\dagger}(\bg_{\std},\SL_n)$ is the same as for $\gc_h^{\dagger}(\bg_{\std},\GL_n)$ except the vertex with $h_{11}$ has to be removed.
\end{itemize}

 \begin{figure}[htb]
 \begin{center}
 \includegraphics[scale=0.2]{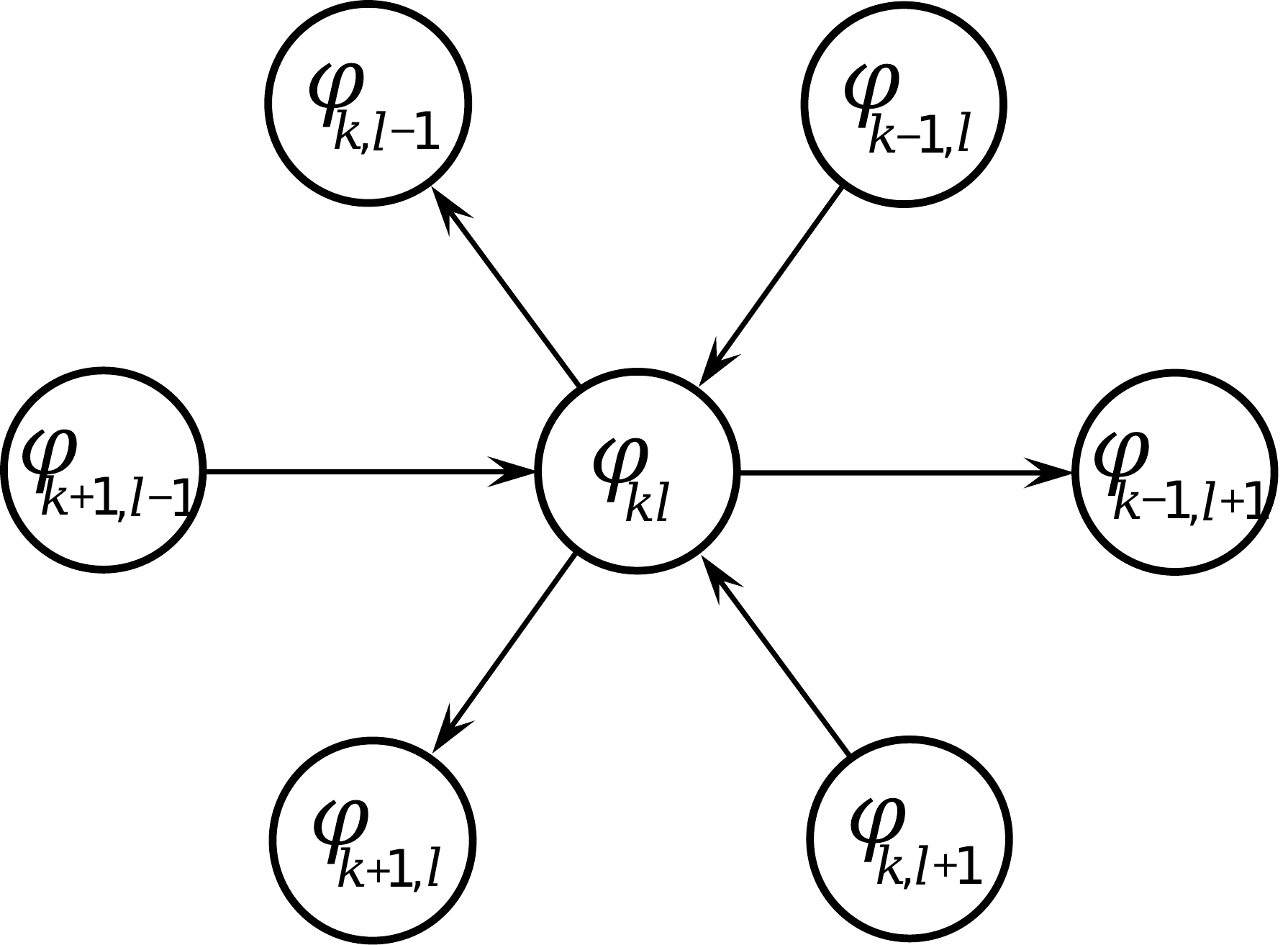}
 \end{center}
 \caption{The neighborhood of $\varphi_{kl}$ for $k,l \neq 1$, $k+l < n$.}
 \label{f:nbd_phikl}
 \end{figure}

 \vspace{5mm}
 \begin{figure}[htb]
 \begin{subfigure}[t]{3in}
 \begin{center}
 \includegraphics[scale=0.2]{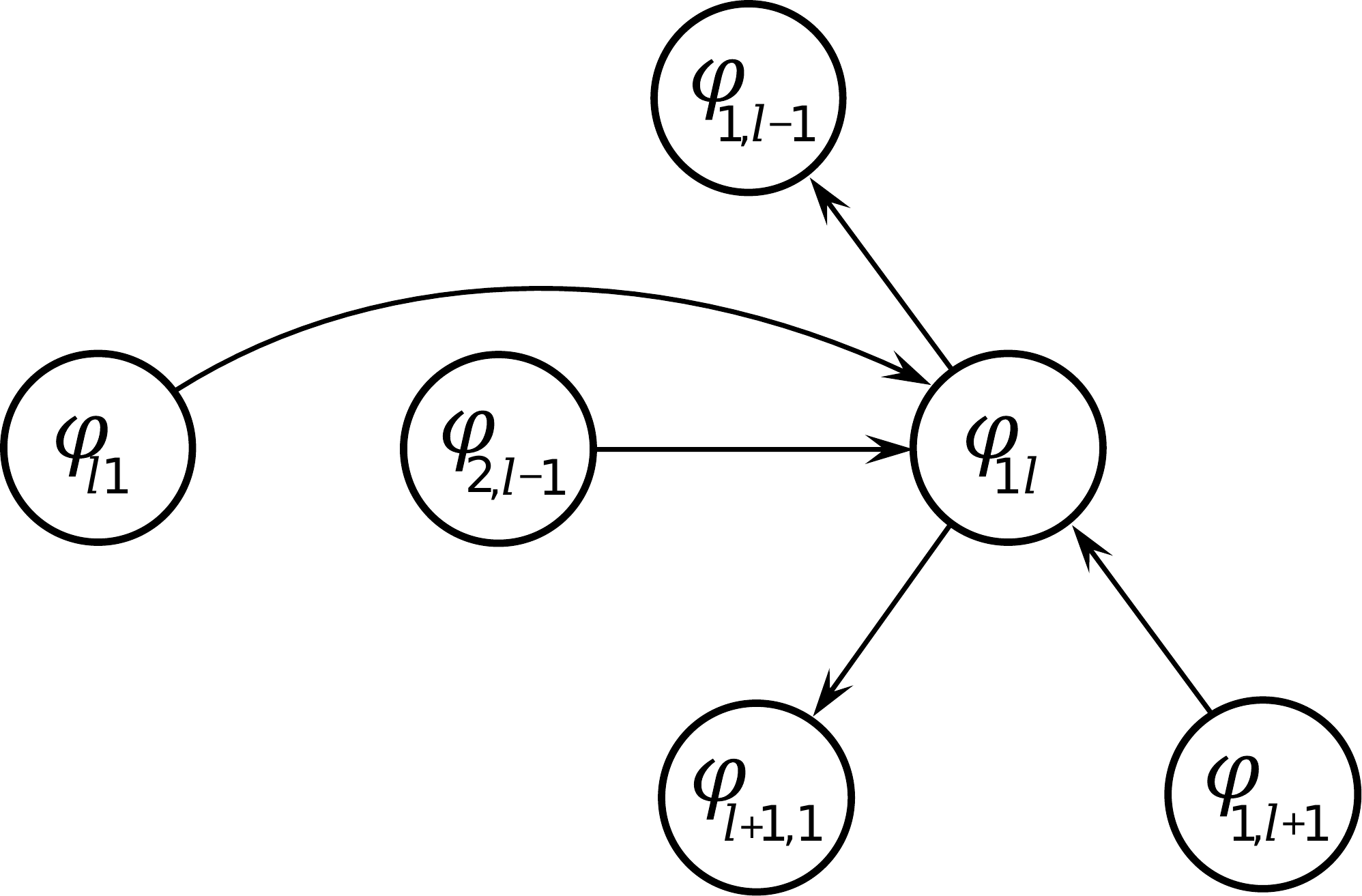}
 \end{center}
 \subcaption{Case $2 \leq l \leq n-2$.}
 \label{f:nbd_phi1l}
 \end{subfigure}
 \begin{subfigure}[t]{2.4in}
 \begin{center}
 \includegraphics[scale=0.2]{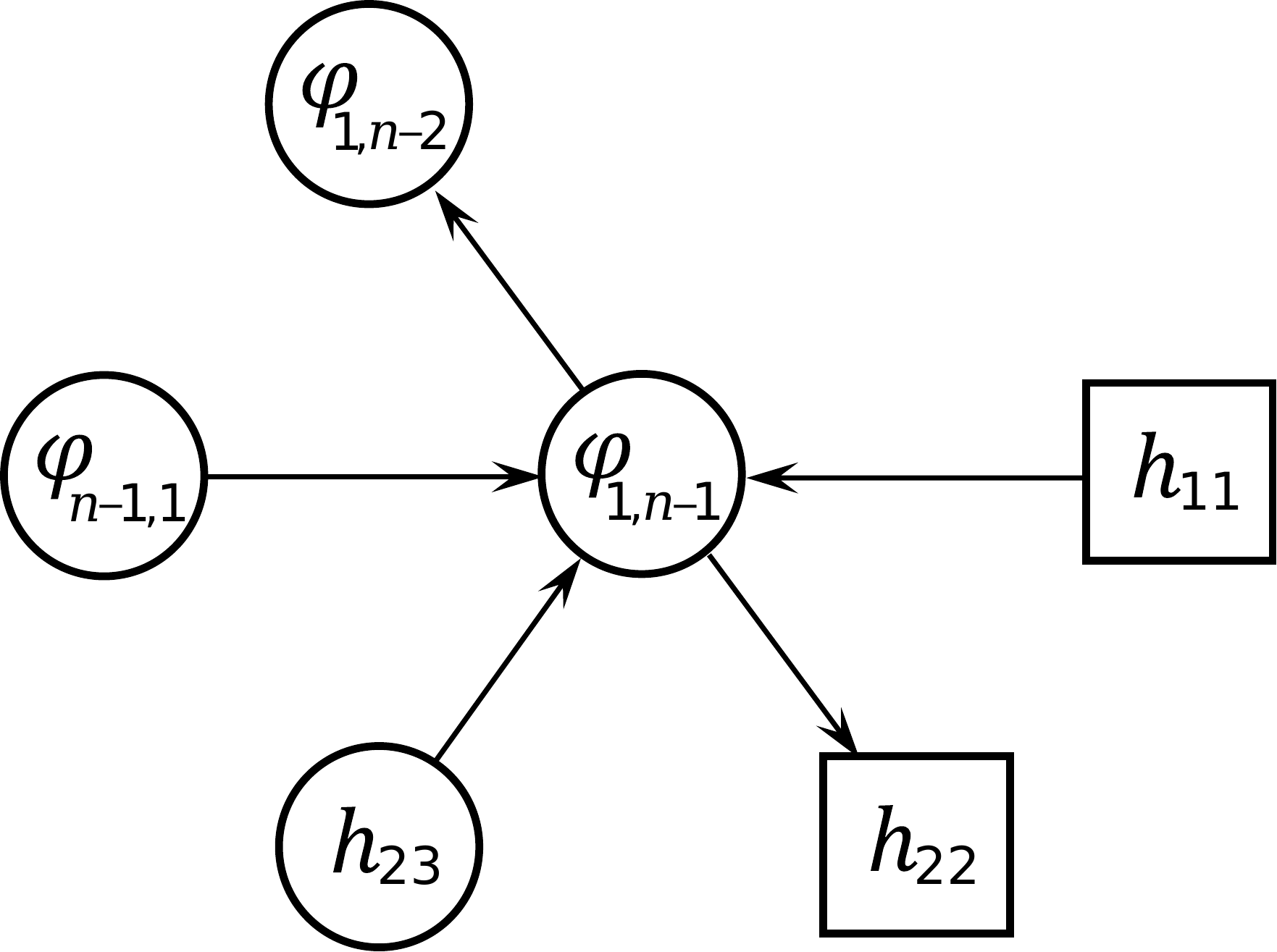}
 \end{center}
 \subcaption{Case $l = n-1$.}
 \label{f:nbd_phi1n1}
 \end{subfigure}
 \caption{The neighborhood of $\varphi_{1l}$ for $2 \leq l \leq n-1$.}
 \label{f:nbd_phi1}
 \end{figure}

 \vspace{5mm}
 \begin{figure}[htb]
 \begin{subfigure}[t]{3.5in}
 \begin{center}
 \includegraphics[scale=0.2]{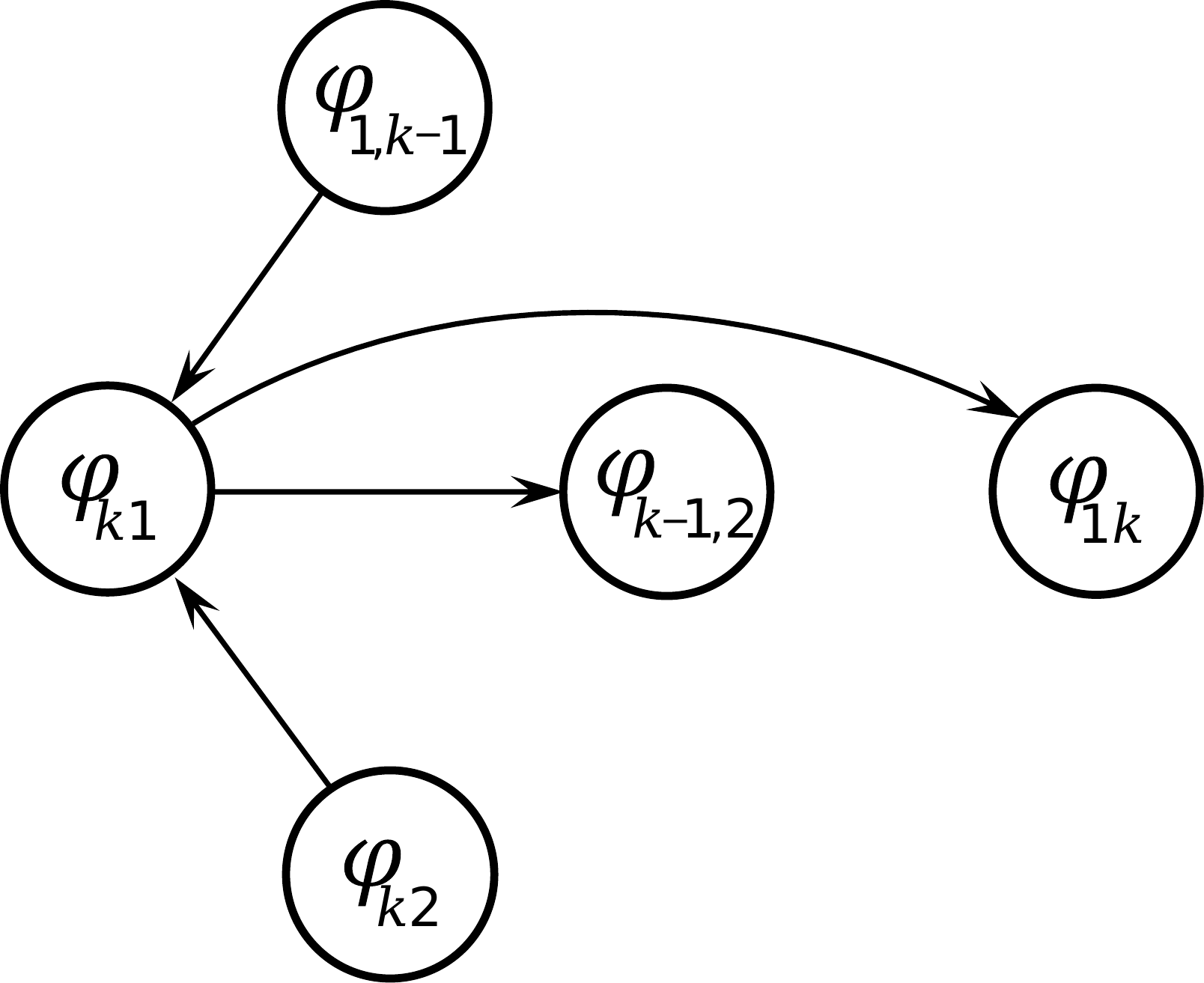}
 \end{center}
 \subcaption{Case $2 \leq k \leq n-2$.}
 \label{f:nbd_phik1}
 \end{subfigure}
 \begin{subfigure}[t]{2.5in}
 \begin{center}
 \includegraphics[scale=0.2]{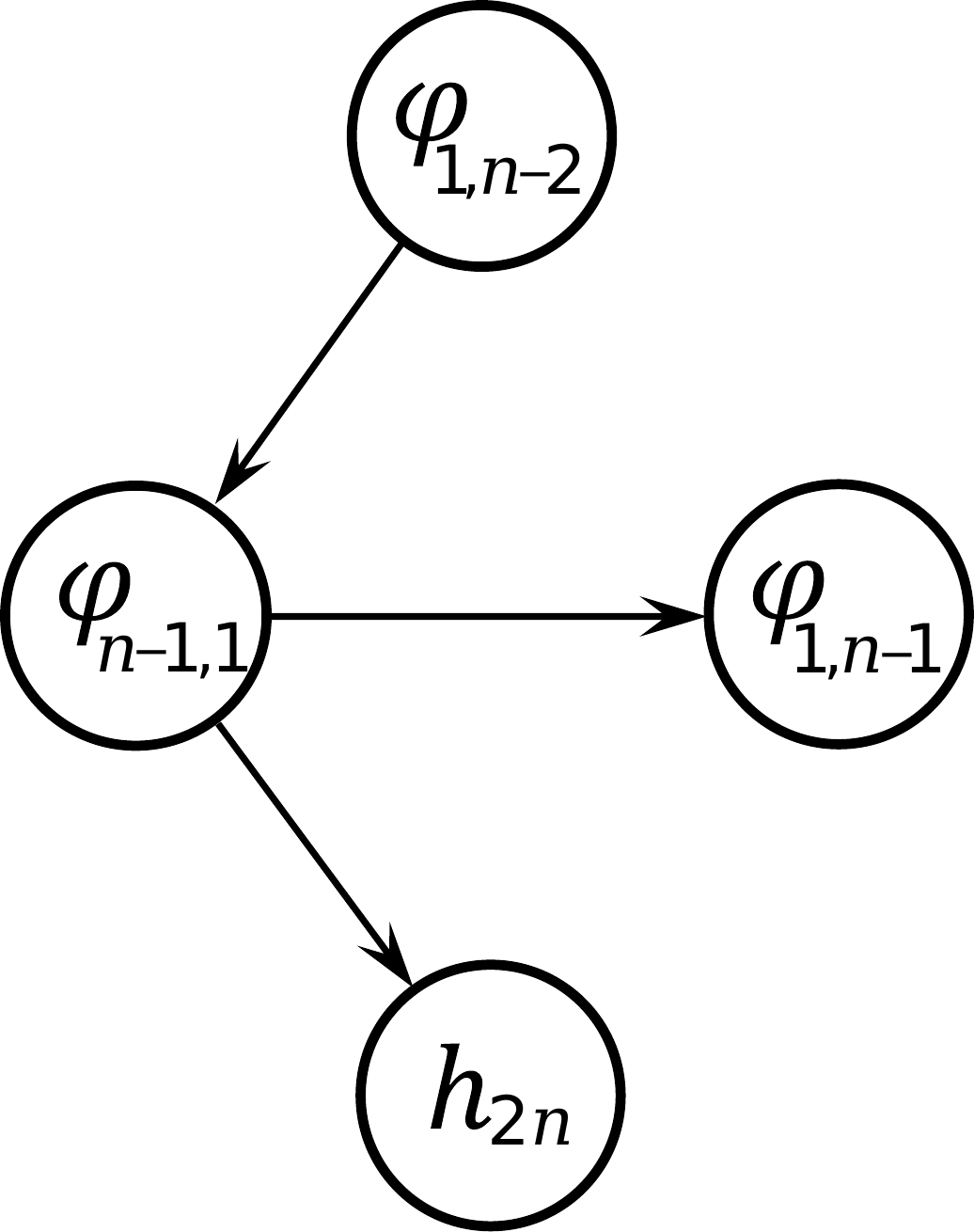}
 \end{center}
 \subcaption{Case $k = n-1$.}
 \label{f:nbd_phin11}
 \end{subfigure}
 \caption{The neighborhood of $\varphi_{k1}$ for $2 \leq k \leq n-1$.}
 \label{f:nbd_phik}
 \end{figure}

 \vspace{5mm}
 \begin{figure}[htb]
 \begin{subfigure}[t]{3.5in}
 \begin{center}
 \includegraphics[scale=0.2]{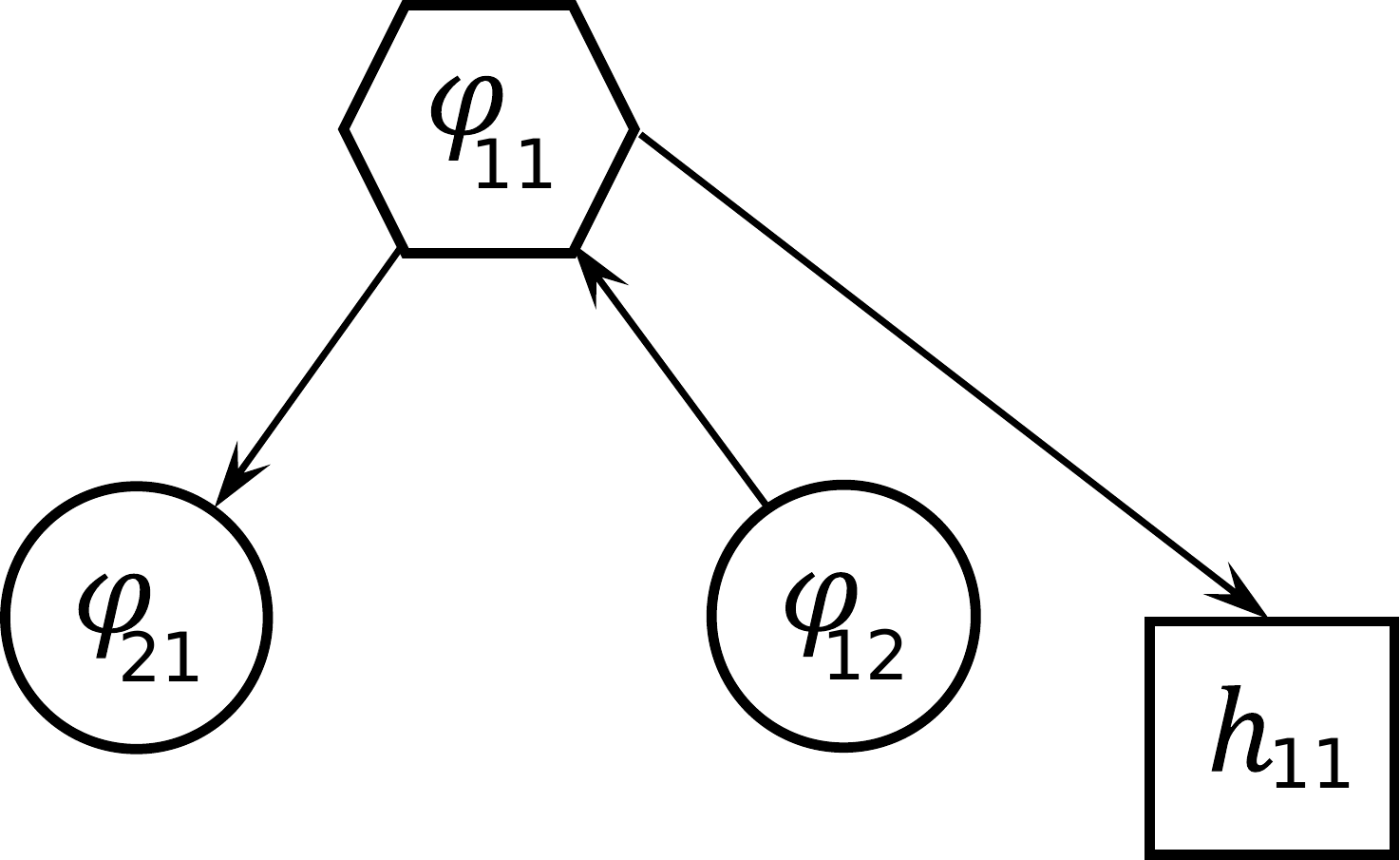}
 \end{center}
 \subcaption{Case $k=l=1$.}
 \label{f:nbd_phi11}
 \end{subfigure}
 \begin{subfigure}[t]{2.5in}
 \begin{center}
 \includegraphics[scale=0.2]{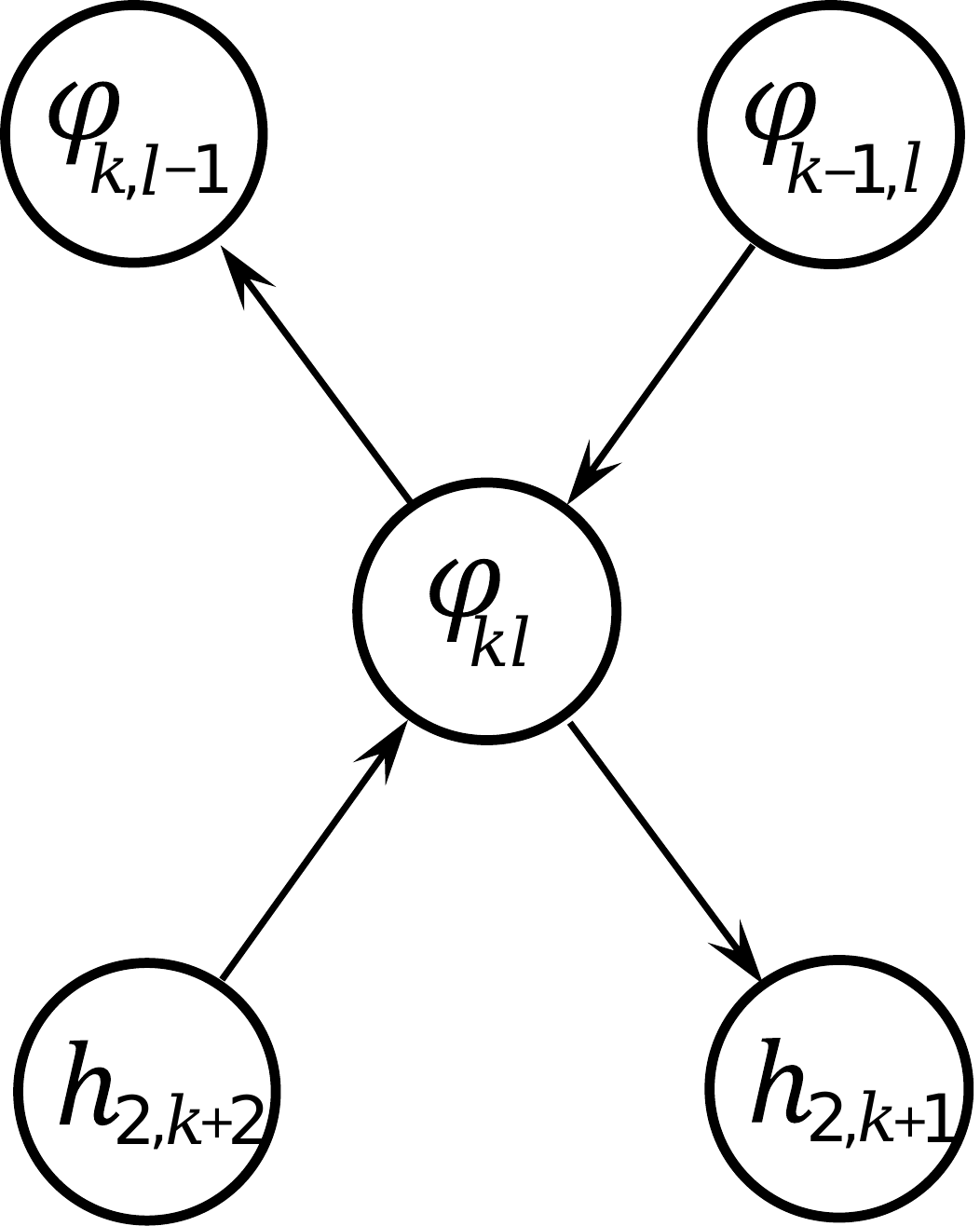}
 \end{center}
 \subcaption{Case $k + l = n$.}
 \label{f:nbd_phikl_b}
 \end{subfigure}
 \caption{The neighborhood of $\varphi_{kl}$ for (a) $k=l=1$ and (b) $k+l=n$.}
 \label{f:nbd_boundary}
 \end{figure}

 \vspace{5mm}
 \begin{figure}[htb]
 \begin{subfigure}[t]{3in}
 \begin{center}
 \includegraphics[scale=0.2]{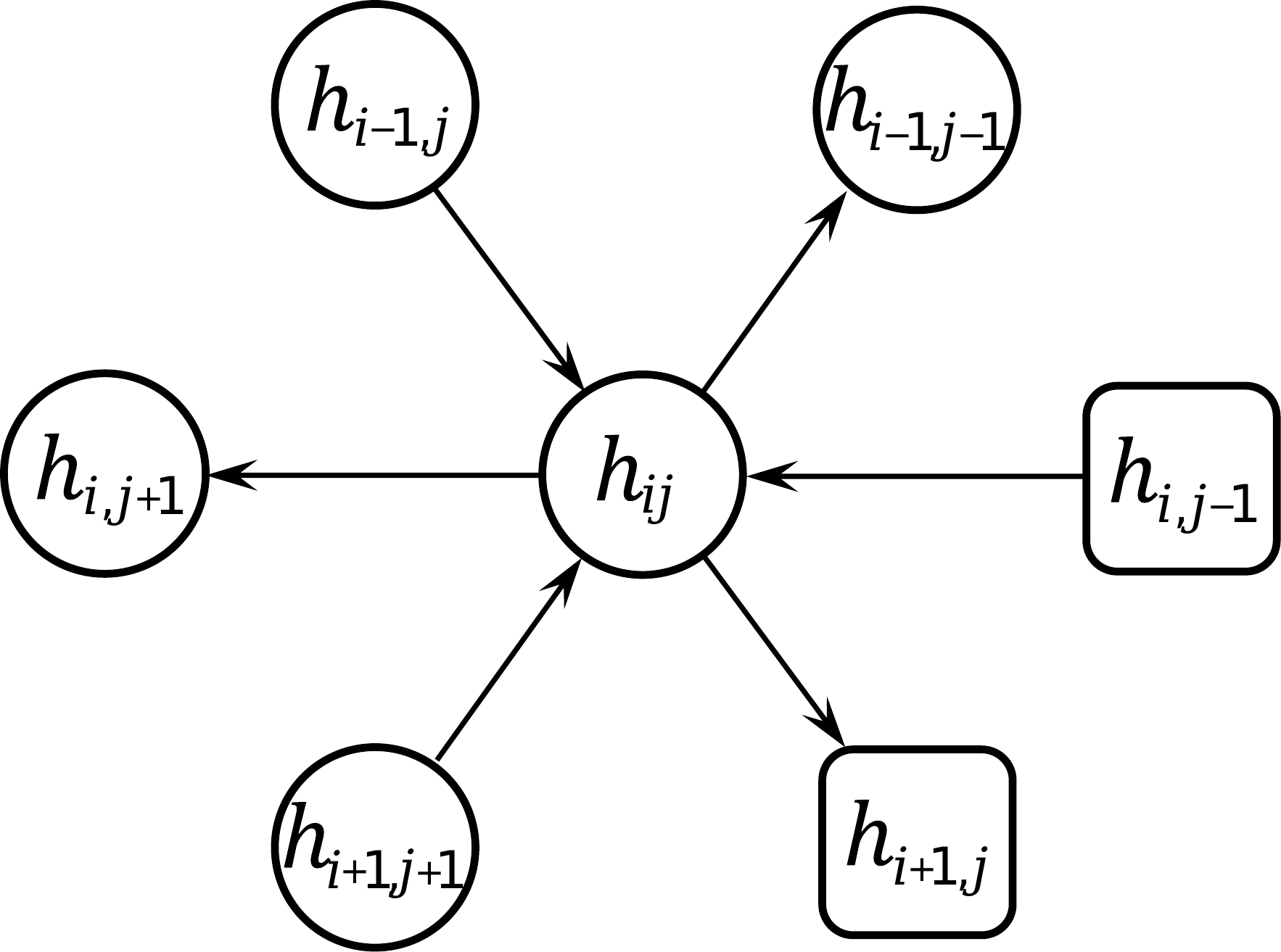}
 \end{center}
 \subcaption{Case $j < n$.}
 \label{f:nbd_hij}
 \end{subfigure}
 \begin{subfigure}[t]{3in}
 \begin{center}
 \includegraphics[scale=0.2]{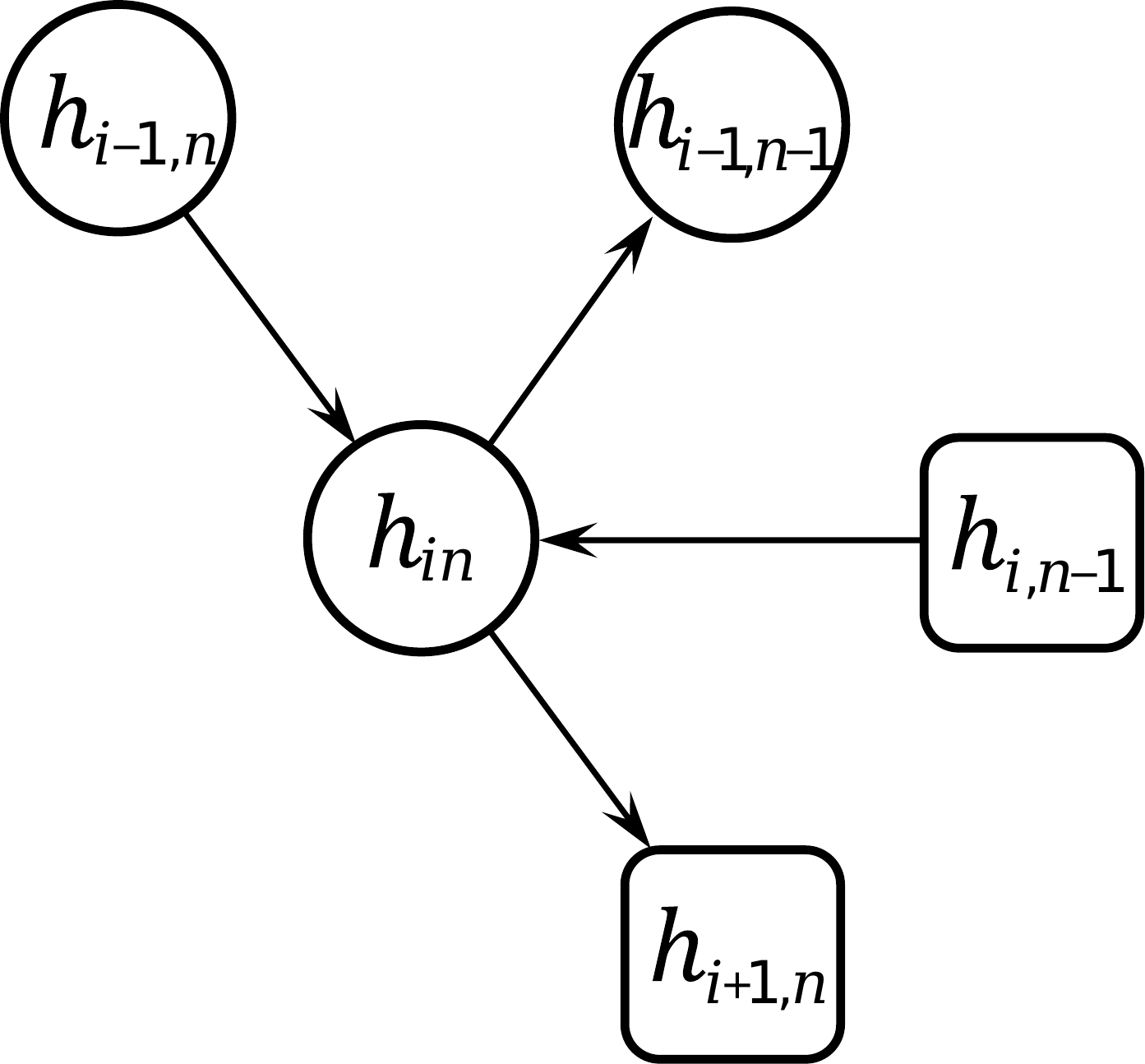}
 \end{center}
 \subcaption{Case $j = n$.}
 \label{f:nbd_hin}
 \end{subfigure}
 \caption{The neighborhood of $h_{ij}$ for $2 \leq i \leq n-1$, $2 \leq j \leq n$, $i \leq j$.}
 \label{f:nbd_h}
 \end{figure}


 \vspace{5mm}
 \begin{figure}[htb]
 \begin{subfigure}[t]{2in}
 \begin{center}
 \includegraphics[scale=0.2]{nbds_h/inbd_h11}
 \end{center}
 \subcaption{Case $i=1$.}
 \label{f:nbd_h11}
 \end{subfigure}
 \begin{subfigure}[t]{2in}
 \begin{center}
 \includegraphics[scale=0.2]{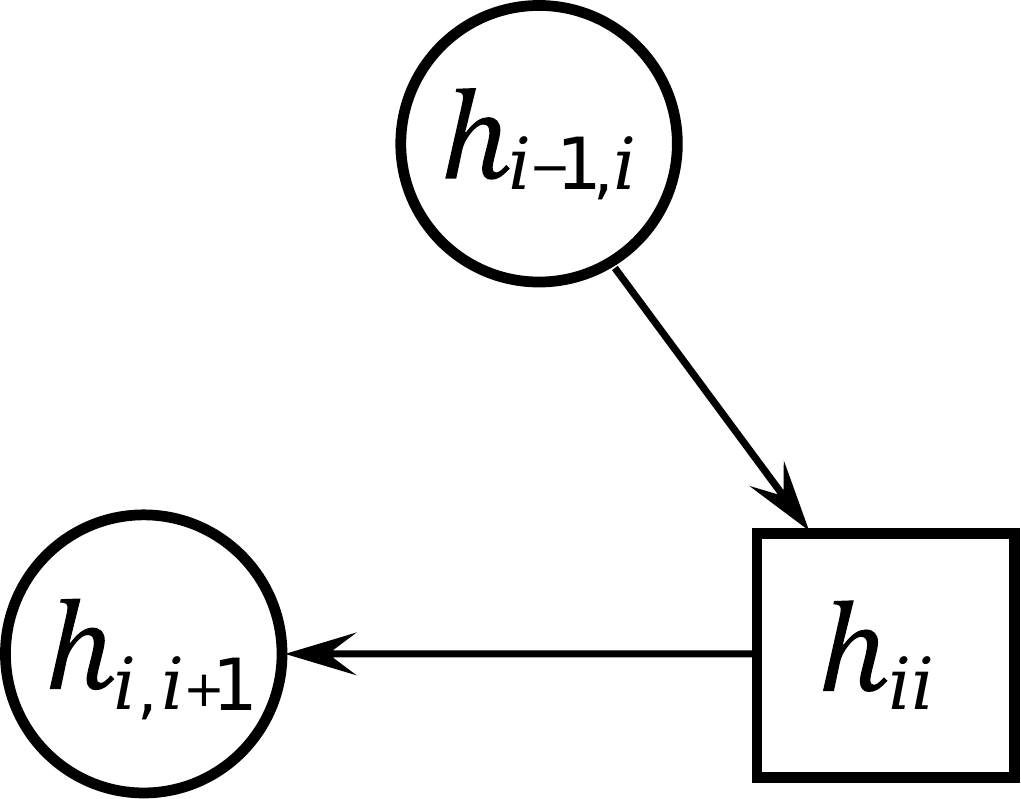}
 \end{center}
 \subcaption{Case $2 \leq i \leq n-1$.}
 \label{f:nbd_hii}
 \end{subfigure}
 \begin{subfigure}[t]{2in}
 \begin{center}
 \includegraphics[scale=0.2]{nbds_h/inbd_hnn}
 \end{center}
 \subcaption{Case $i=n$.}
 \label{f:nbd_hnn}
 \end{subfigure}
 \caption{The neighborhood of $h_{ii}$ for $1 \leq i \leq n$.}
 \label{f:nbd_hh}
 \end{figure}

 \clearpage

\begin{figure}[htb]
\begin{center}
\includegraphics[scale=0.2]{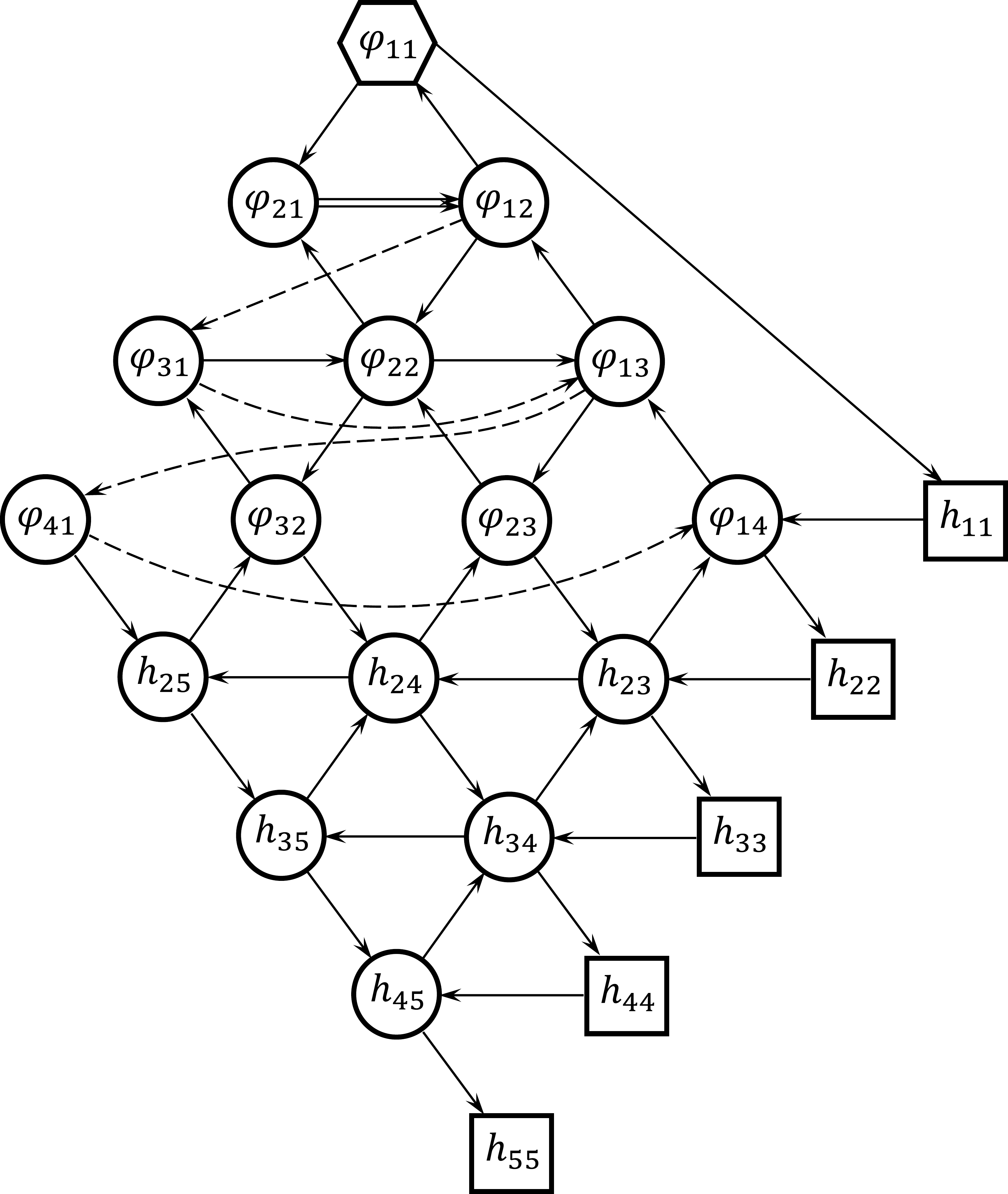}
\end{center}
\caption{The initial quiver for $\gc_h^{\dagger}(\bg_{\std},\GL_5)$. }
\label{f:ex_n=5_std}
\end{figure}

\clearpage

\subsection{The quiver for a nontrivial BD triple (algorithm)}\label{s:quiv_h_alg}
Let $\bg := (\Gamma_1,\Gamma_2,\gamma)$ be a nontrivial BD triple. In order to obtain the quiver $Q_h(\bg)$, one proceeds as follows. First, draw the quiver $Q_h(\bg_{\std})$, as described in Appendix~\ref{s:quiv_h_triv}. Second, add new arrows as prescribed by the following algorithm:
\begin{enumerate}[1)]
\item If $j \in \Gamma_2$ is such that $j-1 \notin \Gamma_2$ or $\gamma|_{\Delta(i)\cap \Gamma_1} < 0$, add new arrows as indicated in Figure~\ref{f:nbd_hj1j1};
\item If $j \in \Gamma_2$ is such that $j-1 \in \Gamma_2$ and $\gamma|_{\Delta(i)\cap \Gamma_1} > 0$, add new arrows as indicated in Figure~\ref{f:nbd_hj1j1_2};
\item Repeat for all roots in $\Gamma_1$.
\end{enumerate}
In the algorithm, $\Delta(i)$ denotes the $X$-run that contains $i:=\gamma^*(j)$ for $j \in \Gamma_2$.

 \vspace{5mm}
 \begin{figure}[htb]
 \begin{subfigure}[t]{3in}
 \begin{center}
 \includegraphics[scale=0.2]{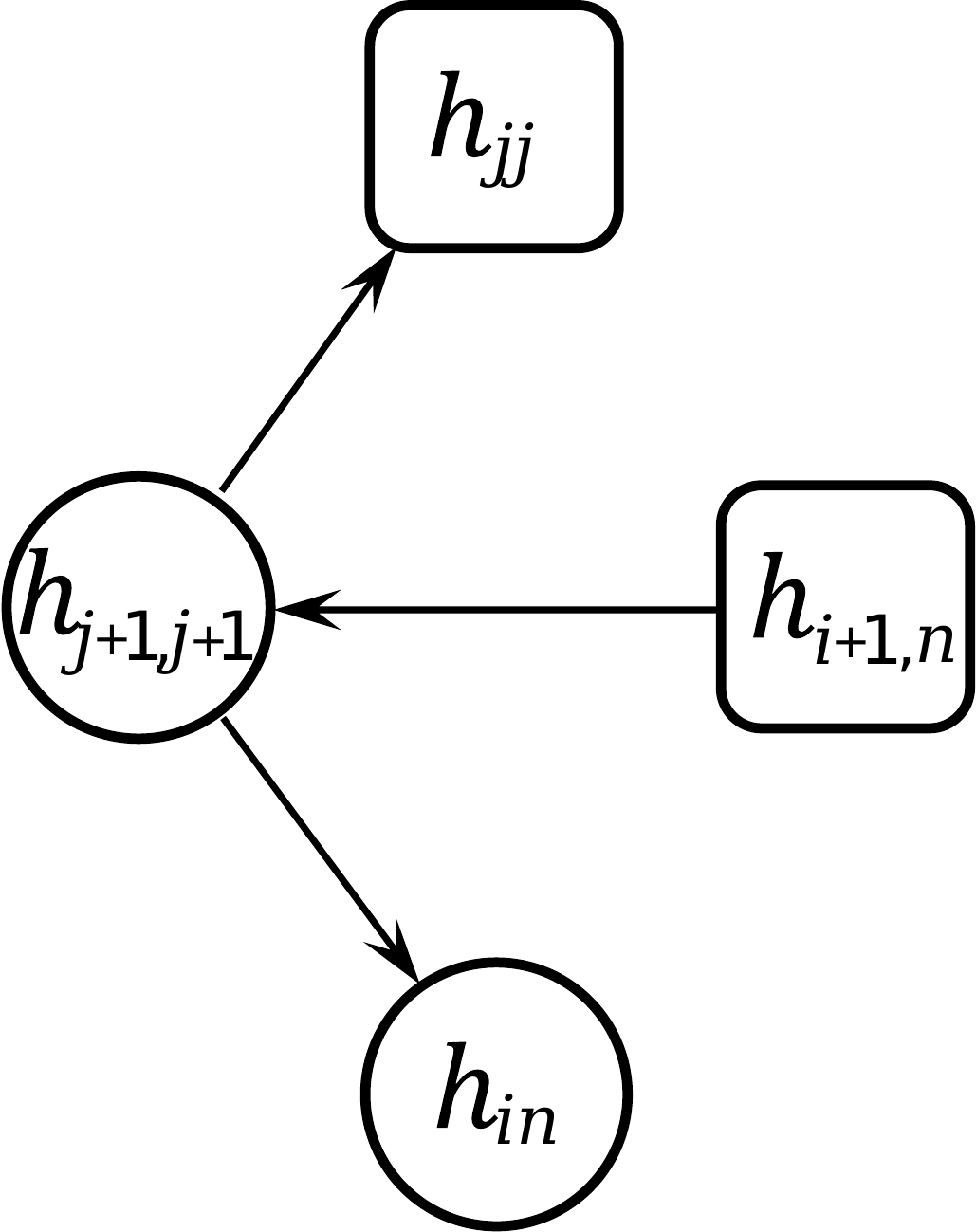} 
 \end{center}
 \subcaption{Case $j-1 \notin \Gamma_2$ or $\gamma|_{\Delta(i)\cap \Gamma_1} < 0$}
 \label{f:nbd_hj1j1}
 \end{subfigure}
 \begin{subfigure}[t]{3in}
 \begin{center}
 \includegraphics[scale=0.2]{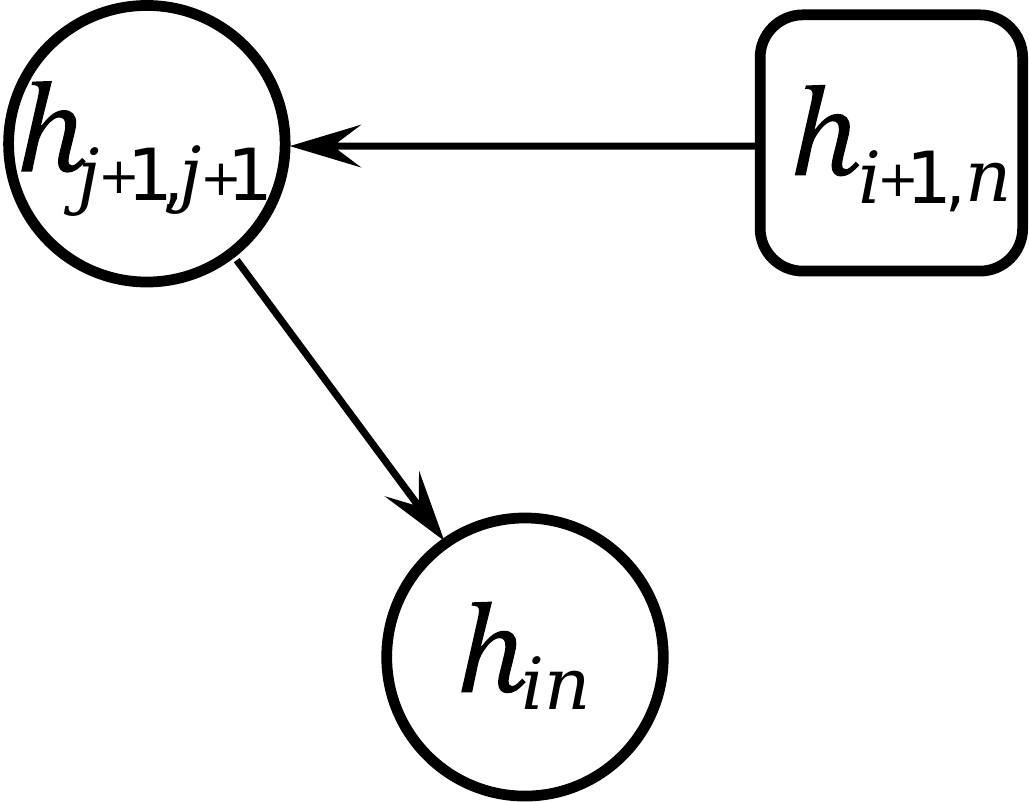}
 \end{center}
 \subcaption{Case $j-1\in \Gamma_2$ and $\gamma|_{\Delta(i)\cap\Gamma_1} > 0$}
 \label{f:nbd_hj1j1_2}
 \end{subfigure}
 \caption{Additional arrows for $j \in \Gamma_2$ and $i:= \gamma^*(j)$.}
 \label{f:nbd_extra_h}
 \end{figure}


\subsection{The quiver for a nontrivial BD triple (explicit)}\label{s:h_quiv_expl}
As an alternative to the algorithm described in the previous subsection, we provide explicit neigh-
borhoods of the variables $h_{jj}$ ($2 \leq j \leq n$), $h_{in}$ ($2\leq i\leq n$), $\varphi_{n-1,1}$ in the case of a nontrivial BD triple. All the other neighborhoods of $Q_h(\bg)$ are the same as in $Q_h(\bg_{\std})$. Let us mention that if $\Delta$ is a nontrivial $X$-run such that $|\Delta| = 2$, then both conditions $\gamma|_{\Delta \cap \Gamma_1} > 0$ and $\gamma|_{\Delta \cap \Gamma_2} < 0$ are considered to be satisfied; however, if $|\Delta|> 2$, then the conditions are mutually exclusive.

 \vspace{5mm}
 \noindent
\begin{figure}[htb]
 \begin{subfigure}[t]{2.2in}
 \begin{center}
 \includegraphics[scale=0.2]{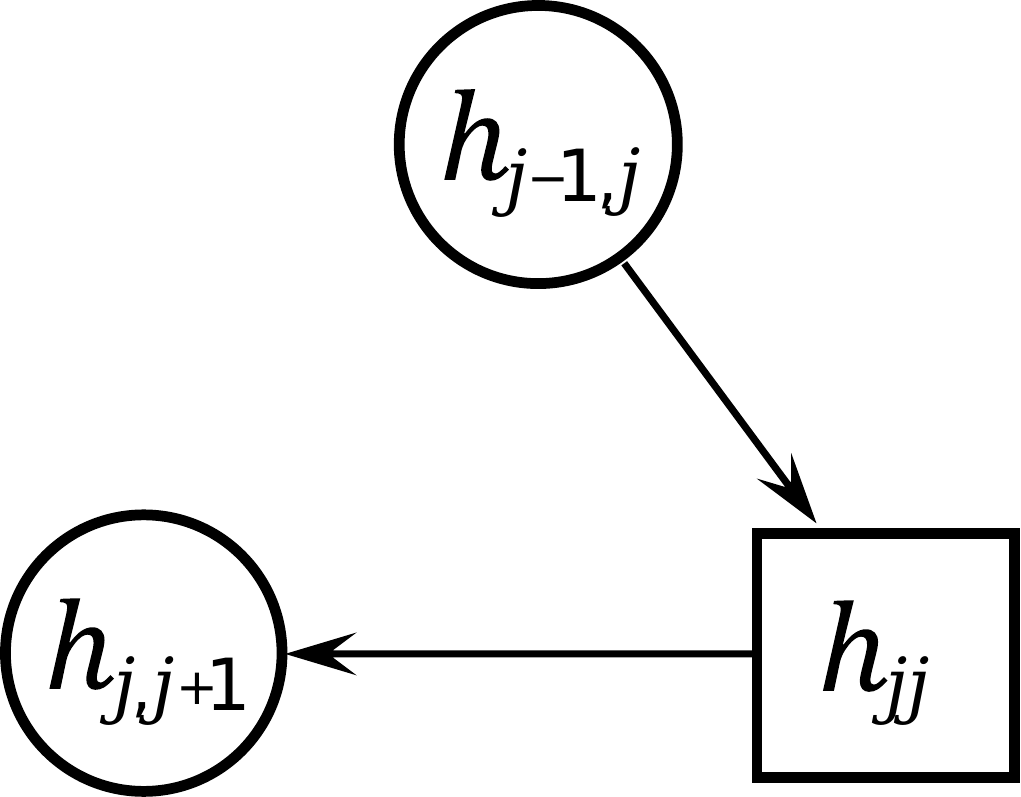}
 \end{center}
 \subcaption{Case $j-1 \notin \Gamma_2$.}
 \label{f:inbd_hjj_0}
 \end{subfigure}
 \begin{subfigure}[t]{3.6in}
 \begin{center}
 \includegraphics[scale=0.2]{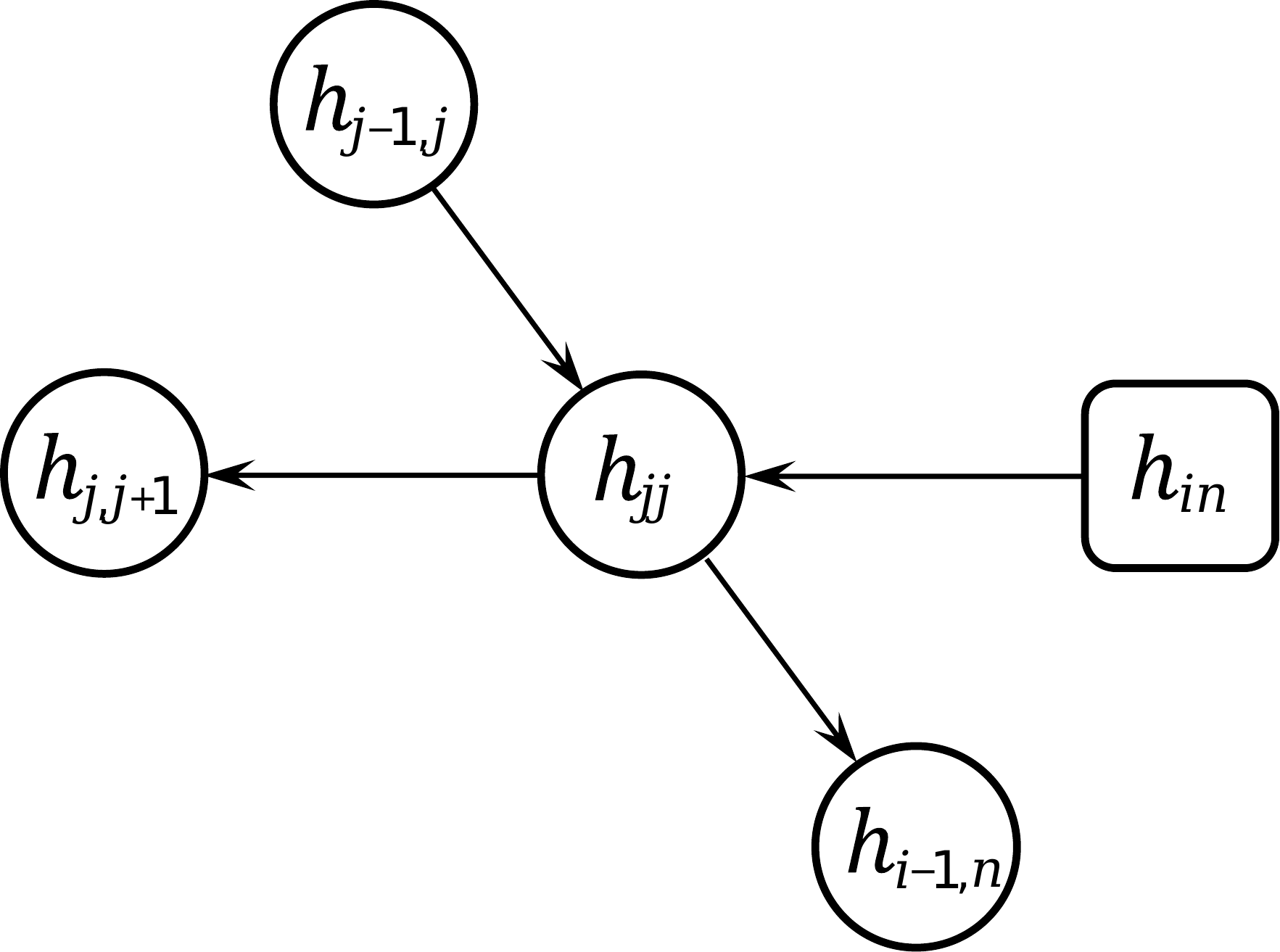}
 \end{center}
 \subcaption{Case $j-2,j-1\in\Gamma_2$ and $\gamma|_{\Delta(i)\cap \Gamma_1}>0$.}
 \label{f:inbd_hjj_1}
 \end{subfigure}
 \begin{subfigure}[t]{3in}
\vspace{4mm}
 \begin{center}
 \includegraphics[scale=0.2]{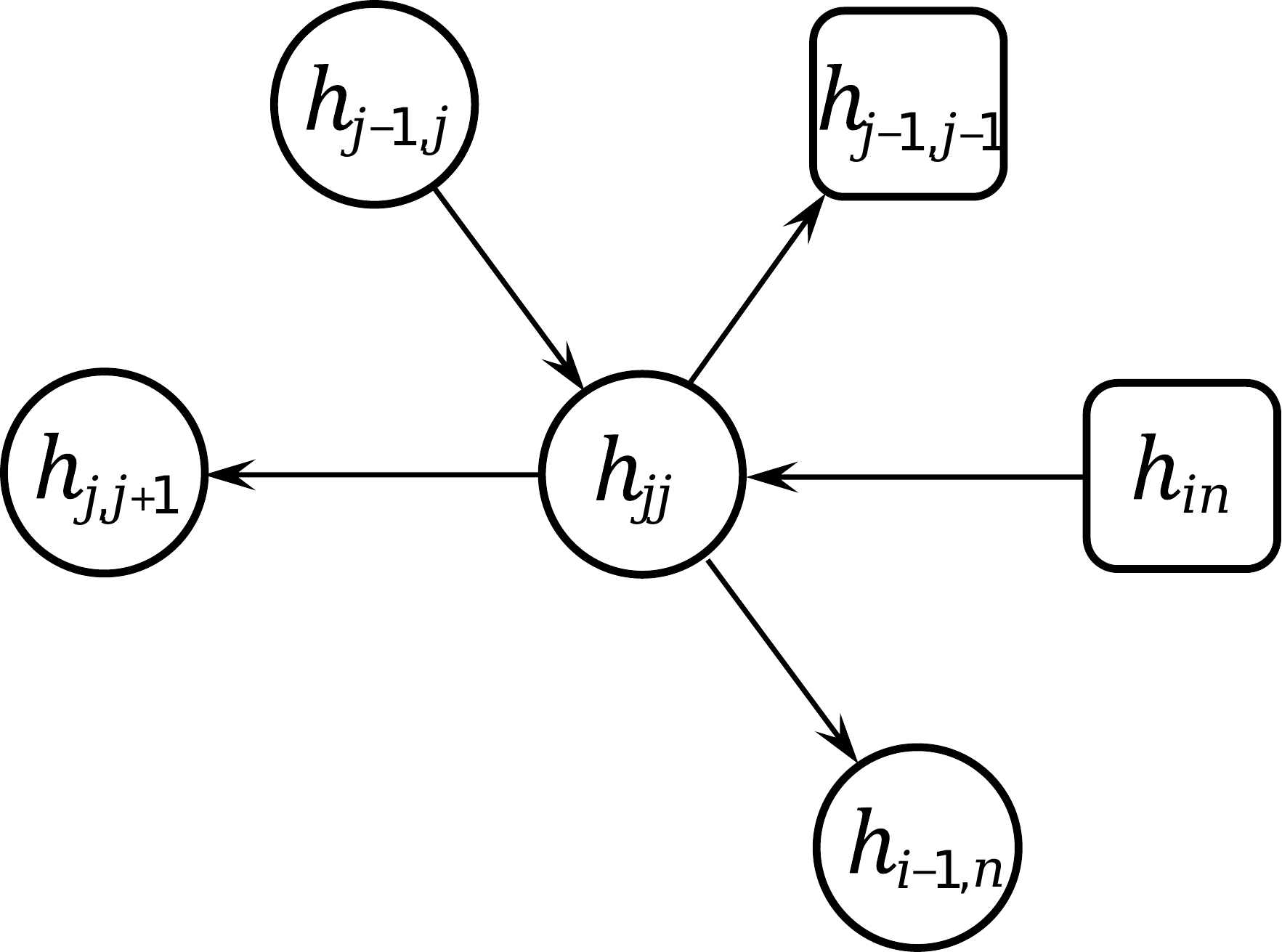}
 \end{center}
 \subcaption{Case $j-1 \in \Gamma_2$, $j \notin \Gamma_2$ and ($j-2\notin\Gamma_2$ or $\gamma|_{\Delta(i)\cap\Gamma_1}<0$).}
 \label{f:inbd_hjj_2}
 \end{subfigure}
 \begin{subfigure}[t]{3in}
\vspace{4mm}
 \begin{center}
 \includegraphics[scale=0.2]{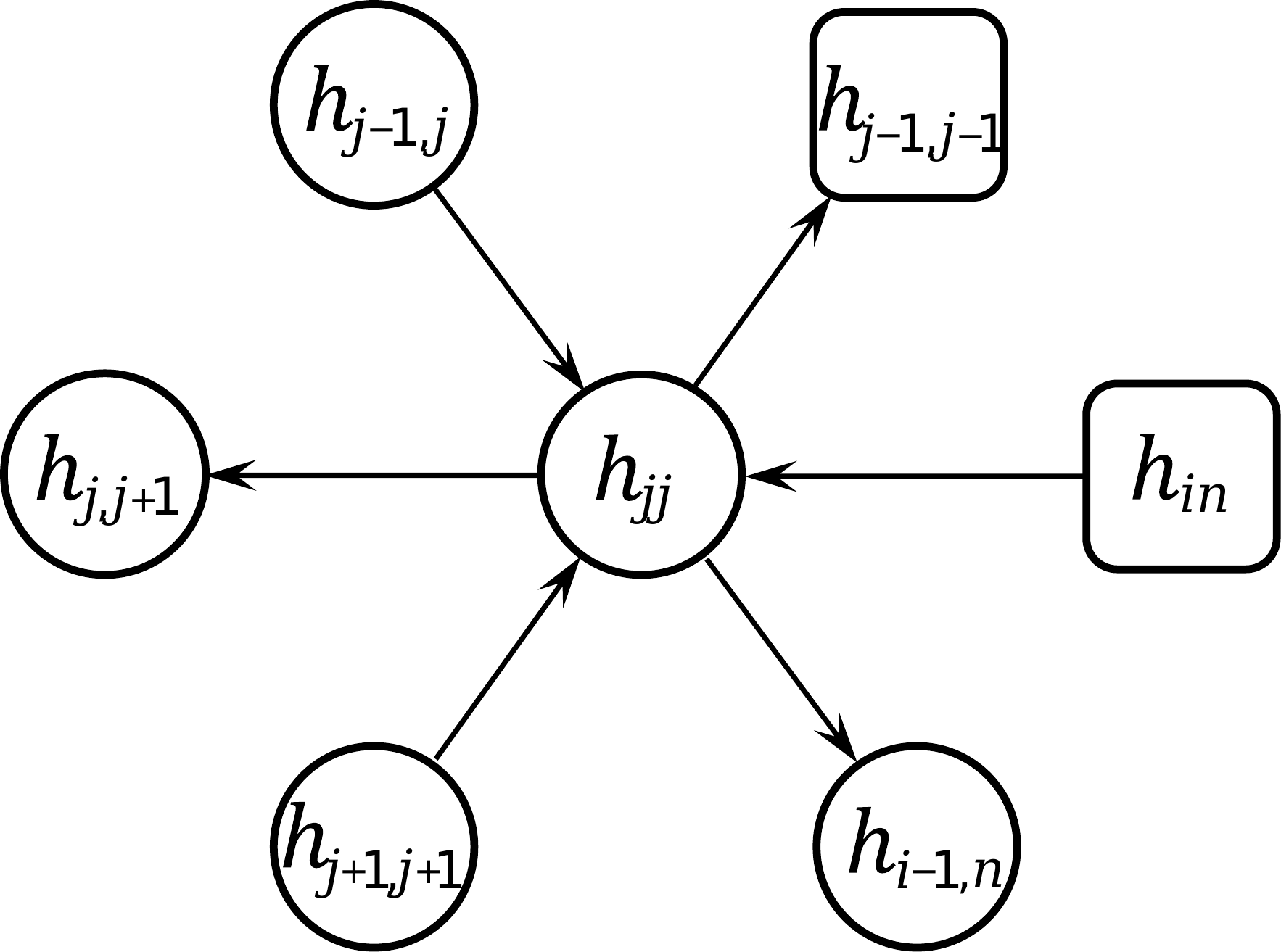}
 \end{center}
 \subcaption{Case $j-1,j \in \Gamma_2$ and $\gamma|_{\Delta(i)\cap \Gamma_1}<0$.}
 \label{f:inbd_hjj_3}
 \end{subfigure}
 \caption{The neighborhood of $h_{jj}$ for $2 \leq j \leq n-1$; here, $i:=\gamma^*(j-1)+1$ when $j-1 \in \Gamma_2$.}
 \label{f:inbd_hjj}
 \end{figure}

\vspace{5mm}
\noindent
\begin{figure}[htb]
\begin{center}
\begin{subfigure}[t]{2.6in}
\begin{center}
\includegraphics[scale=0.2]{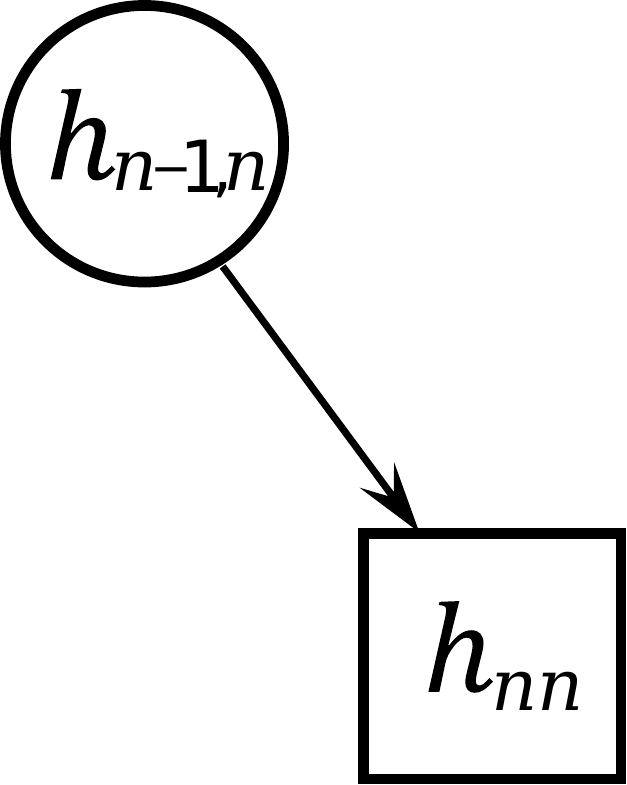} 
\end{center}
\subcaption{Case $n-1 \notin \Gamma_1$, $n-1 \notin \Gamma_2$}
\label{f:inbd_hnn_0}
\end{subfigure}
\begin{subfigure}[t]{2.6in}
\begin{center}
\includegraphics[scale=0.2]{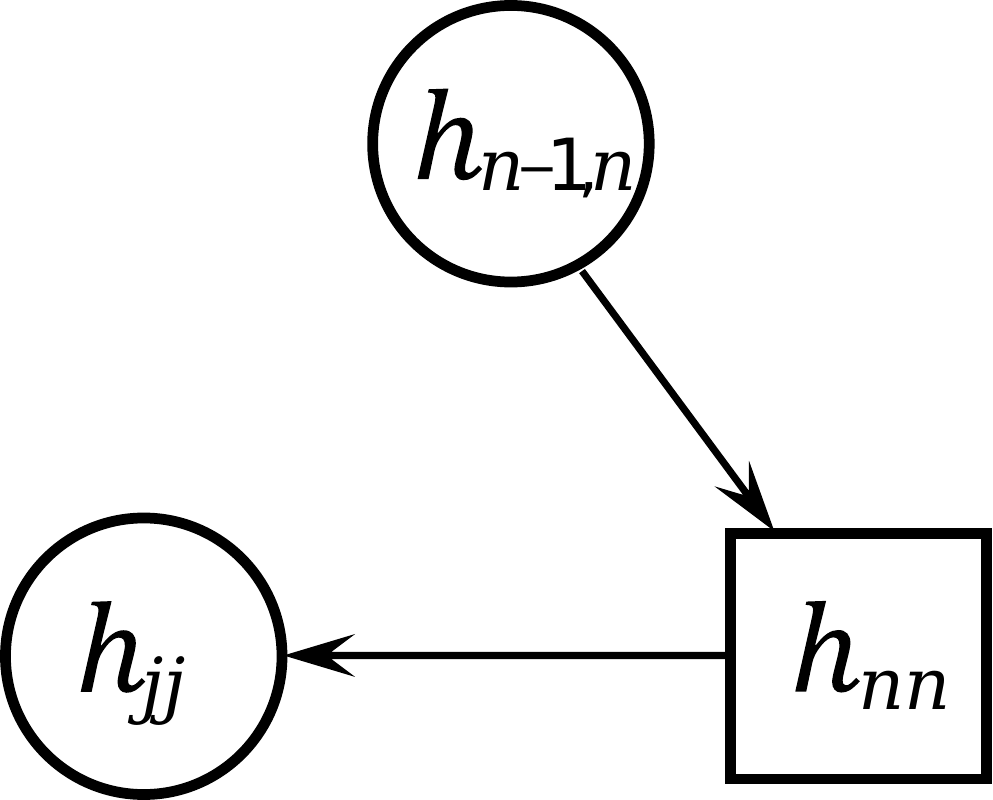}
\end{center}
\subcaption{Case $n-1 \in \Gamma_1$, $n-1 \notin \Gamma_2$.}
\label{f:inbd_hnn_1}
\end{subfigure}
\begin{subfigure}[t]{2.6in}
\vspace{4mm}
\begin{center}
\includegraphics[scale=0.2]{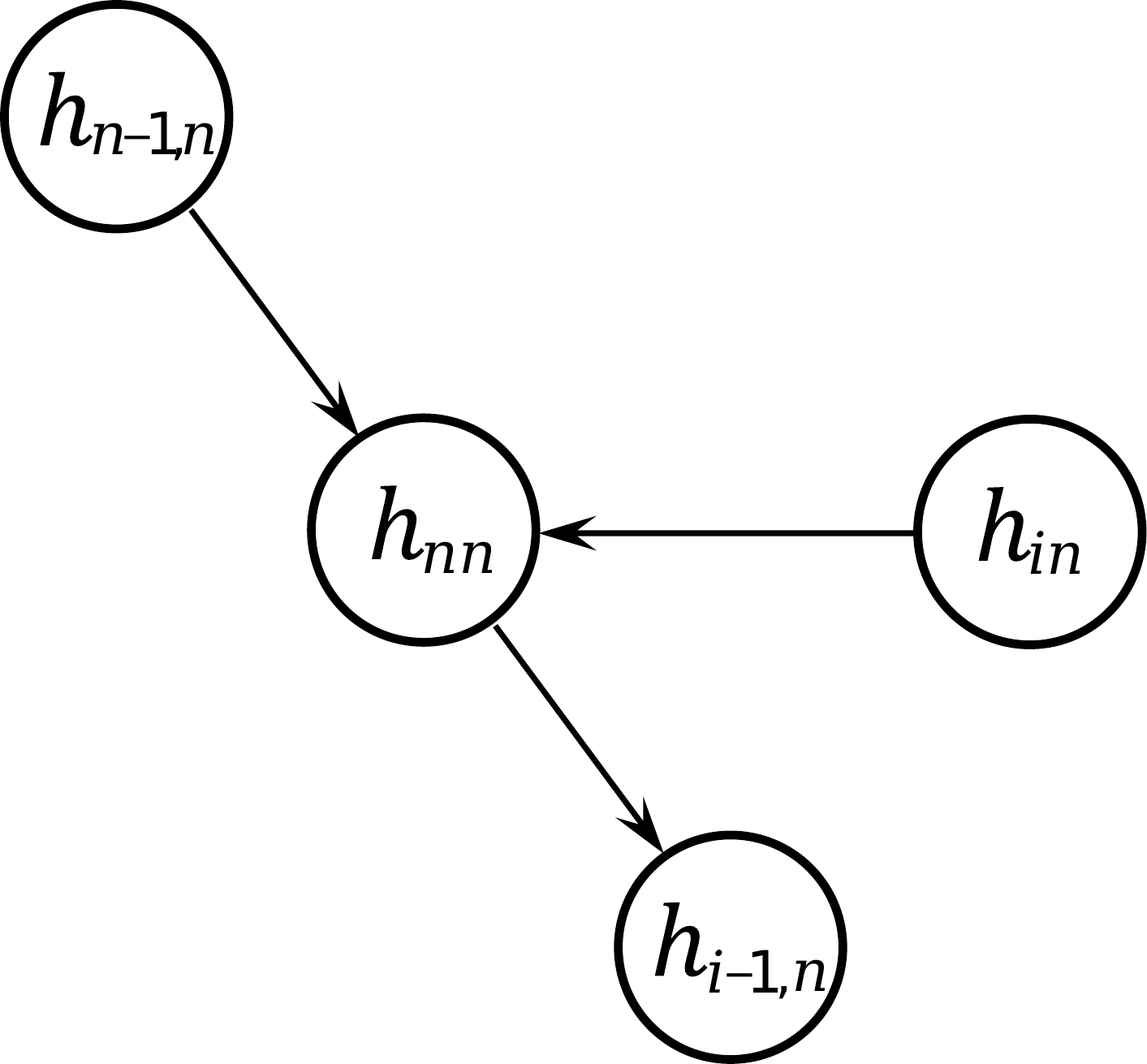}
\end{center}
\subcaption{Case $n-1 \notin \Gamma_1$, $n-1,n-2 \in \Gamma_2$ and $\gamma|_{\Delta(i) \cap \Gamma_1>0}$.}
\label{f:inbd_hnn_2}
\end{subfigure}
\hspace{7mm}
\begin{subfigure}[t]{2.6in}
\vspace{4mm}
\begin{center}
\includegraphics[scale=0.2]{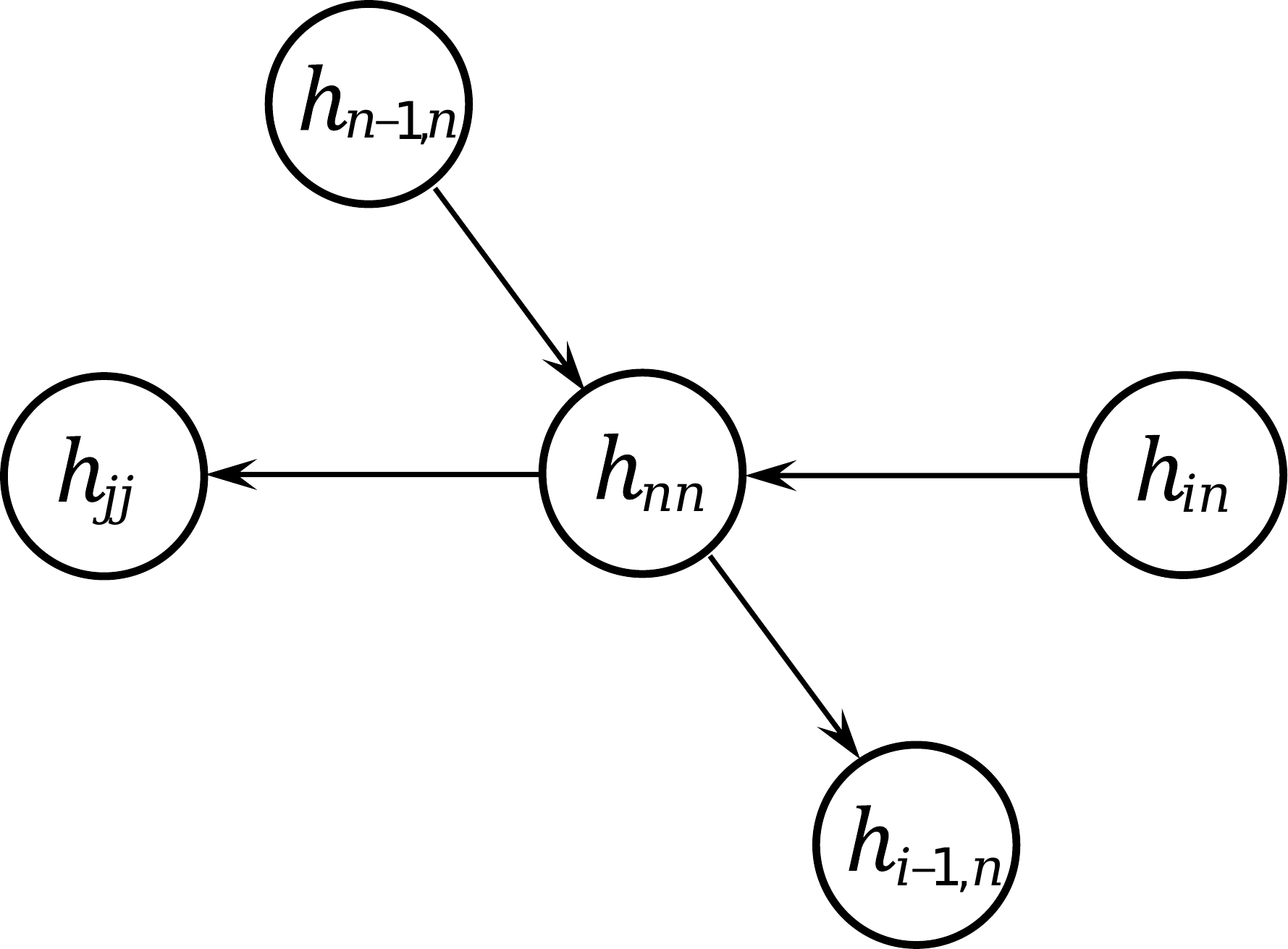}
\end{center}
\subcaption{Case $n-1 \in \Gamma_1$, $n-1,n-2 \in \Gamma_2$ and $\gamma|_{\Delta(i) \cap \Gamma_1>0}$.}
\label{f:inbd_hnn_3}
\end{subfigure}
\begin{subfigure}[t]{2.6in}
\vspace{4mm}
\begin{center}
\includegraphics[scale=0.2]{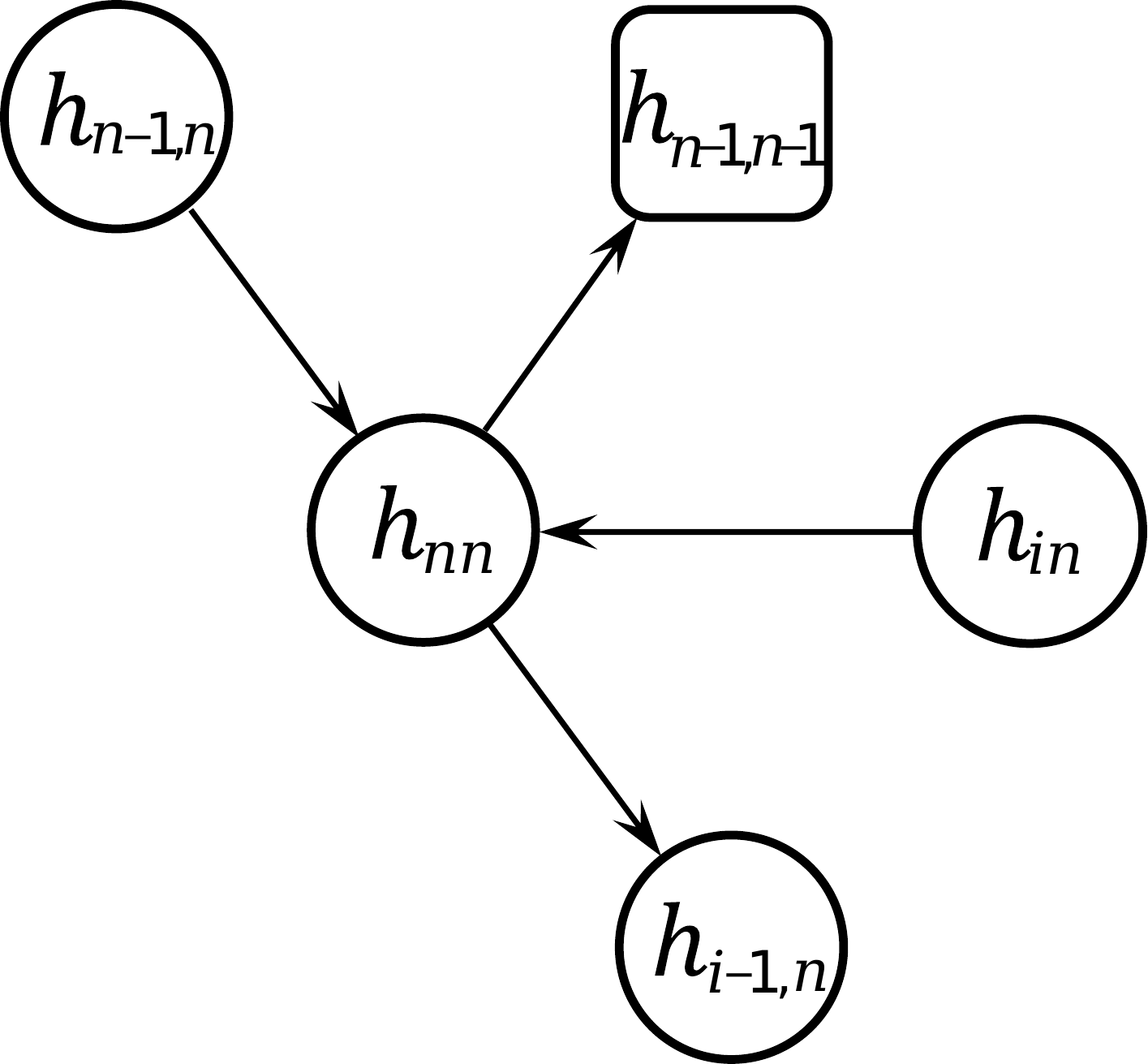}
\end{center}
\subcaption{Case $n-1 \notin \Gamma_1$, $n-1 \in \Gamma_2$ and ($n-2 \notin \Gamma_2$ or $\gamma|_{\Delta(i)\cap \Gamma_1}<0$).}
\label{f:inbd_hnn_4}
\end{subfigure}
\hspace{7mm}
\begin{subfigure}[t]{2.6in}
\vspace{4mm}
\begin{center}
\includegraphics[scale=0.2]{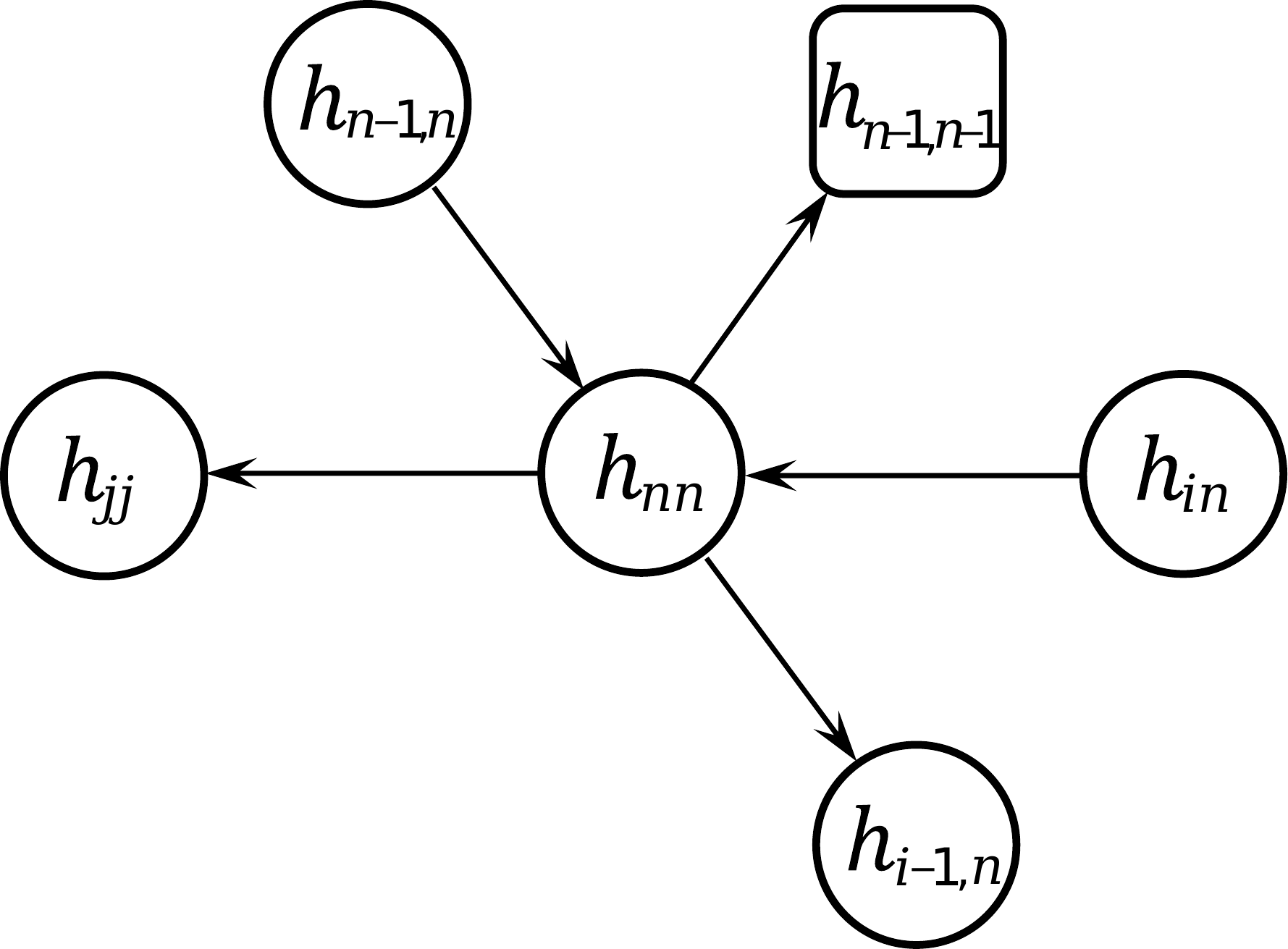}
\end{center}
\subcaption{Case $n-1 \in \Gamma_1$, $n-1 \in \Gamma_2$ and ($n-2 \notin \Gamma_2$ or $\gamma|_{\Delta(i)\cap \Gamma_1}<0$).}
\label{f:inbd_hnn_5}
\end{subfigure}
\caption{The neighborhood of $h_{nn}$. In the figures, we set $j:=\gamma(n-1)+1$ if $n-1 \in \Gamma_1$ and $i:=\gamma^*(n-1)+1$ if $n-1 \in \Gamma_2$.}
\label{f:inbd_hnn}
\end{center}
\end{figure}

 \vspace{5mm}
 \noindent
\begin{figure}[htb]
 \begin{subfigure}[t]{2.2in}
 \begin{center}
 \includegraphics[scale=0.2]{nbds_h/inbd_hin}
 \end{center}
 \subcaption{Case $i-1,i \notin \Gamma_1$.}
 \label{f:inbd_hin_0}
 \end{subfigure}
 \begin{subfigure}[t]{3.6in}
 \begin{center}
 \includegraphics[scale=0.2]{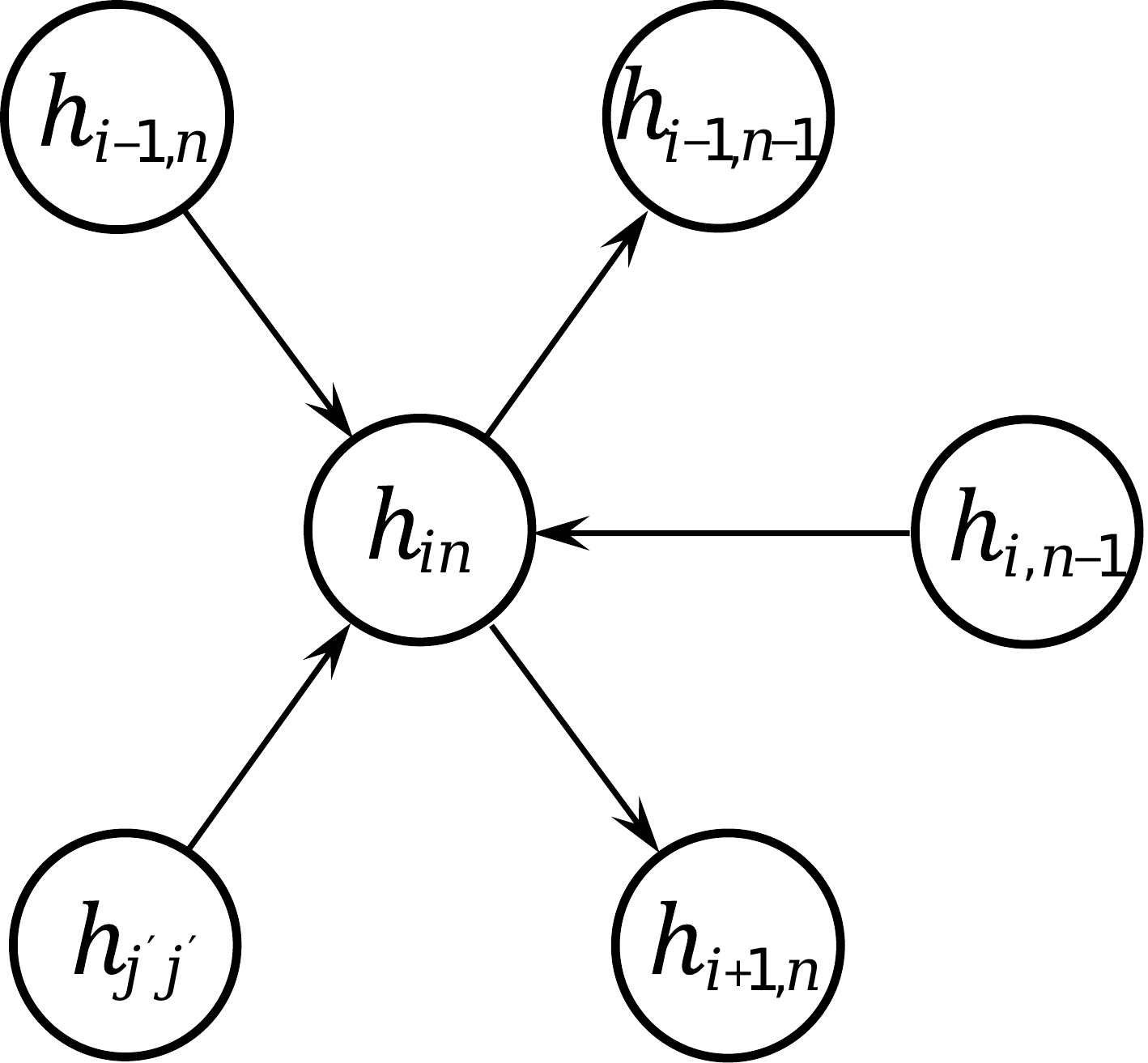}
 \end{center}
 \subcaption{Case $i-1 \notin \in \Gamma_1$, $i \in \Gamma_1$.}
 \label{f:inbd_hin_1}
 \end{subfigure}
 \begin{subfigure}[t]{3in}
\vspace{4mm}
 \begin{center}
 \includegraphics[scale=0.2]{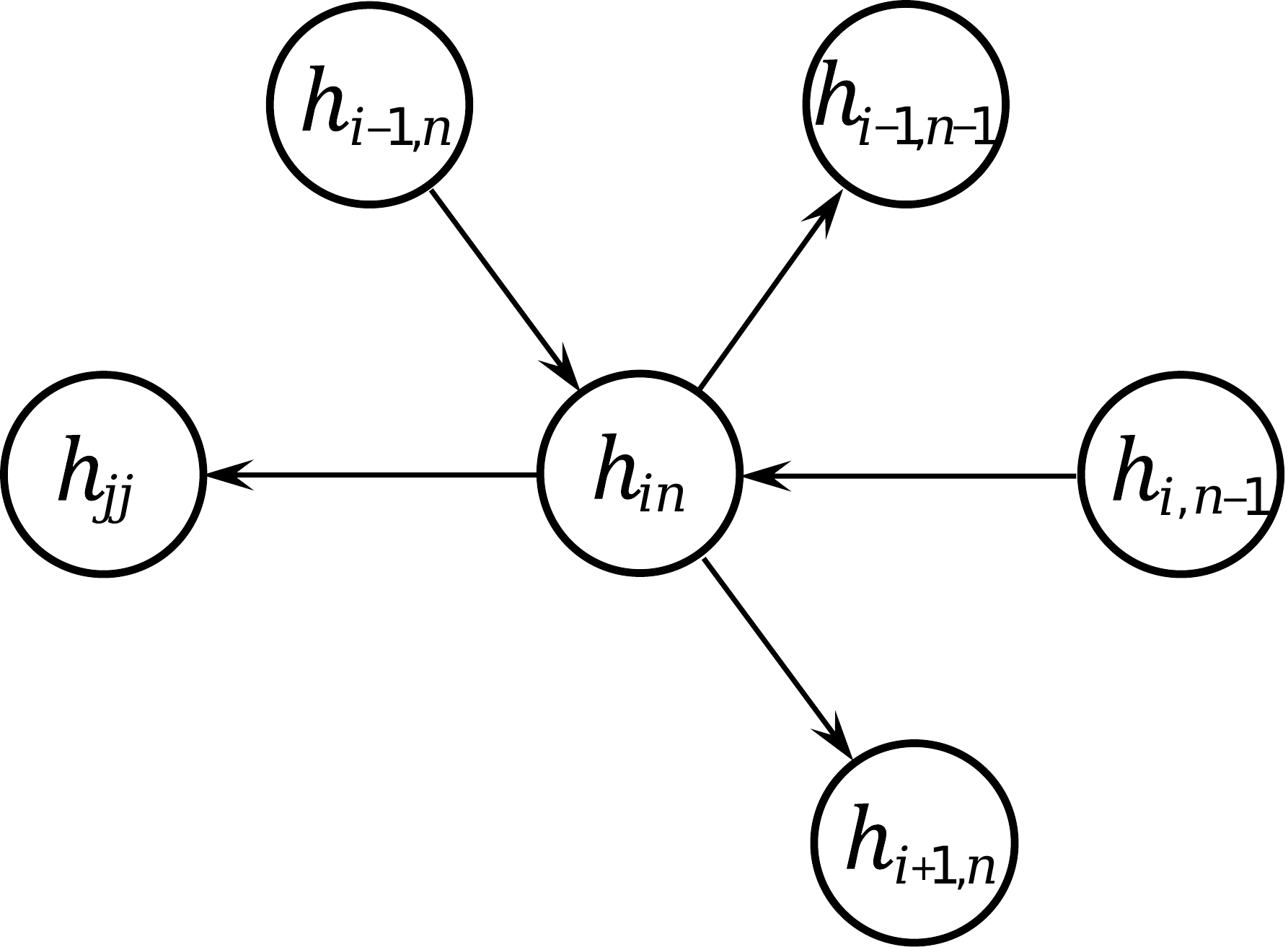}
 \end{center}
 \subcaption{Case $i-1\in\Gamma_1$, $i \notin \Gamma_1$.}
 \label{f:inbd_hin_2}
 \end{subfigure}
 \begin{subfigure}[t]{3in}
\vspace{4mm}
 \begin{center}
 \includegraphics[scale=0.2]{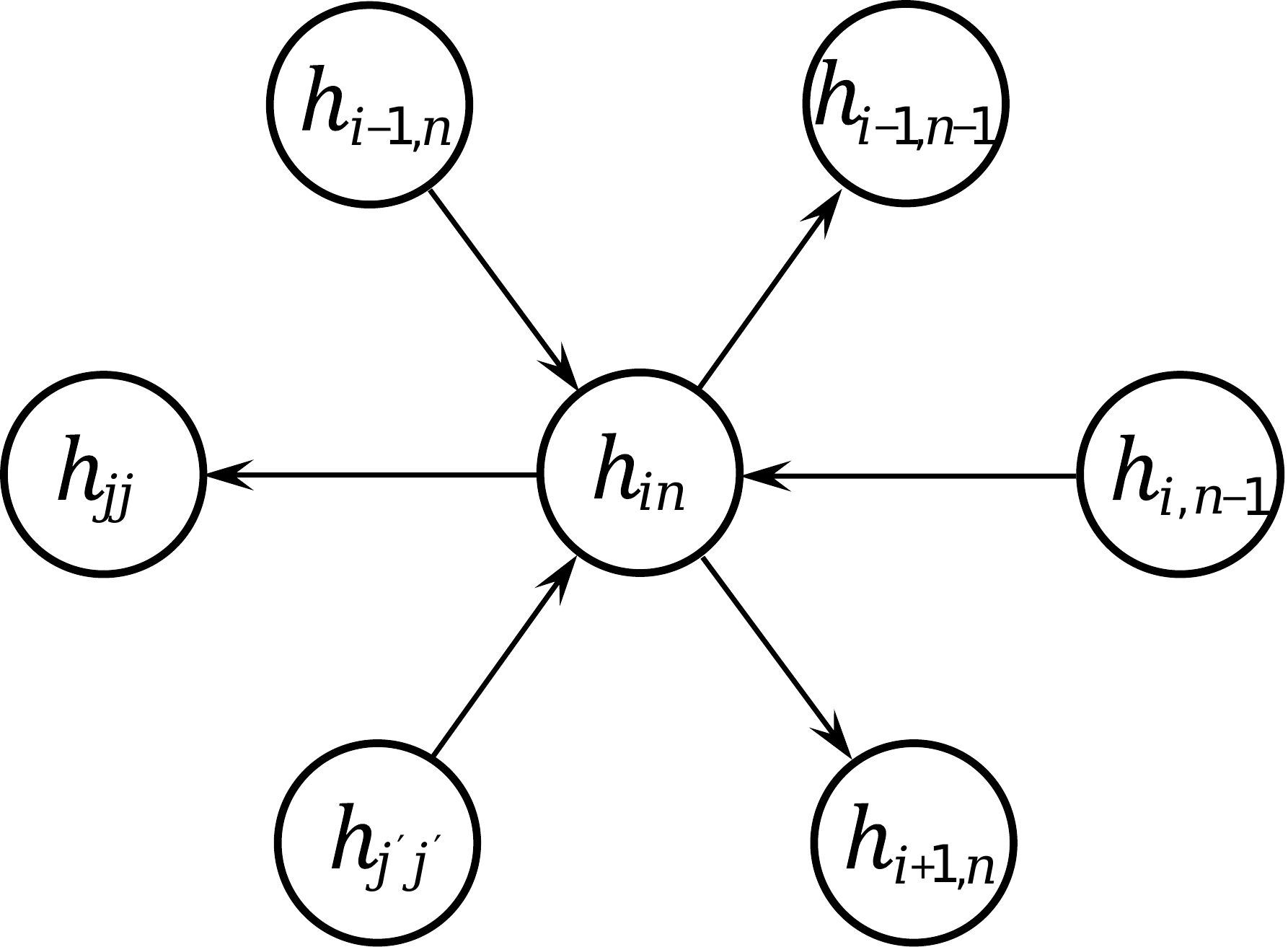}
 \end{center}
 \subcaption{Case $i-1,i\in\Gamma_1$.}
 \label{f:inbd_hin_3}
 \end{subfigure}
 \caption{The neighborhood of $h_{in}$ for $2 \leq i \leq n-1$; here, $j:=\gamma(i-1)+1$ if $i-1 \in \Gamma_1$ and $j^\prime:=\gamma(i)+1$ if $i \in \Gamma_1$.}
 \label{f:inbd_hin}
 \end{figure}

 \begin{figure}[htb]
 \begin{subfigure}[t]{3in}
 \begin{center}
 \includegraphics[scale=0.2]{nbds_h/inbd_phin11}
 \end{center}
 \subcaption{Case $1 \notin \Gamma_1$.}
 \label{f:inbd_phin11_0}
\end{subfigure}
 \begin{subfigure}[t]{3in}
 \begin{center}
 \includegraphics[scale=0.2]{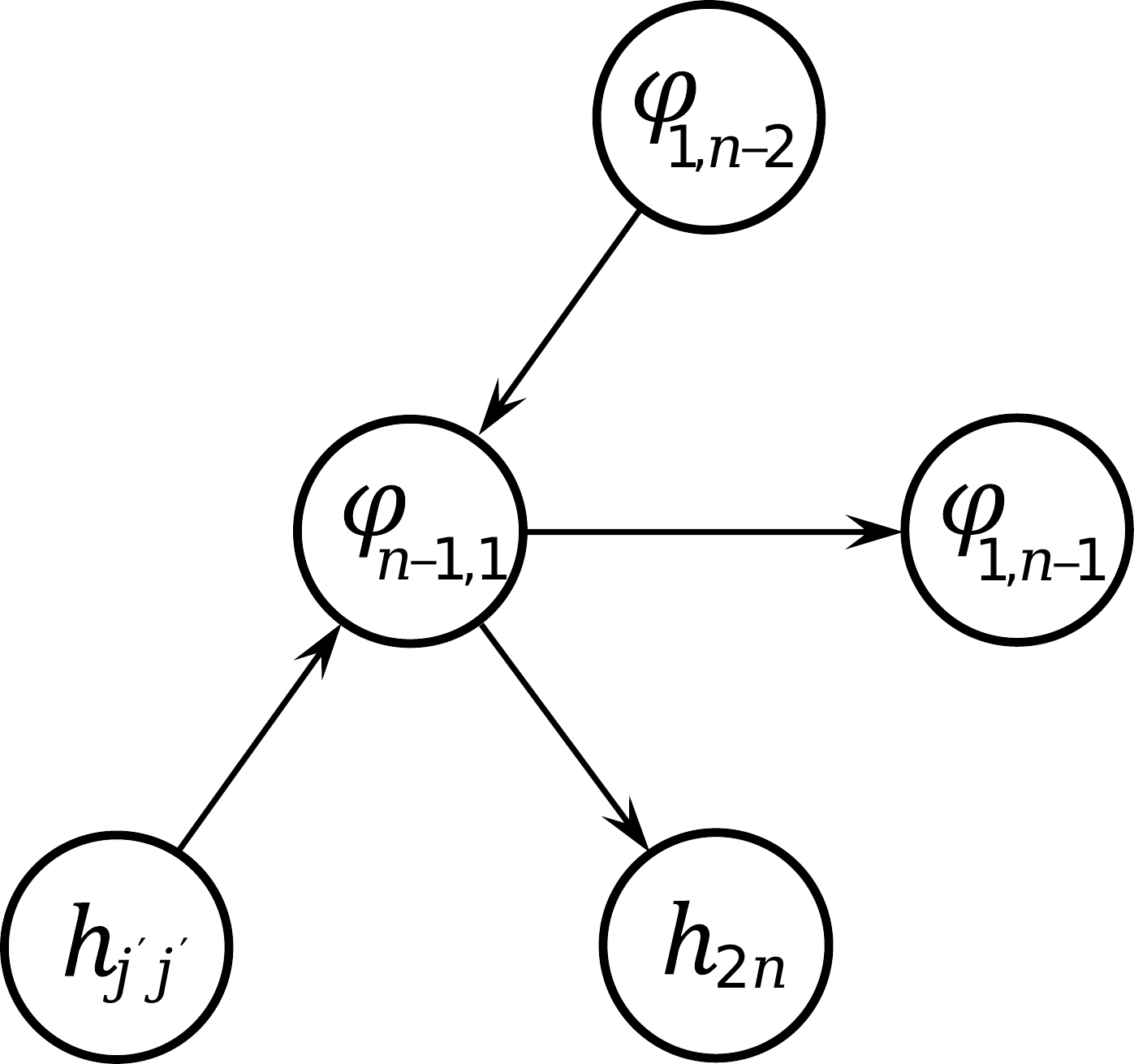}
 \end{center}
 \subcaption{Case $1 \in \Gamma_1$.}
 \label{f:inbd_phin11_1}
 \end{subfigure}
 \caption{The neighborhood of $\varphi_{n-1,1}$. Here, $j^\prime :=\gamma(1)+1$ if $1 \in \Gamma_1$.}
 \label{f:inbd_phin11}
 \end{figure}

\clearpage

 \noindent
 \begin{figure}[htb]
 \begin{subfigure}[t]{3in}
 \begin{center}
 \includegraphics[scale=0.2]{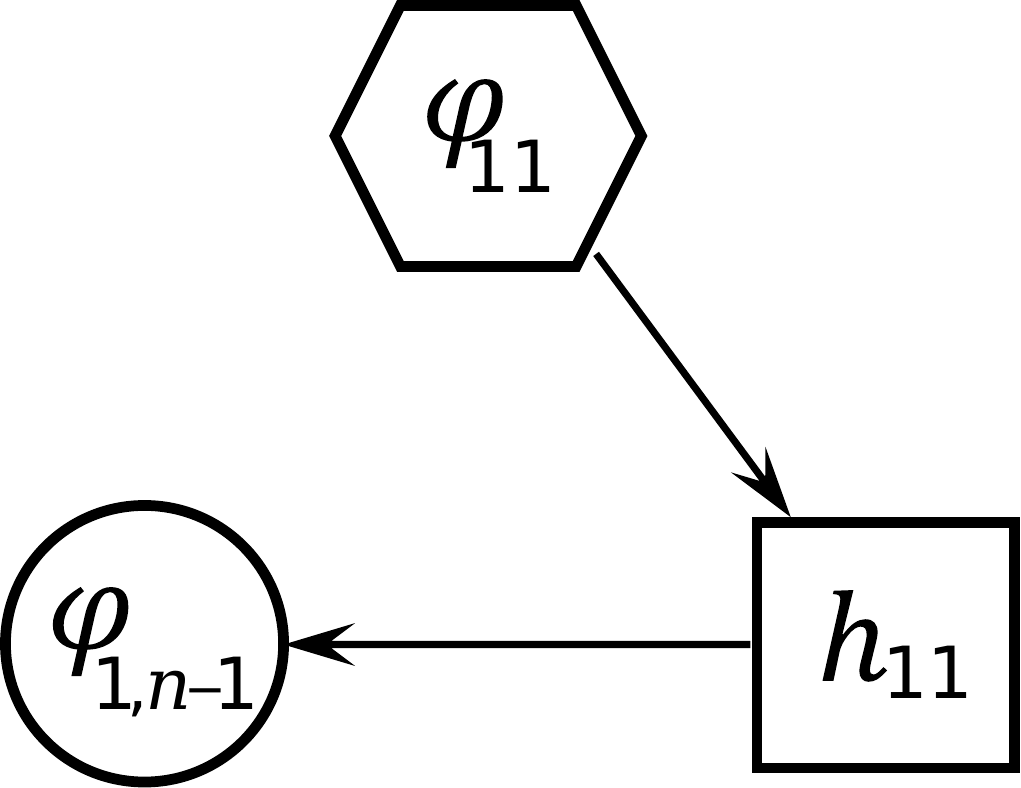} 
 \end{center}
 \subcaption{Case $1 \notin \Gamma_2$}
 \label{f:inbd_h11_0}
 \end{subfigure}
 \begin{subfigure}[t]{3in}
 \begin{center}
 \includegraphics[scale=0.2]{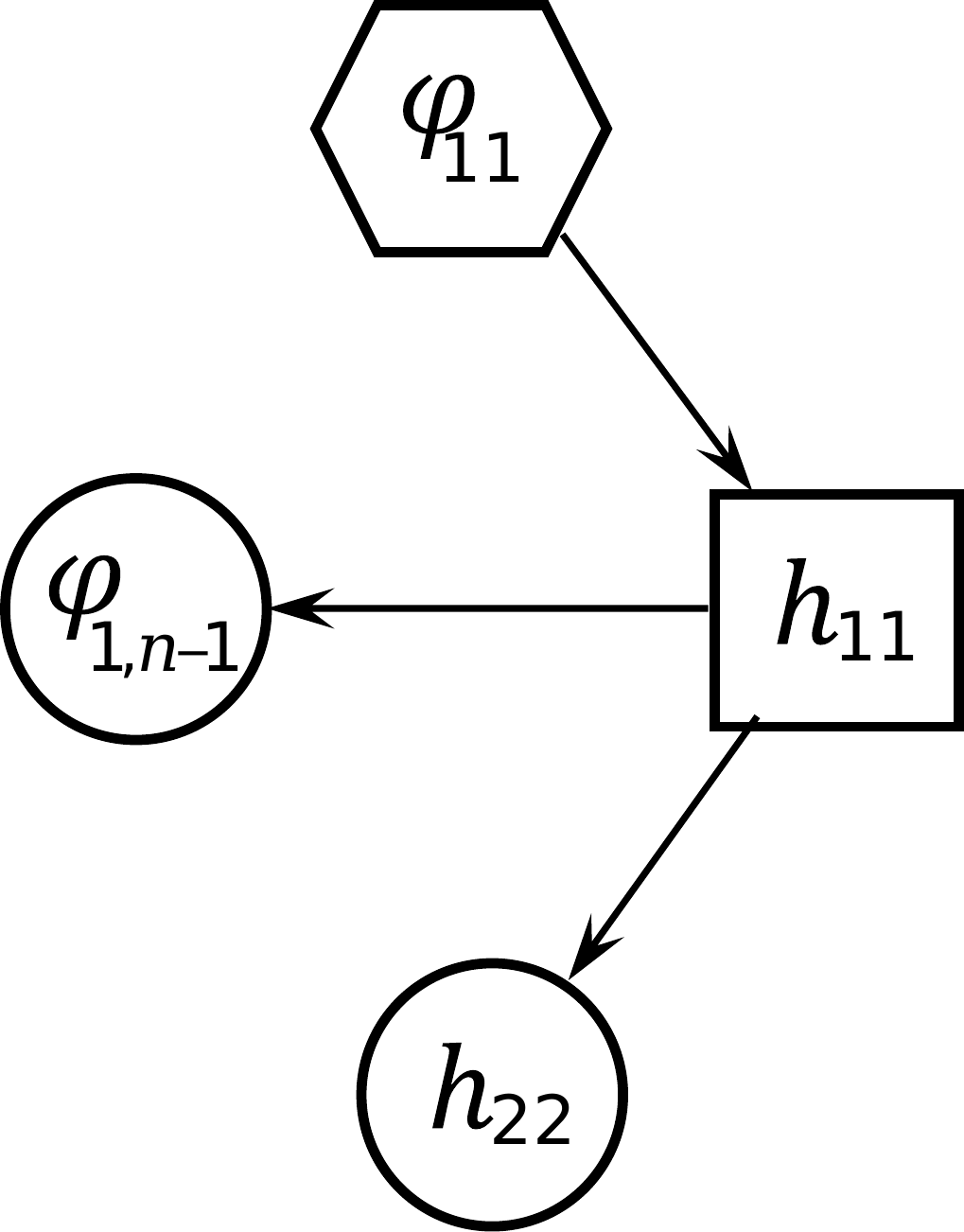}
 \end{center}
 \subcaption{Case $1 \in \Gamma_2$}
 \label{f:inbd_h11_1}
 \end{subfigure}
 \caption{The neighborhood of $h_{11}$.}
 \label{f:inbd_h11}
 \end{figure}

\section{\texorpdfstring{Initial quiver for $\gc_g^{\dagger}(\bg)$}{Initial quiver for the g-convention}}\label{s:ainiquivg}
In this appendix, we describe the initial quiver $Q_g(\bg)$ for $\gc^\dagger_g(\bg)$. As in the $h$-convention, we start with describing the quiver $Q_g(\bg_{\std})$ in Section~\ref{s:quiv_g_triv}, which is exemplified in Figure~\ref{f:ex_n=5_gstd} for $n=5$. In Section~\ref{s:quiv_g_alg}, we explain how to obtain the quiver $Q_g(\bg)$ from $Q_g(\bg_{\std})$ for a nontrivial BD triple $\bg$. The procedure of obtaining $Q_g(\bg)$ from $Q_{g}(\bg_{\std})$ consists in adding extra arrows to $Q_{g}(\bg_{\std})$ based on $\bg$, as well as in unfreezing some of the vertices. Alternatively, in Section~\ref{s:g_quiv_expl} we provide explicit neighborhoods of the vertices of $Q_g(\bg)$ that are different from the corresponding neighborhoods in $Q_g(\bg_{\std})$. Some explicit examples of $Q_g(\bg)$ for nontrivial BD triples are given in the supplementary note~\cite{github}. Throughout the section, we assume that $n \geq 3$.

Let us mention that the quiver $Q_g(\bg)$ can be obtained from $Q_h(\bg^{\op})$ via the following steps:
\begin{itemize}
\item Replace each vertex $\varphi_{kl}$ with $\phi_{kl}$, $2\leq k+l \leq n$, $k,l \geq 1$ and each $h_{ji}$ with $g_{ij}$, $2\leq j \leq i \leq n$;
\item For each $g_{ij}$, $2 \leq j \leq i \leq n$, reverse the orientation of the arrows in its neighborhood;
\item For the vertices $\phi_{kl}$ with $k+l = n$ and $k \geq 2$, add an arrow $\phi_{kl}\rightarrow \phi_{k-1,l+1}$;
\item Remove the arrow $\phi_{1,n-1}\rightarrow g_{11}$.
\end{itemize}

\subsection{The quiver for the trivial BD triple}\label{s:quiv_g_triv}
Below one can find pictures of the neighborhoods of all vertices of the initial quiver $Q_g(\bg_{\std})$ of $\gc_g^{\dagger}(\bg_{\std})$. A few remarks beforehand:
\begin{itemize}
\item The circled vertices are mutable (in the sense of ordinary exchange relations~\eqref{eq:ordexchrel}), the square vertices are frozen, the rounded square vertices may or may not be mutable depending on the indices, and the hexagon vertex is a mutable vertex with a generalized mutation relation (see Formula~\eqref{eq:p11mut});
\item Since $c_1,\ldots,c_{n-1}$ are isolated variables, they are not shown on the resulting quiver;
\item Unlike in the $h$-convention, there are always two arrows $\phi_{21}\rightrightarrows\phi_{12}$ in the resulting quiver;
\item When the indices of the $g$-variables are out of range, we use the following conventions:
\begin{align*}
g_{i1}(U) &:= \phi_{i-1,n-i+1}(U),& &i \in [2,n];\\
g_{n+1,j}(U) &:= 1,& &j \in [1,n].
\end{align*}
\end{itemize}

\begin{figure}[htb]
 \begin{center}
 \includegraphics[scale=0.2]{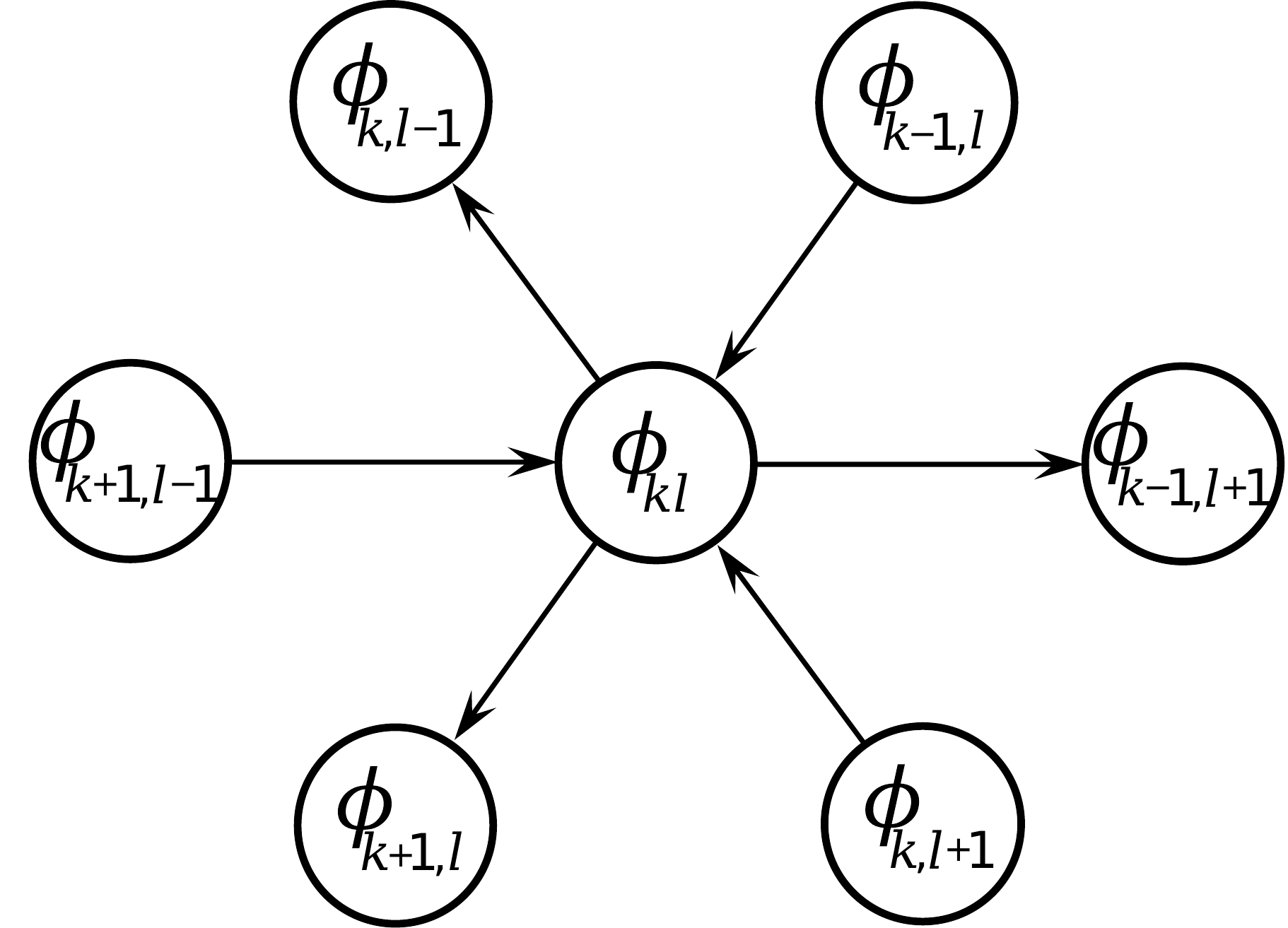}
 \end{center}
 \caption{The neighborhood of $\varphi_{kl}$ for $k,l \neq 1$, $k+l < n$.}
 \label{f:gnbd_phikl}
 \end{figure}

 \vspace{5mm}
 \begin{figure}[htb]
 \begin{subfigure}[t]{3in}
 \begin{center}
 \includegraphics[scale=0.2]{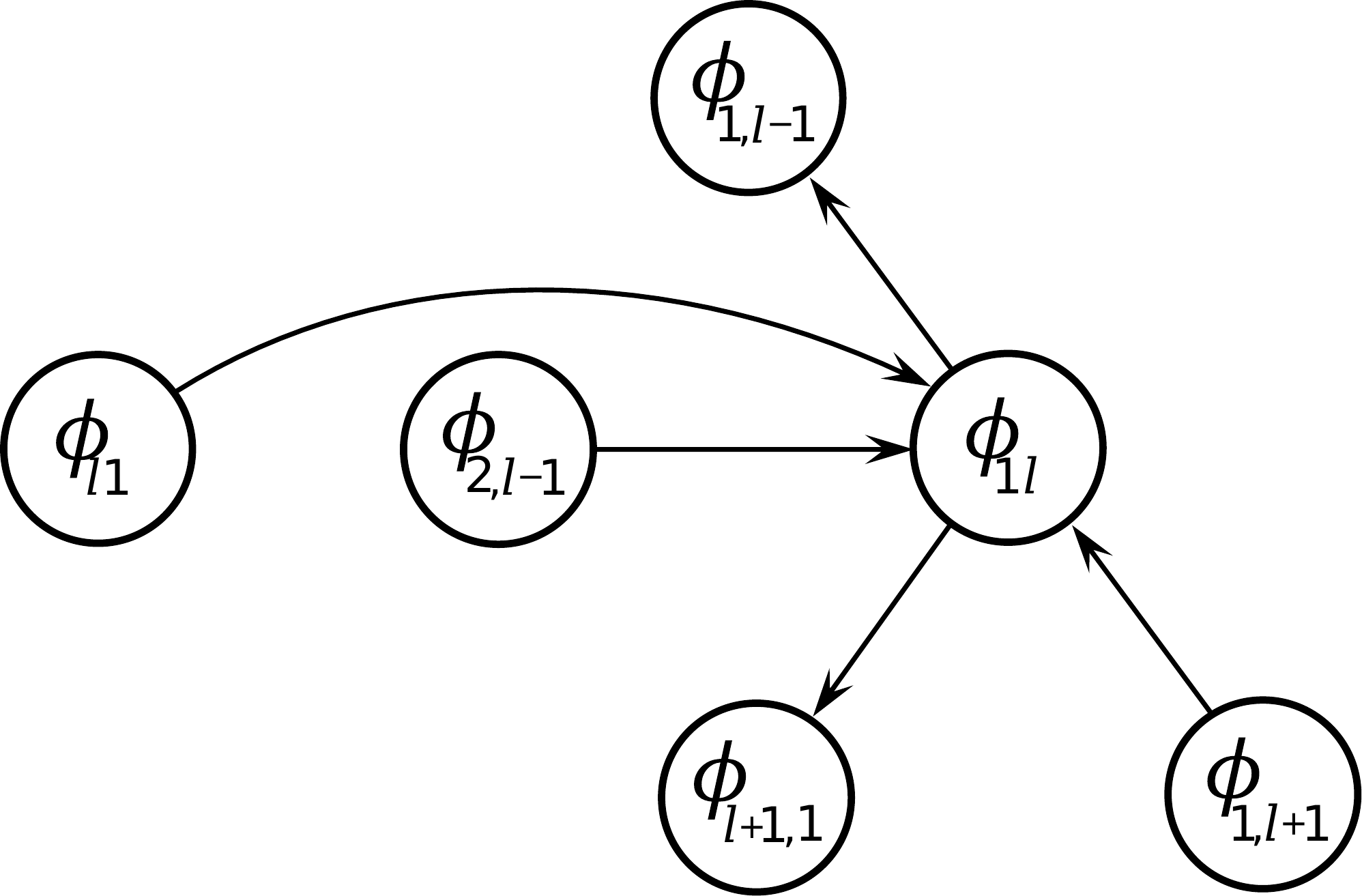}
 \end{center}
 \subcaption{Case $2 \leq l \leq n-2$.}
 \label{f:gnbd_phi1l}
 \end{subfigure}
 \begin{subfigure}[t]{3in}
 \begin{center}
 \includegraphics[scale=0.2]{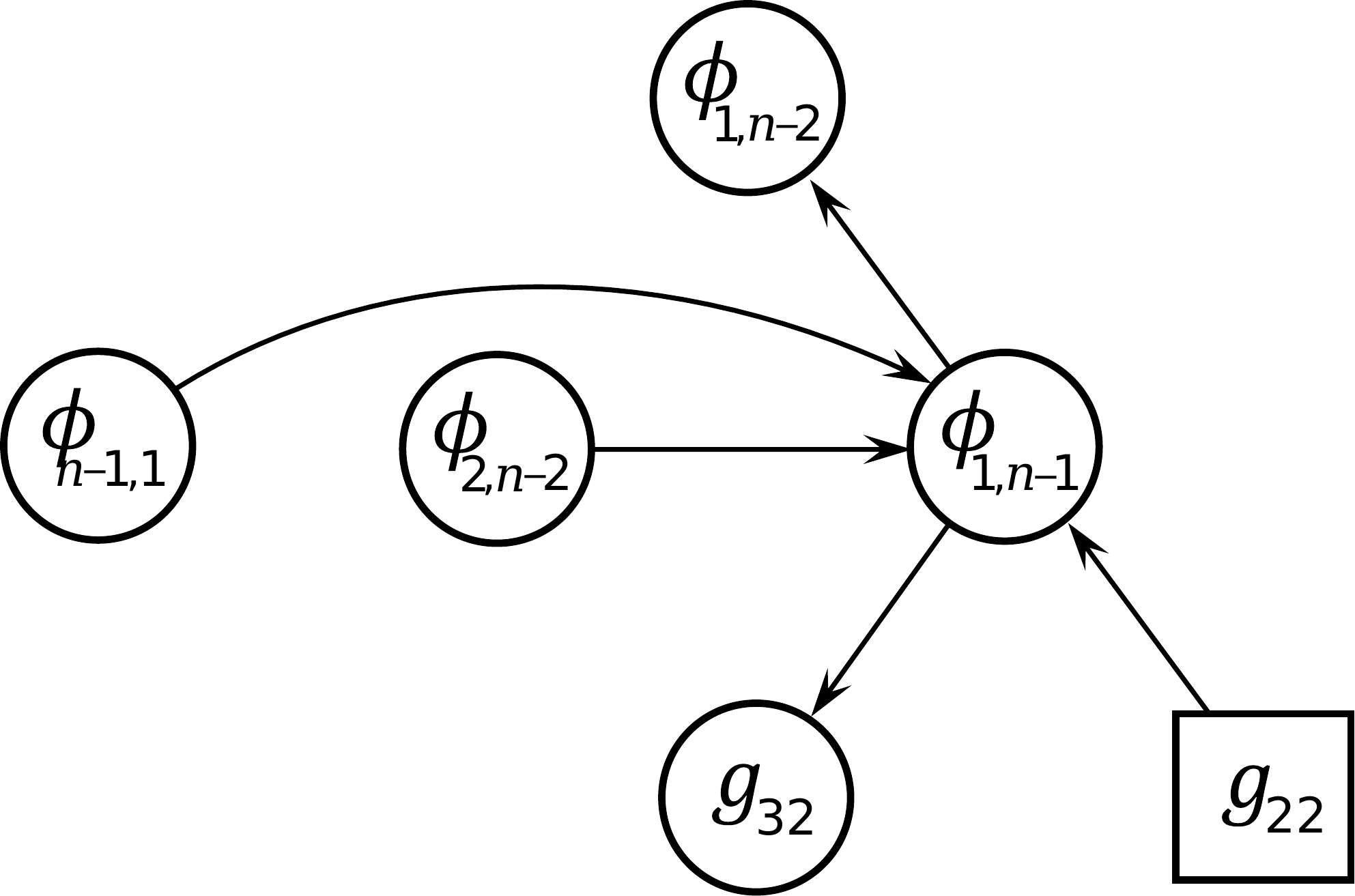}
 \end{center}
 \subcaption{Case $l = n-1$.}
 \label{f:gnbd_phi1n1}
 \end{subfigure}
 \caption{The neighborhood of $\varphi_{1l}$ for $2 \leq l \leq n-1$.}
 \label{f:gnbd_phi1}
 \end{figure}

 \vspace{5mm}
 \begin{figure}[htb]
 \begin{subfigure}[t]{3in}
 \begin{center}
 \includegraphics[scale=0.2]{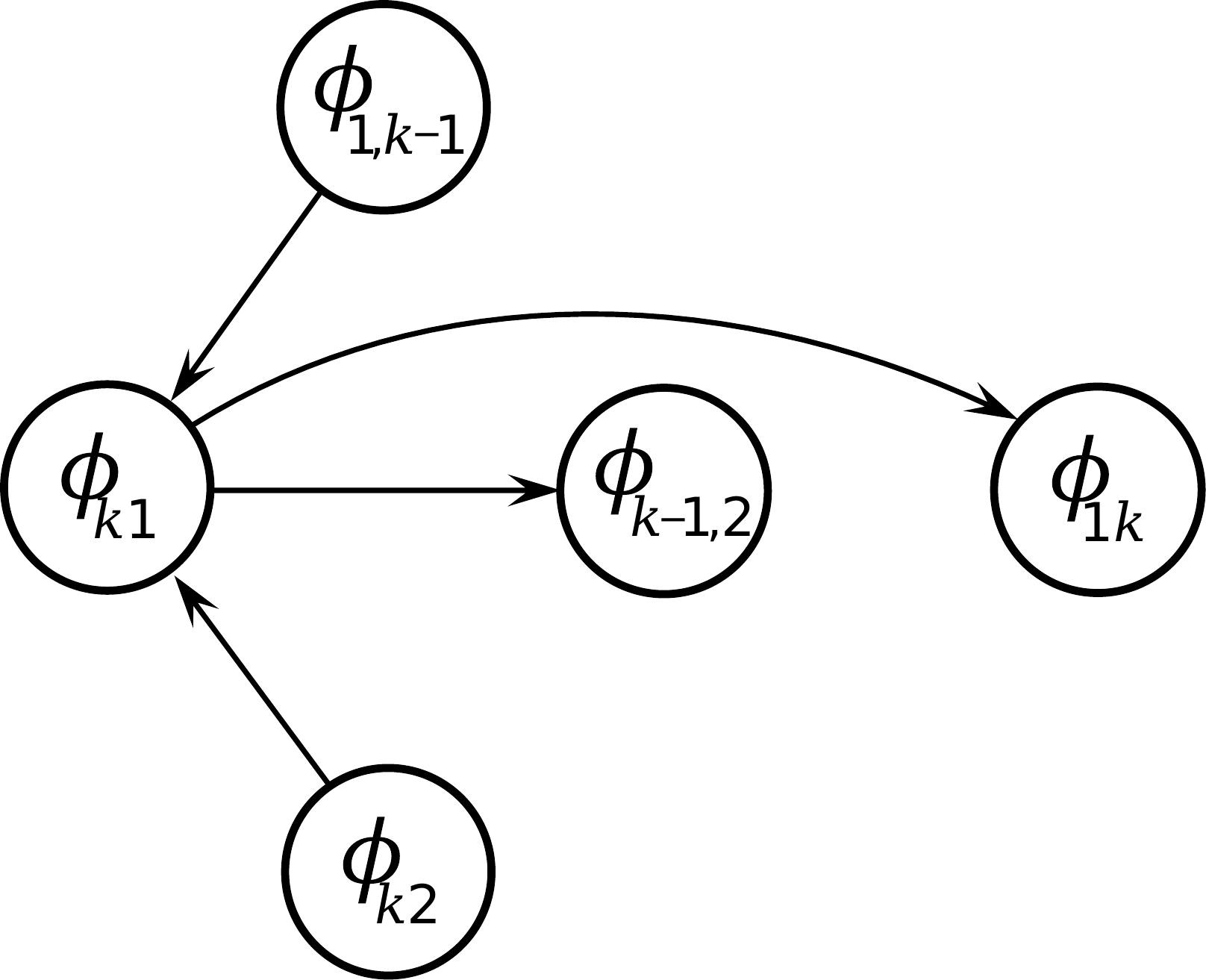}
 \end{center}
 \subcaption{Case $2 \leq k \leq n-2$.}
 \label{f:gnbd_phik1}
 \end{subfigure}
 \begin{subfigure}[t]{3in}
 \begin{center}
 \includegraphics[scale=0.2]{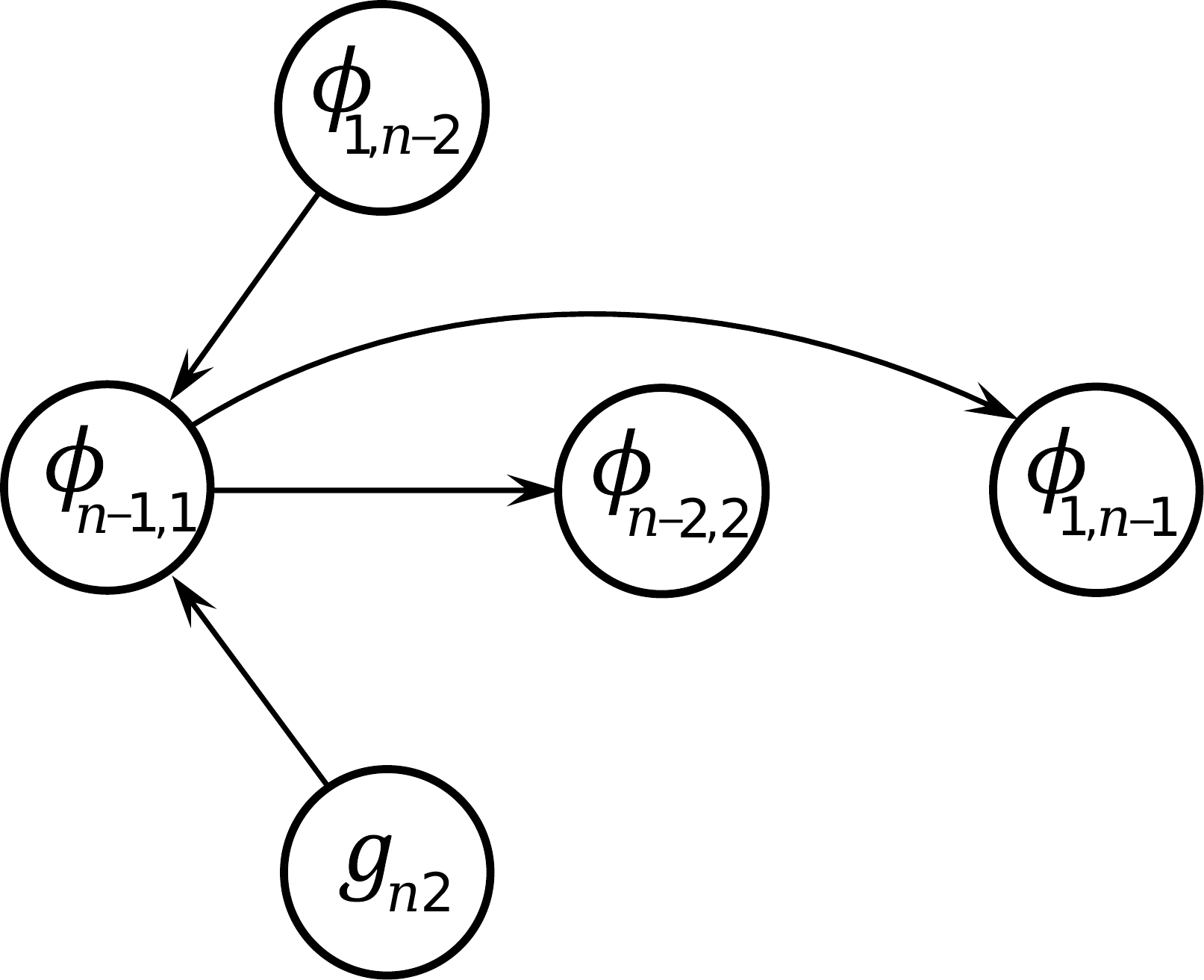}
 \end{center}
 \subcaption{Case $k = n-1$.}
 \label{f:gnbd_phin11}
 \end{subfigure}
 \caption{The neighborhood of $\varphi_{k1}$ for $2 \leq k \leq n-1$.}
 \label{f:gnbd_phik}
 \end{figure}

 \vspace{5mm}
 \begin{figure}[htb]
 \begin{subfigure}[t]{3in}
 \begin{center}
 \includegraphics[scale=0.2]{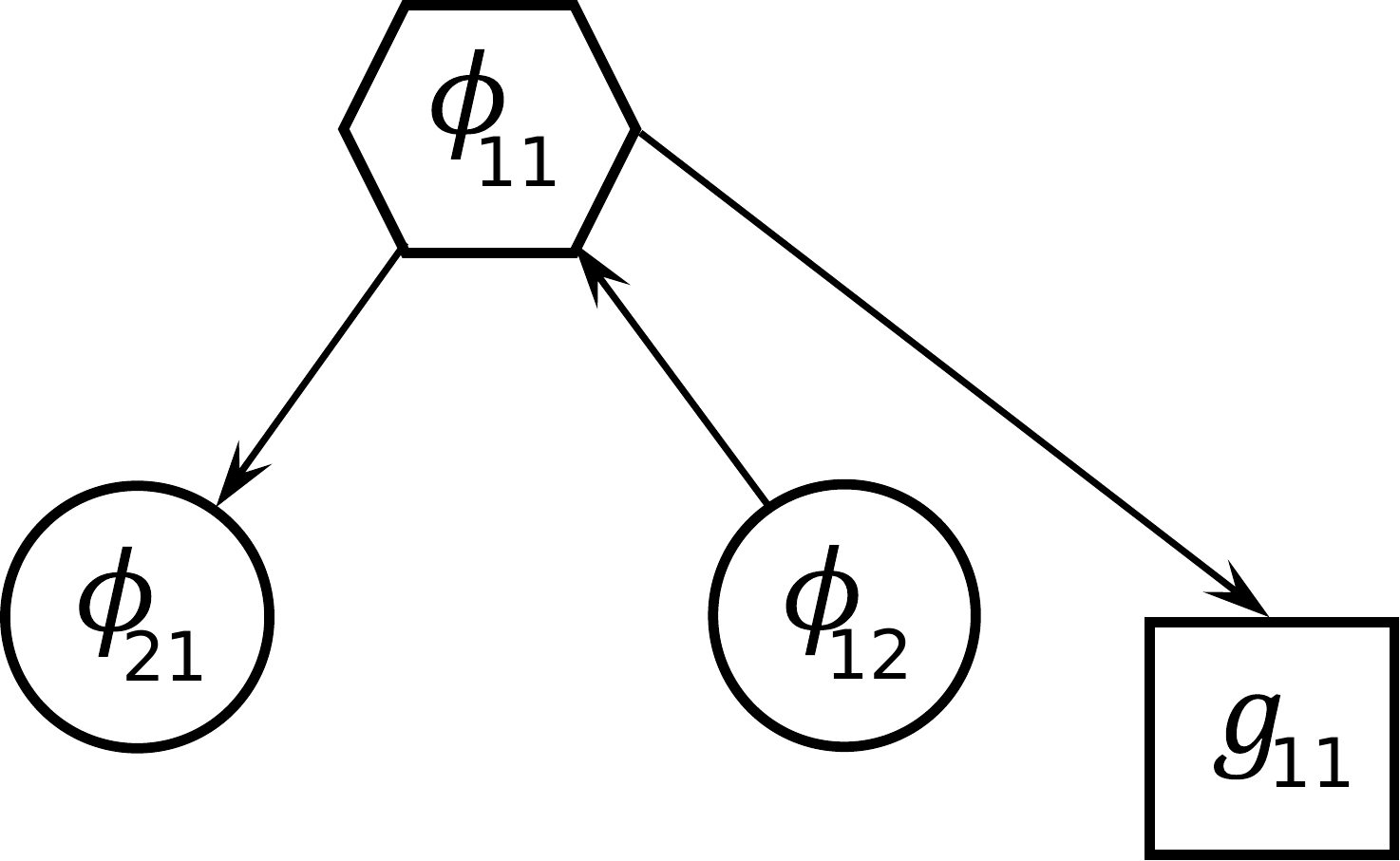}
 \end{center}
 \subcaption{Case $k=l=1$.}
 \label{f:gnbd_phi11}
 \end{subfigure}
 \begin{subfigure}[t]{3in}
 \begin{center}
 \includegraphics[scale=0.2]{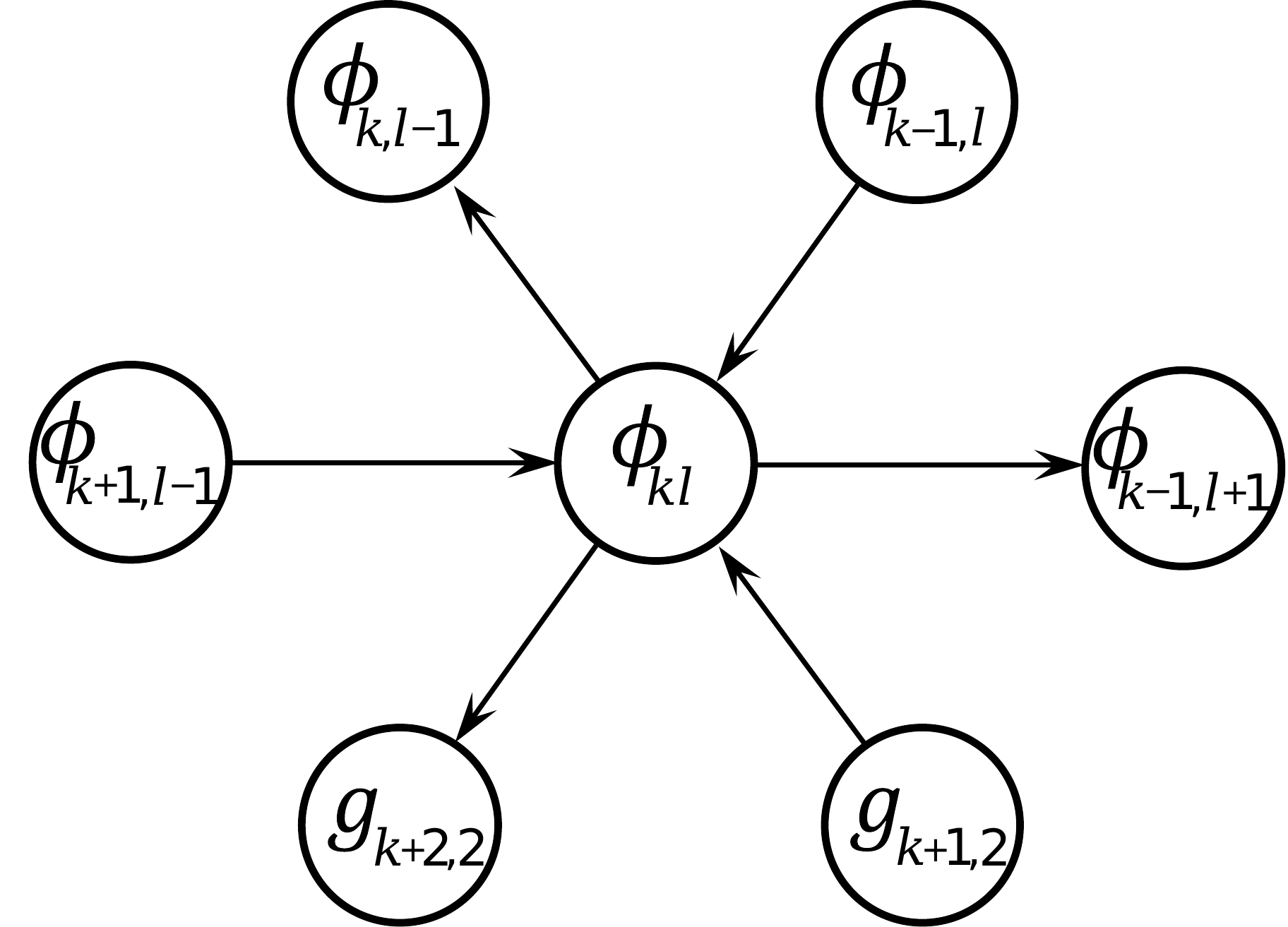}
 \end{center}
 \subcaption{Case $k + l = n$.}
 \label{f:gnbd_phikl_b}
 \end{subfigure}
 \caption{The neighborhood of $\varphi_{kl}$ for (a) $k=l=1$ and (b) $k+l=n$.}
 \label{f:gnbd_boundary}
 \end{figure}

 \vspace{5mm}
 \begin{figure}[htb]
 \begin{subfigure}[t]{3in}
 \begin{center}
 \includegraphics[scale=0.2]{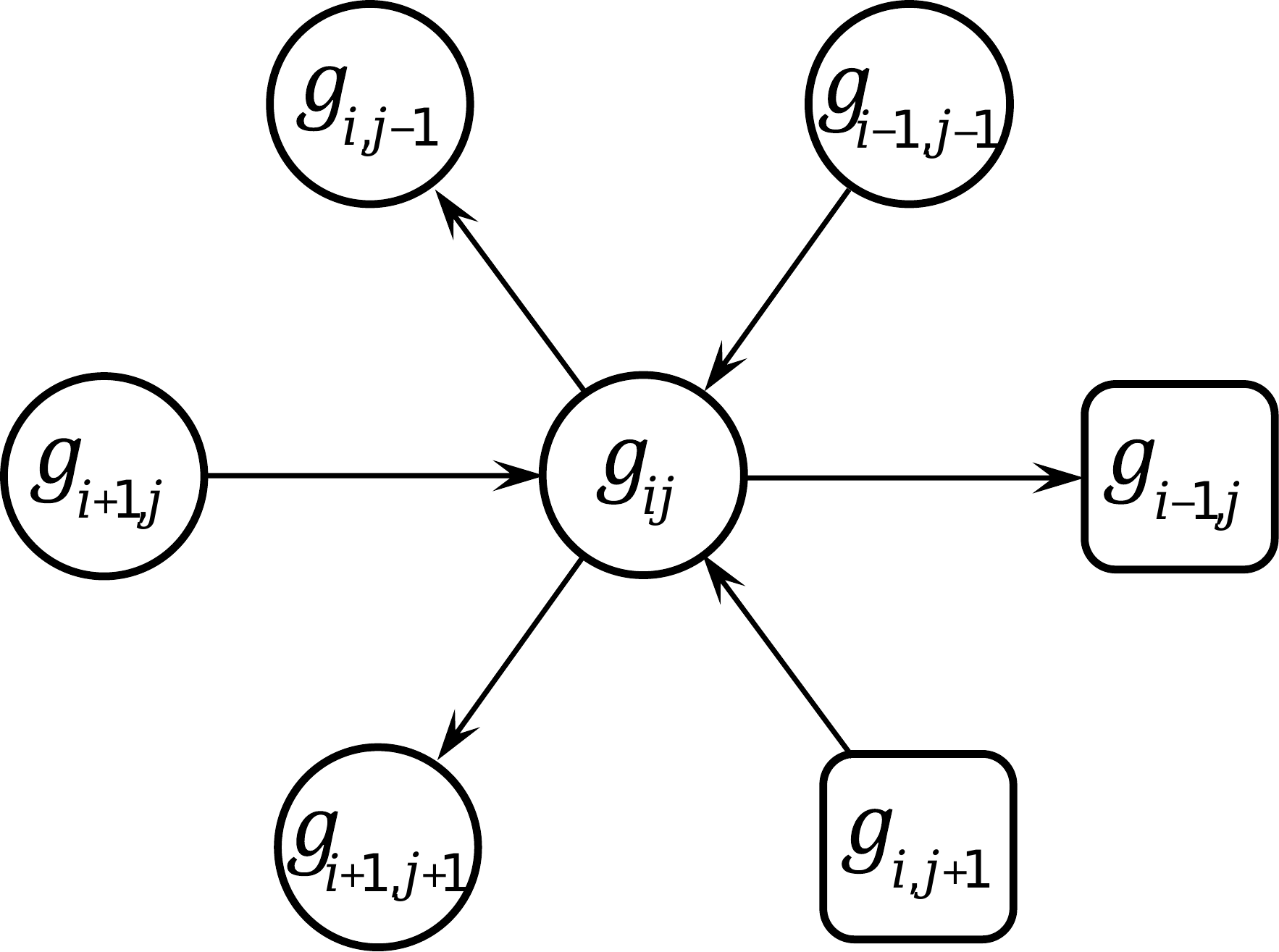}
 \end{center}
 \subcaption{Case $j < n$.}
 \label{f:nbd_gij}
 \end{subfigure}
 \begin{subfigure}[t]{3in}
 \begin{center}
 \includegraphics[scale=0.2]{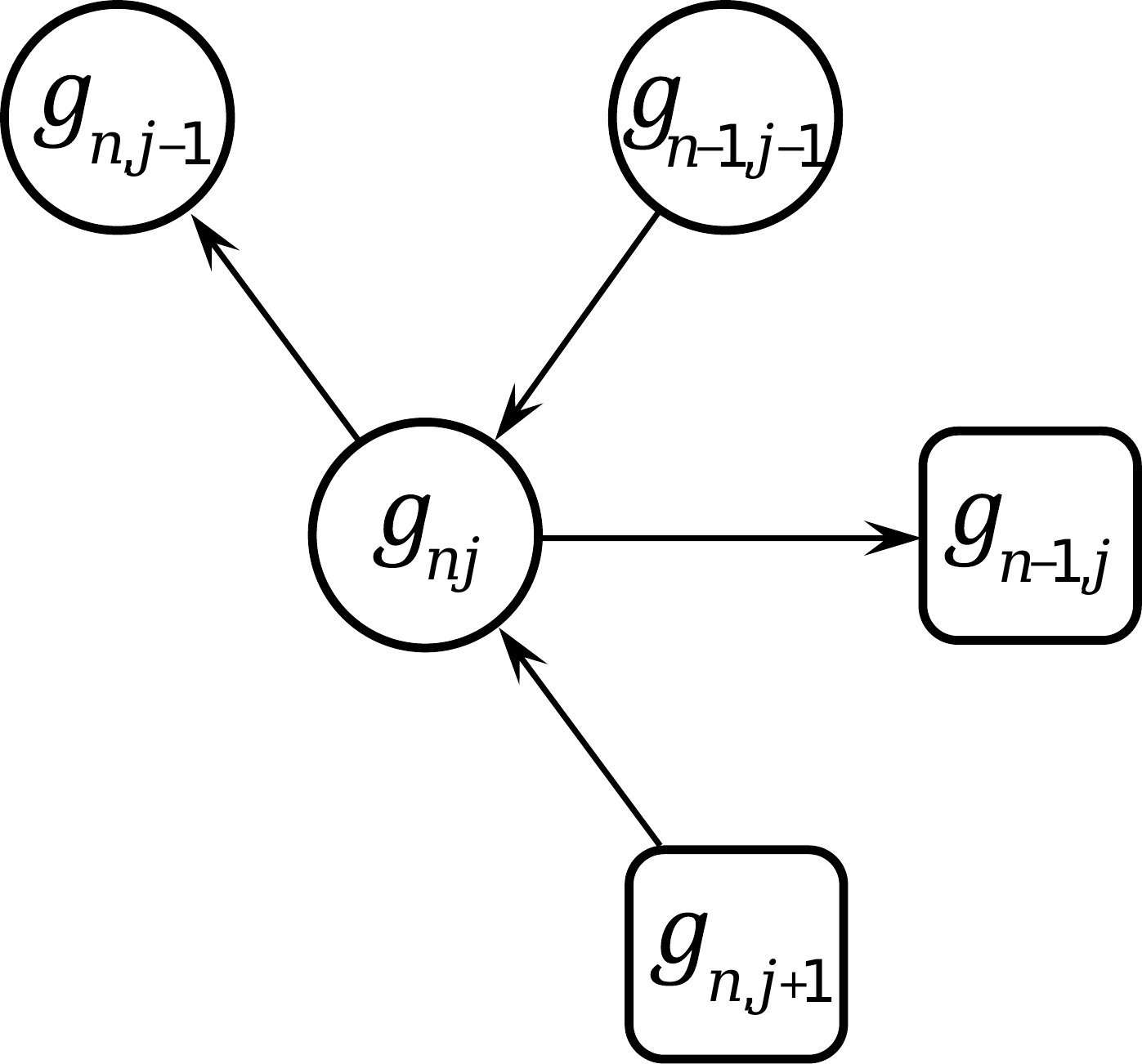}
 \end{center}
 \subcaption{Case $j = n$.}
 \label{f:nbd_gnj}
 \end{subfigure}
 \caption{The neighborhood of $g_{ij}$ for $2 \leq i \leq n-1$, $2 \leq j \leq n$, $i \geq j$.}
 \label{f:nbd_g}
 \end{figure}


 \vspace{5mm}
 \begin{figure}[htb]
 \begin{subfigure}[t]{2in}
 \begin{center}
 \includegraphics[scale=0.2]{nbds_g/gnbd_g11}
 \end{center}
 \subcaption{Case $i=1$.}
 \label{f:nbd_g11}
 \end{subfigure}
 \begin{subfigure}[t]{2in}
 \begin{center}
 \includegraphics[scale=0.2]{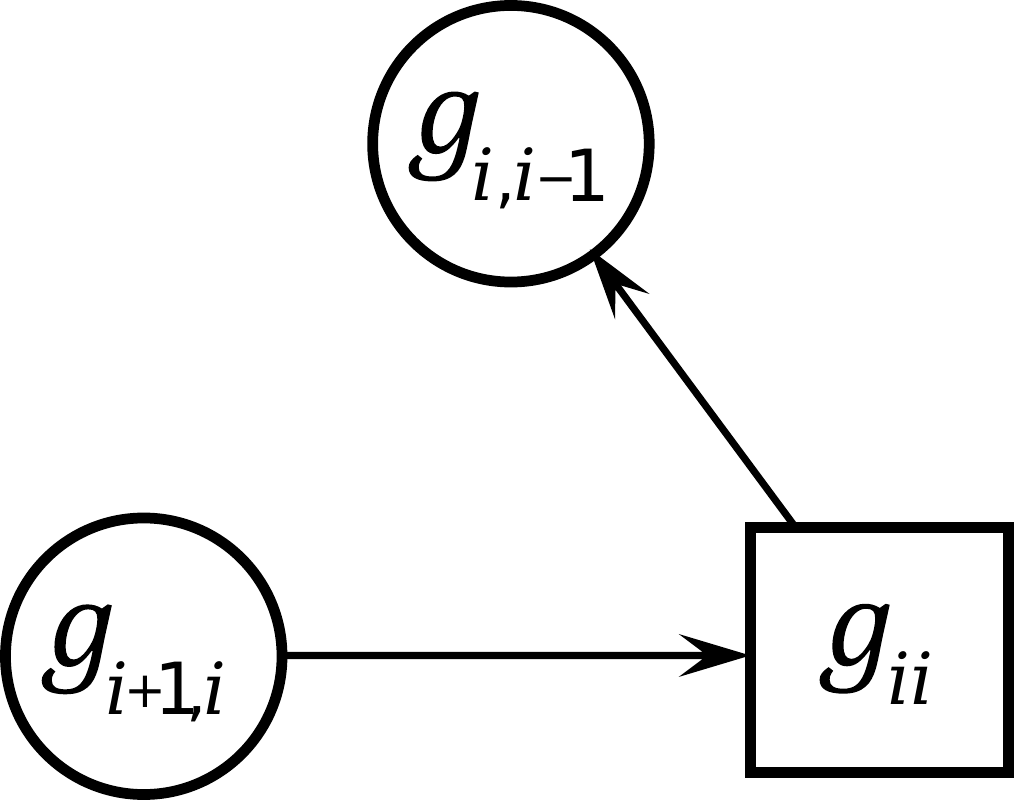}
 \end{center}
 \subcaption{Case $2 \leq i \leq n-1$.}
 \label{f:nbd_gii}
 \end{subfigure}
 \begin{subfigure}[t]{2in}
 \begin{center}
 \includegraphics[scale=0.2]{nbds_g/gnbd_gnn}
 \end{center}
 \subcaption{Case $i=n$.}
 \label{f:nbd_gnn}
 \end{subfigure}
 \caption{The neighborhood of $g_{ii}$ for $1 \leq i \leq n$.}
 \label{f:nbd_gg}
 \end{figure}

 \begin{figure}[htb]
\begin{center}
\includegraphics[scale=0.20]{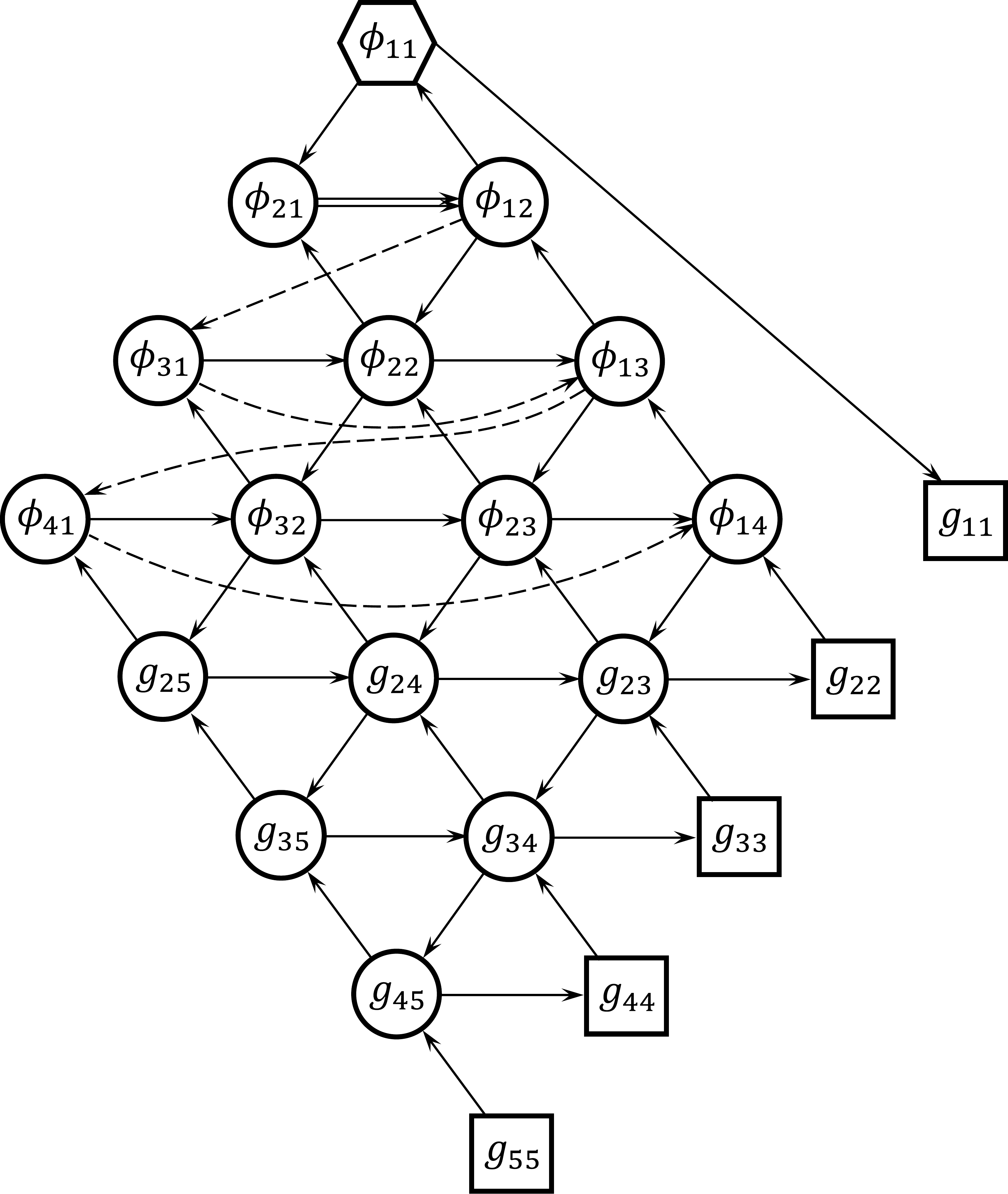}
\end{center}
\caption{The initial quiver for $\gc_g^{\dagger}(\bg_{\std},\GL_5)$. }
\label{f:ex_n=5_gstd}
\end{figure}

\clearpage

\subsection{The quiver for a nontrivial BD triple (algorithm)}\label{s:quiv_g_alg}
Let $\bg := (\Gamma_1,\Gamma_2,\gamma)$ be a nontrivial BD triple. In order to obtain the quiver $Q_g(\bg)$, one proceeds as follows. First, draw the quiver $Q_g(\bg_{\std})$, as described in Appendix~\ref{s:quiv_g_triv}. Second, add new arrows as prescribed by the following algorithm:
\begin{enumerate}[1)]
\item If $i \in \Gamma_1$ is such that $i-1\notin\Gamma_1$ or $\gamma|_{\Delta(i)\cap \Gamma_1} < 0$, add new arrows as indicated in Figure~\ref{f:nbd_gi1i1_0};
\item If $i \in \Gamma_1$ is such that $i-1 \in \Gamma_1$ and $\gamma|_{\Delta(i)\cap \Gamma_1} > 0$, add new arrows as indicated in Figure~\ref{f:nbd_gi1i1_1};
\item Repeat for all roots in $\Gamma_1$.
\end{enumerate}
In the algorithm, ${\Delta}(i)$ denotes the $X$-run that contains $i$. 

\vspace{5mm}
 \begin{figure}[htb]
 \begin{subfigure}[t]{3in}
 \begin{center}
 \includegraphics[scale=0.2]{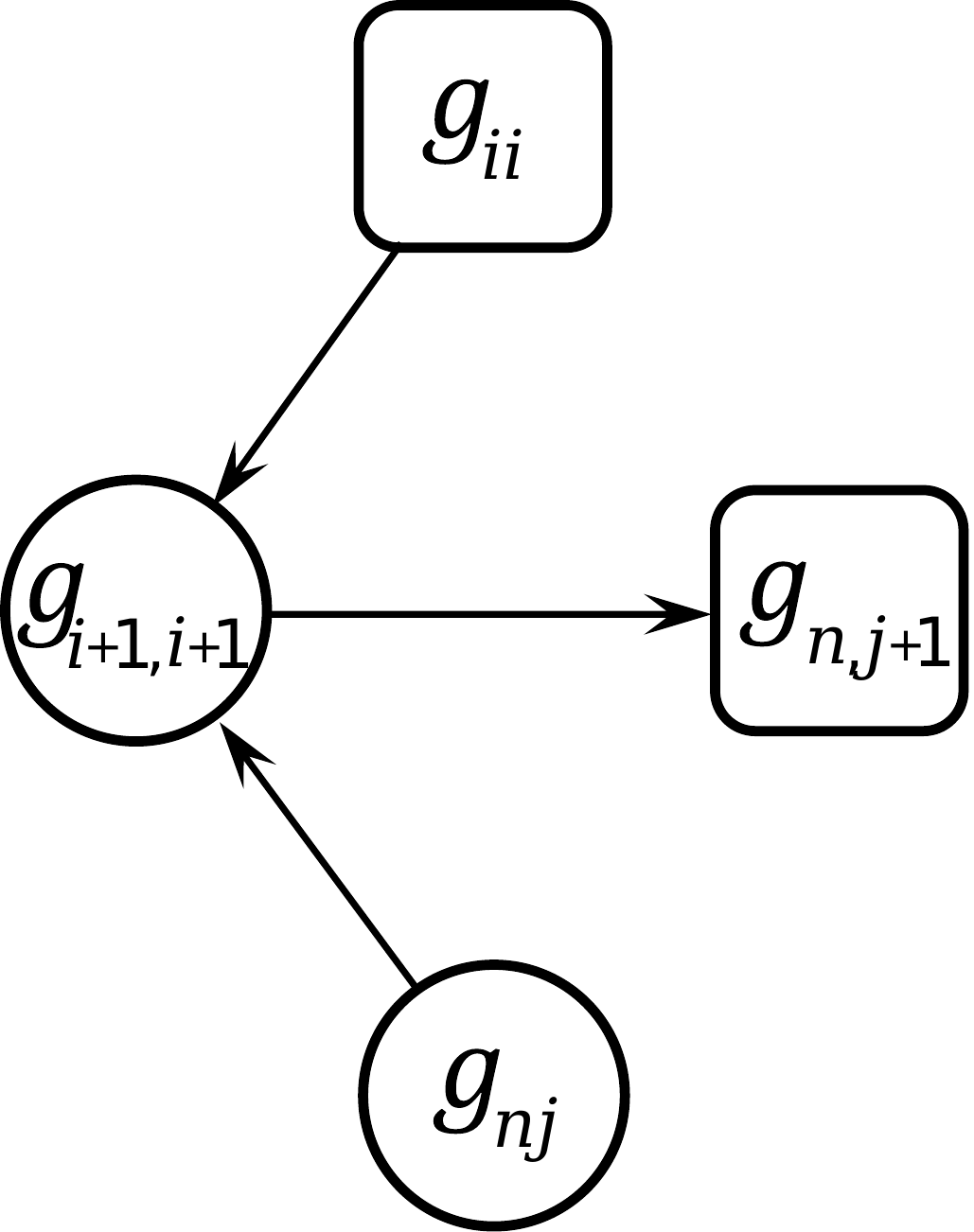} 
 \end{center}
 \subcaption{Case $i-1 \notin \Gamma_1$ or $\gamma|_{\Delta(i)\cap \Gamma_1} < 0$.}
 \label{f:nbd_gi1i1_0}
 \end{subfigure}
 \begin{subfigure}[t]{3in}
 \begin{center}
 \includegraphics[scale=0.2]{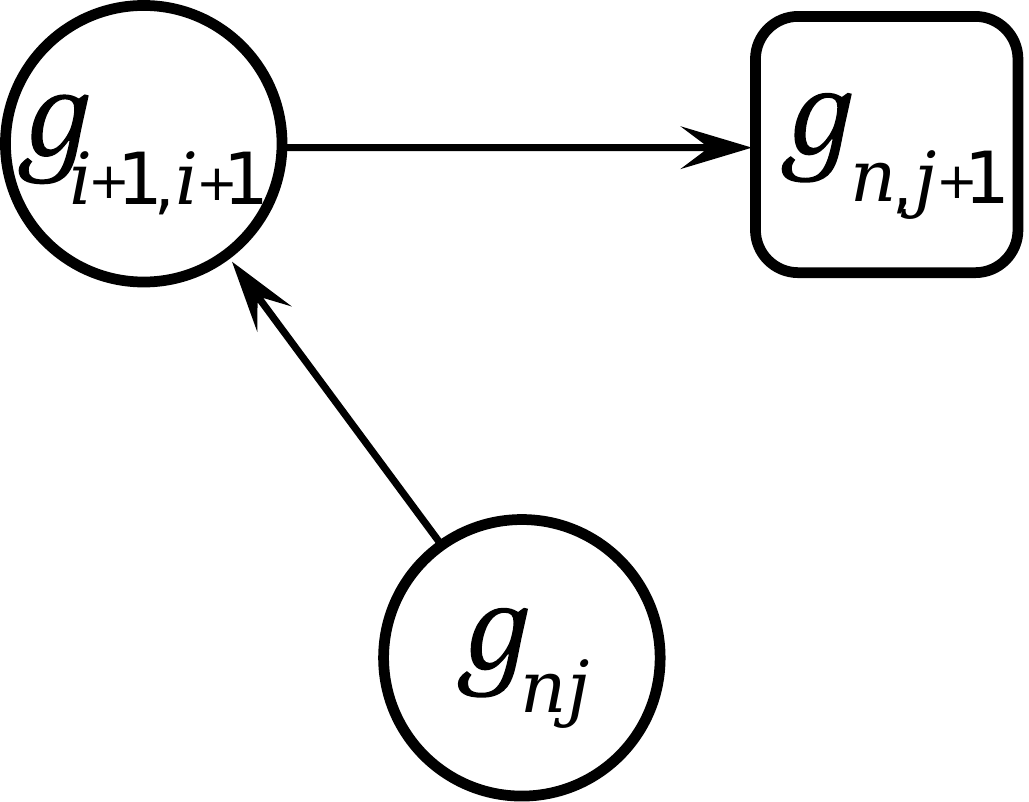}
 \end{center}
 \subcaption{Case $i-1\in \Gamma_1$ and $\gamma|_{\Delta(i)\cap\Gamma_1} > 0$.}
 \label{f:nbd_gi1i1_1}
 \end{subfigure}
 \caption{Additional arrows for $i \in \Gamma_1$ and $j:= \gamma(i)$.}
 \label{f:nbd_extra_g}
 \end{figure}

\subsection{The quiver for a nontrivial BD triple (explicit)}\label{s:g_quiv_expl}
As an alternative to the algorithm described in the previous subsection, we provide explicit neigh-
borhoods of the variables $g_{jj}$ ($2 \leq j \leq n$), $g_{in}$ ($2\leq i\leq n$), $\phi_{n-1,1}$ in the case of a nontrivial BD triple. All the other neighborhoods of $Q_g(\bg)$ are the same as in $Q_g(\bg_{\std})$. Let us mention that if $\Delta$ is a nontrivial $X$-run such that $|\Delta| = 2$, then both conditions $\gamma|_{\Delta \cap \Gamma_1} > 0$ and $\gamma|_{\Delta \cap \Gamma_2} < 0$ are considered to be satisfied; however, if $|\Delta|> 2$, then the conditions are mutually exclusive.

  \vspace{5mm}
 \noindent
\begin{figure}[htb]
 \begin{subfigure}[t]{2.2in}
 \begin{center}
 \includegraphics[scale=0.2]{nbds_g/gnbd_gii}
 \end{center}
 \subcaption{Case $i-1 \notin \Gamma_1$.}
 \label{f:inbd_gii_0}
 \end{subfigure}
 \begin{subfigure}[t]{3.6in}
 \begin{center}
 \includegraphics[scale=0.2]{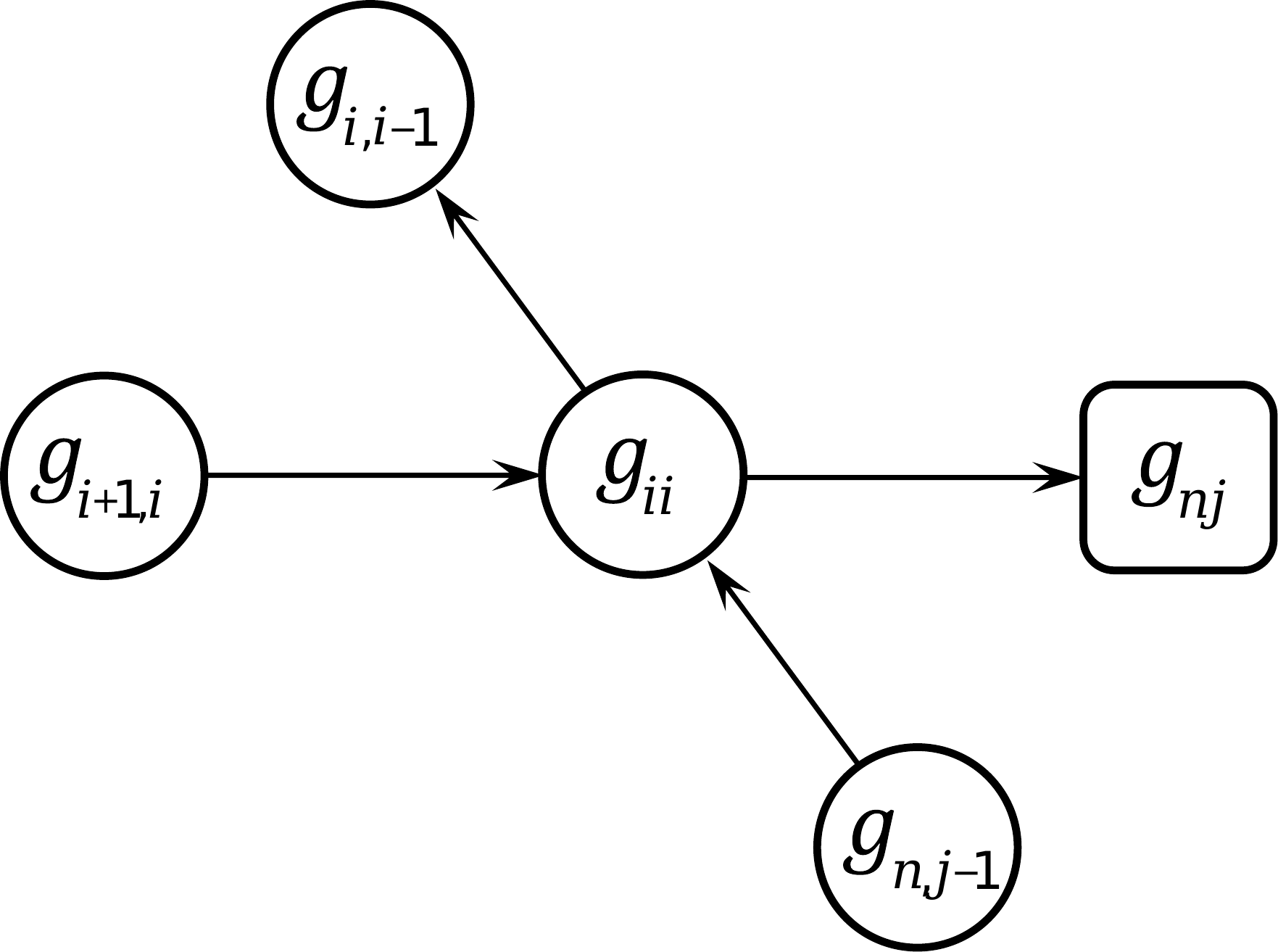}
 \end{center}
 \subcaption{Case $i-2,i-1\in\Gamma_1$ and $\gamma|_{\Delta(i)\cap \Gamma_1}>0$.}
 \label{f:inbd_gii_1}
 \end{subfigure}
 \begin{subfigure}[t]{3in}
\vspace{4mm}
 \begin{center}
 \includegraphics[scale=0.2]{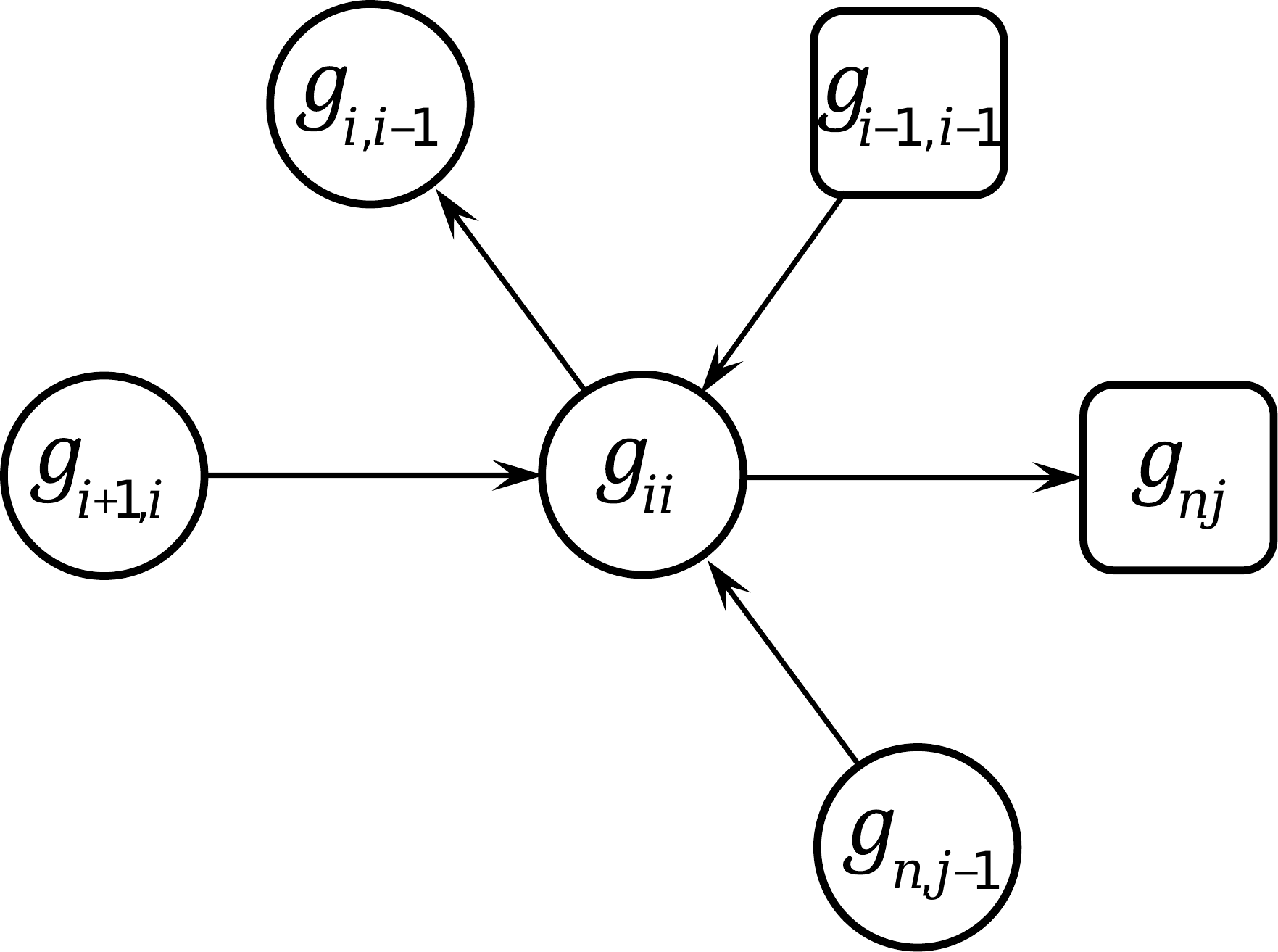}
 \end{center}
 \subcaption{Case $i-1 \in \Gamma_1$, $i \notin \Gamma_1$ and ($i-2\notin\Gamma_1$ or $\gamma|_{\Delta(i)\cap\Gamma_1}<0$).}
 \label{f:inbd_gii_2}
 \end{subfigure}
 \begin{subfigure}[t]{3in}
\vspace{4mm}
 \begin{center}
 \includegraphics[scale=0.2]{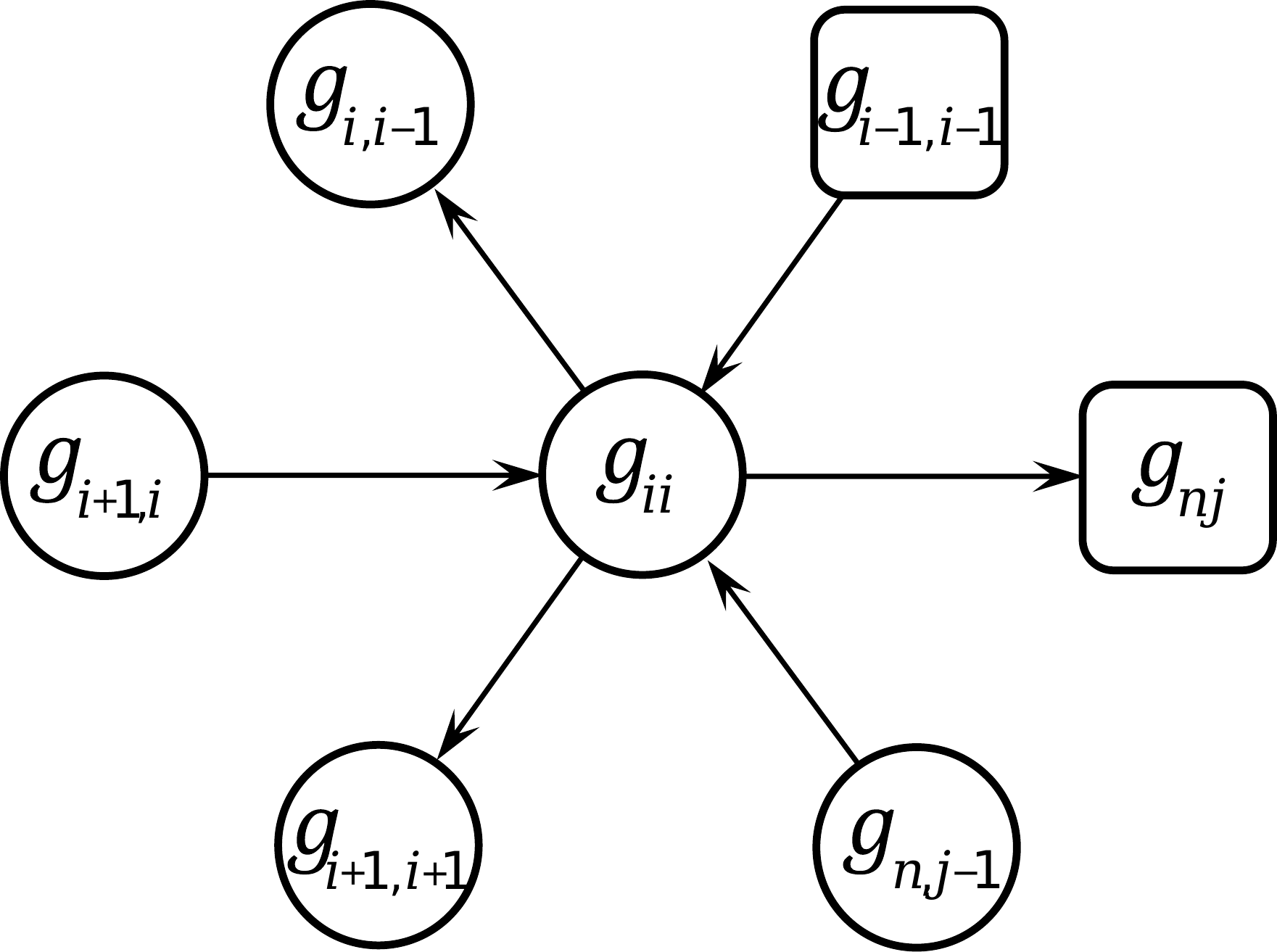}
 \end{center}
 \subcaption{Case $i-1,i \in \Gamma_1$ and $\gamma|_{\Delta(i)\cap \Gamma_1}<0$.}
 \label{f:inbd_gii_3}
 \end{subfigure}
 \caption{The neighborhood of $g_{ii}$ for $2 \leq i \leq n-1$; here, $j:=\gamma(j-1)+1$ when $i-1 \in \Gamma_1$.}
 \label{f:inbd_gii}
 \end{figure}

 \vspace{5mm}
\noindent
\begin{figure}[htb]
\begin{center}
\begin{subfigure}[t]{2.6in}
\begin{center}
\includegraphics[scale=0.2]{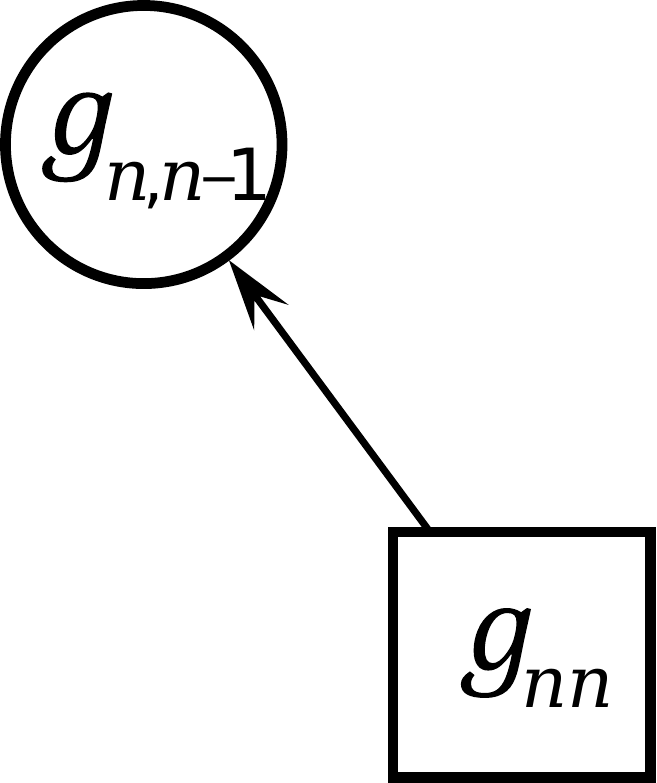} 
\end{center}
\subcaption{Case $n-1 \notin \Gamma_2$, $n-1 \notin \Gamma_1$.}
\label{f:gnbd_gnn_0}
\end{subfigure}
\begin{subfigure}[t]{2.6in}
\begin{center}
\includegraphics[scale=0.2]{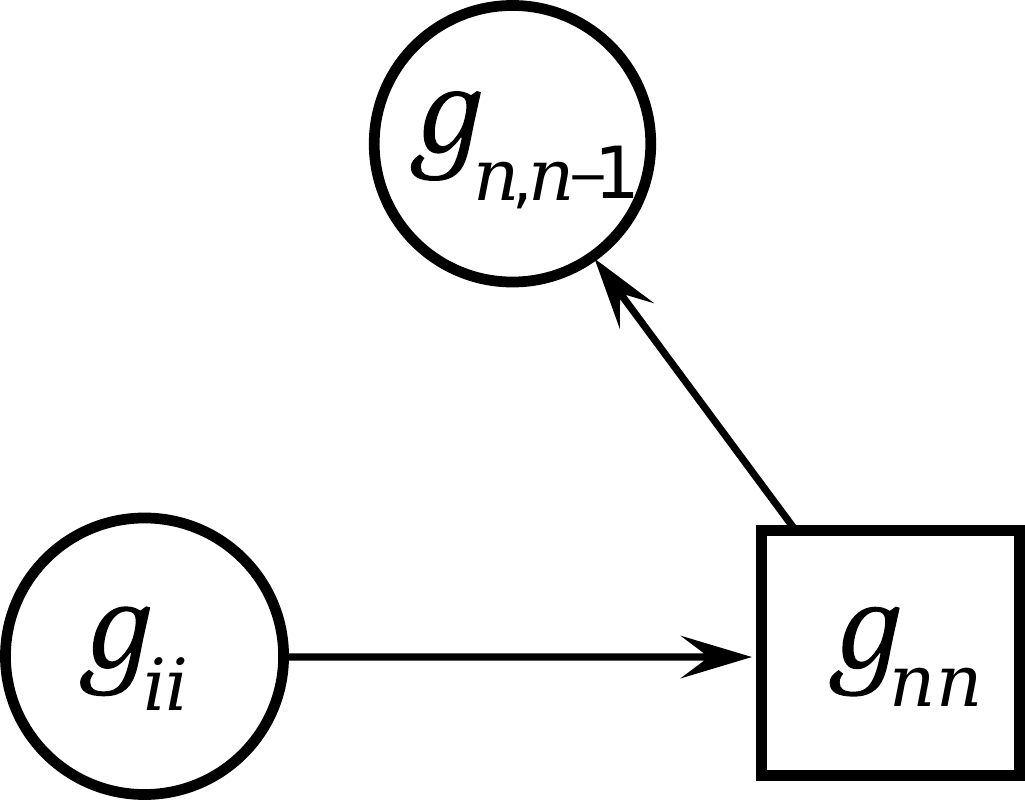}
\end{center}
\subcaption{Case $n-1 \in \Gamma_2$, $n-1 \notin \Gamma_1$.}
\label{f:gnbd_gnn_1}
\end{subfigure}
\begin{subfigure}[t]{2.6in}
\vspace{4mm}
\begin{center}
\includegraphics[scale=0.2]{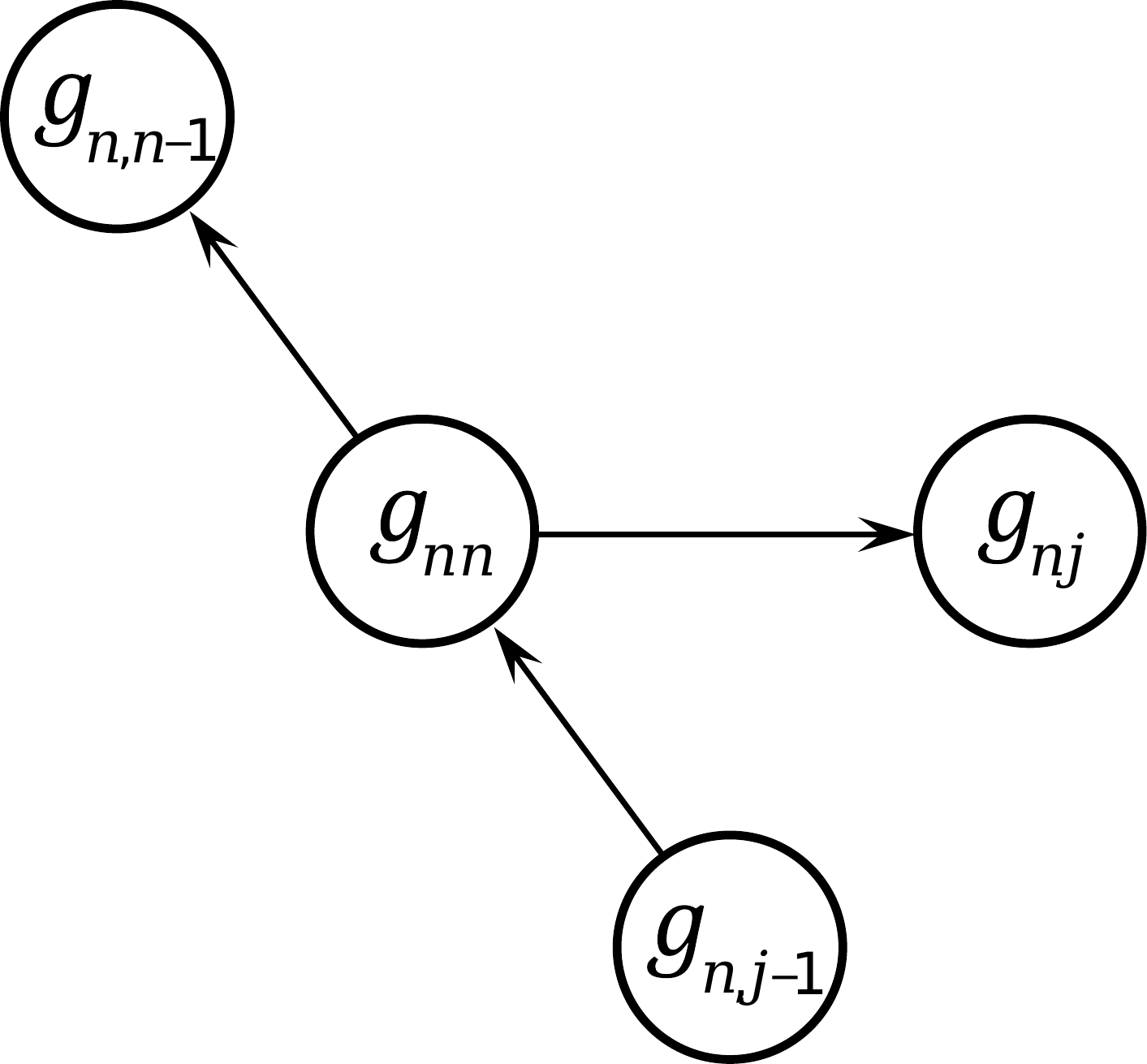}
\end{center}
\subcaption{Case $n-1 \notin \Gamma_2$, $n-1,n-2 \in \Gamma_1$ and $\gamma|_{\Delta(n-1) \cap \Gamma_2>0}$.}
\label{f:gnbd_gnn_2}
\end{subfigure}
\hspace{7mm}
\begin{subfigure}[t]{2.6in}
\vspace{4mm}
\begin{center}
\includegraphics[scale=0.2]{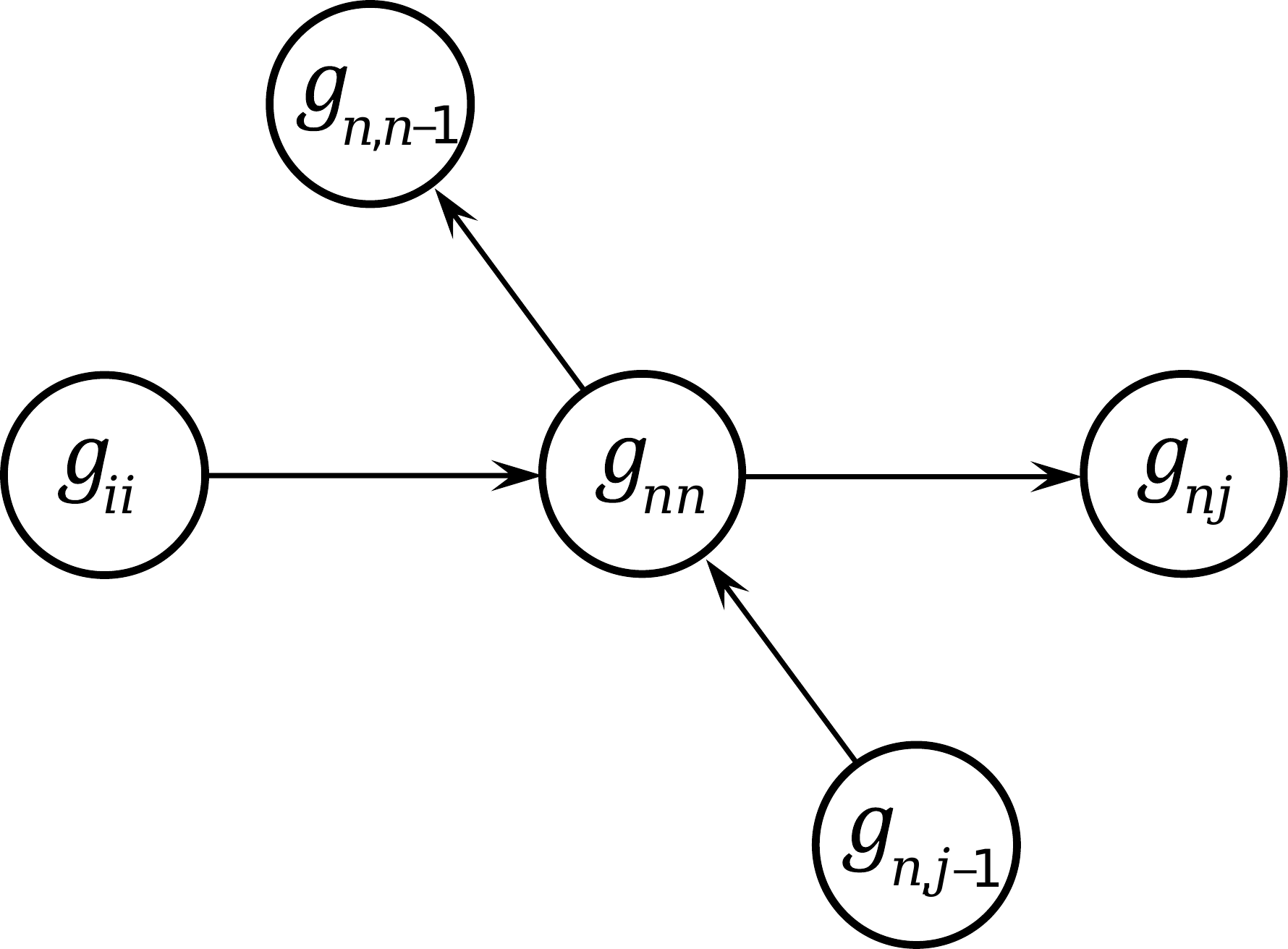}
\end{center}
\subcaption{Case $n-1 \in \Gamma_2$, $n-1,n-2 \in \Gamma_1$ and $\gamma|_{\Delta(n-1) \cap \Gamma_2>0}$.}
\label{f:gnbd_gnn_3}
\end{subfigure}
\begin{subfigure}[t]{2.6in}
\vspace{4mm}
\begin{center}
\includegraphics[scale=0.2]{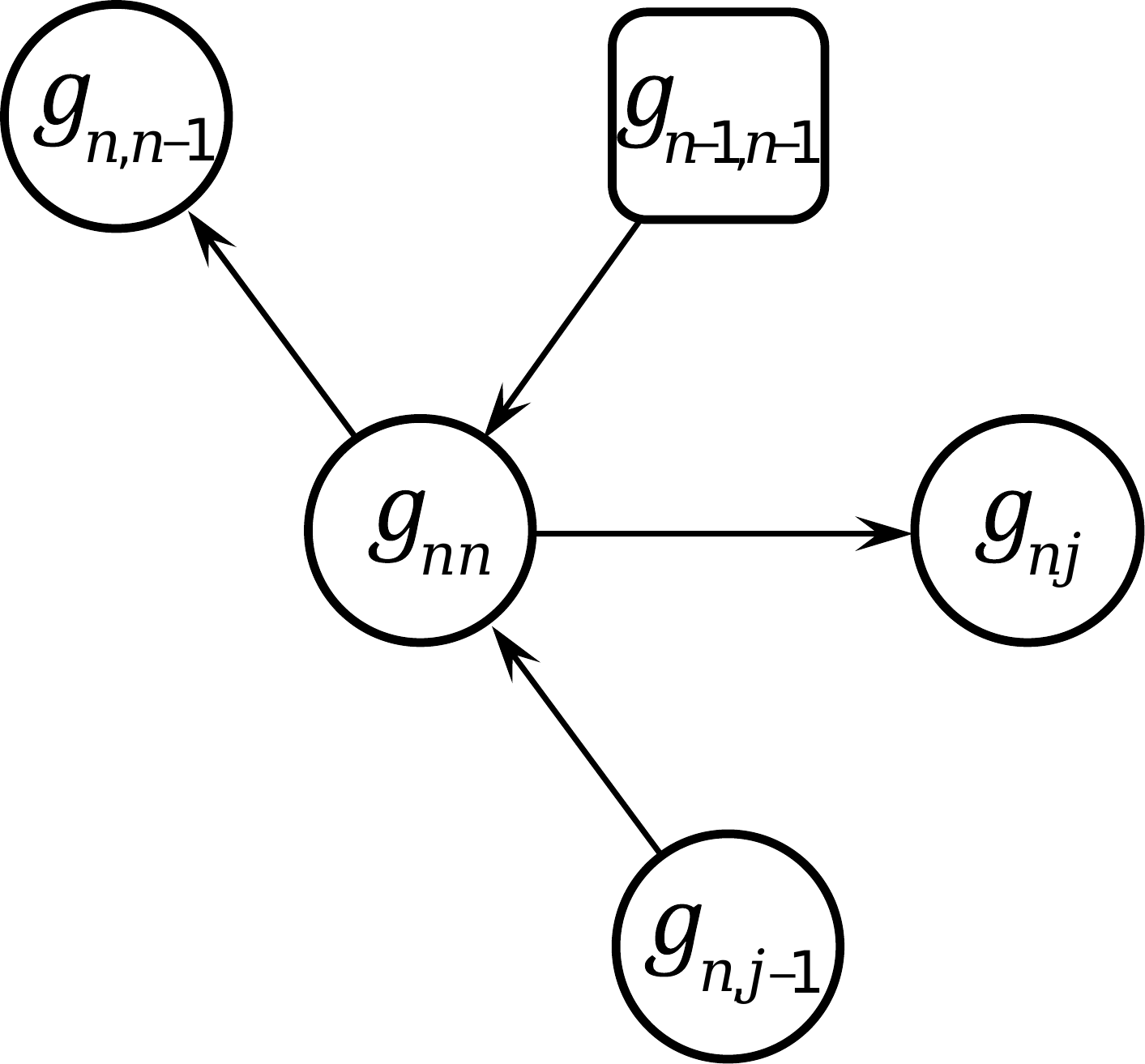}
\end{center}
\subcaption{Case $n-1 \notin \Gamma_2$, $n-1 \in \Gamma_1$ and ($n-2 \notin \Gamma_1$ or $\gamma|_{\Delta(n-1)\cap \Gamma_2}<0$).}
\label{f:gnbd_gnn_4}
\end{subfigure}
\hspace{7mm}
\begin{subfigure}[t]{2.6in}
\vspace{4mm}
\begin{center}
\includegraphics[scale=0.2]{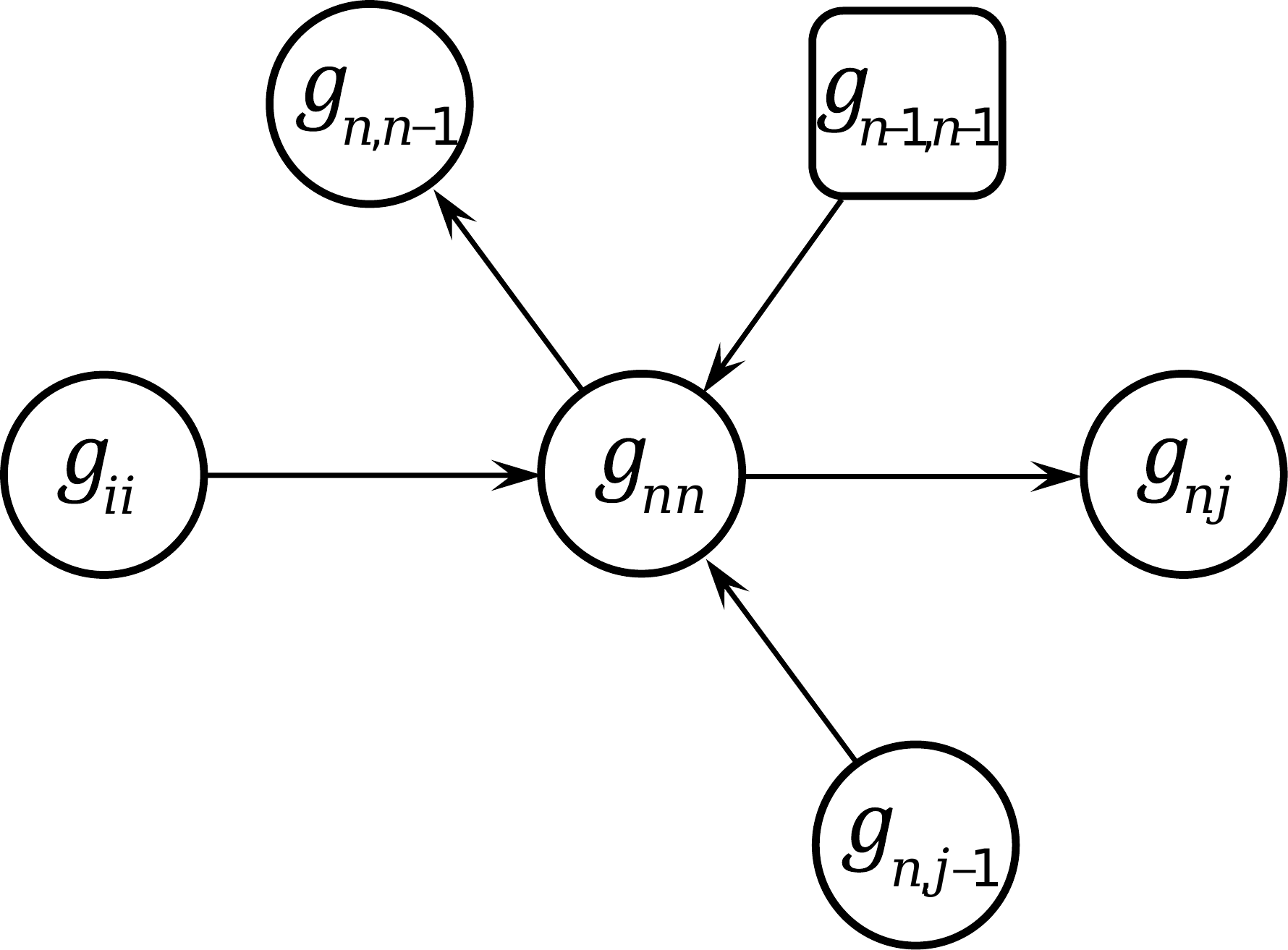}
\end{center}
\subcaption{Case $n-1 \in \Gamma_2$, $n-1 \in \Gamma_1$ and ($n-2 \notin \Gamma_1$ or $\gamma|_{\Delta(n-1)\cap \Gamma_2}<0$).}
\label{f:gnbd_gnn_5}
\end{subfigure}
\caption{The neighborhood of $g_{nn}$. In the figures, we set $j:=\gamma(n-1)+1$ if $n-1 \in \Gamma_1$ and $i:=\gamma^*(n-1)+1$ if $n-1 \in \Gamma_2$.}
\label{f:gnbd_gnn}
\end{center}
\end{figure}

 \vspace{5mm}
 \noindent
\begin{figure}[htb]
 \begin{subfigure}[t]{2.2in}
 \begin{center}
 \includegraphics[scale=0.2]{nbds_g/gnbd_gnj}
 \end{center}
 \subcaption{Case $j-1,j \notin \Gamma_2$.}
 \label{f:gnbd_gnj_0}
 \end{subfigure}
 \begin{subfigure}[t]{3.6in}
 \begin{center}
 \includegraphics[scale=0.2]{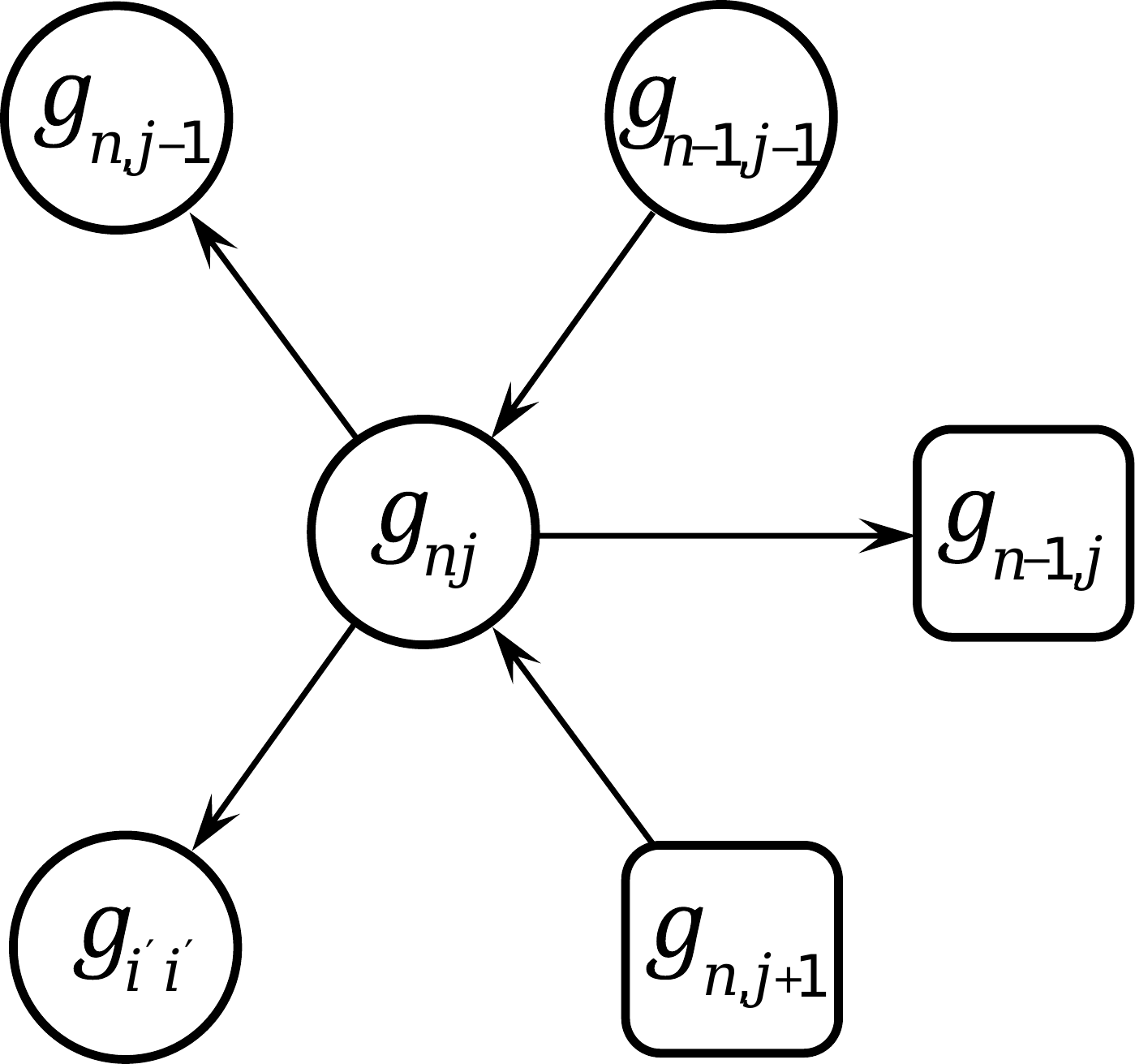}
 \end{center}
 \subcaption{Case $j-1 \notin \in \Gamma_2$, $j \in \Gamma_2$.}
 \label{f:gnbd_gnj_1}
 \end{subfigure}
 \begin{subfigure}[t]{3in}
\vspace{4mm}
 \begin{center}
 \includegraphics[scale=0.2]{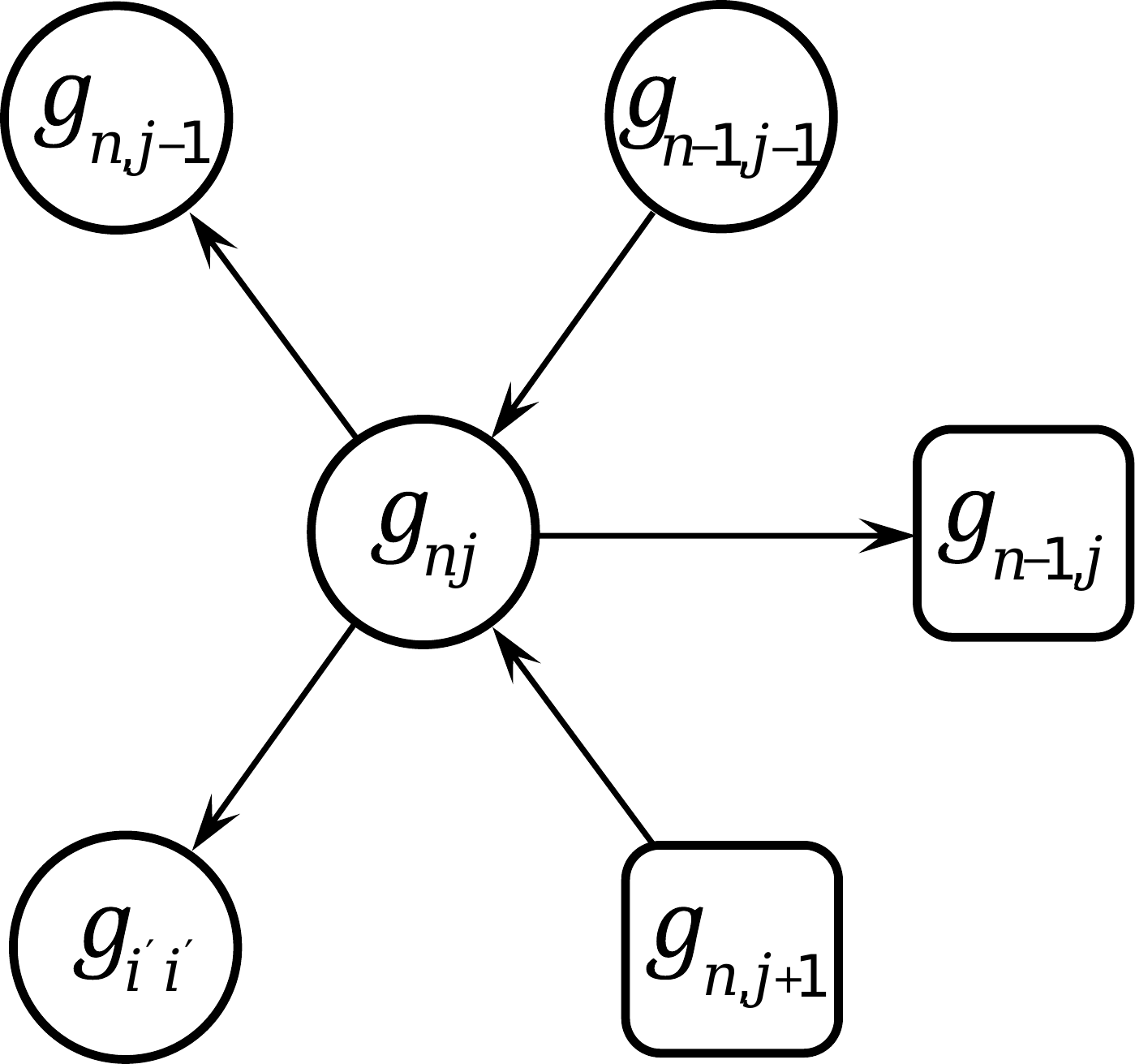}
 \end{center}
 \subcaption{Case $j-1\in\Gamma_2$, $j \notin \Gamma_2$.}
 \label{f:gnbd_gnj_2}
 \end{subfigure}
 \begin{subfigure}[t]{3in}
\vspace{4mm}
 \begin{center}
 \includegraphics[scale=0.2]{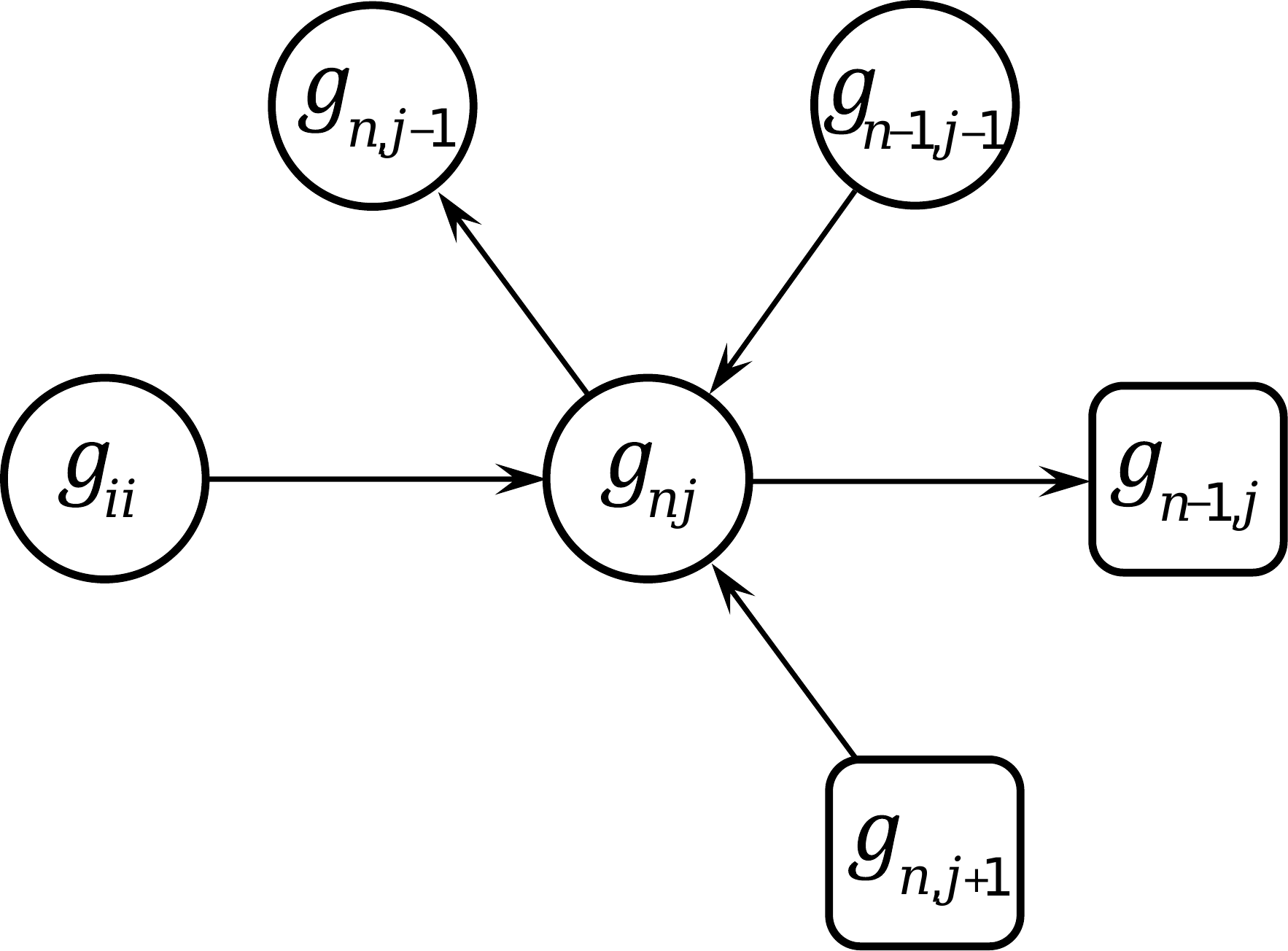}
 \end{center}
 \subcaption{Case $j-1,j\in\Gamma_2$.}
 \label{f:gnbd_gnj_3}
 \end{subfigure}
 \caption{The neighborhood of $g_{nj}$ for $2 \leq j \leq n-1$; here, $i:=\gamma^*(j-1)+1$ if $j-1 \in \Gamma_2$ and $i^\prime:=\gamma^*(j)+1$ if $j \in \Gamma_2$.}
 \label{f:gnbd_gnj}
 \end{figure}

 \begin{figure}[htb]
 \begin{subfigure}[t]{2.6in}
 \begin{center}
 \includegraphics[scale=0.2]{nbds_g/gnbd_phin11}
 \end{center}
 \subcaption{Case $1 \notin \Gamma_2$.}
 \label{f:gnbd_phin11_0}
\end{subfigure}
 \begin{subfigure}[t]{3.5in}
 \begin{center}
 \includegraphics[scale=0.2]{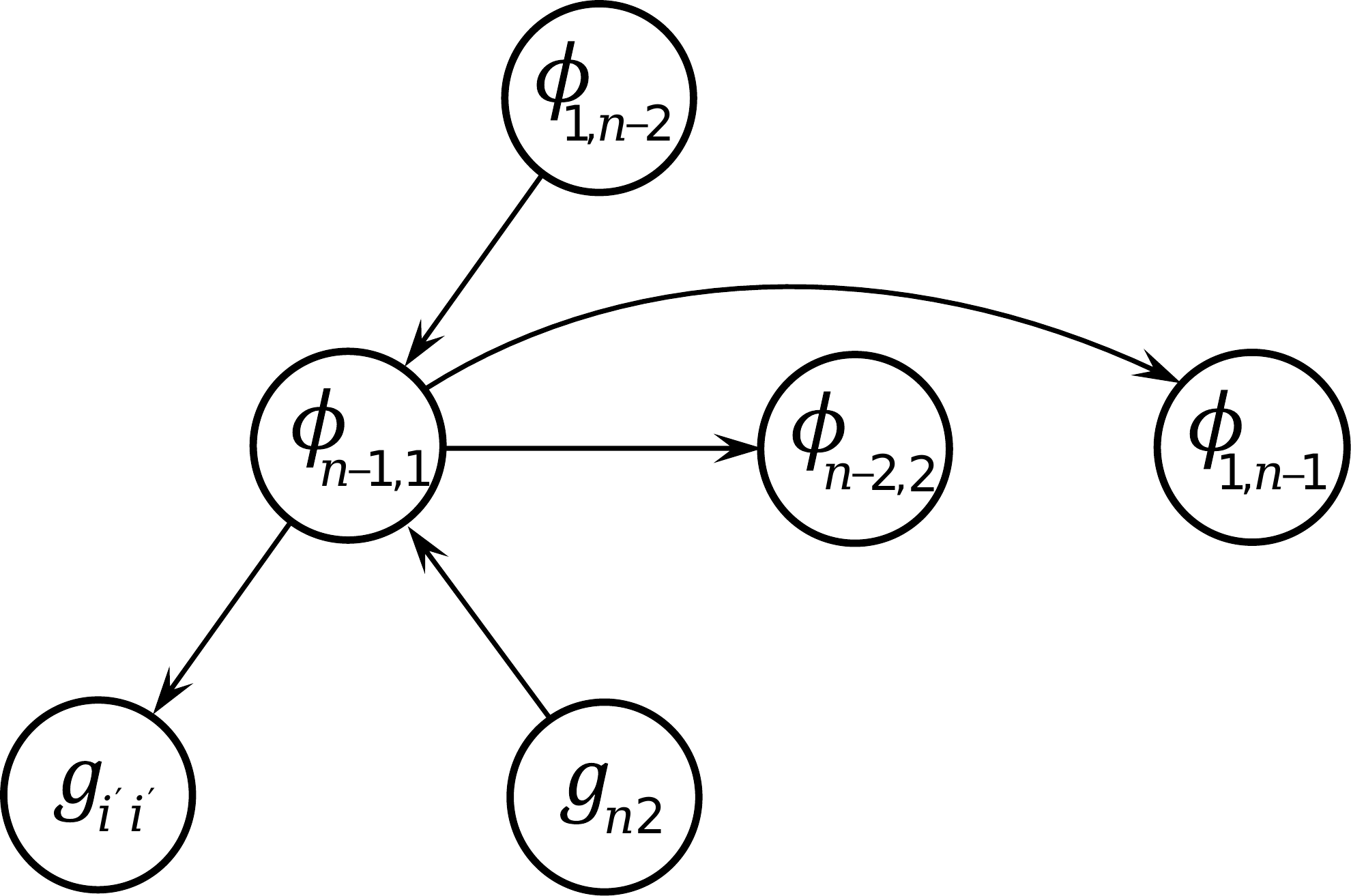}
 \end{center}
 \subcaption{Case $1 \in \Gamma_2$.}
 \label{f:gnbd_phin11_1}
 \end{subfigure}
 \caption{The neighborhood of $\varphi_{n-1,1}$. Here, $i^\prime:=\gamma^*(1)+1$ if $1 \in \Gamma_2$.}
 \label{f:ginbd_phin11}
 \end{figure}

 \clearpage

 \noindent
 \begin{figure}[htb]
 \begin{subfigure}[t]{3in}
 \begin{center}
 \includegraphics[scale=0.2]{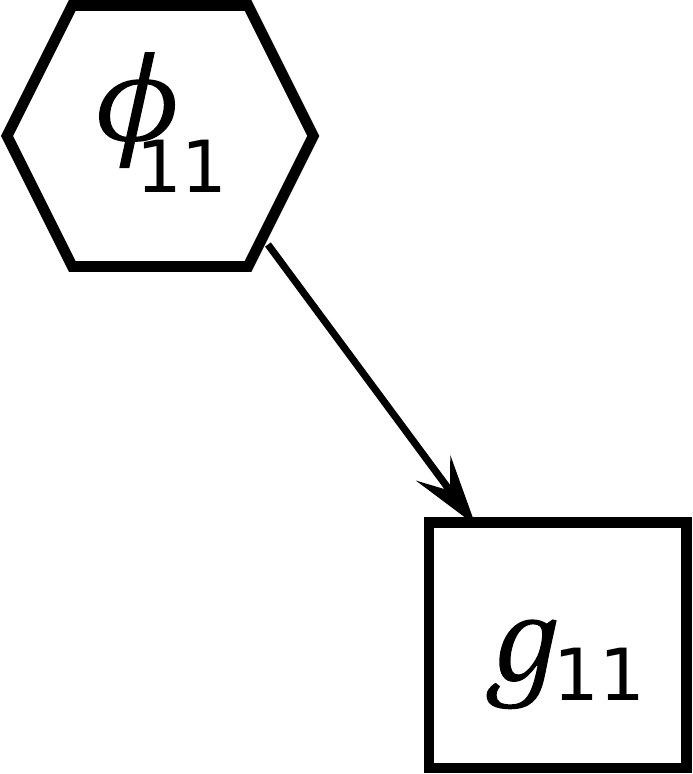} 
 \end{center}
 \subcaption{Case $1 \notin \Gamma_1$.}
 \label{f:gnbd_g11_0}
 \end{subfigure}
 \begin{subfigure}[t]{3in}
 \begin{center}
 \includegraphics[scale=0.2]{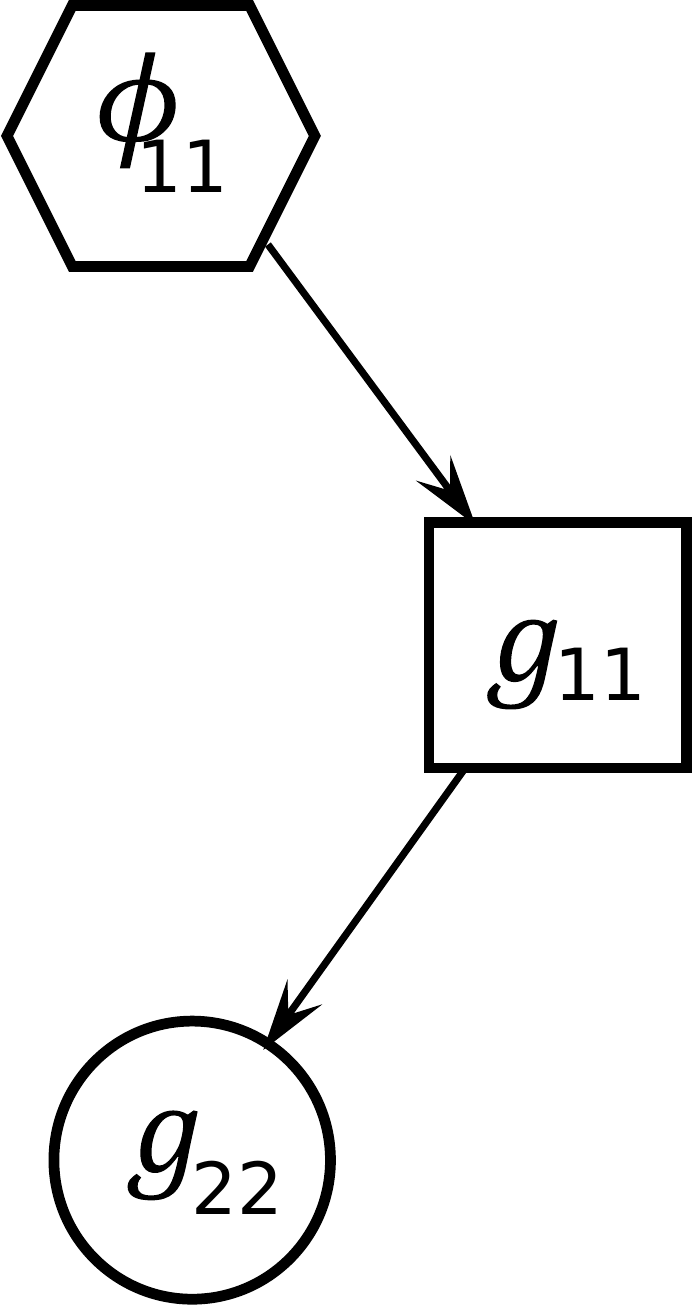}
 \end{center}
 \subcaption{Case $1 \in \Gamma_1$.}
 \label{f:gnbd_g11_1}
 \end{subfigure}
 \caption{The neighborhood of $g_{11}$.}
 \label{f:gnbd_g11}
 \end{figure}

\section{Determinantal identities}\label{s:aidentminors}
In this appendix, we list various determinantal identities that are used throughout the text. 

\paragraph{Gauss decomposition.} Let $U \in \GL_n$ be a generic element. Decompose $U$ as $U = U_{\oplus} U_-$ where $U_\oplus$ is an upper-triangular and $U_-$ is a unipotent lower-triangular matrix. For $1 \leq j \leq i \leq n$, the entries of $U_{\oplus}$, $U_{-}$ and their inverses are given by the well-known formulas
(see, e.g. \cite{gantmacher}).
\begin{align*}
&[U_{\oplus}]_{ji} = \frac{\det U^{[i,n]}_{\{j\}\cup[i+1,n]}}{\det U^{[i+1,n]}_{[i+1,n]}}, & &[U_-]_{ij} = \frac{ \det U^{\{j\}\cup [i+1,n]}_{[i,n]}}{\det U^{[i,n]}_{[i,n]}}, \\
&[\sinv{U_{\oplus}}]_{ji} = (-1)^{j+i} \frac{\det U^{[j+1,n]}_{[j,n]\setminus \{i\}}}{\det U^{[j,n]}_{[j,n]}}, & &[\sinv{U_-}]_{ij} = (-1)^{i+j} \frac{\det U^{[j,n]\setminus \{i\}}_{[j+1,n]}}{\det U^{[j+1,n]}_{[j+1,n]}}.
\end{align*}

The next two identities are needed in the derivation of formulas for minors of $\tilde{\gamma}^*(U_-)$. They are used in Section~\ref{s:bsimpl} for deriving a more explicit formula for birational quasi-isomorphisms, as well as in Sections~\ref{s:p_prelim}-\ref{s:p_explicit} for deriving explicit formulas for the $h$-variables.

\begin{lemma}\label{l:min_ump}
Let $U \in \GL_n$ be a generic element, $[a,b] \subseteq [1,n]$, $J\subseteq [1,n]$ be nontrivial subintervals such that $\max J \leq b$ and $|J|=|[a,b]|$. Then
\begin{equation}\label{eq:min_ump}
     \det [U_-]_{[a,b]}^{J} = \frac{\det U^{J\cup [b+1,n]}_{[a,n]} }{\det U^{[a,n]}_{[a,n]}}.
\end{equation}
\end{lemma}
\begin{proof}
The identity follows from an application of the Cauchy--Binet formula to $U_\oplus U_-$.
\end{proof}

\begin{lemma}\label{l:min_umn}
Let us set $\Delta := [p+1,p+k] \subseteq [1,n]$, $I:=[a,b] \subseteq \Delta$ with $b := p+k$, $J \subseteq \Delta$ a subset such that $|J| = |[a,b]|$. Let $\hat{I}:=[\hat{a},\hat{b}]$ and $\hat{J}$ be the subintervals of $[1,k]$ obtained as images of $I$ and $J$ under the unique increasing bijection $\Delta \rightarrow [1,k]$. Set $w_0 : \Delta \rightarrow \Delta$ and $\hat{w}_0 :[1,k]\rightarrow[1,k]$ to be the unique decreasing bijections, and let $W_0 := \sum_{i=1}^{k}(-1)^{i-1}e_{k-i+1,i}$. For a generic $U \in \GL_n$, the following identity holds:
\begin{equation}\label{eq:min_umn}
\det \left[W_0 [(U_-)^{-T}]_{\Delta}^{\Delta} W_0^{-1}\right]_{\hat{I}}^{\hat{J}} = \frac{\det U^{w_0(\Delta \setminus J) \cup [b+1,n]}_{[w_0(a) +1,n]}}{\det U^{[w_0(a) +1,n]}_{[w_0(a) +1,n]}}.
\end{equation}
\end{lemma}
\begin{proof}
It follows from the Cauchy--Binet formula that
\[\begin{split}
\det [W_0&[(U_-)^{-T}]_{\Delta}^{\Delta}W_0^{-1}]_{\hat{I}}^{\hat{J}} = (-1)^{\varepsilon} \det\left[ [U^T_\oplus U^{-T}]_{\Delta}^{\Delta} \right]_{[1,\hat{w_0}(\hat{a})]}^{\hat{w_0}(\hat{J})} = (-1)^{\varepsilon} \det [U^T_\oplus U^{-T}]_{[p+1,w_0(a)]}^{w_0(J)} = \\ = & (-1)^\varepsilon \det [U^T_\oplus U^{-T}]_{[1,w_0(a)]}^{[1,p]\cup w_0(J)} = (-1)^{\varepsilon} \frac{\det U}{\det U^{[w_0(a) +1,n]}_{[w_0(a) + 1,n]}} \det [U^{-T}]_{[1,w_0(a)]}^{[1,p]\cup w_0(J)} = \\ = &(-1)^{\varepsilon + \varepsilon^\prime} \frac{\det U^{w_0(\Delta \setminus J) \cup [b+1,n]}_{[w_0(a) +1,n]}}{\det U^{[w_0(a) +1,n]}_{[w_0(a) +1,n]}}.
\end{split}
\]
where we introduced
\[\begin{array}{l}
\varepsilon := \sum_{j \in \hat{w}_0(\hat{J})}(j-1) + \sum_{j \in \hat{w}_0(\hat{I})}(j-1),\\ \varepsilon^\prime := \sum_{i= w_0(a)+1}^{n} i + \sum_{i \in w_0(\Delta\setminus J)} i + \sum_{i=p+k+1}^n i.
\end{array}
\]
A simple computation shows that $\varepsilon + \varepsilon^\prime \equiv 0 \mod 2$. Thus the formula holds.
\end{proof}

\paragraph{Desnanot--Jacobi identity.} The following identity is used in Section~\ref{s:completeness}.

\begin{proposition}\label{p:dj22}
Let $A$ be an $m\times m$ matrix with entries in an arbitrary field. If $1 \leq i < j\leq m$ and $1\leq k < l\leq m$, then the following identity holds:
\[
\det A \det A^{\hat{i}\hat{j}}_{\hat{k}\hat{l}} + \det A_{\hat{l}}^{\hat{i}} \det A_{\hat{k}}^{\hat{j}} = \det A^{\hat{i}}_{\hat{k}} \det A_{\hat{l}}^{\hat{j}},
\]
where the hatted upper (lower) indices indicate that the corresponding column (row) is removed.
\end{proposition}

\paragraph{Derivatives of flag minors.} For $1 \leq i \leq j \leq n$, let us set
\begin{equation*}
    \hat{h}_{ij}(U):= \det U_{[i,n-j+i]}^{[j,n]}, \ \ U \in \GL_n.
\end{equation*}
In our conventions, $\nabla_U^R f = U\cdot \nabla_U f$ and $\nabla_U^L f = \nabla_U f \cdot U$, $U \in \GL_n$, where $\nabla_U f$ is given by formula~\eqref{eq:glngrad}. The following identities are used in Sections~\ref{s:toric}-\ref{s:compatibility}:

\begin{lemma}\label{l:grrhijks}
    The following formula holds:
    \begin{equation}\label{eq:grrhijks}
    \langle \pi_{>} \nabla_U^R \log\hat{h}_{ij}, \nabla_U^R \log \hat{h}_{ks}\rangle = \max \{ \min\{i-k,n-s+1\}-\max\{j-s,0\},0\}.
    \end{equation}
\end{lemma}
\begin{proof}
    It is a matter of a direct computation that
    \[
        (\nabla^R_U \hat{h}_{ij})_{qr} = \begin{cases}
        (-1)^{r-i}\det U^{[j,n]}_{\{q\}\cup([i,n-j+i]\setminus\{r\})}  &\text{if} \ q < i \ \text{and}\ r \in[i,n-j+i];\\
        (-1)^{n-j+i-r} \det U^{[j,n]}_{([i,n-j+i]\setminus\{r\})\cup\{q\}}  &\text{if} \ q > n-j+i\ \text{and} \ r \in[i,n-j+i]; \\
        \hat{h}_{ij} \ &\text{if} \ q=r\  \text{and} \ r \in[i,n-j+i];\\
        0 \ &\text{if} \ r \notin [i,n-j+i] \ \text{or} \ q \in [i,n-j+i]\setminus\{r\}.
        \end{cases}
    \]
    We see that
    \[
        \begin{split}
            \langle \pi_{>}\nabla_U^R \hat{h}_{ij},&\nabla^R_U \hat{h}_{ks}\rangle =\\ = &\sum_{\substack{(q,r):\, q\in [k,\min\{i-1,n-s+k\}] \\ r \in [\max\{i,n-s+k+1\},n-j+i]}} (-1)^{n-s+k-q+r-i} \det U^{[j,n]}_{\{q\}\cup([i,n-j+i]\setminus\{r\})} \det U^{[s,n]}_{([k,n-s+k]\setminus\{q\})\cup\{r\}}.
        \end{split}
    \]
    Now, the above sum simplifies via a long Pl\"ucker relation and is proportional to the product $\hat{h}_{ij}\hat{h}_{ks}$. A case-by-case analysis shows that the proportionality constant is given by~\eqref{eq:grrhijks}.
\end{proof}

\begin{lemma}\label{l:hijlrd}
The following identities hold:
\begin{equation*}
\pi_0\left( \nabla_U^R\log \hat{h}_{ij} \right)  = \sum_{k=i}^{n-j+i} e_{kk}, \ \ \pi_0\left(\nabla_U^L \log \hat{h}_{ij}\right) = \sum_{k=j}^{n} e_{kk}.
\end{equation*}
\end{lemma}
\begin{proof}
The identities follow from the fact that, given $T = \diag(t_1,\ldots,t_n)$,
\[
\hat{h}_{ij}(TU) = \hat{h}_{ij}(U)\prod_{k=i}^{n-j+i}t_k, \ \ \hat{h}_{ij}(UT) = \hat{h}_{ij}(U)\prod_{k=j}^{n}t_k.\qedhere
\]
\end{proof}


\section{Selected examples}\label{s:exs}
In this appendix, we provide explicit examples of the initial extended seeds of $\gc_h^{\dagger}(\bg,\GL_n)$. More examples (including the $g$-convention) can be found in the supplementary note~\cite{github}.

\subsection{\texorpdfstring{Cremmer--Gervais in $n=3$, $\gamma:1 \mapsto 2$}{Cremmer--Gervais in n=3}}
The initial quiver is illustrated in Figure~\ref{f:ex_n=3}. 
\begin{figure}[htb]
\begin{center}
\includegraphics[scale=0.2]{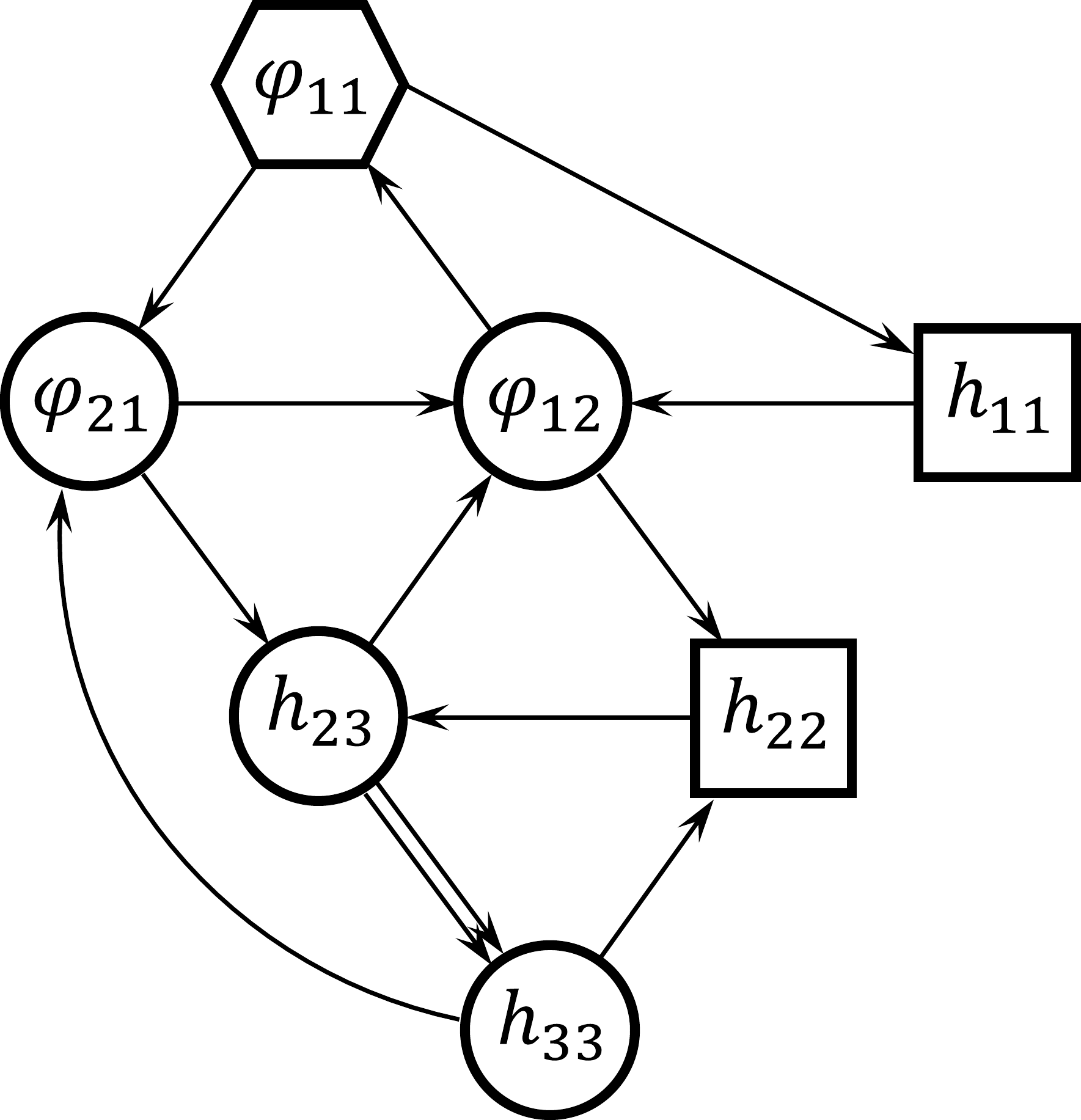}
\end{center}
\caption{The initial quiver for $\gc_h^{\dagger}(\bg,\GL_3)$ with $\gamma:1\mapsto 2$. }
\label{f:ex_n=3}
\end{figure}

The variables in the initial extended cluster are given as follows:
\begin{equation*}
    c_1(U) = \tr(U), \ \ c_2(U) = \frac{1}{2!} (\tr(U)^2 - \tr(U^2));
\end{equation*}
\begin{equation*}
    \varphi_{21}(U) = u_{13}, \ \ \varphi_{12}(U) = \det U^{[2,3]}_{[1,2]}, \ \  \varphi_{11}(U) = u_{23}\det U^{[2,3]}_{[1,2]} + u_{13}\det U^{\{1,3\}}_{[1,2]};
\end{equation*}
\begin{equation*}
    h_{23}(U) = -u_{23}u_{33}-u_{13}u_{32}, \ \ h_{22}(U) = u_{33}\det U^{[2,3]}_{[2,3]} + u_{32}\det U^{[2,3]}_{\{1,3\}};
\end{equation*}
\begin{equation*}
    h_{11}(U) = \det U, \ \ h_{33}(U) = u_{33}.
\end{equation*}

\subsection{\texorpdfstring{Cremmer--Gervais in $n=4$, $\gamma:i \mapsto i-1$}{Cremmer--Gervais in n=4}}

The initial quiver is illustrated in Figure~\ref{f:ex_n=4}. In the initial extended cluster, all cluster and frozen variables are given as in $\gc_h^{\dagger}(\bg_{\std},\GL_4)$ except for the variables $h_{34}$, $h_{33}$, $h_{44}$. These are given by:
\begin{align*}
    &h_{34}(U) = -u_{34}\det U^{[2,4]}_{[2,4]} - u_{24} \det U^{\{1\}\cup[3,4]}_{[2,4]};\\
    &h_{33}(U) = \det U^{[3,4]}_{[3,4]} \det U^{[2,4]}_{[2,4]} + \det U^{[3,4]}_{\{2,4\}} \det U^{\{1\}\cup[3,4]}_{[2,4]} + \det U^{[3,4]}_{[2,3]}\det U^{[1,2]\cup\{4\}}_{[2,4]};
\end{align*}
\begin{equation*}
    \begin{split}
     h_{44}(U) = &u_{44}\left(\det U^{[3,4]}_{[3,4]} \det U^{[2,4]}_{[2,4]} + \det U^{[3,4]}_{\{2,4\}} \det U^{\{1\}\cup[3,4]}_{[2,4]} + \det U^{[3,4]}_{[2,3]}\det U^{[1,2]\cup\{4\}}_{[2,4]}\right) + \\ + &u_{34}\left(\det U^{\{2,4\}}_{[3,4]} \det U^{[2,4]}_{[2,4]} + \det U^{\{2,4\}}_{\{2,4\}} \det U^{\{1\}\cup[3,4]}_{[2,4]} + \det U^{\{2,4\}}_{[2,3]}\det U^{[1,2]\cup\{4\}}_{[2,4]}\right) + \\ + &u_{24}\left(\det U^{\{1,4\}}_{[3,4]} \det U^{[2,4]}_{[2,4]} + \det U^{\{1,4\}}_{\{2,4\}} \det U^{\{1\}\cup[3,4]}_{[2,4]} + \det U^{\{1,4\}}_{[2,3]}\det U^{[1,2]\cup\{4\}}_{[2,4]}\right).
    \end{split}
\end{equation*}

\begin{figure}[htb]
\begin{center}
\includegraphics[scale=0.2]{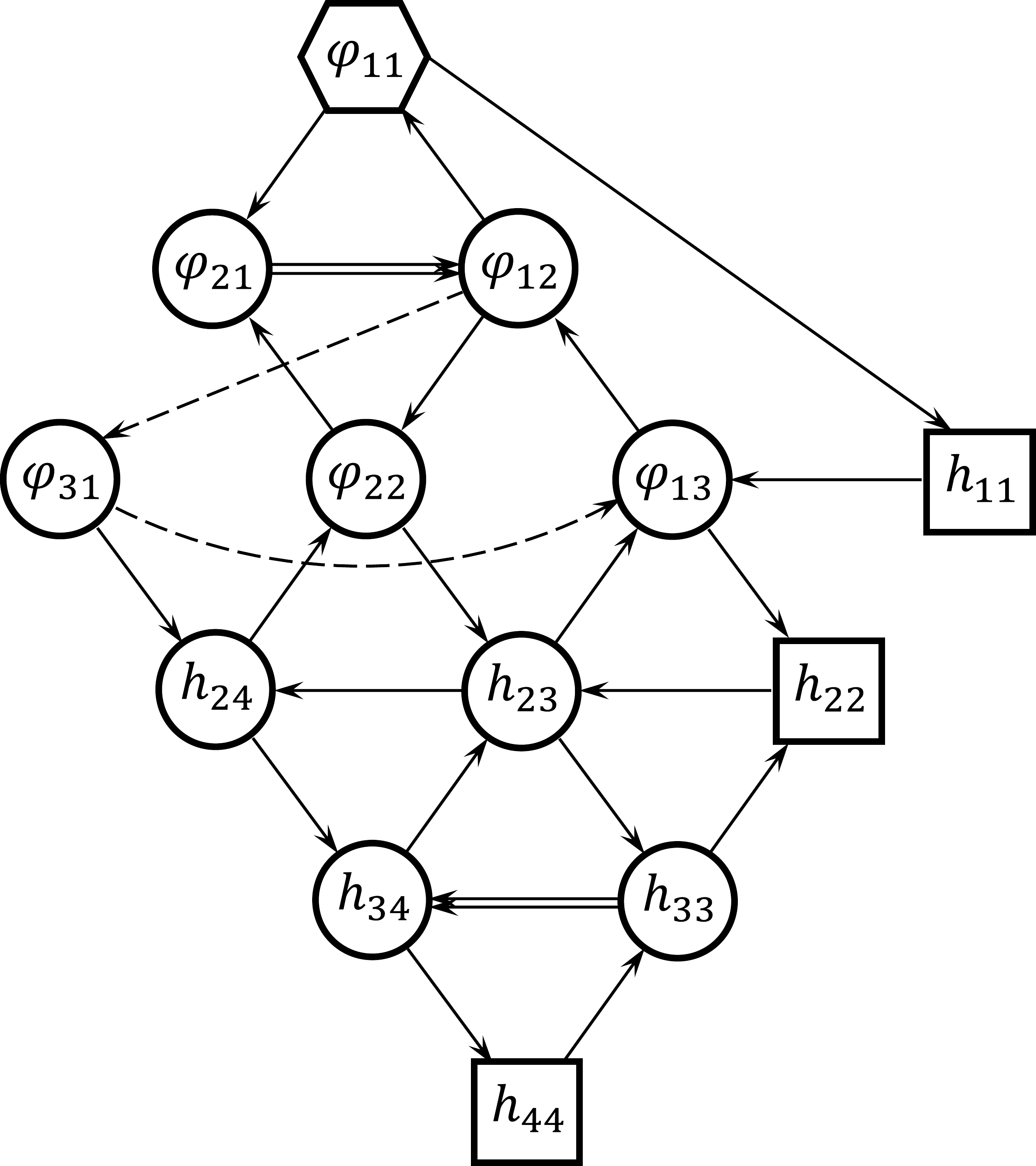}
\end{center}
\caption{The initial quiver for $\gc_h^{\dagger}(\bg,\GL_4)$, $\gamma:i \mapsto i-1$, $i\in\{2,3\}$. }
\label{f:ex_n=4}
\end{figure}

\subsection{\texorpdfstring{Negative orientation in $n=5$, $\gamma:1 \mapsto 4$, $\gamma:2 \mapsto 3$}{Negative orientation in n=5}}

The initial quiver is illustrated in Figure~\ref{f:ex_n=5}. In the initial extended cluster, all cluster and frozen variables are given as in $\gc_h^{\dagger}(\bg_{\std},\GL_5)$ except for the variables $h_{25}$, $h_{24}$, $h_{23}$, $h_{35}$, $h_{34}$, $h_{33}$. These variables are given by the following formulas:
\begin{align*}
&h_{25}(U) = -u_{25} u_{55} - u_{15} u_{54};\\
&h_{24}(U) = \det U^{[4,5]}_{[2,3]}u_{55} + \det U^{[4,5]}_{\{1,3\}}u_{54} + \det U^{[4,5]}_{[1,2]}u_{53};\\
&h_{23}(U) = -\det U^{[3,5]}_{[2,4]}u_{55} - \det U^{[3,5]}_{\{1\}\cup[3,4]}u_{54} - \det U^{[3,5]}_{[1,2]\cup\{4\}}u_{53};\\
&h_{22}(U) = \det U^{[2,5]}_{[2,5]}u_{55} + \det U^{[2,5]}_{\{1\}\cup[3,5]}u_{54} + \det U^{[2,5]}_{[1,2]\cup[4,5]}u_{53};\\
&h_{35}(U) = u_{35} \det U^{[4,5]}_{[4,5]} + u_{25} \det U^{\{3,5\}}_{[4,5]} + u_{15} \det U^{[3,4]}_{[4,5]};\\
&h_{34}(U) = \det U^{[4,5]}_{[3,4]} \det U^{[4,5]}_{[4,5]} +  \det U^{[4,5]}_{\{2,4\}}\det U^{\{3,5\}}_{[4,5]} + \det U^{[4,5]}_{\{1,4\}} \det U^{[3,4]}_{[4,5]};\\
&h_{33}(U) = \det U^{[3,5]}_{[3,5]} \det U^{[4,5]}_{[4,5]} +  \det U^{[3,5]}_{\{2\}\cup [4,5]}\det U^{\{3,5\}}_{[4,5]} + \det U^{[3,5]}_{\{1\}\cup[4,5]} \det U^{[3,4]}_{[4,5]}.
\end{align*}
\begin{figure}[htb]
\begin{center}
\includegraphics[scale=0.2]{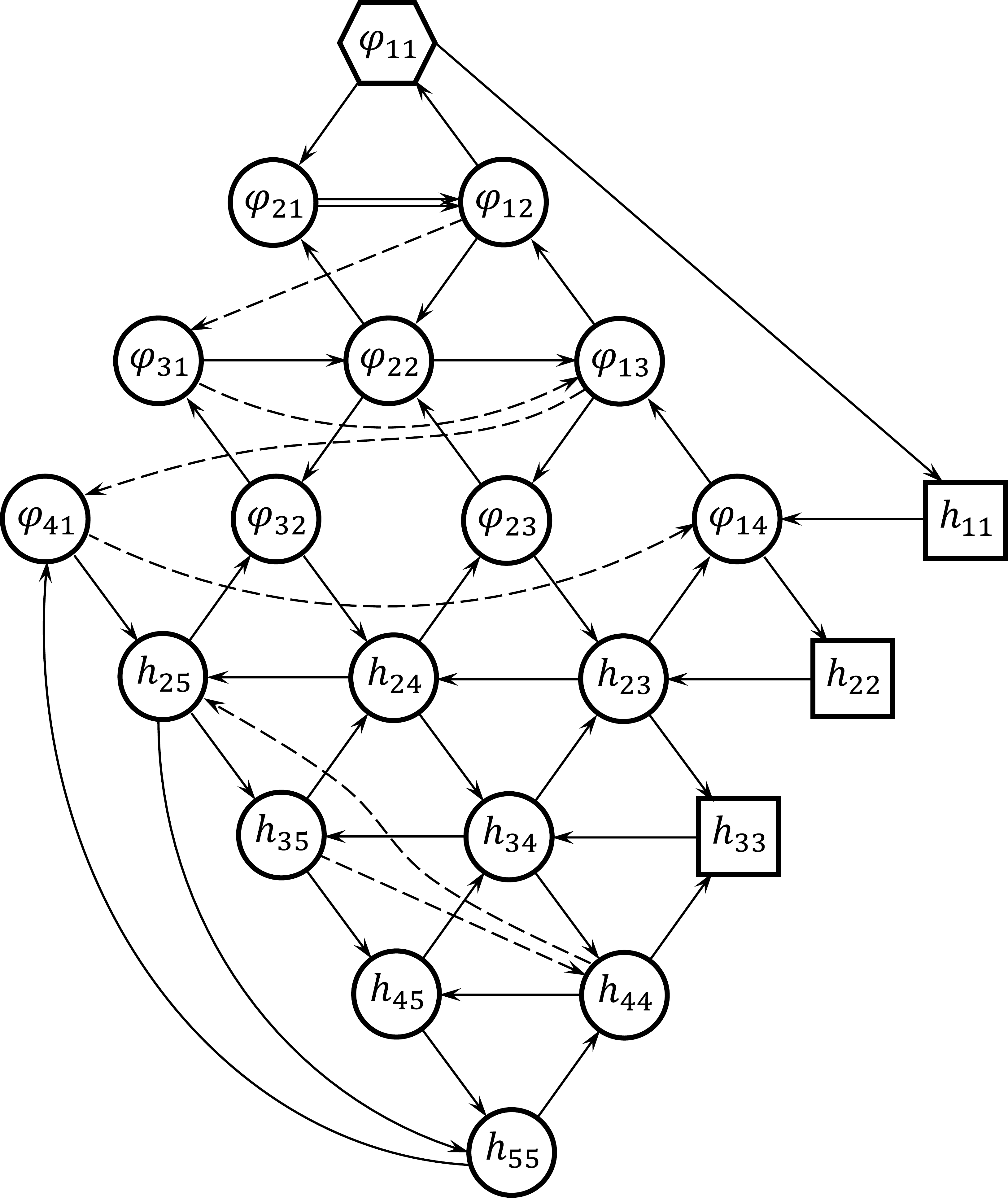}
\end{center}
\caption{The initial quiver for $\gc_h^{\dagger}(\bg,\GL_5)$ with $\gamma:1\mapsto 4$, $\gamma:2\mapsto 3$. }
\label{f:ex_n=5}
\end{figure}

\end{document}